\RequirePackage[l2tabu,orthodox]{nag}
\documentclass[11pt,table]{article}

\usepackage{placeins}	
\usepackage{amsmath,amssymb,amsthm,amsfonts}
\usepackage{multirow}
\usepackage[T1]{fontenc}
\usepackage[utf8]{inputenc}
\usepackage{ragged2e}
\usepackage[justification=centering]{caption}
\usepackage[english]{babel}
\usepackage[numbers,square]{natbib}
\usepackage{makecell}
\usepackage{setspace}
\usepackage{centernot}
\usepackage[dvipsnames]{xcolor}
\usepackage{colortbl}
\usepackage{tikz}
\usetikzlibrary{arrows.meta,calc}
\definecolor{Gray}{gray}{0.9}
\newcolumntype{C}[1]{>{\Centering}m{#1}}
\newcolumntype{L}[1]{>{\raggedright\arraybackslash}m{#1}}
\newcolumntype{R}[1]{>{\raggedleft}m{#1}}
\newcolumntype{a}{>{\columncolor{Gray}}c}
\usepackage[margin=0.5in,bmargin=1.5in]{geometry}
\usepackage[symbol,hang,flushmargin]{footmisc}
\usepackage{diagbox}
\usepackage{longtable}
\usepackage{scalerel}

\usepackage{etoolbox}
\usepackage{xcolor,soul}
\usepackage[colorlinks,urlcolor=blue,citecolor=,linkcolor=,hyperfootnotes=false]{hyperref}
\renewcommand{\thefootnote}{\arabic{footnote}}
\renewcommand{\thempfootnote}{\arabic{mpfootnote}}

\renewcommand{\citenumfont}[1]{\textbf{#1}}
\makeatletter
\def\@biblabel#1{[\textbf{#1}]}
\makeatother

\newcommand{\define}[1]{\emph{#1}}
\newcommand{\from}[1]{\emph{(#1)}}
\newcommand{\mth}[1]{\emph{#1}}
\newcommand{\nit}[1]{\emph{#1}}
\newcommand{\nrm}[1]{\emph{#1}}
\newcommand{\journ}[1]{\emph{#1}}
\newcommand{\case}[1]{\emph{#1}}
\newcommand{\KurZar}[1]{\textbf{#1}}
\newcommand{\inlinequote}[1]{\textit{#1\/}}
\newcommand{\monoid}[1]{\mathbf{#1}}
\newcommand{\semigp}[1]{\mathbf{#1}}
\newcommand{\opset}[1]{\mathcal{#1}}
\newcommand{\topology}[1]{\mathcal{#1}}
\newcommand{\ident}{\mathsf{id}}

\newcommand{\bksp}{\hspace{.16667em plus .08333em}}

\newcommand{\XT}{(X,\mathcal{T})}
\newcommand{\XU}[1]{\textstyle\sum_{i=1}^{#1}\XT}
\newcommand{\N}{\mathbb{N}}
\newcommand{\R}{\mathbb{R}}
\newcommand{\Kf}{K_{\!f}}
\newcommand{\Knum}{\mbox{\emph{K}\nobreakdash-number}}
\newcommand{\Kfnum}{\mbox{$K_{\!f}$\nobreakdash-number}}
\newcommand{\kf}{k_{\hspace{-.5pt}f}}
\newcommand{\knum}{\mbox{\emph{k}\nobreakdash-number}}
\newcommand{\kfnum}{\mbox{$k_{\hspace{-.5pt}f}$\nobreakdash-number}}
\newcommand{\vsp}{\raise7pt\hbox{$\phantom-$}}
\newcommand{\vsy}{\vrule height7pt width0pt depth0pt}
\newcommand{\vsz}{\vrule height11pt width0pt depth0pt}
\newcommand{\hs}{\vrule width40pt height0pt depth0pt}
\newcommand{\hc}{\hspace{-.5pt},}
\newcommand{\hd}{\hspace{-1.5pt},}
\newcommand{\specialpar}[1]{\par\hangindent=5pt \hangafter=1#1\ignorespaces}

\newlength\fx
\newcommand{\hfx}{\hspace{\fx}}
\newlength\fxf
\newcommand{\hfxf}{\hspace{\fxf}}
\newlength\fxb
\newcommand{\hfxb}{\hspace{\fxb}}
\newlength\fxA
\newcommand{\hfxA}{\hspace{\fxA}}

\newlength\cw

\newcommand{\f}{f\hfx}
\newcommand{\ff}{f\hfxf f\hfx}
\newcommand{\ffb}{f\hfxf f\hfxb b}
\newcommand{\ffi}{f\hfxf f\hfx i}
\newcommand{\fff}{f\hfxf f\hfxf f\hfx}
\newcommand{\ei}{f\hfx i}
\newcommand{\fii}{f\hfx ii}
\newcommand{\ie}{i\hfx f\hfx}
\newcommand{\be}{b\hfx f\hfx}
\newcommand{\iee}{i\hfx f\hfxf f\hfx}
\newcommand{\fif}{f\hfx i\hfx f\hfx}
\newcommand{\bif}{bi\hfx f\hfx}
\newcommand{\fb}{f\hfxb b}
\newcommand{\fib}{f\hfx ib}
\newcommand{\fibi}{f\hfx ibi}
\newcommand{\fiba}{f\hfx iba}
\newcommand{\fibf}{f\hfx ib\hfx f\hfx}
\newcommand{\fifb}{f\hfx i\hfx f\hfxb b}
\newcommand{\fifi}{f\hfx i\hfx f\hfx i}
\newcommand{\fbb}{f\hfxb bb}
\newcommand{\fbi}{f\hfxb bi}
\newcommand{\fbf}{f\hfxb b\hfx f\hfx}
\newcommand{\fbg}{f\hfxb bg}
\newcommand{\fbga}{f\hfxb bga}
\newcommand{\fo}{f\hfx o}
\newcommand{\ifo}{i\hfx f\hfx o}
\newcommand{\ifa}{i\hfx f\hfx a}
\newcommand{\ifb}{i\hfx f\hfxb b}
\newcommand{\ifi}{i\hfx f\hfx i}
\newcommand{\fg}{f\hfx g}
\newcommand{\ifg}{i\hfx f\hfx g}
\newcommand{\fa}{f\hfx a}
\newcommand{\fA}{f\hfxA A}
\newcommand{\ffA}{f\hfxf f\hfxA A}
\newcommand{\fiA}{f\hfx iA}
\newcommand{\ifA}{i\hfx f\hfxA A}
\newcommand{\fifA}{f\hfx i\hfx f\hfxA A}
\newcommand{\bifA}{bi\hfx f\hfxA A}
\newcommand{\fbA}{f\hfxb bA}
\newcommand{\fibA}{f\hfx ibA}
\newcommand{\fbiA}{f\hfxb biA}
\newcommand{\fbgA}{f\hfxb bgA}
\newcommand{\foA}{f\hfx oA}
\newcommand{\ifaA}{i\hfx f\hfx aA}

\newcommand{\af}{a\hfx f\hfx}
\newcommand{\aff}{a\hfx f\hfxf f\hfx}
\newcommand{\afi}{a\hfx f\hfx i}
\newcommand{\aif}{ai\hfx f\hfx}
\newcommand{\afif}{a\hfx f\hfx i\hfx f\hfx}
\newcommand{\abif}{abi\hfx f\hfx}
\newcommand{\afb}{a\hfx f\hfxb b}
\newcommand{\afib}{a\hfx f\hfx ib}
\newcommand{\afbi}{a\hfx f\hfxb bi}
\newcommand{\afbg}{a\hfx f\hfxb bg}
\newcommand{\afA}{a\hfx f\hfxA A}
\newcommand{\affA}{a\hfx f\hfxf f\hfxA A}
\newcommand{\afiA}{a\hfx f\hfx iA}
\newcommand{\aifA}{ai\hfx f\hfxA A}
\newcommand{\afifA}{a\hfx f\hfx i\hfx f\hfxA A}
\newcommand{\abifA}{abi\hfx f\hfxA A}
\newcommand{\afbA}{a\hfx f\hfxb bA}
\newcommand{\afibA}{a\hfx f\hfx ibA}
\newcommand{\afbiA}{a\hfx f\hfxb biA}

\newcommand{\su}{\subseteq}
\newcommand{\nsu}{\nsubseteq}
\newcommand{\sm}{\setminus}
\newcommand{\es}{\varnothing}
\newcommand{\cn}{\centernot}
\newcommand{\sym}{\mathop\triangle}
\newcommand{\KF}{\monoid{KF}}
\newcommand{\KFZ}{\monoid{KF^0}}
\newcommand{\KFG}{\monoid{KFG}}
\newcommand{\KFGZ}{\monoid{KFG^0}}
\newcommand{\K}{\monoid{K}}
\newcommand{\KZ}{\monoid{K^0}}
\newcommand{\F}{\semigp{F}}
\newcommand{\G}{\semigp{G}}
\newcommand{\FZ}{\semigp{F^0}}
\newcommand{\FG}{\semigp{FG}}
\newcommand{\GZ}{\semigp{G^0}}
\newcommand{\FGZ}{\semigp{FG^0}}
\newcommand{\Y}{\opset{O}}
\newcommand{\KFA}{\KF\hspace{-1pt}A}
\newcommand{\KFZA}{\KFZ\!A}
\newcommand{\KFGA}{\monoid{KFG}A}
\newcommand{\KFGZA}{\monoid{KFG^0}\!A}
\newcommand{\KA}{\K A}
\newcommand{\KZA}{\KZ\!A}
\newcommand{\FA}{\F A}
\newcommand{\FZA}{\FZ\!A}
\newcommand{\PK}{P_{\,\K}}
\newcommand{\PKF}{P_{\,\KF}}
\newcommand{\PKFG}{P_{\,\KFG}}
\newcommand{\PO}{P_\opset{O}}
\newcommand{\Ord}[1]{\mbox{Ord}(#1)}
\newcommand{\Ext}[1]{\mbox{Ext}(#1)}
\newcommand{\YA}{\Y\hfx A}
\newcommand{\Iopt}{I_{\mbox{\scriptsize opt}}}
\newcommand{\eqy}{\!=\!}
\newcommand{\neqy}{\!\neq\!}
\newcommand{\leqy}{\!\leq\!}
\newcommand{\ify}{\!\iff\!}
\newcommand{\imp}{\!\implies\!}
\newcommand{\eqz}{\,=\,}
\newcommand{\ifz}{\,\iff}
\newcommand{\ds}{\mathop{\dot{\cup}}}
\newcommand{\pn}{\makebox{$\phi$-number}}
\newcommand{\sn}{\makebox{$\psi$-number}}
\newcommand{\h}{\vrule width 2pt height0pt depth0pt}
\newcommand{\hz}{\vrule width 4pt height0pt depth0pt}
\newcommand{\vadj}{\vrule width0pt height10pt}
\newcommand{\p}[1]{\raisebox{-1pt}{$#1$}}
\newcommand{\ph}[1]{\phantom{#1}}
\newcommand{\e}[1]{\raisebox{-2.5pt}{#1}}
\newcommand{\com}{\smash{,}\,}
\newcommand{\y}[1]{\raisebox{-.8pt}{#1}}
\newcommand{\z}[1]{\raisebox{-1pt}{#1}}
\newcommand{\squeezeup}{\vspace{-10pt}}
\newcommand{\mystrut}{\rule[\dimexpr-\dp\strutbox-\arrayrulewidth]{0pt}{%
 \dimexpr\baselineskip+\arrayrulewidth}}
\newcommand{\citecirc}[9]{\node (#1) at (#2\textwidth,#9)
{\scriptsize\hypersetup{citecolor=black}\text{\citealp{#1}}};
\path [draw=none,fill=black,even odd rule] (#2\textwidth,#9)
circle (.3) (#2\textwidth,#9) circle (.28);
\ifodd#4\path [draw=none,fill=red!#5,even odd rule] (#2\textwidth,#9)
circle (#6) (#2\textwidth,#9) circle (.3);\fi
\path [draw=none,fill=#7,even odd rule] (#2\textwidth,#9)
circle (#8) (#2\textwidth,#9) circle (#3);}
\newtheorem{thm}{Theorem}
\newtheorem{lem}{Lemma}
\newtheorem{cor}{Corollary}
\newtheorem*{prop*}{Proposition}
\newtheorem{prop}{Proposition}
\newtheorem{dfn}{Definition}

\begin{document}
\bibliographystyle{plainnatbib}
\setlength{\bibsep}{6pt}

\renewcommand{\abstractname}{\vspace{-\baselineskip}}

\title{Boundary-Border Extensions of the Kuratowski Monoid}
\author{Mark Bowron}

\maketitle

\begin{abstract}\noindent
The Kuratowski monoid $\K$ is generated under operator composition by closure and
complement in a nonempty topological space.
It satisfies $2\leq|\K|\leq14$.
The Gaida\textendash{}Eremenko (or GE) monoid $\KF$ extends $\K$ by adding the boundary
operator.
It satisfies $4\leq|\KF|\leq34$.
We show that when $|\K|<14$ the GE~monoid is determined by $\K$.
When $|\K|=14$ if the interior of the boundary of every subset is clopen, then
$|\KF|=28$.
This defines a new type of topological space we call \define{Kuratowski disconnected}.
Otherwise $|\KF|=34$.
When applied to an arbitrary subset the GE monoid collapses in one of $70$ possible ways.
We investigate how these collapses and $\KF$ interdepend, settling two questions raised
by Gardner and Jackson \cite{2008_gardner_jackson}.
Computer experimentation played a key role in our research.
C source can be found at \url{https://www.mathtransit.com/c.php}.



\medskip\noindent
\emph{Published version} \cite{2024_bowron}:
\url{https://doi.org/10.1016/j.topol.2023.108703}
\end{abstract}

\vspace{8pt}
\begin{minipage}{7in}
\section{Introduction}\label{sec:Intro}
\mbox{Kelley's} classic textbook \cite{1955_kelley} posed the following result
to generations of students as the \textsc{kuratowski closure and complement problem
\cite{1922_kuratowski,1927_zarycki}:\/}
\inlinequote{If $A$ is a subset of a topological space, then at most $14$ sets can be
constructed from $A$ by complementation and closure. There is a subset
of the real numbers \nit(with the usual topology\nit) from which $14$ different
sets can be so constructed.}
It now usually goes by the name \define{\mbox{Kuratowski} closure-complement theorem}
(or \define{problem}, or \define{$14$-set theorem}).
Modern treatments of the subject can be found in Gardner and Jackson
\cite{2008_gardner_jackson} (we refer to it/them as GJ; some familiarity with GJ is
recommended but not necessary to understand the present paper) and \mbox{Sherman's}
informative Monthly article \cite{2010_sherman}.\footnote{For more references see
\mbox{Bowron's} list at \url{https://www.mathtransit.com/cornucopia.php}.}
In this paper we study variants of \mbox{Kuratowski's} theorem involving boundary and
border.

\hspace{1.2em}
\mbox{Anatolii Gol'dberg} posed the closure-complement-boundary problem in a course on
discrete mathematics at the University of Lviv in $1972$ (personal communication,
\mbox{A.~Eremenko}, $6$~Dec~$2019$).
First-year undergraduates \mbox{Yurii Gaida} and \mbox{Alexandre Eremenko} solved it
independently.\footnote{The problem occurred to Gol'dberg while browsing Zarycki's Ph.D.
thesis~\citep{1927_zarycki} to prepare for a talk in his memory. Eremenko heard of it
through word of mouth as he was not attending Gol'dberg's course at the time.}
They showed that an arbitrary subset generates at most $34$ distinct sets and found
all inclusions that hold in general among them.
The \journ{Ukrainian Mathematical Journal} published their work
\cite{1974_gaida_eremenko} (we refer to it/them as GE).\footnote{This was also published
as a Monthly problem in 1986 \cite{1986_buchman_ferrer}.
In $1982$ Soltan \cite{1982_soltan} solved the version that assumes a general closure
operator (see GJ, Figure~2.2 and Section~4.2 for further discussion).
The closure-complement-boundary problem also appears
in~\citep{2021_canilang_cohen_graese_seong,2008_gardner_jackson,1977_kleiner,2002_luo}.
In 1927 Zarycki \cite{1927_zarycki} replaced closure in the closure-complement problem
with various other operators including boundary and border.
GE's result is for
Boolean algebras with a closure, i.e., closure algebras.
It implies the corresponding topological result~\citep{1944_mckinsey_tarski}.
Various boundary operators also appear in
\citep{1991_ahmad_khan_noiri,1941_albuquerque,2020_cohen_johnson_kral_li_soll,
2021_fuenmayor,1964_gabai,1956_gravett,1960_hammer,2000_he_zhang,2021_hoque_modak,
1971_isomichi,2015_jirasek_jiraskova,1966_joseph,2020_khodabocus_sookia,1922_kuratowski,
2019_lei_zhang,1995_moslehian_tavallaii,1980_radulescu,1984_shum,1993_shum,1996_shum,
2017_shum,1968_staley,1939_wallace,1975_yip_shum,2020_yu_lee_au,1927_zarycki}.
For a philosophical take on the boundary concept see Varzi \citep{1997_varzi}.}
\mbox{Eremenko} recently posted their previously unpublished diagram of inclusions at
\url{https://www.math.purdue.edu/~eremenko/dvi/table1.pdf}.
\end{minipage}
\vfill\eject\setcounter{footnote}{3}

\begin{table}[!t]
\renewcommand{\arraystretch}{1.4}
\renewcommand{\tabcolsep}{3pt}
\centering\small
\caption{Glossary of symbols.}
\begin{tabular}{|c|c|}\hline
$\monoid{O}$&$(2^X)^{2^X}$\\\hline
$\KZ,\FZ,\GZ$&see Table~\ref{tab:multiplication}\\\hline
$\KFZ$&$\KZ\cup\FZ$\\\hline
$\KFGZ$&$\KFZ\cup\GZ$\\\hline
$\K$&$\KZ\cup a\KZ$\\\hline
\end{tabular}\hspace{6pt}
\begin{tabular}{|c|c|}\hline
$k(A)$&$|\K A|$\\\hline
$k(\XT)$&$\max\{k(A):A\su X\}$\\\hline
$K(\XT)$&$|\K|$\\\hline
$\kf,\Kf$&use $\KF$ above\\\hline
$a$&complement\\\hline
\end{tabular}\hspace{6pt}
\begin{tabular}{|c|l|}\hline
$b$&closure\\\hline
$d$&dual\\\hline
$\f$&boundary\\\hline
$g$&border\\\hline
$i$&interior\\\hline
\end{tabular}\label{tab:notation}
\end{table}

GJ found connections between the monoid of operators generated under composition
by closure and complement and its action on individual subsets.
As we show in Section~\ref{sec:interplay}, the boundary and border operators help settle
two questions they raised in this area.
Section~\ref{sec:Theorem} addresses \mbox{Gaida} and \mbox{Eremenko's} theorem and lays
the foundation for later sections.
Quotient monoids under operator and set equality are investigated in
Sections~\ref{sec:kf_monoid} and~\ref{sec:kfa}.

\section{The Closure-Complement-Boundary Theorem}\label{sec:Theorem}

Most of the definitions below are used throughout the paper.
The first few only apply when no further information is given.
For the reader's convenience a brief glossary of symbols is provided in
Table~\ref{tab:notation}.

\subsection{Notation and terminology.}

The pair $(X,\topology{T})$ denotes an arbitrary nonempty topological space.
The symbols $A,A_j$ denote arbitrary subsets of $X$.
The set $(2^X)^{2^X}$ of all set operators on $2^X$ is denoted by $\monoid{O}$.
The symbol $\Y$ denotes an arbitrary subset of $\monoid{O}$ and $o,o_j$ arbitrary
set operators in $\monoid{O}$.

The set $\monoid{O}$ forms a monoid under
composition with the identity operator $\ident$ as the identity element and
$(o_1o_2)A:=o_1(o_2A)$ for all $o_1,o_2\in\monoid{O}$ and $A\su X$.
For $p\in\monoid{O}$ we define $p\Y:=\{po:o\in\Y\}$ and
$\Y\hspace{-1pt}p:=\{op:o\in\Y\}$.
For $A\su X$ we define $\Y\hspace{-1pt} A:=\{oA:o\in\Y\}$.

The relations $o_1=o_2$, $o_1\leq o_2$ are \define{satisfied} by $A$ if and only if
$o_1A=o_2A$, $o_1A\su o_2A$, respectively, and by $\XT$ if and only if all $A\su X$
satisfy them.

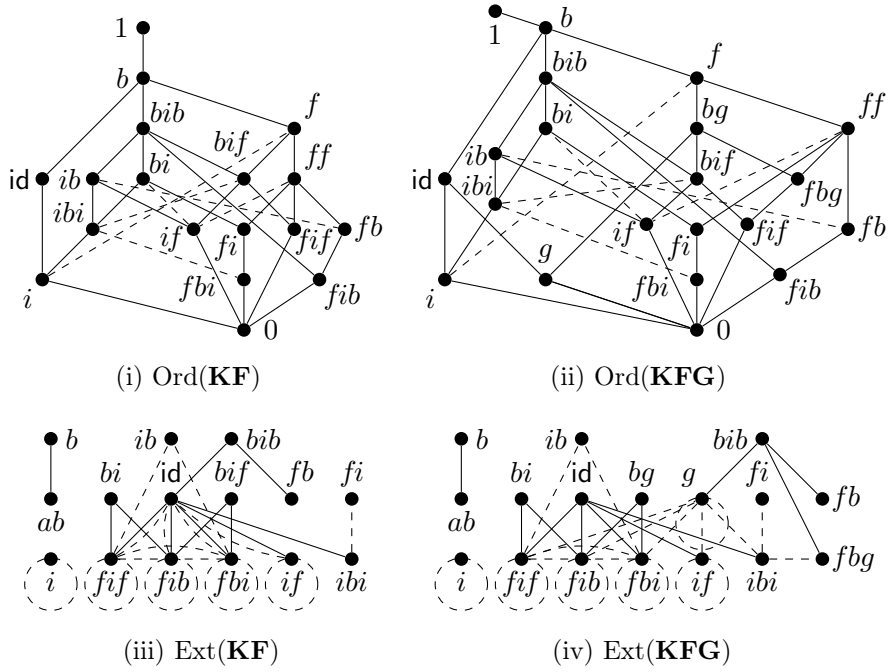
\begin{figure}[!t]
\centering
\begin{tikzpicture}[scale=.67]
	\node (1) at (-1,6) [circle,fill,scale=.5,label={[label distance=-1pt]180:$1$}] {};
	\node (b) at (-1,5) [circle,fill,scale=.5,label={[label distance=-1pt]180:$b$}] {};
	\node (bib) at (-1,4) [circle,fill,scale=.5,label={[label distance=-4.5pt]45:$bib$}] {};
	\node (ib) at (-2,3) [circle,fill,scale=.5,label={[label distance=-2pt]180:$ib$}] {};
	\node (bi) at (-1,3) [circle,fill,scale=.5,label={[label distance=-5pt]45:$bi$}] {};
	\node (ibi) at (-2,2) [circle,fill,scale=.5,label={[label distance=-4pt]135:$ibi$}] {};
	\node (i) at (-3,1) [circle,fill,scale=.5,label={[label distance=-3pt]-135:$i$}] {};
	\node (id) at (-3,3) [circle,fill,scale=.5,label={[label distance=-2pt]180:$\ident$}] {};
	\node (0) at (1,0) [circle,fill,scale=.5,label={[label distance=1pt]0:$0$}] {};
	\node (fib) at (2.5,1) [circle,fill,scale=.5,label={[label distance=-6pt]-45:$\fib$}] {};
	\node (fb) at (3,2) [circle,fill,scale=.5,label={[label distance=-3pt]0:$\fb$}] {};
	\node (ff) at (2,3) [circle,fill,scale=.5,label={[label distance=-3pt]45:$\ff$}] {};
	\node (f) at (2,4) [circle,fill,scale=.5,label={[label distance=-3pt]45:$\f$}] {};
	\node (fbi) at (1,1) [circle,fill,scale=.5,label={[label distance=4pt]180:\raisebox{-15pt}{$\fbi$}}] {};
	\node (fi) at (1,2) [circle,fill,scale=.5,label={[label distance=-5pt]-135:$\ei$}] {};
	\node (fif) at (2,2) [circle,fill,scale=.5,label={[label distance=-8pt]-45:$\fif$}] {};
	\node (bif) at (1,3) [circle,fill,scale=.5,label={[label distance=2pt]90:\hspace{-10pt}$\bif$}] {};
	\node (if) at (0,2) [circle,fill,scale=.5,label={[label distance=-2pt]180:\raisebox{-15pt}{$\ie$}}] {};
	\node (title) at (-.1,-.5) [label={[label distance=-4pt]-90:\small(i)~$\Ord\KF$}] {};
	\draw (0) -- (i) -- (id) -- (b) -- (1);
	\draw (0) -- (fib) -- (bib) -- (b);
	\draw (0) -- (fbi) -- (fi) -- (bi) -- (bib);
	\draw (0) -- (fif) -- (ff) -- (f) -- (b);
	\draw (0) -- (if) -- (bif) -- (f);
	\draw (i) -- (ibi) -- (ib) -- (bib);
	\draw (ibi) -- (bi);
	\draw (fib) -- (fb) -- (ff);
	\draw (fi) -- (ff);
	\draw (fif) -- (bif) -- (bib);
	\draw (if) -- (ib);
	\draw[dashed] (ibi) -- (bif);
	\draw[dashed] (ff) -- (if);
	\draw[dashed] (ib) -- (fb);
	\draw[dashed] (i) -- (f);
	\draw[dashed] (ibi) -- (fbi);
	\draw[dashed] (if) -- (bi);
\end{tikzpicture}
\begin{tikzpicture}[scale=.67]
	\node (1) at (-2,6.33) [circle,fill,scale=.5,label={[label distance=-1pt]-90:$1$}] {};
	\node (b) at (-1,6) [circle,fill,scale=.5,label={[label distance=-1pt]0:\raisebox{8pt}{$b$}}] {};
	\node (bib) at (-1,5) [circle,fill,scale=.5,label={[label distance=-4.5pt]45:$bib$}] {};
	\node (ib) at (-2,3.5) [circle,fill,scale=.5,label={[label distance=-3.5pt]180:\raisebox{-9pt}{$ib$}}] {};
	\node (bi) at (-1,4) [circle,fill,scale=.5,label={[label distance=-5pt]45:$bi$}] {};
	\node (ibi) at (-2,2.5) [circle,fill,scale=.5,label={[label distance=-6pt]180:\raisebox{14pt}{$ibi$}}] {};
	\node (i) at (-3,1) [circle,fill,scale=.5,label={[label distance=-3pt]-135:$i$}] {};
	\node (g) at (-1,1) [circle,fill,scale=.5,label={[label distance=0pt]90:$g$}] {};
	\node (id) at (-3,3) [circle,fill,scale=.5,label={[label distance=-2pt]180:$\ident$}] {};
	\node (0) at (2,0) [circle,fill,scale=.5,label={[label distance=1pt]0:$0$}] {};
	\node (fib) at (3.65,1.1) [circle,fill,scale=.5,label={[label distance=-6pt]-45:$\fib$}] {};
	\node (fb) at (5,2) [circle,fill,scale=.5,label={[label distance=-3pt]0:$\fb$}] {};
	\node (ff) at (5,4) [circle,fill,scale=.5,label={[label distance=-3pt]45:$\ff$}] {};
	\node (fbg) at (4,3) [circle,fill,scale=.5,label={[label distance=-7pt]-45:$\fbg$}] {};
	\node (f) at (2,5) [circle,fill,scale=.5,label={[label distance=-3pt]45:$\f$}] {};
	\node (fbi) at (2,1) [circle,fill,scale=.5,label={[label distance=4pt]180:\raisebox{-15pt}{$\fbi$}}] {};
	\node (fi) at (2,2) [circle,fill,scale=.5,label={[label distance=-5pt]-135:$\ei$}] {};
	\node (fif) at (3,2.1) [circle,fill,scale=.5,label={[label distance=-5pt]0:\raisebox{-14pt}{$\fif$}}] {};
	\node (bg) at (2,4) [circle,fill,scale=.5,label={[label distance=-5.5pt]45:$bg$}] {};
	\node (bif) at (2,3) [circle,fill,scale=.5,label={[label distance=-6.5pt]45:$\bif$}] {};
	\node (if) at (1,2.1) [circle,fill,scale=.5,label={[label distance=-2pt]180:\raisebox{-15pt}{$\ie$}}] {};
	\node (title) at (.84,-.5) [label={[label distance=-4pt]-90:\small(ii)~$\Ord\KFG$}] {};
	\draw (0) -- (i) -- (id) -- (b) -- (1);
	\draw (0) -- (g) -- (id);
	\draw (0) -- (g) -- (bg);
	\draw (fbg) -- (bg);
	\draw (0) -- (fib) -- (bib) -- (b);
	\draw (0) -- (fbi) -- (fi) -- (bi) -- (bib);
	\draw (0) -- (fif) -- (fbg) -- (ff) -- (f) -- (b);
	\draw (0) -- (if) -- (bif) -- (bg) -- (f);
	\draw (i) -- (ibi) -- (ib) -- (bib);
	\draw (ibi) -- (bi);
	\draw (fib) -- (fb) -- (ff);
	\draw (fi) -- (ff);
	\draw (fif) -- (bif) -- (bib);
	\draw (if) -- (ib);
	\draw[dashed] (ibi) -- (bif);
	\draw[dashed] (ff) -- (if);
	\draw[dashed] (ib) -- (fb);
	\draw[dashed] (i) -- (f);
	\draw[dashed] (ibi) -- (fbi);
	\draw[dashed] (if) -- (bi);
\end{tikzpicture}\medskip

\hspace{2pt}
\begin{tikzpicture}[scale=.8]
	\node (b) at (-2,2) [circle,fill,scale=.5,label={[label distance=-1pt]0:$b$}] {};
	\node (ab) at (-2,1) [circle,fill,scale=.5,label={[label distance=-2pt]-90:$ab$}] {};
	\node (bib) at (1,2) [circle,fill,scale=.5,label={[label distance=-1pt]0:$bib$}] {};
	\node (ib) at (0,2) [circle,fill,scale=.5,label={[label distance=-1pt]180:$ib$}] {};
	\node (bi) at (-1,1) [circle,fill,scale=.5,label={[label distance=0pt]90:$bi$}] {};
	\node (ibi) at (3,0) [circle,fill,scale=.5,label={[label distance=-2pt]-90:$ibi$}] {};
	\node (id) at (0,1) [circle,fill,scale=.5,label={[label distance=0pt]90:$\ident$}] {};
	\node (fib) at (0,0) [circle,fill,scale=.5,label={[label distance=-2pt]-90:$\fib$}] {};
	\node (fb) at (2,1) [circle,fill,scale=.5,label={[label distance=-2pt]90:\hspace{6pt}$\fb$}] {};
	\node (fbi) at (1,0) [circle,fill,scale=.5,label={[label distance=-2pt]-90:$\fbi$}] {};
	\node (fi) at (3,1) [circle,fill,scale=.5,label={[label distance=-2pt]90:$\ei$}] {};
	\node (fif) at (-1,0) [circle,fill,scale=.5,label={[label distance=-2pt]-90:\hspace{-1pt}$\fif$}] {};
	\node (bif) at (1,1) [circle,fill,scale=.5,label={[label distance=-2pt]90:$\bif$}] {};
	\node (if) at (2,0) [circle,fill,scale=.5,label={[label distance=-2pt]-90:$\ie$}] {};
	\node (i) at (-2,0) [circle,fill,scale=.5,label={[label distance=-2pt]-90:$i$\vphantom{$\f$}}] {};
	\node (fif_loop) at (-1,-.42) [circle,draw,dashed,inner sep=7pt] {};
	\node (fib_loop) at (0,-.42) [circle,draw,dashed,inner sep=7pt] {};
	\node (fbi_loop) at (1,-.42) [circle,draw,dashed,inner sep=7pt] {};
	\node (if_loop) at (2,-.42) [circle,draw,dashed,inner sep=7pt] {};
	\node (i_loop) at (-2,-.42) [circle,draw,dashed,inner sep=7pt] {};
	\node (title) at (.5,-1) [label={[label distance=0pt]-90:\small(iii)~$\Ext\KF$}] {};
	\draw (ab) -- (b);
	\draw (ibi) -- (id) -- (bib);
	\draw (fbi) -- (id);
	\draw (fib) -- (id);
	\draw (fif) -- (id);
	\draw (if) -- (id);
	\draw (fb) -- (bib);
	\draw (fib) -- (bi);
	\draw (fif) -- (bi);
	\draw (fbi) -- (bif);
	\draw (fib) -- (bif);
	\draw[dashed] (ibi) -- (fi);
	\draw[dashed] (fif) -- (ib);
	\draw[dashed] (fbi) -- (ib);
	\draw[dashed] (fif) -- (fib);
	\draw[dashed] (fbi) -- (fib);
	\path (fif) edge[out=20,in=160,dashed] (fbi);
	\path (fbi) edge[out=150,in=300,dashed] (id);
	\path (fib) edge[out=110,in=250,dashed] (id);
	\path (fif) edge[out=30,in=240,dashed] (id);
	\path (if) edge[out=170,in=320,dashed] (id);
\end{tikzpicture}\hspace{18pt}
\begin{tikzpicture}[scale=.8]
	\node (b) at (-2,2) [circle,fill,scale=.5,label={[label distance=-1pt]0:$b$}] {};
	\node (ab) at (-2,1) [circle,fill,scale=.5,label={[label distance=-2pt]-90:$ab$}] {};
	\node (bib) at (3,2) [circle,fill,scale=.5,label={[label distance=-1pt]180:$bib$}] {};
	\node (ib) at (0,2) [circle,fill,scale=.5,label={[label distance=-1pt]180:$ib$}] {};
	\node (bi) at (-1,1) [circle,fill,scale=.5,label={[label distance=0pt]90:$bi$}] {};
	\node (ibi) at (3,0) [circle,fill,scale=.5,label={[label distance=-2pt]-90:$ibi$}] {};
	\node (id) at (0,1) [circle,fill,scale=.5,label={[label distance=0pt]90:$\ident$}] {};
	\node (fib) at (0,0) [circle,fill,scale=.5,label={[label distance=-2pt]-90:$\fib$}] {};
	\node (fb) at (4,1) [circle,fill,scale=.5,label={[label distance=-3pt]0:$\fb$}] {};
	\node (fbg) at (4,0) [circle,fill,scale=.5,label={[label distance=-3pt]0:$\fbg$}] {};
	\node (fbi) at (1,0) [circle,fill,scale=.5,label={[label distance=-2pt]-90:$\fbi$}] {};
	\node (fi) at (3,1) [circle,fill,scale=.5,label={[label distance=-2pt]90:\hspace{-4pt}$\ei$}] {};
	\node (fif) at (-1,0) [circle,fill,scale=.5,label={[label distance=-2pt]-90:\hspace{-1pt}$\fif$}] {};
	\node (bg) at (1,1) [circle,fill,scale=.5,label={[label distance=-2pt]90:$bg$}] {};
	\node (if) at (2,0) [circle,fill,scale=.5,label={[label distance=-2pt]-90:$\ie$}] {};
	\node (g) at (2,1) [circle,fill,scale=.5,label={[label distance=-2pt]90:\hspace{-10pt}$g$}] {};
	\node (i) at (-2,0) [circle,fill,scale=.5,label={[label distance=-2pt]-90:$i$\vphantom{$\f$}}] {};
	\node (fif_loop) at (-1,-.42) [circle,draw,dashed,inner sep=7pt] {};
	\node (fib_loop) at (0,-.42) [circle,draw,dashed,inner sep=7pt] {};
	\node (fbi_loop) at (1,-.42) [circle,draw,dashed,inner sep=7pt] {};
	\node (if_loop) at (2,-.42) [circle,draw,dashed,inner sep=7pt] {};
	\node (g_loop) at (2,.59) [circle,draw,dashed,inner sep=7pt] {};
	\node (i_loop) at (-2,-.42) [circle,draw,dashed,inner sep=7pt] {};
	\node (title) at (1,-1) [label={[label distance=0pt]-90:\small(iv)~$\Ext\KFG$}] {};
	\draw (ab) -- (b);
	\draw (ibi) -- (id);
	\draw (fbi) -- (id);
	\draw (fib) -- (id);
	\draw (fif) -- (id);
	\draw (if) -- (id);
	\draw (fb) -- (bib);
	\draw (fbg) -- (bib);
	\draw (g) -- (bib);
	\draw (fib) -- (bi);
	\draw (fif) -- (bi);
	\draw (fbi) -- (bg);
	\draw (fib) -- (bg);
	\draw[dashed] (ibi) -- (fi);
	\draw[dashed] (ibi) -- (g);
	\draw[dashed] (ibi) -- (fbg);
	\draw[dashed] (fif) -- (ib);
	\draw[dashed] (fbi) -- (ib);
	\draw[dashed] (fif) -- (fib);
	\draw[dashed] (fbi) -- (fib);
	\draw[dashed] (if) -- (g);
	\draw[dashed] (fif) -- (g);
	\draw[dashed] (fbi) -- (g);
	\draw[dashed] (fib) -- (g);
	\path (fif) edge[out=20,in=160,dashed] (fbi);
\end{tikzpicture}
\caption{$\Ord{\KF},\Ord{\KFG}$ and their extenders (see Definition~\ref{dfn:extender}).\\
\vrule width-27pt height0pt depth0pt
\small\begin{tabular}{l}\\[-10pt]
Solid edges represent $(o_1,o_2)$, dashed $(o_1,ao_2)$.\\
Transitivity only applies to (i) and (ii).\\
All diagrams are up to left duality.\end{tabular}}
\label{fig:kf_monoid}
\end{figure}

The relation $\leq$ is a partial order on $\monoid{O}$.\footnote{As usual, by
\define{partial order} we mean a reflexive, transitive, asymmetric binary relation.}
The join $(\vee)$ and meet $(\wedge)$ of $\Y$ exist and satisfy $\left(\bigvee_{\!\Y}o
\right)\!A\,=\,\bigcup_{\Y}oA$ and $\left(\bigwedge_{\Y}o\right)\!A\,=\,\bigcap_{\Y}oA$.
Let $\Ord{\Y}:=\{(o_1,o_2)\in\Y\times\Y:o_1\leq o_2$ for all $\XT\}$.\footnote{$o_1\leq
o_2\mbox{ for all }\XT\iff o_1\leq o_2\mbox{ in }(\mathbb{R},\mbox{usual topology})$ by
McKinsey\textendash{}Tarski
\cite{2018_bezhanishvili_bezhanishvili_lucero-bryan_van_mill,1944_mckinsey_tarski}.}

Operators $o_1,o_2$ are \define{disjoint} if $o_1\wedge o_2$ is the
\define{zero operator} $0$ defined by $0A:=\varnothing$.\footnote{We use dashed edges to
represent disjointness in poset diagrams (see Figure~\ref{fig:kf_monoid}).}
The \define{one operator} $1$ equals $a\hspace{.5pt}0$ where
$A\longmapsto$\llap{\raisebox{1pt}{$^a$}\hspace{7pt}}\,$X\setminus A$ is the complement
operator on $2^X$.\footnote{The set $\monoid{O}$ is a Boolean algebra under
$\{\vee,\wedge\}$.
For details see Sherman \cite{2010_sherman}.}
The \define{difference} $o_1\setminus o_2$ equals $o_1\wedge ao_2$.

An operator $o$ is \define{open} if $oA$ is open for all $A\su X$.
Closed operators are defined similarly.

Let $b$ be topological closure, $i:=aba$ interior, $\f:=b\wedge ba$ boundary (aka
frontier), and $g:=\ident\wedge ba$ border.
Define
$\KZ:=\{\ident,$\bksp$b,$\bksp$i,$\bksp$bi,$\bksp$ib,$\bksp$bib,$\bksp$ibi\}$,
$\FZ:=\{0,$\bksp$\f,$\bksp$\ie,$\bksp$\fif,$\bksp$\bif,$\bksp$\ff,$\bksp$\fb,$\bksp$\ei,$\bksp$\fbi,$\bksp$\fib\}$,
$\GZ:=\{g,$\bksp$bg,$\bksp$\fbg\}$,
$\KFZ:=\KZ\cup\FZ$,
$\FGZ:=\FZ\cup\GZ$,
$\KFGZ:=\KFZ\cup\GZ$, and
$\semigp{S}:=\semigp{S^0}\cup a\semigp{S^0}$ for each $\semigp{S^0}$ above.

GJ call $\K$ the \define{Kuratowski monoid} and its elements \define{Kuratowski
operators.}
It will be shown below that $\KF$ is generated by $\{a,b,\f\,\}$, hence we call $\KF$ the
\define{Gaida\textendash{}Eremenko monoid} and its elements
\define{Gaida\textendash{}Eremenko operators}.

Kuratowski \cite{1922_kuratowski} proved that the four \define{closure axioms} $b(A_1\cup A_2)=bA_1\cup bA_2$,
$b\leq\ident$, $bb=b$, $b\es=\es$ imply $\K$ is the submonoid of $\monoid{O}$ generated
by $\{a,b\}$.\footnote{This follows from the identities $bababab=bab\iff bibi=bi\iff
ibib=ib$ (right- and left-multiply by $a$).
The literature usually credits Hammer \cite{1960_hammer} with proving that the axioms
$b\es=\es$ and $b(A_1\cup A_2)\su bA_1\cup bA_2$ are not necessary for $\{a,b\}$ to
generate $14$ distinct operators but Chittenden \cite{1941_chittenden_a} proved this
$19$ years before with $bb=b$ weakened to $bbb=b$.}

The axiom $b(A_1\cup A_2)=bA_1\cup bA_2$ implies that every $o\in\KZ$ is isotone.
It follows that we can left- or right-multiply both sides of any equation or inequality
in $\KFG$ by an operator in $\KFG$ (reversing order when necessary) with one exception:
since $\f$ and $g$ are neither isotone nor antitone, we generally cannot left-multiply
inequalities in $\KFG$ by operators in $\FG\sm\{0,1\}$.

In Sections~\ref{sec:kf_monoid}-\ref{sec:kfa} we find all quotients
$\KF/\textstyle{\sim}$ and $\KF/\textstyle{\sim}_A$ where $o_1\sim o_2\iff o_1=o_2$ and
$o_1\sim_Ao_2\iff o_1A=o_2A$.
For brevity we suppress the quotient notation and refer to equivalence classes by their
constituent operators.\footnote{\vrule width3.6pt height0pt depth0ptThis shorthand is
permissible since the projections from $\KF$ onto $\KF/\textstyle{\sim}$ and
$\KF/\textstyle{\sim}_A$ are isotonic monoid homomorphisms.
We \vrule width3.6pt height0pt depth0ptapply it to other monoids as well.}

GJ call $K(\XT):=|\K|$ the \define{K-number of} $\XT$, $k(A):=|\KA|$ the \define{k-number
of} $A$, and $k(\XT):=\max\{k(A):A\su X\}$ the \define{k-number of} $\XT$.
We denote the $\KF$ analogues by $\Kf$ and $\kf$.

Note that $k(\XT)\leq K(\XT)$.
GJ call spaces satisfying $k(\XT)=K(\XT)$ \define{full}.
If $\kf(\XT)=\Kf(\XT)$ we call $\XT$ \define{completely full}.

GJ call spaces with \Knum\ $14$ \define{Kuratowski spaces}.
Subsets with \knum\ $14$ were given the name \define{Kuratowski $14$-set} by Langford
\cite{1971_langford}.\footnote{The first $14$-set appeared in Zarycki \cite{1927_zarycki}
(Kuratowski used two complementary sets with \knum\ $12$ to prove $\Ord\K$).
Since finding a $14$-set is half the closure-complement problem, Zarycki's name could
have been added to it.
In the Closing Remarks we name a certain class of problems
\define{Kuratowski\textendash{}Zarycki} problems.}
We naturally call the $\KF$ analogues \define{Gaida\textendash{}Eremenko spaces} and
\define{$34$-sets}.

Kuratowski \cite{1922_kuratowski} calls operators in $\KZ$ \define{even} and $a\KZ$
\define{odd} due to the number of times ``$a$'' must appear in any product representing
them.
Canilang et~al\text{.} \cite{2021_canilang_cohen_graese_seong} observed that these terms
extend naturally to $\KF$ in nonempty spaces since $\KFZ\cap a\KFZ=\es$ $(o_1(\es)=\es$
and $ao_2(\es)=X$ for all $o_1,o_2\in\KFZ)$.

The \define{duals} of $o,\Y$ are $d(o):=aoa,d(\Y):=\{d(o):o\in\Y\}$.
Respectively, the \define{left dual}, \define{right dual}, and \define{dual} of the
equation $o=p$ (inequality $o\leq p$) are: $ao=ap$ $(ap\leq ao)$, $oa=pa$ $(oa\leq pa)$,
and $aoa=apa$ $(apa\leq aoa)$.

\begin{table}
\centering
\caption{Relations in $\KZA\cup\{\es,X\}$ implied by inclusions in $\KFGA$.}
\begin{tabular}{|c|c|c|}
\multicolumn{1}{c}{$o_1\!\in$}&\multicolumn{1}{c}{$o_2\!\in$}&
\multicolumn{1}{c}{$o_1A\su o_2A\implies$}\\\Xhline{2\arrayrulewidth}
\multirow{2}{*}{\vrule width0pt height14pt$\KZ$}&$a\KZ$\vrule width0pt height12pt&
\multirow{3}{*}{\vrule width0pt height14pt$iA=\es,\,biA=iA,\,A\neq X$}\\\cline{2-2}
&\multirow{2}{*}{\vrule width0pt height14pt$\FGZ$}\vrule width0pt height12pt&\\
\cline{1-1}
$a\FGZ$\vrule width0pt height12pt&&\\\hline
$a\KFGZ$\vrule width0pt height12pt&$\KFGZ$&$bA=X,\,ibA=bA,\,A\neq\es$\\\hline
\end{tabular}\label{tab:kf_akf}
\end{table}

\subsection{Gaida and Eremenko's theorem.}

The following are obvious consequences of the definition of the dual.

\begin{lem}\label{lem:eqs}
\nit{(i)}~\mth{$d(o_1o_2)=d(o_1)d(o_2)$,}\quad
\nit{(ii)}~\mth{$a(d(o))=oa$,}\quad
\nit{(iii)}~\mth{$d(d(o))=o$.}
\end{lem}

\noindent
Parts~(ii)-(iii) imply the operator $a$ can be moved across $o$ in either direction by
changing $o$ to $d(o)$.
The reader is advised to memorize this rule for $o\in\KZ$.

\begin{cor}\label{cor:eqs}
The dual of an operator in \mth{$\KZ$} is obtained by interchanging \mth{$i$} and
\mth{$b$} in its product representation.
Thus \mth{$d(bib)=ibi$,} \mth{$d(\ident)=\ident$,} \mth{$d(ib)=bi$,} etc\nit.
\end{cor}

\begin{proof}
Have $d(b)=i$ and $d(i)=b$.
Apply Lemma~\ref{lem:eqs}(i).
\end{proof}

\begin{lem}\label{lem:f_cancels}
$\be=\f=\fa$.
\end{lem}

\begin{proof}
$\be=b(b\wedge ba)=b\wedge ba=ba\wedge baa=\fa$.
\end{proof}

The Hasse diagram of $\Ord\K$ dates back to Kuratowski \cite{1922_kuratowski}.
Its extension to $\KF$ forms the bedrock of this paper.

\begin{thm}\label{thm:partial_order}\from{Gaida and Eremenko \cite{1974_gaida_eremenko}}
Figure~\nit{\ref{fig:kf_monoid}(i)} and its left dual represent \mth{$\Ord\KF$.}
\end{thm}

\begin{proof}
GE proved all edges hold.
Computer results complete the proof.\footnote{All C programs for this paper can be found
at \url{https://www.mathtransit.com/c.php}.}
\end{proof}

\begin{lem}\label{lem:kf_akf}
All entries in Table~\nit{\ref{tab:kf_akf}} are correct\nit.
\end{lem}

\begin{proof}
Suppose $o_1A\su o_2A$.
If $o_1\in\KZ,o_2\in a\KZ$ then $iA\su o_1A\su o_2A\su aiA$, hence $iA=\es$.
If $o_1\in\KZ,o_2\in\FGZ$ then $iA\su o_1A\su o_2A\su\fA\su aiA$.
If $o_1\in a\FGZ,o_2\in\FGZ$ then $iA\su ao_2A\su ao_1A\su\fA\su aiA$.
If $o_1\in a\KFGZ,o_2\in\KFGZ$ then $abA\su o_1A\su o_2A\su bA$, hence $bA=X$.
The inequations hold since $X\neq\es$.  
\end{proof}

\begin{cor}\label{cor:dual}
\mth{$o_1\cn\leq ao_2$} and \mth{$ao_3\cn\leq o_4$} for all \mth{$o_1,o_2\in\KZ$}
and \mth{$o_3,o_4\in\KFGZ$.}
\end{cor}

\begin{proof}
Table~\ref{tab:kf_akf} implies $o_1X\cn\su ao_2X$ and $ao_3\es\cn\su o_4\es$.
\end{proof}

\begin{cor}\label{cor:maximal_family}
\nit{(i)}~\mth{$|\KFZA|=17\implies|\KFA|=34$,}\quad
\nit{(ii)}~\mth{$|\KFGZA|=20\implies|\KFGA|=40$.}
\end{cor}

\begin{proof}
(i)~Suppose $|\KFZA|=17$.
Then $|(a\KFZ)A|=17$.
Since $|\KZA|=7$, Table~\ref{tab:kf_akf} implies $o_1A\neq o_2A$ for all $o_1\in a\KFZ$,
$o_2\in\KFZ$.
Conclude $|\KFA|=34$.
The proof of (ii) is similar.
\end{proof}

\begin{lem}\label{lem:half_dist}
\nit{(i)}~\mth{$bi(A\cup B)\su b(iA\cup B)$,}\quad
\nit{(ii)}~\mth{$i(bA\cap B)\su ib(A\cap B)$.}
\end{lem}

\begin{proof}
(i)~Since $i(A\cup B)\sm bB$ is an open set contained in $A$, $i(A\cup B)\su iA\cup bB$.
Hence $bi(A\cup B)\su b(iA\cup bB)=biA\cup bB=b(iA\cup B)$.
(ii) is the dual of (i).
\end{proof}
\vfill\eject

\begin{table}
\caption{Values of $o_1o_2$ for $o_1,o_2\in\KFGZ\sm\{0,\ident\}$.}
\footnotesize
\centering
\renewcommand{\arraystretch}{1.142}
\renewcommand{\tabcolsep}{.5pt}
\begin{tabular}{|C{20pt}!{\vrule width1pt}C{14pt}|C{14pt}|C{14pt}|C{14pt}|
C{14pt}|C{14pt}!{\vrule width1pt}C{14pt}|C{14pt}|C{14pt}|C{14pt}|
C{14pt}|C{14pt}|C{14pt}|C{14pt}|C{14pt}!{\vrule width1pt}C{14pt}|C{14pt}|C{14pt}|}
\multicolumn{1}{c}{}&\multicolumn{6}{c}{$\KZ\sm\{\ident\}$}&\multicolumn{9}{c}{$\FZ\sm
\{0\}$}&\multicolumn{3}{c}{$\GZ$}\\\hline
\diagbox[width=20pt, height=20pt, linewidth=1pt]{\raisebox{0.5pt}{\hspace*{3pt}$o_1$}}
{\raisebox{1.5pt}{\hspace{-11pt}$o_2$}}&\e{$b$}&\e{$i$}&\e{$bi$}&\e{$ib$}&\e{$bib$}&
\e{$ibi$}&\e{$\f$}&\e{$\ff$}&\e{$\ei$}&\e{$\fb$}&\e{$\fbi$}&\e{$\fib$}&\e{$\fif$}&
\e{$\bif$}&\e{$\ie$}&\e{$g$}&\e{$bg$}&\e{$\fbg$}\\\Xhline{1pt}
$b$&$b$&$bi$&$bi$&$bib$&$bib$&$bi$&$\f$&$\ff$&$\ei$&$\fb$&
$\fbi$&$\fib$&$\fif$&$\bif$&$\bif$&$bg$&$bg$&$\fbg$\\\hline
$i$&$ib$&$i$&$ibi$&$ib$&$ib$&$ibi$&$\ie$&$0$&$0$&$0$&
$0$&$0$&$0$&$\ie$&$\ie$&$0$&$\ie$&$0$\\\hline
$bi$&$bib$&$bi$&$bi$&$bib$&$bib$&$bi$&$\bif$&$0$&$0$&$0$&
$0$&$0$&$0$&$\bif$&$\bif$&$0$&$\bif$&$0$\\\hline
$ib$&$ib$&$ibi$&$ibi$&$ib$&$ib$&$ibi$&$\ie$&$0$&$0$&$0$&
$0$&$0$&$0$&$\ie$&$\ie$&$\ie$&$\ie$&$0$\\\hline
$bib$&$bib$&$bi$&$bi$&$bib$&$bib$&$bi$&$\bif$&$0$&$0$&$0$&
$0$&$0$&$0$&$\bif$&$\bif$&$\bif$&$\bif$&$0$\\\hline
$ibi$&$ib$&$ibi$&$ibi$&$ib$&$ib$&$ibi$&$\ie$&$0$&$0$&$0$&
$0$&$0$&$0$&$\ie$&$\ie$&$0$&$\ie$&$0$\\\Xhline{1pt}
$\f$&$\fb$&$\ei$&$\fbi$&$\fib$&$\fib$&$\fbi$&$\ff$&$\ff$&$\ei$&$\fb$&
$\fbi$&$\fib$&$\fif$&$\fif$&$\fif$&$bg$&$\fbg$&$\fbg$\\\hline
$\ff$&$\fb$&$\ei$&$\fbi$&$\fib$&$\fib$&$\fbi$&$\ff$&$\ff$&$\ei$&$\fb$&
$\fbi$&$\fib$&$\fif$&$\fif$&$\fif$&$\fbg$&$\fbg$&$\fbg$\\\hline
$\ei$&$\fib$&$\ei$&$\fbi$&$\fib$&$\fib$&$\fbi$&$\fif$&$0$&$0$&$0$&
$0$&$0$&$0$&$\fif$&$\fif$&$0$&$\fif$&$0$\\\hline
$\fb$&$\fb$&$\fbi$&$\fbi$&$\fib$&$\fib$&$\fbi$&$\ff$&$\ff$&$\ei$&$\fb$&
$\fbi$&$\fib$&$\fif$&$\fif$&$\fif$&$\fbg$&$\fbg$&$\fbg$\\\hline
$\fbi$&$\fib$&$\fbi$&$\fbi$&$\fib$&$\fib$&$\fbi$&$\fif$&$0$&$0$&$0$&
$0$&$0$&$0$&$\fif$&$\fif$&$0$&$\fif$&$0$\\\hline
$\fib$&$\fib$&$\fbi$&$\fbi$&$\fib$&$\fib$&$\fbi$&$\fif$&$0$&$0$&$0$&
$0$&$0$&$0$&$\fif$&$\fif$&$\fif$&$\fif$&$0$\\\hline
$\fif$&$0$&$0$&$0$&$0$&$0$&$0$&$0$&$0$&$0$&$0$&
$0$&$0$&$0$&$0$&$0$&$\fif$&$0$&$0$\\\hline
$\bif$&$0$&$0$&$0$&$0$&$0$&$0$&$0$&$0$&$0$&$0$&
$0$&$0$&$0$&$0$&$0$&$\bif$&$0$&$0$\\\hline
$\ie$&$0$&$0$&$0$&$0$&$0$&$0$&$0$&$0$&$0$&$0$&
$0$&$0$&$0$&$0$&$0$&$\ie$&$0$&$0$\\\Xhline{1pt}
$g$&$\fb$&$0$&$\fbi$&$0$&$\fib$&$0$&$\ff$&$\ff$&$\ei$&$\fb$&
$\fbi$&$\fib$&$\fif$&$\fif$&$0$&$g$&$\fbg$&$\fbg$\\\hline
$bg$&$\fb$&$0$&$\fbi$&$0$&$\fib$&$0$&$\ff$&$\ff$&$\ei$&$\fb$&
$\fbi$&$\fib$&$\fif$&$\fif$&$0$&$bg$&$\fbg$&$\fbg$\\\hline
$\fbg$&$\fb$&$0$&$\fbi$&$0$&$\fib$&$0$&$\ff$&$\ff$&$\ei$&$\fb$&
$\fbi$&$\fib$&$\fif$&$\fif$&$0$&$\fbg$&$\fbg$&$\fbg$\\\hline
\end{tabular}\label{tab:multiplication}
\end{table}

\begin{lem}\label{lem:ifo}
\nit{(i)}~\mth{$\ifo=0\mbox{ for each }o\in\{b,i,\f\hspace{1pt}\}$,}\quad
\nit{(ii)}~\mth{$gi=ig=0$,}\quad
\nit{(iii)}~\mth{$\fg=bg$,}\quad
\nit{(iv)}~\mth{$ibg=\ie$.}
\end{lem}

\begin{proof}
(i)~$\ifb=i(bb\wedge bab)=ib\wedge ibab=ib\wedge abib=0$, $\ifi=\ifa i=\ifb a=0$,
and $\iee=\ie\hspace{.6pt}\be=0$.
(ii)~$gi=\ident(i)\wedge ai(i)=i\wedge ai=0$ and
$ig=i(\ident\wedge ai)=i\wedge iai=i\wedge abi=0$.
(iii)~$\fg=bg\wedge bag=bg\wedge aig=bg$.
(iv)~$ibg\leq i(bg\vee bga)=ib(g\vee ga)=i\be=\ie$.
By Lemma~\ref{lem:half_dist}(ii) we have $\ie=i(\f\wedge b)\leq ib(\f\wedge\ident)=ibg$.
\end{proof}

\begin{prop}\label{prop:multiplication}
All entries in Table~\nit{\ref{tab:multiplication}} are correct\nit.
\end{prop}

\begin{proof}
Lemma~\ref{lem:ifo}, $\be=\f$, and $\{b,i,\f,g\}0=0\{b,i,\f,g\}=\{0\}$ imply the zero
entries.
All other nontrivial entries with $o_1\in\KZ$ and $o_2\in\KFZ$ are implied by
$i\bif=ibi\be=i\be=\ie$ and/or idempotence of $b$, $i$, $bi$, and $ib$.
The remaining nontrivial entries with $o_1,o_2\in\KFZ$ follow from $\fff=\ff$,
$\ffb=\fb$, $\fb\ie=\fif$, $\fbi b=\fib$ and their right duals (see GE for proofs).
All entries involving $g$ are easy consequences of Lemma~\ref{lem:ifo} and the results
above.
\end{proof}

\begin{cor}\label{cor:semigroups}
\mth{$\FZ$,} \mth{$\GZ$,} \mth{$\FGZ$,} and \mth{$\F$} are semigroups under
composition\nit.\footnote{In 2013 Plewik and Walczyńska \cite{2013_plewik_walczynska}
showed that $\K$ contains exactly $56$ nonisomorphic and $118$ total subsemigroups.}
\end{cor}

\begin{proof}
This is evident in Table~\ref{tab:multiplication} for the first three sets.
It follows for $\F$ by Corollary~\ref{cor:eqs} and $\fa=\f$.
\end{proof}

Table~\ref{tab:multiplication} implies that $\KFZ$ and $\KFGZ$ are the submonoids
of $\monoid{O}$ generated by $\{b,i,\f\hspace{1pt}\}$ and $\{b,i,\f,g\}$,
respectively.
Since $\fa=\f$ it follows by Corollary~\ref{cor:eqs} that $\KF$ is the submonoid of
$\monoid{O}$ generated by $\{a,b,\f\hspace{1pt}\}$.
This gives us the upper bound of $34$ in the closure-complement-boundary theorem.
We note in passing that $\KFGZ\sm\{\f,\ff,\ei\}$ is the submonoid of $\monoid{O}$
generated by $\{b,i,g\}$.\footnote{In 1914 Hausdorff \cite{1914_hausdorff} introduced
the term \define{border} (also called \define{partial boundary} by Joseph
\cite{1966_joseph} and \define{rim} by Shum \cite{1969_shum,1993_shum}) and showed that
the monoid generated by $\{a,g\}$ can be infinite.
Zarycki \cite{1927_zarycki} showed that $\cdots\leq ghg\leq gh\leq g\leq\ident\leq h\leq
hg\leq hgh\leq\cdots$ where $h=aga$.}

GE's set $A=(0,1)\cup(1,2]\cup(\mathbb{Q}\cap(2,3))\cup\{4\}\cup[5,6]$ satisfies $|\KFZA|
=17$.\footnote{In $\mathbb{R}$:\raisebox{-11pt}
{$\begin{array}[c]{r@{\hspace{3pt}}c@{\hspace{3pt}}l@{\hspace{10pt}}r@{\hspace{3pt}}
c@{\hspace{3pt}}l@{\hspace{10pt}}r@{\hspace{3pt}}c@{\hspace{3pt}}l@{\hspace{10pt}}
r@{\hspace{3pt}}c@{\hspace{3pt}}l@{\hspace{10pt}}r@{\hspace{3pt}}c@{\hspace{3pt}}l}
bA&=&[0,3]\cup\{4\}\cup[5,6],&ibA&=&(0,3)\cup(5,6),&bibA&=&[0,3]\cup[5,6],&\fibA&=&\{0,3,5,6\},&\ifA&=&(2,3),\\
iA&=&(0,1)\cup(1,2)\cup(5,6),&biA&=&[0,2]\cup[5,6],&ibiA&=&(0,2)\cup(5,6),&\fbiA&=&\{0,2,5,6\},&\bifA&=&[2,3],\\
\fA&=&\{0,1\}\cup[2,3]\cup\{4,5,6\},&\ffA&=&\{0,1,2,3,4,5,6\},&\fbA&=&\{0,3,4,5,6\},&\fiA&=&\{0,1,2,5,6\},&\fifA&=&\{2,3\}.\\
\end{array}$}}
Corollary~\ref{cor:maximal_family}(i) implies $|\KFA|=34$, completing the proof of the
theorem.\footnote{The existence of a subset that distinguishes every unequal pair in $\Y$
is guaranteed by the disjoint union construction but authors traditionally supply an
example in $\mathbb{R}$.}

\begin{thm}\label{thm:ccb}\from{Gaida and Eremenko \cite{1974_gaida_eremenko}}
The monoid \mth{$\KF$} of operators generated by \mth{$\{a,b,\f\hspace{1pt}\}$} in a
given topological space has cardinality at most \mth{$34$} and \mth{$|\KFA|=34$} for some
\mth{$A$} in some space\nit.
\end{thm}

\noindent
The set $A$ above also satisfies $|\KFGZA|=20$ for we have
$gA=(\mathbb{Q}\cap(2,3))\cup\{2,4,5,6\}$, $bgA=[2,3]\cup\{4,5,6\}$, and
$\fbgA=\{2,3,4,5,6\}$.
Thus $|\KFGA|=40$ by Corollary~\ref{cor:maximal_family}(ii).

\subsection{The partial orders on \texorpdfstring{\mth{$\KF$}}{KF} and
\texorpdfstring{\mth{$\KFG$}}{KFG}.}

We begin this subsection with several basic results.
The next lemma is clear.

\begin{lem}\label{lem:equivalent_inclusions}
If \mth{$o_1\leq ao_2$} and \mth{$o\leq o_1\vee o_2$} then \mth{$oA\su o_1A\iff oA\su
ao_2A$.}
\end{lem}

\noindent
We verified by computer that the following lemma gives all decompositions of the form
$o_1=o_2\vee o_3$ in $\KFGZ\sm\{0\}$.
It is often combined with Lemma~\ref{lem:equivalent_inclusions}.

\begin{lem}\label{lem:decompositions}
The operands of each join below are disjoint\nit.
\[
\begin{array}[c]{l@{\hspace{4pt}}l@{\hspace{4pt}}l@{\hspace{4pt}}l}
\begin{array}{r@{\hspace{3pt}}l}
\nit{(i)}&b=\f\vee i=\fb\vee ib\nit,\\
\nit{(ii)}&bib=\fib\vee ib\nit,\\
\end{array}&
\begin{array}{r@{\hspace{3pt}}l}
\nit{(iii)}&bi=\ei\vee i=\fbi\vee ibi\nit,\\
\nit{(iv)}&\f=\ff\vee\ie=g\vee ga\nit,\\
\end{array}&
\begin{array}{r@{\hspace{3pt}}l}
\nit{(v)}&\bif=\fif\vee\ie\nit,\\
\nit{(vi)}&bg=\fbg\vee\ie\nit,\\
\end{array}&
\begin{array}{r@{\hspace{3pt}}l}
\nit{(vii)}&ai=g\vee a\nit,\\
\nit{(viii)}&\ident=g\vee i\nit.\\
\end{array}\\
\end{array}
\]
\end{lem}

\begin{proof}
Part~(viii) is clear.
The decomposition $1=i\vee g\vee ga\vee ia$ implies $\f=b\wedge ba=a(ia\vee i)=g\vee ga$,
$ai=g\vee ga\vee ia=g\vee a$, and $b=aia=\f\vee i$.
To get the other decompositions right-multiply $b=\f\vee i$ by $b,$ $ib,$ $i,$ $bi,$
$\f,$ $\ie,$ $bg$.
\end{proof}

\begin{lem}\label{lem:op_relations}
\hspace{18pt}\raisebox{-.5\baselineskip}{$\begin{array}{l@{\hspace{20pt}}l
@{\hspace{20pt}}l@{\hspace{14pt}}l}
\begin{array}{r@{\hspace{3pt}}l}
\nit{(i)}&ib\sm bi=\ie\nit,\\
\nit{(ii)}&bib=\bif\vee bi\nit,\\
\end{array}&
\begin{array}{r@{\hspace{3pt}}l}
\nit{(iii)}&ib\sm ibi=ib\wedge\bif\nit,\\
\nit{(iv)}&bi\sm ib=\fb\wedge\ei=\fib\wedge\fbi\nit,\\
\end{array}&
\begin{array}{r@{\hspace{3pt}}l}
\nit{(v)}&\ff=\fb\vee\ei\nit,\\
\nit{(vi)}&b=bg\vee bi\nit.\\
\end{array}
\end{array}$}
\end{lem}

\begin{proof}
(i)~$ib\sm bi=ib\wedge iba=i(b\wedge ba)=\ie$.
(ii)~Left-multiply $ib\leq (ib\sm bi)\vee bi=\ie\vee bi$ by $b$.
(iii)~By (ii), $ib\wedge biba=abia\wedge(b\ifa\vee bia)=ib\wedge\bif$.
(iv)~By $\Ord\KZ$, $ai=bai$, and Lemmas~\ref{lem:eqs}(ii) and
\ref{lem:decompositions}(ii)-(iii), $\fb\wedge\ei=(b\wedge aib)\wedge(bi\wedge ai)=bi
\wedge aib=(bib\wedge aib)\wedge(bi\wedge aibi)$ $=$ $\fib\wedge\fbi$.
(v)~By (i) and Lemmas~\ref{lem:equivalent_inclusions} and~\ref{lem:decompositions}, $\ff=
\f\sm\ie=\f\wedge(aib\vee bi)=\fb\vee\ei$.
(vi)~$b=b(\ident)=b(g\vee i)=bg\vee bi$.
\end{proof}

\begin{prop}\label{prop:order_implications}
Let \mth{$(\pi_1,\pi_2,\pi_3)$} be any permutation of \mth{$\fib,$} \mth{$\fif,$}
\mth{$\fbi$.}
Then \nit{(i)~$\pi_1\sm\pi_3=\pi_2\sm\pi_3$.} Also\nit,
\[
\begin{array}[c]{@{}l@{\hspace{8pt}}l}
\begin{array}{@{}r@{\hspace{3pt}}l}
\nit{(ii)}&\fbi\sm\fib=(ib\wedge bi)\sm ibi\nit,\\
\nit{(iii)}&\fib\sm\fbi=bib\sm(ib\vee bi)\nit,\\
\nit{(iv)}&\pi_1\leq\pi_2\vee\pi_3\nit,\\
\end{array}&
\begin{array}{r@{\hspace{3pt}}l}
\nit{(v)}&\pi_1A\su\pi_3A\iff\pi_2A\su\pi_3A\nit,\\
\nit{(vi)}&\pi_1A=\es\implies\pi_2A=\pi_3A\nit,\\
\nit{(vii)}&\pi_1A\su a\pi_3A\implies\pi_1A\su\pi_2A\nit.\\
\end{array}
\end{array}
\]
\end{prop}

\begin{proof}
(i)--(iii)~Lemma~\ref{lem:decompositions} implies $\fif\sm\fib=\fif\wedge ib=(\bif\sm\ie
\hspace{1pt})\wedge ib=\bif\wedge(ib\wedge bi)\leq(ib\wedge bi)\sm ibi=ib\wedge\fbi=\fbi
\sm\fib$.
Conversely, Lemma~\ref{lem:op_relations}(ii) implies $(ib\wedge bi)\sm ibi\leq\bif
\wedge(ib\wedge bi)$.
This proves (ii) and the case $\pi_3=\fib$ in (i).
(iii) and $\pi_3=\fbi$ follow by right duality.
Suppose $\pi_3=\fif$.
The proven cases imply $\fbi\sm\fib\leq\fif$ and $\fib\sm\fbi\leq\fif$.
Thus $\fbi\sm\fif\leq\fib$ and $\fib\sm\fif\leq\fbi$, giving us the result.
(iv)~By (i), $\pi_1=(\pi_1\wedge\pi_2)\vee
(\pi_1\sm\pi_2)=(\pi_1\wedge\pi_2)\vee(\pi_3\sm\pi_2)\leq\pi_2\vee\pi_3$.
Clearly $\mbox{(i)}\imp\mbox{(v)}\imp\mbox{(vi)}$ and $\mbox{(iv)}\imp\mbox{(vii)}$.
\end{proof}

\noindent
Parts~(ii) and~(iii) imply the following corollary.

\begin{cor}\label{cor:equivalences}
\nit{(i)~$biA\cap ibA\su ibiA\iff\fbiA\su\fibA$,}\quad
\nit{(ii)~$bibA\su biA\cup ibA\iff\fibA\su\fbiA$.}
\end{cor}

\begin{lem}\label{lem:subset_implications}The following hold for all \mth{$A\su X$}\nit.
\hspace{-224pt}\raisebox{-2.5\baselineskip}{$\begin{array}[c]{r@{\hspace{3pt}}lcr@{\hspace{3pt}}l}
\nit{(i)}&\ifA\su A\iff\ifA=\es\iff\ifA\su aA\nit,&\quad&\nit{(iv)}&oA\su biA\implies oA\su\fbiA\mbox{ for }o\in\{\fb,\fib,\bif,\fif\hspace{.6pt}\}\nit,\\
\nit{(ii)}&\fbA\su bibA\iff bibA=bA\nit,&&\nit{(v)}&oA\su\fbA\implies oA\su\fibA\mbox{ for }o\in\{\ei,\fbi,\bif,\fif\hspace{.6pt}\}\nit,\\
\nit{(iii)}&oA\su ibA\implies oA\su\fiA\mbox{ for }o\in\{\ff,\fif\hspace{.6pt}\}\nit,&&\nit{(vi)}&oA\su \bifA\implies oA\su\fifA\mbox{ for }o\in\{\ff,\ei,\fb,\fbi,\fib\}\nit.\\
\end{array}$}
\end{lem}

\begin{proof}
(i)~Left-multiply by $i$ to get $\ifA\su A\implies\ifA\su iA\implies\ifA=\es\implies\ifA
\su A$ then substitute $aA$ for $A$.
(ii)~Apply $b=\fb\vee ib$.
(iii)~Since $o\wedge(ib\sm bi)=0$ we get $oA\su biA$.
The result follows by Lemma~\ref{lem:equivalent_inclusions}.
(iv)--(vi)~Apply Lemma~\ref{lem:equivalent_inclusions}.
\end{proof}

\begin{dfn}\label{dfn:extender}
Let \mth{$P$} be a partial order on a set \mth{$S$.}
An ordered pair \mth{$(x,y)$} in \mth{$(S\times S)\sm P$} is a \define{potential cover}
of \mth{$P$} if and only if \mth{$P\cup\{(x,y)\}$} is a partial order.
We call the set of potential covers the \define{extender} of \mth{$P$,} denoted
\mth{$\Ext P$.}
For brevity we define \mth{$\Ext\Y:=\Ext{\Ord\Y}$} for each \mth{$\Y\su\monoid{O}$.}
\end{dfn}

\begin{lem}\label{lem:poset_result}
Suppose \mth{$P\su Q$} are partial orders on a finite set \mth{$S$.}
Then \[P=Q\iff\nit{Ext}(P)\cap Q=\es.\]
\end{lem}

\begin{proof}
This holds since the poset under set inclusion of all partial orders on $S$ is graded by
cardinality (see Lemma~2.1 in Culberson and Rawlins \cite{1988_culberson_rawlins}).
\end{proof}

\noindent
The next proposition implies GE's $34$-set in $\mathbb{R}$ satisfies $\Ord\KF$.
It is sharp up to Proposition~\ref{prop:order_implications}(v).

\begin{prop}\label{prop:partial_order}
A GE \mth{$34$}-set \mth{$A$} satisfies \mth{$\Ord\KF$} if and only if
\mth{$o_1A \nsu o_2A$} for all \mth{$(o_1,o_2)\in\{(\fib,\fif)$,} \mth{$(\fif,\fbi)$},
\mth{$(\fbi,\fib)\}\cup(\{\fib,\fif,\fbi\}\times\{\ident,a\})$.}
\end{prop}

\begin{proof}
Suppose $A$ is a GE $34$-set.
The ``only if'' holds by Theorem~\ref{thm:partial_order}.
Conversely, suppose $A$ satisfies the condition.
By Figure~\ref{fig:kf_monoid}(iii), up to left duality we have $\Ext\KF=
\bigcup_{j=1}^7S_j$ where
$S_1=\{(o,ao):o\in\{i,\ie,\fib,\fif,\fbi\}\}\cup\{(ab,b)\}$,
$S_2=\{(\ident,bib),$ $(ibi,\ident)\}$,
$S_3=\{(\ie,\ident),$ $(\ie,a),$ $(\fb,bib)\}$,
$S_4=\{(\fif,aib),$ $(\fbi,aib),$ $(\ei,aibi)\}$,
$S_5=\{(\fib,\afbi),$ $(\fif,\afib),$ $(\fbi,\afif\hspace{1pt})\}$,
$S_6=\{(\fib,bi),$ $(\fif,bi),$ $(\fib,\bif\hspace{1pt}),$ $(\fbi,\bif\hspace{1pt})\}$, and
$S_7=\{\fib,\fif,\fbi\}\times\{\ident,a\}$.
Let $(o_1,o_2)\in\Ext\KF$.
Claim $o_1A\nsu o_2A$.
This holds for $S_1$ since $o_1A\su o_2A\implies o_1A=\es$.
It holds for $S_2$ since $A\su bibA\implies bibA=bA$ (left-multiply by $b$) and $ibiA\su
A\implies ibiA=iA$ (left-multiply by $i$).
It holds for $S_3$ by Lemma~\ref{lem:subset_implications}(i)--(ii), $S_4$ by
Lemmas~\ref{lem:equivalent_inclusions}-\ref{lem:decompositions},
Proposition~\ref{prop:order_implications}(v), and $\Ord\KF$, $S_5$ by
Proposition~\ref{prop:order_implications}(vii), and $S_6$ by
Lemma~\ref{lem:subset_implications} parts (iv),(vi) and
Proposition~\ref{prop:order_implications}(v).
The hypothesis covers $S_7$.
The result follows by Theorem~\ref{thm:partial_order} and Lemma~\ref{lem:poset_result}.
\end{proof}

\begin{prop}\label{prop:kfgz_ordering}
Figure~\nit{\ref{fig:kf_monoid}(ii)} and its left dual represent \mth{$\Ord\KFG$.}
\end{prop}

\begin{proof}
All three edges with $g$ as an endpoint are clear.
Left-multiply $\ie=ibg\leq bg$ and $g\leq f$ by $b$ to get $\bif\leq bg\leq f$.
Right-multiply $\fib\leq\fb\leq b$ by $g$ to get $\fif\leq\fbg\leq bg$.
Have $\fbg=bg\sm ibg\leq\f\sm\ie=\ff$.
The other edges hold by Theorem~\ref{thm:partial_order}.
We verified by computer that no further edges hold.
\end{proof}

\subsection{Basic relationships.}

This subsection lays the foundation for Section~\ref{sec:kfa}.

\begin{lem}\label{lem:equations_kf}
In Figure~\nit{\ref{fig:kf_monoid}(i),} each of the six edge equations in
\mth{$(\KZ\sm\{\ident\})A$} is equivalent to one of the nine edge equations in
\mth{$(\FZ\sm\{\f,\bif\hspace{.6pt}\})A$.}
Among the latter\nit, only \mth{$\fifA=\es$} fails to imply any equation in \mth{$\KZA$.}
Specifically\nit, we have

\bigskip\noindent
\makebox[36pt][r]{\nit{(i)}}~\mth{$bibA=biA\iff ibiA=ibA\iff\ifA=\es$,}\\
\makebox[36pt][r]{\nit{(ii)}}~\mth{$bibA=ibA\iff\fibA=\es$,}\\
\makebox[36pt][r]{\nit{(iii)}}~\mth{$ibiA=biA\iff\fbiA=\es$,}\\
\makebox[36pt][r]{\nit{(iv)}}~\mth{$bA=ibA\implies\ffA=\fiA\implies\fbA=\fibA\iff bibA=bA$,}\\
\makebox[36pt][r]{\nit{(v)}}~\mth{$iA=biA\implies\ffA=\fbA\implies\fiA=\fbiA\iff ibiA=iA$,}\\
\makebox[36pt][r]{\nit{(vi)}}~\mth{$\ffA=\fifA\implies(bibA=bA\mbox{ and }ibiA=iA)$.}

\bigskip\noindent
The following five non-edge equations in \mth{$\{\ei,\fb,\fib,\fbi,\fif\hspace{.6pt}\}A$}
each imply at least one edge equation in \mth{$(\KZ\sm\{\ident\})A$.}

\noindent
\makebox[36pt][r]{\nit{(vii)}}\vrule width0pt height16pt depth0pt~\mth{$\fbA=(\fbiA
\mbox{ or }\fifA)\implies bibA=bA$,}\\
\makebox[36pt][r]{\nit{(viii)}}~\mth{$\fiA=(\fibA\mbox{ or }\fifA)\implies ibiA=iA$,}\\
\makebox[36pt][r]{\nit{(ix)}}~\mth{$\fbA=\fiA\implies(bibA=bA\mbox{ and }ibiA=iA)$.}

\medskip\noindent
We also have

\medskip\noindent
\makebox[36pt][r]{\nit{(x)}}~\mth{$\ffA\neq X$,}\\
\makebox[36pt][r]{\nit{(xi)}}~\mth{$bA=X\iff iA=\afA\iff biA=\aifA\iff ibiA=\abifA$,}\\
\makebox[36pt][r]{\nit{(xii)}}~\mth{$\fbgA=\es\implies\fbgA=\fifA\iff bgA=\bifA\implies bibA=bA$.}
\end{lem}

\begin{proof}
(i)~Left-multiply $bibA=biA$ by $i$ and $ibiA=ibA$ by $b$ to get the first equivalence.
The second one holds since $\ifA=\es\implies ibA\su biA\implies ibA=ibiA\imp\ifA\su
ibA\su biA\imp\ifA=\es$.
(ii)~Apply $bib=\fib\vee ib$.
(iii) is the dual of (ii).
(iv)~The equivalence holds by Lemma~\ref{lem:decompositions}(i)--(ii).
The implications hold by Lemmas~\ref{lem:decompositions}(i),
\ref{lem:op_relations}(iii), and~\ref{lem:subset_implications}(ii).
(v) is the dual of (iv).
(vi)~Have $\fA=\bifA$ by Lemma~\ref{lem:decompositions}(iv)--(v).
It follows that $bA\sm bibA\su\fbA\su\bifA\su bibA$ and
$ibiA\sm iA\su\fiA\su\bifA\su aibiA$.
(vii)~Apply Lemma~\ref{lem:subset_implications}(ii).
(viii) is the dual of (vii).
(ix)~$\ffA=\fiA=\fbA$ by Lemma~\ref{lem:op_relations}(iii).
Apply (iv) and (v).
(x)~$\ffA=X$ implies the contradiction $\es=i\ff A=iX=X$.
(xi)~Each implies $bA=X$ since $o_1A=ao_2A\implies X=o_1A\cup o_2A\su bA$ for all $o_1,o_2
\in\KFZ$.
Apply $b=\f\vee i$ to get $bA= X\implies iA=\afA$.
Left-multiply by $b$ and $i$ to complete the proof.
(xii)~The equivalence holds by Lemma~\ref{lem:decompositions}(v)--(vi).
The implications hold by $\Ord\KFG$ and
Lemma~\ref{lem:op_relations} parts (i) and (v).
\end{proof}

\begin{dfn}\label{dfn:collapses}
We call any set \mth{$C\su\{\{o_1,o_2\}:o_1,o_2\in\Y\}$} a \define{collapse} of
\mth{$\Y$.}
A subset \mth{$A$} \define{satisfies} \mth{$C$ [}\define{on} \mth{$\Y$]}\footnote{We
often leave out ``on $\Y$\hspace{.8pt}'' when the context is clear.} if and only if for
all \mth{$o_1,o_2\in\Y$,} \mth{$o_1A=o_2A\iff\{o_1,o_2\}\in C$.}
For \mth{$\XT$} to \define{satisfy} \mth{$C$,} replace \mth{$o_1A=o_2A$} with
\mth{$o_1=o_2$.}
Satisfaction of a partial order \mth{$P$} on \mth{$\Y$} by a space or subset is defined
similarly.

If some subset in some topological space satisfies a collapse \mth{$C$} on \mth{$\Y$} we
call \mth{$C$} a \define{local collapse of} \mth{$\Y$.}
If some space satisfies \mth{$C$} on \mth{$\Y$} we call \mth{$C$} a \define{global
collapse of} \mth{$\Y$.}
Local and global orderings are defined similarly.\footnote{Previous authors have used
various names for collapses and orderings.
The term ``collapse'' doubtlessly derives from the graphical collapsing that occurs when
comparable nodes merge in a Hasse diagram.
GJ refer to local collapses specifically as \define{equivalence patterns}.
McCluskey et al\text{.} \cite{2007_mccluskey_mcintyre_watson} call local orderings
\define{properties} and global orderings \define{universals}.
Staiger and Wagner \cite{2010_staiger_wagner} call local orderings of \mth{$\KZ$}
\define{Kuratowski lattices}.}
We call a collapse \define{closed} if it contains the pair \mth{$\{\ident,b\}$} and
\define{open} if it contains \mth{$\{\ident,i\}$.}
\end{dfn}

Every global collapse or partial order is local but the converse is not true in general.

\begin{lem}\label{lem:fza_collapse}
The collapse of \mth{$\FZ$} that \mth{$A$} satisfies is determined by the collapse
\mth{$\{\{o_1,o_2\}\in C_1\cup C_2\cup C_3: o_1A=o_2A\}$} where \mth{$C_1$} is the set of
all nine edge pairs in \mth{$\FZ\sm\{\f,\bif\hspace{.6pt}\}$} in
Figure~\nit{\ref{fig:kf_monoid}(i),} \mth{$C_2=\{\{\fb,\ei\},$} \mth{$\{\fb,\fbi\},$}
\mth{$\{\fb,\fif\hspace{.6pt}\},$} \mth{$\{\ei,\fib\},$}
\mth{$\{\ei,\fif\hspace{.6pt}\}\}$,} and \mth{$C_3=\{\{\fib,\fbi\},$}
\mth{$\{\fbi,\fif\hspace{.6pt}\},$} \mth{$\{\fif,\fib\}\}$.}
\end{lem}

\begin{proof}
Since $\ie\wedge\ff=0$ it follows that every equation in $\FZA$ involving an operator
in $\{0,\ie,\bif,\f\hspace{1pt}\}$ is equivalent to some combination of edge equations in
$\FZA$.
Every such combination is determined by the nine edge equations in
$(\FZ\sm\{\bif,\f\hspace{1pt}\})A$ since $\fA=\bifA\iff\ffA=\fifA$, $\fA=\ffA\iff\bifA=
\fifA\iff\ifA=\es$, and $\bifA=\ifA\iff\fifA=\es$ by
Lemma~\ref{lem:decompositions}(iv)--(v).
Every incomparable pair in $\FZ\sm\{0,\ie,\bif,\f\hspace{1pt}\}$ is in $C_2\cup C_3$.
\end{proof}

Kleiner \cite{1977_kleiner} published a follow-up to GE in which the next two results
were stated without proof.
The first is an equivalent restatement of the original.

\begin{table}
\caption{The six space types based on $\K$
(\cite{1982_chagrov},\,\cite{2008_gardner_jackson}) where
$S=\{o\,\in\,\{ib,bi,b\}{:}\ o{\eqz}bib\}$.}
\centering\small
\begin{tabular}{r|c|c|c|c|c|c|}
\multicolumn{1}{c}{}&\multicolumn{1}{c}{Kuratowski}&\multicolumn{1}{c}{ED}&
\multicolumn{1}{c}{OU}&\multicolumn{1}{c}{EO}&\multicolumn{1}{c}{partition}&
\multicolumn{1}{c}{discrete}\\\Xcline{2-7}{2\arrayrulewidth}\vsz
$K(\XT)$&\vsz$14$&$10$&$10$&$8$&$6$&$2$\\\cline{2-7}
$S$&\vsz$\es$&$\{ib\}$&$\{bi\}$&$\{ib,bi\}$&$\{ib,b\}$&$\{ib,bi,b\}$\\\cline{2-7}
\end{tabular}
\label{tab:six_Kuratowski_monoids}
\end{table}

\begin{prop}\label{prop:kleiner}\from{Kleiner \cite{1977_kleiner}}
A Kuratowski \nit{$14$-}set \nit{$A$} is a GE \mth{$34$-}set if and only if at most one
of the following inclusions holds\nit: \mth{$\fbiA\su\fibA$,} \mth{$\fibA\su\fifA$,}
\mth{$\fifA\su\fbiA$.}
\end{prop}

\begin{proof}Let $\Y=\{\fbi,\fif,\fib\}$. If two of the inclusions hold then $|\YA|<3$
by Proposition~\ref{prop:order_implications}(v).
Conversely suppose $|\YA|=3$.
Since $|\KZA|=7$, Proposition~\ref{prop:order_implications}(vi) and
Lemmas~\ref{lem:equations_kf}-\ref{lem:fza_collapse} imply $|\FZA|=10$.
By Corollary~\ref{cor:maximal_family} we conclude $|\KFA|=34$.
\end{proof}

\begin{cor}\label{cor:kleiner}\from{Kleiner \cite{1977_kleiner}}
If \mth{$A$} is a Kuratowski \mth{$14$-}set and \mth{$biA\su ibA$} then \mth{$A$} is a GE
\mth{$34$-}set\nit.
\end{cor}

\begin{proof}
Apply Corollary~\ref{cor:equivalences} and Proposition~\ref{prop:kleiner} to
$ibA\cap biA=biA\nsu ibiA$ and $bibA\nsu ibA=ibA\cup biA$.
\end{proof}

\noindent
Kleiner's paper concludes by exhibiting a GE $34$-set $A$ in $\mathbb{R}$ with the novel
property that all $34$ sets in $\KFA$ are pairwise nonhomeomorphic.

\section{The Gaida\textendash{}Eremenko Monoid}\label{sec:kf_monoid}

In this section we find all global collapses and orderings of $\KF$ and $\KFG$.

\subsection{The monoid \texorpdfstring{\mth{$\KF$}}{KF}.}

GJ proved (see Theorems~2.1 and~2.10) that $\KZ$ has:
(i)~six global collapses, each of which extends uniquely to $\K$, and
(ii)~$30$ local collapses, one of which extends to two local collapses of $\K$ and the
other $29$ of which extend uniquely to $\K$ (see Tables~\ref{tab:six_Kuratowski_monoids}
and~\ref{tab:subset_types}).

It is surprising the local collapses of $\KZ$ were not studied exhaustively until the
$2000$s.
Chapman \cite{1962_chapman_a} studied a coarser equivalence on $2^X$.
In the widely-read \journ{American Mathematical Monthly}, Chapman \cite{1962_chapman_b}
gave necessary and sufficient conditions for each of the $\binom{7}{2}$ possible
equations in $\KZA$.\footnote{Chapman wrote both papers as an undergraduate student
(personal communication, H.~W.~Gould, $1999$).}
With hindsight, the question jumps out: Which combinations of these equations occur?
It evidently escaped notice for over $40$ years.
In the next section we show that they extend to $62$ local collapses of $\KFZ$ and $70$
local collapses of $\KF$.

McCluskey et al\text{.} \cite{2007_mccluskey_mcintyre_watson} proved that $\KZ$ has six
global and $49$ local orderings.
The latter result was also found independently by Staiger and Wagner
\cite{2010_staiger_wagner}.\footnote{Their unpublished $2010$ conference paper was
uploaded to ResearchGate in late $2021$ by Staiger.
It appears that it and the papers by McCluskey et~al\text{.} $(2007)$ and GJ $(2008)$
were all written independently of one another.}

It turns out that $\K$, $\KFZ$, and $\KF$ have $66$, $274$, and $496$ local orderings,
respectively (see Theorem~\ref{thm:kfa_inclusions}).

Calling $|\K|$ the \define{Kuratowski number} of $\XT$, Chagrov \cite{1982_chagrov} was
the first author to exhibit the six global collapses of $\K$.\footnote{Chagrov omitted
proof that no further ones exist (for the proof see GJ).
Though there is no citation to confirm it,  Kleiner \cite{1977_kleiner} may have inspired
Chagrov's work in this area.
Langford \cite{1971_langford} showed that $A$ is a Kuratowski $14$-set if
and only if five independent inclusions in $\KZA$ all fail.
Kleiner considered the five cases when each holds individually for all $A\su X$.
They reduce to these three: (i)~partition or discrete, (ii)~OU, EO, or discrete,
(iii)~ED, EO, or discrete.
For each, Kleiner gave a correct upper bound for $\kf(\XT)$ and an example claimed to
realize this bound, but only one was entirely correct.
The family of closures of singletons in the OU example is misidentified as a base, which
instead yields an EO space, and the English version of the journal forgets to mention the
set $\{a,b\,'\!,c\}$ that generates $20$ distinct subsets.
The ED example fails because it is not a topological space.
A subset of $\mathbb{R}^2$ is defined to be closed if and only if it is of the form
$U\cup V$ where $U$ is finite and $V$ is any collection of lines parallel to the
$x$-axis.
The set $\bigcap_{n=1}^\infty(\{(n,n)\}\cup\{(x,m):x\in\mathbb{R},m\in\mathbb{N}^+,m\neq
n\})=\{(n,n):n\in\mathbb{N}^+\}$ is not closed under this definition.
Leaving $V$ out gives us the ED space $(\mathbb{R}^2,$ cofinite topology) but it only has
\kfnum\ $4$.}
Neither Chagrov nor GJ named all six space types per se; for brevity we do
so using Table~\ref{tab:six_Kuratowski_monoids}.

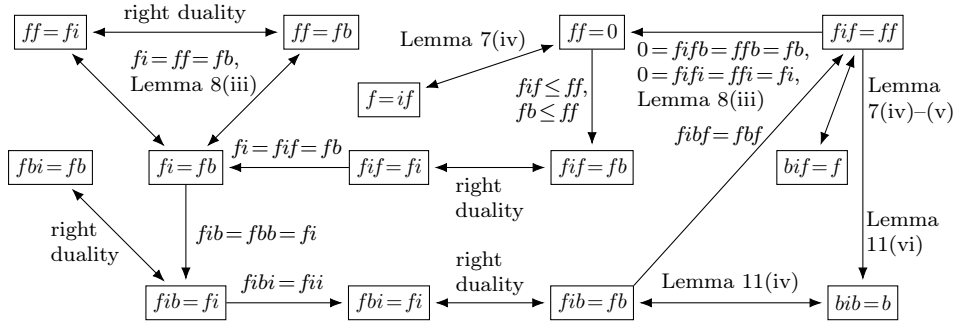
\begin{figure}
\small
\centering
\begin{tikzpicture}
[auto,
 block/.style={rectangle,draw=black,minimum height=1.6em, text centered,scale=.8},
 line/.style ={draw},
 doublearrow/.style={{Latex[sep=2pt]}-{Latex[sep=2pt]}},
 myarrow/.style={-{Latex[sep=2pt]}},scale=.9]
	\node[block,anchor=south] (1) at (-3,4) {$\f\eqy\ie$};
	\node[block,anchor=south] (2) at (0,5) {$\ff\eqy0$};
	\node[block,anchor=south] (3) at (-3,3) {$\fif\eqy\ei$};
	\node[block,anchor=south] (4) at (0,3) {$\fif\eqy\fb$};
	\node[block,anchor=south] (5) at (-6,3) {$\ei\eqy\fb$};
	\node[block,anchor=south] (6) at (-8,5) {$\ff\eqy\ei$};
	\node[block,anchor=south] (7) at (-4,5) {$\ff\eqy\fb$};
	\node[block,anchor=south] (8) at (-8,3) {$\fbi\eqy\fb$};
	\node[block,anchor=south] (9) at (-6,1) {$\fib\eqy\ei$};
	\node[block,anchor=south] (10) at (-3,1) {$\fbi\eqy\ei$};
	\node[block,anchor=south] (11) at (0,1) {$\fib\eqy\fb$};
	\node[block,anchor=south] (12) at (4,5) {$\fif\eqy\ff$};
	\node[block,anchor=south] (13) at (4,1) {$bib\eqy b$};
	\node[block,anchor=south] (14) at (3.25,3) {$\bif\eqy\f$};
	\draw[doublearrow] (10) -- (11) node[midway,above=-2pt] {\scriptsize\begin{tabular}{l}right\\duality\end{tabular}};
	\draw[doublearrow] (11) -- (13) node[midway,above] {\scriptsize Lemma~\ref{lem:equations_kf}(iv)};
	\draw[myarrow] (12) -- (13) node[midway,below=12pt] {\hspace{30pt}\scriptsize\begin{tabular}{l}Lemma\\\ref{lem:equations_kf}(vi)\end{tabular}};
	\draw[myarrow] (11.north east) -- (12) node[below=32pt] {\hspace{-110pt}\scriptsize$\fibf\eqy\fbf$};
	\draw[myarrow] (12) -- (2) node[midway,below] {\scriptsize\vbox{\hbox{$0\eqy\fifb\eqy\ffb\eqy\fb$,}\hbox{$0\eqy\fifi\eqy\ffi\eqy\ei$,}\hbox{Lemma~\ref{lem:op_relations}(iii)}}};
	\draw[doublearrow] (1) -- (2) node[midway,above=2pt] {\hspace{-20pt}\scriptsize Lemma~\ref{lem:decompositions}(iv)};
	\draw[myarrow] (2) -- (4) node[midway,left=-2pt] {\scriptsize\vbox{\hbox{$\fif\leqy\ff$,}\hbox{$\fb\leqy\ff$}}};
	\draw[doublearrow] (3) -- (4) node[midway,below=-1pt] {\scriptsize\begin{tabular}{l}right\\duality\end{tabular}};
	\draw[myarrow] (3) -- (5) node[midway,above] {\scriptsize$\ei\eqy\fif\eqy\fb$};
	\draw[doublearrow] (6) -- (7) node[midway,above] {\scriptsize right duality};
	\draw[doublearrow] (5) -- (6);
	\draw[doublearrow] (5) -- (7) node[midway,above=-2pt] {\hspace{-45pt}\scriptsize\vbox{\hbox{$\ei\eqy\ff\eqy\fb$,}\hbox{Lemma~\ref{lem:op_relations}(iii)}}};
	\draw[myarrow] (9) -- (10) node[midway,above] {\hspace{-4pt}\scriptsize$\fibi\eqy\fii$};
	\draw[doublearrow] (8) -- (9) node[midway,below left=-10pt] {\scriptsize\begin{tabular}{l}right\\duality\end{tabular}};
	\draw[myarrow] (5) -- (9) node[midway,right] {\scriptsize$\fib\eqy\fbb\eqy\ei$};
	\draw[doublearrow] (12) -- (14) node[midway,right] {\hspace{2pt}\scriptsize\begin{tabular}{l}Lemma\\\ref{lem:decompositions}(iv)--(v)\end{tabular}};
\end{tikzpicture}
\caption{Equations in $\FZ$ equivalent to $bib=b$.}
\label{fig:equivalent_equations}
\end{figure}

\begin{table}
\caption{The GE monoid in non-Kuratowski spaces (part numbers refer to
Lemma~\ref{lem:imps}).}
\centering\small
\renewcommand{\arraystretch}{1.1}
\renewcommand{\tabcolsep}{5pt}
\begin{tabular}{r|c|c|c|c|c|}
\multicolumn{1}{c}{}&\multicolumn{1}{c}{ED}&\multicolumn{1}{c}{OU}&
\multicolumn{1}{c}{EO}&\multicolumn{1}{c}{P}&\multicolumn{1}{c}{D}\\
\Xcline{2-6}{2\arrayrulewidth}
equations hold by\vrule width0pt height10pt&(i),\,(ii)&(i),\,(iii)&(i)--(iii)&
(i),\,(ii),\,Fig{.}~\ref{fig:equivalent_equations}&(iv)\\\cline{2-6}
equations fail by\vrule width0pt height10pt&
(iii),\,Fig{.}~\ref{fig:equivalent_equations}&
(ii),\,Fig{.}~\ref{fig:equivalent_equations}&Fig{.}~\ref{fig:equivalent_equations}&
(iv)&$X\neq\es$\\\cline{2-6}
\end{tabular}
\label{tab:GE_monoid}
\end{table}

\setlength{\cw}{50pt}
\begin{table}[!ht]
\caption{Frequencies of the seven GE monoids up to homeomorphism for $|X|\leq11$.}
\centering\small
\renewcommand{\arraystretch}{1}
\renewcommand{\tabcolsep}{0pt}
\begin{tabular}{|C{16pt}|C{\cw}|C{\cw}|C{\cw}|C{\cw}|
C{\cw}|C{24pt}|C{24pt}|}
\multicolumn{1}{c}{$|X|$}&\multicolumn{1}{c}{GE}&\multicolumn{1}{c}{KD}&
\multicolumn{1}{c}{ED}&\multicolumn{1}{c}{OU}&\multicolumn{1}{c}{EO}&
\multicolumn{1}{c}{P}&\multicolumn{1}{c}{D}\\\Xhline{2\arrayrulewidth}
\z{1}&\z{0}&\z{0}&\z{0}&\z{0}&\z{0}&\z{0}&\z{1}\\\hline
\z{2}&\z{0}&\z{0}&\z{0}&\z{0}&\z{1}&\z{1}&\z{1}\\\hline
\z{3}&\z{0}&\z{0}&\z{1}&\z{1}&\z{4}&\z{2}&\z{1}\\\hline
\z{4}&\z{1}&\z{0}&\z{6}&\z{7}&\z{14}&\z{4}&\z{1}\\\hline
\z{5}&\z{11}&\z{1}&\z{25}&\z{45}&\z{50}&\z{6}&\z{1}\\\hline
\z{6}&\z{88}&\z{9}&\z{99}&\z{306}&\z{205}&\z{10}&\z{1}\\\hline
\z{7}&\z{697}&\z{65}&\z{397}&\z{2375}&\z{986}&\z{14}&\z{1}\\\hline
\z{8}&\z{5993}&\z{454}&\z{1784}&\z{21906}&\z{5820}&\z{21}&\z{1}\\\hline
\z{9}&\z{59525}&\z{3425}&\z{9442}&\z{247357}&\z{43304}&\z{29}&\z{1}\\\hline
\z{10}&\z{712639}&\z{29816}&\z{62679}&\z{3497270}&\z{415241}&\z{41}&\z{1}\\\hline
\z{11}&\z{10592049}&\z{315322}&\z{543735}&\z{62855093}&\z{5195399}&\z{55}&\z{1}\\\hline
\end{tabular}
\label{tab:monoid_frequencies}
\end{table}

\begin{lem}\label{lem:imps}
The equations in each part below are equivalent\emph.

\medskip
\renewcommand{\tabcolsep}{2pt}
\centering\begin{tabular}{rlcrl}
\emph{(i)}&\emph{$\bif\eqy\ie,\,\,\fif\eqy0,\,\,\fib\eqy\fbi$,}&\quad\quad&
\emph{(iii)}&\emph{$\f\eqy\ff,\,\,\ie\eqy0,\,\,bib\eqy bi$,}\\
\emph{(ii)}&\emph{$\fif\eqy\fib,\,\,\fif\eqy\fbi,\,\,\fbi\eqy0,\,\,\fib\eqy0,\,\,bib\eqy ib$,}&&
\emph{(iv)}&\emph{$\ie\eqy\fb,\,\,\ie\eqy\ei,\,\,\ie\eqy\ff,\,\,\f\eqy0,\,\,b\eqy i$.}\\
\end{tabular}
\end{lem}

\begin{proof}
(i)~$\bif=\ie\iff\fif=0\implies\fib=\fbi\implies\fif=\fib g=\fbi g=0$ by
Proposition~\ref{prop:order_implications}(vi) and Lemmas~\ref{lem:decompositions}(v)
and~\ref{lem:ifo}(iv).
(ii)~$\fif=\fib\iff\fif=\fbi\iff\fbi=\fib=\fif=0\iff\fbi=\fib=0\iff\fbi=0\iff\fib=0\iff
bib=ib$ by right duality, part~(i), and Lemma~\ref{lem:decompositions}(ii).
(iii)~Apply Lemma~\ref{lem:decompositions}(iv) and Lemma~\ref{lem:equations_kf}(i).
(iv)~$\ie=\fb\iff\ie=\ei\iff\ie=\ei=\fb=0\iff\ie=\ff\iff\f=0\iff b=i$ by right duality,
$\Ord\KF$, Lemma~\ref{lem:op_relations}(iii), and Lemma~\ref{lem:decompositions}
parts~(i) and~(iv).
\end{proof}

\begin{thm}\label{thm:kf_monoids}
\mth{$\KF$} has exactly seven global collapses.
Figure~\nit{\ref{fig:Hasse_diagrams}} displays the six nonempty ones\nit.
Hence the only possible \Kfnum s are\nit: \mth{$4,10,16,20,22,28,34$.}
\end{thm}

\begin{proof}
We verified by computer that seven exist (see Proposition~\ref{prop:minimal_spaces})
and claim no further ones do.

Since $X\neq\es\implies i\neq0$ we have $o_1\neq o_2$ for all $o_1\in\KZ,o_2\in\FZ$ by
Table~\ref{tab:kf_akf}.
Moreover, $\KFZ\cap a\KFZ=\es$.
Hence our task is limited to finding all possible equations in $\FZ$ for each
of the six global collapses of $\K$ (see Table~\ref{tab:six_Kuratowski_monoids}).

\case{Case}~$1$. $(|\KZ|=7$ and $\fif\neq0)$
Have $\bif\neq\ie\neq0$.
Hence $\ie\cn\leq\ff$.
Have $\bif\neq\f$ by Figure~\ref{fig:equivalent_equations}.
Thus by $\Ord\KF$, $o_1\neq o_2$ for all $o_1\in\{\ie,\bif,f\hspace{1pt}\}$, $o_2\in\FZ
\sm\{\ie,\bif,f\hspace{1pt}\}$.
By Lemma~\ref{lem:imps}(i)--(ii), $\Ord\KF$, and Figure~\ref{fig:equivalent_equations},
$o_1\neq o_2$ for all $o_1,o_2\in\mathop\downarrow\{\ff\hspace{1pt}\}$.
Conclude $|\FZ|=10$.

\case{Case}~$2$. $(|\KZ|=7$ and $\fif=0)$
Lemma~\ref{lem:imps} and Figure~\ref{fig:equivalent_equations} imply the collapse of
$\FZ$ in Figure~\ref{fig:Hasse_diagrams}(i).

All other cases are covered by Table~\ref{tab:GE_monoid}.
\end{proof}

\Kfnum\ $28$ defines a special type of disconnected Kuratowski space we call
\define{Kuratowski disconnected} (KD). 
Note that $\KZ$ and $\FZ$ are order isomorphic in KD, EO, and discrete spaces.

\begin{prop}\label{prop:minimal_spaces}
By Table~\nit{\ref{tab:monoid_frequencies}}\nit, each GE monoid occurs in a unique space
of minimal cardinality up to homeomorphism\nit.
Bases for these spaces appear below\nit.

\begin{center}
\small
\renewcommand{\arraystretch}{1.1}
\renewcommand{\tabcolsep}{2pt}
\begin{tabular}{c@{\hspace{5pt}}c@{\hspace{5pt}}c@{\hspace{5pt}}c}
\begin{tabular}{|c|l|}\hline
\nit{\raisebox{-.75pt}{GE}}&\mth{$\{w\},\{x,y\},\{w,x,y,z\}$}\\\hline
\nit{\raisebox{-.75pt}{KD}}&\mth{$\{v\},\{w\},\{v,w,x\},\{y,z\}$}\\\hline
\end{tabular}&
\begin{tabular}{|c|l|}\hline
\nit{\raisebox{-.75pt}{ED}}&\mth{$\{x,y\},\{x,y,z\}$}\\\hline
\nit{\raisebox{-.75pt}{OU}}&\mth{$\{x\},\{y\},\{x,y,z\}$}\\\hline
\end{tabular}&
\begin{tabular}{|c|l|}\hline
\nit{\raisebox{-.75pt}{EO}}&\mth{$\{x\},\{x,y\}$}\\\hline
\end{tabular}&
\begin{tabular}{|c|l|}\hline
\nit{\raisebox{-.75pt}{P}}&\mth{$\{x,y\}$}\\\hline
\nit{\raisebox{-.75pt}{D}}&\mth{$\{x\}$}\\\hline
\end{tabular}
\end{tabular}
\end{center}
\end{prop}

The number of nonhomeomorphic partition spaces on $n$ points is one less than the number
of partitions of $n$ ($1+1+\dots+1$ corresponds to the discrete space).
This is sequence \href{https://oeis.org/A000065}{A000065} in the OEIS~\cite{2021_sloane}.
Besides the ``all $1$'s'' sequence \href{https://oeis.org/A000012}{A000012}, no other
column in Table~\ref{tab:monoid_frequencies} appears in the OEIS.

\begin{prop}\label{prop:kf_ordering}
GE spaces satisfy \mth{$\Ord\KF$.}
\end{prop}

\begin{proof}
Let $X$ be GE and $S_1,\dots,S_7$ be the sets in the proof of
Proposition~\ref{prop:partial_order}.
By the arguments in that proof, we have $o_1\cn\leq o_2$ for $(o_1,o_2)\in S_1\cup S_2
\cup S_3$.
Right duality, Proposition~\ref{prop:order_implications}(v), and right-multiplication
by $i$ yield $\fif\leq\fbi\iff\fif\leq\fib\iff\fbi\leq\fib\iff\fib\leq\fbi\iff\fib=\fbi$,
$\fib\leq\fif\iff\fbi\leq\fif\iff\fbi=0\implies\fbi\leq a\iff\fib\leq\ident\implies\fbi
\leq i\iff\fbi=0\iff\fbi\leq\ident\iff\fib\leq a$, and $\fif\leq\ident\iff\fif\leq a\iff
\fif\leq\ident\wedge a=0\iff\fif=0$.
Conclude $o_1\cn\leq o_2$ for $(o_1,o_2)\in S_4\cup S_5\cup S_6\cup S_7$.
\end{proof}

\begin{thm}\label{thm:kf_orderings}
\mth{$\KF$} has exactly nine global orderings \nit(\hspace{-.5pt}see
Figures~\nit{\ref{fig:kf_monoid}(i)} and~\nit{\ref{fig:Hasse_diagrams}).}
\end{thm}

\begin{proof}
By Proposition~\ref{prop:kf_ordering}, GE spaces satisfy only one global ordering of
$\KF$.
We verified by computer that each ordering in Figure~\ref{fig:Hasse_diagrams}
is satisfied by some space.\footnote{The minimal ED and EO spaces each satisfy $\fb\leq
\ident$; the ED space with base $\{\{w,x\}$, $\{w,x,y,z\}\}$ does not, nor does the EO
space with base $\{\{x\}$, $\{x,y\}$, $\{x,y,z\}\}$.}
Since the zero operator cannot appear in the extender of a partial order on $\KF$, all
inequalities in rows $1$-$8$ of the table in Figure~\ref{fig:Hasse_diagrams} fail.
The same is true of rows $9$-$10$ since $bib\neq b$ in KD and OU spaces.
The result follows by Lemma~\ref{lem:poset_result}.
\end{proof}

\subsection{The monoid \texorpdfstring{$\KFG$}{KFG}.}\label{sub:kfg}
We now show that $\KFG$ has ten global collapses and $12$ global orderings.

\begin{lem}\label{lem:equations_kfg}
\noindent\hangindent=74pt
\nit{(i)}~\mth{$\fbg=g\implies bg=g\implies\fbg=bg\implies bib=bi$.}

\noindent\hangindent=74pt\phantom{\textbf{Lemma 14.\!}}
\nit{(ii)}~\mth{$g=ga\implies\fbg=\fbga\iff bg=bga\implies\ei=0$.}
\end{lem}

\begin{proof}
(i)~Left-multiply $g=\fbg$ by $b$ to get $bg=\fbg=g$.
Left-multiply by $\f$ to get $\fbg=bg$.
By Lemma~\ref{lem:half_dist}(ii), $bib=bi(bg\vee bi)=bi(\fbg\vee bi)\leq b(\ifb g\vee bi)
=bi$.
(ii)~Apply $bg=\fbg\vee\ie$ and $bg=bga\implies bg=bg\vee bga=\f$.
\end{proof}

The next lemma holds by Lemma~\ref{lem:op_relations}(i).

\begin{lem}\label{lem:kd_bib}
In non-GE spaces\nit, \mth{$bib=bi\vee\ie$,} hence \mth{$iA=\es\implies bibA=\ifA$.}
\end{lem}

\begin{lem}\label{lem:ed_eo_lemma}
If \mth{$\XT$} is ED or EO then \mth{$\fb\leq\fbg$.}
\end{lem}

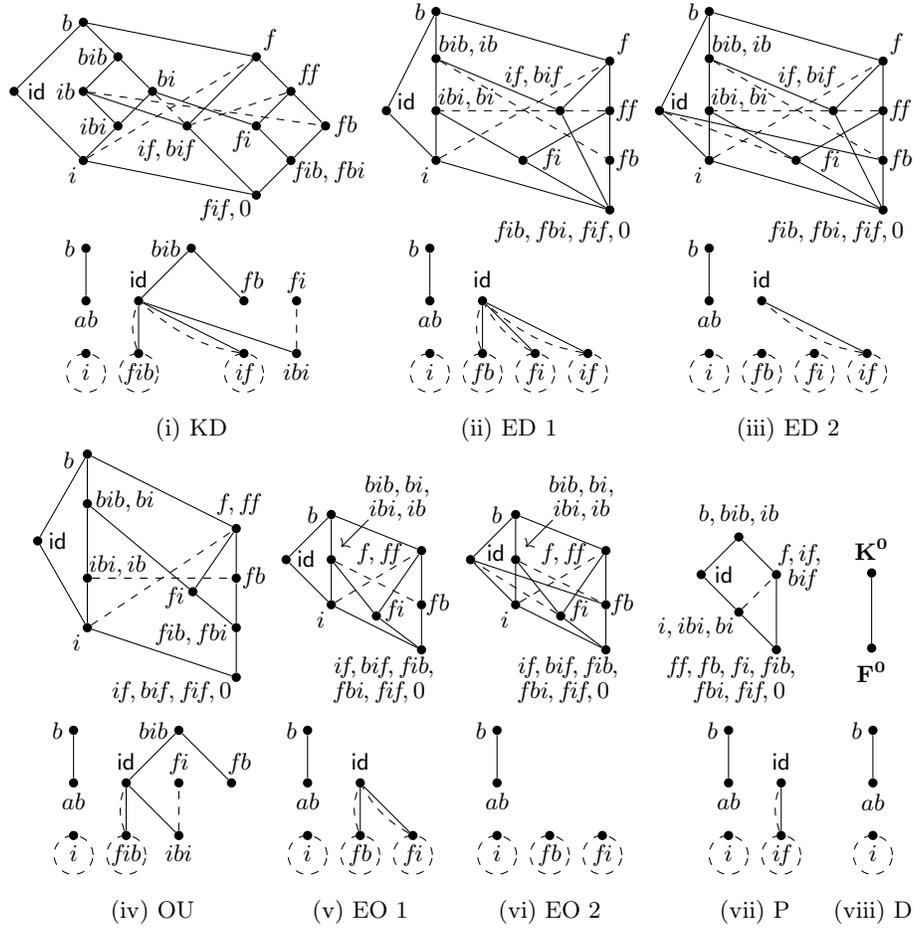
\begin{figure}
\centering\footnotesize
\raisebox{8pt}{

\smallskip
\caption{The eight global orderings on $\KF$ in non-GE spaces.\\\smallskip
\vrule width14pt height0pt depth0pt
\small\begin{tabular}{l}\vrule width0pt height8pt depth0pt
Extenders appear in both the figure and table.\\
Transitivity only applies to the poset diagrams.\\
All diagrams are up to left duality.
\end{tabular}}
\label{fig:Hasse_diagrams}
\end{figure}

\begin{proof}
By Lemma~\ref{lem:kd_bib}, $b=bg\vee bi=\fbg\vee\ie\vee bi=\fbg\vee ib$ in such spaces.
\end{proof}

\begin{thm}\label{thm:kfg_monoids}
\mth{$\KFG$} has exactly ten global collapses \nit(\hspace{-.5pt}see
Figure~\nit{\ref{fig:kfg_monoids}).}
\end{thm}

\begin{proof}
We verified by computer that ten exist and claim no further ones do.
\case{Case}~$1$.~$(|\KZ|=7)$
Suppose $o=p$ for some $o\in\GZ$ and $p\in\KFZ$.
Right-multiply by $i$ to get $pi=0$.
Thus $p\in\{0,\ie,\fif,\bif\hspace{.6pt}\}$.
But then $o=oa$, contradicting $\ei\neq0$ by Lemma~\ref{lem:equations_kfg}(ii).
Conclude $o\neq p$ for all $o\in\GZ$ and $p\in\KFZ$.
Lemma~\ref{lem:equations_kfg}(i) implies $|\GZ|=3$.
\case{Case}~$2$.~(ED~$1,2$)
The Case~$1$ proof yields $|\GZ|=3$ and $o\neq p$ for all $o\in\GZ$ and
$p\in\KFZ\sm\{\fb\}$.
Have $g\neq\fb$ since $g=\fb\implies bg=\fb=g$ and $bg\neq\fb$ since $bg=\fb\implies\ie=
ibg=\ifb=0$.
Thus $\fbg=\fb$ is the only optional equation in $\KFGZ$.
\case{Case}~$3$.~(OU~$1,2$)
The Case~$1$ proof yields $o\neq p$ for all $o\in\GZ$ and $p\in\KFZ$.
Right-multiply $\ff=\f$ by $g$ to get $\fbg=bg$.\footnote{The minimal OU space satisfies
$g=bg$; the OU space with base $\{\{w\}$, $\{x\}$, $\{w,x,y,z\}\}$ does not.
For ED and EO examples see Theorem~\ref{thm:kf_orderings}.}
\case{Case}~$4$.~(EO~$1,2$)
The Case~$1$ proof yields $o\neq p$ for all $o\in\GZ$ and $p\in\KFZ\sm\{\fb\}$.
Lemma~\ref{lem:ed_eo_lemma} and $\ff=\f$ imply $\fb\leq\fbg=bg$.
Claim~$1$: $\fb=bg\implies bg=g$.
Suppose $\fb=bg$.
Then $\ei=\fb a=bga$.
Hence $bg\wedge ga\leq bg\wedge bga=\fb\wedge\ei=0$.
Since $bg\leq\f=g\vee ga$, conclude $bg=g$.
Claim~$2$: $\fb=g\iff bg=g$.
$(\Rightarrow)$~Left-multiply by $b$ and apply Claim~$1$.
$(\Leftarrow)$~Since $bi=ib$ we have $\fb\leq b=bg\vee ib$.
Hence $\fb\leq bg=g$.
Conversely $g\leq\f=(bg\vee bi)\wedge(bga\vee bia)=(g\vee ib)\wedge(ga\vee aib)=(g\wedge
aib)\vee(ib\wedge ga)$.
Since $g\wedge ga=0$ we conclude $g\leq g\wedge aib\leq\fb$.
Together with $\fbg=bg$, Claims~$1$ and~$2$ imply $\fb=g\iff bg=g\iff\fb=bg\iff\fb=\fbg$.
\case{Case}~$5$.~(partition)
Apply $\Ord\KFG$ and Lemma~\ref{lem:equations_kfg}(i).
\case{Case}~$6$.~(discrete)
Apply $\f=0$.
\end{proof}

\begin{lem}\label{lem:inequalities}
\raisebox{-\baselineskip}{\hspace{-11pt}$\begin{array}[c]{r@{\hspace{3pt}}l}
\nit{(i)}&\fb\leq\ident\implies\fbg\leq\ident\iff\fbg\leq\fb\implies\fbg\leq aibi\iff bg
\leq aibi\iff g\leq aibi\nit.\\
\nit{(ii)}&\mbox{If }\XT\mbox{ is not GE then }\fbg\leq\fb\iff\fbg\leq aibi\nit.\\
\nit{(iii)}&\mbox{If }\XT\mbox{ is ED or EO then }\fb\leq\ident\iff\fbg\leq\afi\nit.\\
\end{array}$}
\end{lem}

\begin{proof}
(i)~Right-multiply $\fb\leq\ident$ by $g$ to get $\fbg\leq\ident$.
$(\Rightarrow)$~Have $\fbg\leq\f=b\wedge(bga\vee bia)\leq b(bga)\vee\fb=(\fbga\vee\ie
\hspace{1pt})\vee\fb$.
The result follows since $\fbg\wedge\ie=0$ and $\fbg\leq\ident\implies\fbga\leq a$.
$(\Leftarrow)$~By Lemmas~\ref{lem:op_relations}(iv) and~\ref{lem:subset_implications}(v),
Proposition~\ref{prop:order_implications}(v), and right duality,
$\fif\leq\fbg\leq\fb\implies\fif\leq\fib\implies\fbi\leq\fib\iff\fbi=\fib\implies\fib=
\fib\wedge\fbi=\fb\wedge\ei$.
Thus $\fbg\wedge a=\fbg\wedge ga\leq\fbg\wedge bga=\fbg\wedge\fbga\leq\fb
\wedge\ei=\fib$.
Since $\fib=\fbi\iff\fif=0$ by Lemma~\ref{lem:imps}, conclude $\fbg\wedge ag=(\fbg\wedge
a)g\leq\fib g=\fif=0$.
Thus $\fbg\leq g\leq\ident$, giving us the equivalence.
Clearly $\fbg\leq\fb\implies\fbg\leq aibi\iff bg\ (=\fbg\vee\ie\hspace{1pt})\leq aibi
\implies g\leq aibi$.
Since $aibi$ is closed, $g\leq aibi\implies bg\leq aibi$.
(ii)~Have $(\Rightarrow)$ by (i).
$(\Leftarrow)$ Since $\XT$ is not GE, $\fbg\leq biba=\fiba\vee iba=\fbi\vee abi=\fib\vee
abi\leq\fb\vee abi$.
But $\fbg\wedge abi\leq(b\sm\ie\hspace{1pt})\wedge abi=b\wedge(bi\vee aib)\wedge abi=
(bi\vee\fb)\wedge abi\leq\fb$.
The result follows.
(iii)~By (i), $\fb\leq\ident\implies\fbg\leq\fb\leq\afi$.
Since $\fbg\leq\ff=\fb\vee\ei$, we conclude $\fbg\leq\afi\implies\fbg\leq\fb\implies\fbg
=\fb\implies\fb=\fbg\leq\ident$ by Lemma~\ref{lem:ed_eo_lemma} and (i).
\end{proof}

\begin{thm}\label{thm:kfg_inequalities}
\mth{$\KFG$} has exactly \mth{$12$} global orderings \nit(\hspace{-.5pt}see
Figures~\nit{\ref{fig:kf_monoid}(ii)} and~\nit{\ref{fig:kfg_monoids}).}
\end{thm}

\begin{proof}
We verified by computer that each ordering in Figures~\ref{fig:kf_monoid}(ii)
and~\ref{fig:kfg_monoids} is satisfied by some space.\footnote{The minimal GE space
satisfies $bg\leq aibi$; $\mathbb{R}$ under the usual topology does not (consider $A=
\mathbb{R}\sm\{1/n:n=1,2,\dots\})$.
The minimal KD and OU spaces each satisfy $\fbg\leq\ident$; the KD space with base
$\{\{u\}$, $\{v\}$, $\{u,w\}$, $\{u,v,w,x\}$, $\{y,z\}\}$ does not, nor does the OU space
with base $\{\{w\}$, $\{x\}$, $\{w,x,y,z\}\}$.}
Since $\fib\leq bi\iff\fbi\leq aib\iff\fbi\leq\fib\implies\fbi=\fib$, the inequalities
$\fib\leq bi$ and $\fbi\leq aib$ fail in GE spaces.
All other inequalities in the $12$ extenders besides those in
Table~\ref{tab:kfg_orderings} fail by the $\KFG$ analogue of the
Theorem~\ref{thm:kf_orderings} argument.
By Lemma~\ref{lem:poset_result} it remains only to show that for each space type, the
inequalities in Table~\ref{tab:kfg_orderings} are equivalent.
This holds by Lemma~\ref{lem:inequalities}.
\end{proof}

As GJ point out (see Figure~$2.3$), a natural partial order exists on the six Kuratowski
monoids by setting $\Y_1\leq\Y_2$ if and only if there exists a monoid homomorphism from
$\Y_2$ onto $\Y_1$.
The $\KFG$ analogue appears in Figure~\ref{fig:kfg_projections}.

\begin{figure}
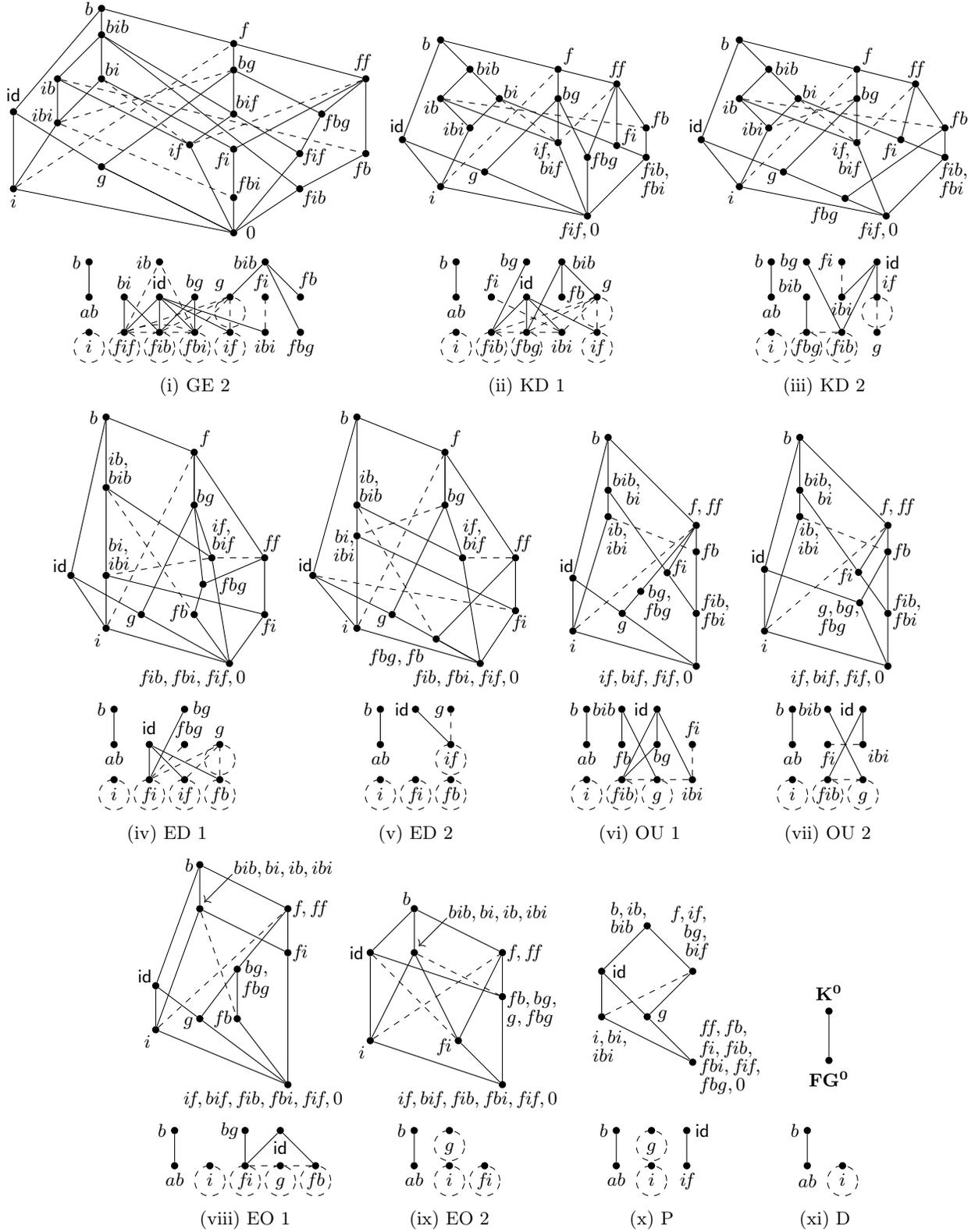

\footnotesize
\centering

\caption{The $11$ global orderings on $\KFG$ in non-GE $1$ spaces and their extenders.\\
\smallskip\small Transitivity only applies to the poset diagrams.
All diagrams are up to left duality.}
\label{fig:kfg_monoids}
\end{figure}

\begin{table}[!ht]
\centering\small
\renewcommand{\arraystretch}{1.2}
\renewcommand{\tabcolsep}{3pt}
\caption{Optional extender inequalities in Figure~\ref{fig:kfg_monoids}.}
%
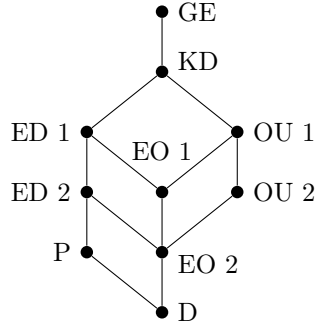
\caption{The ten $\KFG$ monoids ordered by monoid homomorphism.}
\label{fig:kfg_projections}
\end{figure}

\section{The Family \texorpdfstring{$\KFA$}{KFA}}\label{sec:kfa}

In this section we find all local collapses and orderings of $\KF$.\footnote{Local
collapses and orderings of $\KFGZ$ are beyond our scope.
We verified by computer that for $n=2,\dots,11$ there are respectively $5$, $12$, $26$,
$47$, $72$, $106$, $129$, $134$, $134$, $134$ local collapses and $5$, $12$, $28$, $61$,
$131$, $262$, $459$, $614$, $657$, $666$ local orderings of $\KFGZ$ over all $\XT$ such
that $|X|=n$.
These numbers suggest (imply?) that $\KFGZ$ has exactly $134$ local collapses but they
are inconclusive for local orderings.}

\subsection{Equations in \texorpdfstring{\mth{$\KFA$}}{KFA}.}

There are too many local collapses to name, so we number them instead.

\begin{dfn}\label{dfn:phi_number}
Let \define{$\phi A$ $(\psi A)$} be the number of the collapse of \nit{$\KZ$ $(\KF)$}
in Table~\nit{\ref{tab:subset_types}} that \mth{$A$} satisfies\nit.
These so-called \define{\pn s $($\sn s$)$} also refer to their associated collapses\nit.
\end{dfn}

\begin{thm}\label{thm:set_types}
\mth{$\KF$} has exactly \mth{$70$} local collapses \nit(\hspace{-.5pt}see
Table~\nit{\ref{tab:subset_types}).}
\end{thm}

\begin{proof}
Let $A\su X$.
By Theorem~2.10 in GJ, $A$ satisfies one of the $30$ collapses of $\KZ$ in
Table~\ref{tab:subset_types}.
As the authors point out, each extends uniquely to a local collapse of $\K$ by adding
left duals, with one exception: When $bA$ and $iA$ are both clopen and unequal, the equation
$bA=aiA$ $(\!{\iff}(bA=X\mbox{ and }iA=\es))$ may or may not hold.
These $31$ local collapses contain every equation in $\K$ a subset can satisfy.
Hence we need only find equations $o_1A=o_2A$ that involve at least one operator in
$\semigp{F}$.
We can assume without loss of generality that $o_1\in\KFZ$ and $o_2\in\F$.
Note that $(bA\neq X\mbox{ and }iA\neq\es)\implies o_1,o_2\in\FZ$ by
Table~\ref{tab:kf_akf} under this assumption.

We need only consider $19$ of the $30$ local collapses of $\KZ$ since there are
$11$ dual pairs.
Let $\Phi_1=\{4,$ $7,$ $11,$ $12,$ $13,$ $18,$ $20,$ $24,$ $26\}$
and $\Phi_2=\{9,$ $14,$ $16,$ $22,$ $25,$ $28,$ $30\}$.
By Lemma~\ref{lem:fza_collapse}, for $\phi A\in\Phi_1\cup\Phi_2$ the collapse of $\FZ$
that $A$ satisfies is determined by Proposition~\ref{prop:order_implications}(v)--(vi)
and Lemmas~\ref{lem:op_relations}(iii), \ref{lem:subset_implications}(v),
and~\ref{lem:equations_kf} (see Table~\ref{tab:subset_types}).
This completes the proof for $\phi A\in\Phi_1$ since $bA\neq X$ and $iA\neq\es$.

Suppose $\phi A\in\Phi_2$.
The proof is done for the case $(bA\neq X$ and $iA\neq\es)$.

Suppose $bA=X$.
Clearly $\phi A=30\implies\psi A=69$ and $iA=\es\implies\psi A=61$.
Suppose $iA\neq\es$.
Table~\ref{tab:subset_types} covers $o_1,o_2\in\FZ$ and Table~\ref{tab:kf_akf} eliminates
the cases $o_1\in\KZ,\,o_2\in\FZ$ and $o_1\in\FZ,\,o_2\in a\FZ$.
In the case that remains, $o_1\in\KZ$ and $o_2\in a\FZ$.
By Lemma~\ref{lem:equations_kf}(xi) we can assume $o_1=\ident$.

\case{Case}~$1$.~$(\phi A\in\{9,14,16,25\})$
Suppose $A=o_2A$.
Table~\ref{tab:multiplication} implies $\fa o_2A\in\{\es$, $\ffA$, $\fiA$, $\fbA$,
$\fbiA$, $\fibA$, $\fifA\}$.
Hence $\fA\neq\fa o_2A$ by Table~\ref{tab:subset_types}.
Since this contradicts $\fA=\fo_2A=\fa o_2A$ we conclude $A\neq o_2A$.
\case{Case}~$2$.~$(\phi A\in\{22,28\})$
By Table~\ref{tab:subset_types}, $a\FZA=\{\afA,X\}$.
Suppose $\phi A=22$.
Since $A$ is neither open nor equal to $X$, $A\neq o_2A$.
If $\phi A=28$ then $A=iA=\afA$ by Lemma~\ref{lem:equations_kf}(xi).

In the only remaining case we have $bA\neq X$ and $iA=\es$.
Since $iA=\es\implies bi=i$ we get $\phi A\cn\in\{9$, $14$, $16$, $22$, $28\}$.
Clearly $\phi A=25\implies\psi A=60$ and $\phi A=30\implies\psi A=70$.
This completes the proof for $\phi A\in\Phi_2$.

Since $bA=X\implies ibA=bA$ and $iA=\es\implies biA=iA$, it remains only to find all
possible equations $o_1A=o_2A$ for $o_1,o_2\in\FZ$ when $\phi A\in\{1,2,6\}$.

Let $\mathcal{E}_nA$ stand for: ``The set $A$ satisfies exactly $n$ of the equations
$\fib=\fbi$, $\fib=\fif$, $\fbi=\fif$.''
Note that $\mathcal{E}_2A$ is impossible.
Let $C_1$, $C_2$, $C_3$ be the sets defined in Lemma~\ref{lem:fza_collapse}.

\case{Case}~$1$.~$(\phi A=1)$
By Lemma~\ref{lem:equations_kf}, $\fif=0$ is the only equation in $C_1\cup C_2$ that $A$
can satisfy.
By Proposition~\ref{prop:order_implications}(vi), $\fifA=\es\implies\fibA=\fbiA$.
Thus $\mathcal{E}_0A\implies\psi A=1$.
If $\mathcal{E}_1A$ then $\fibA=\fifA\implies\psi A=3$, $\fbiA=\fifA\implies\psi A=4$,
and $\fibA=\fbiA$ implies $\psi A=2$ when $\fifA\neq\es$, $\psi A=6$ when $\fifA=\es$.
Clearly $\mathcal{E}_3A\implies\psi A=5$.

\case{Case}~$2$.~$(\phi A=2)$
$\mathcal{E}_0A$ implies $\psi A=7$ when $\ffA\neq\fiA$ and $\psi A=9$ when $\ffA=\fiA$.
Suppose $\mathcal{E}_1A$.
Since $\fibA=\fifA\implies\fbiA\su\fibA$ by Proposition~\ref{prop:order_implications}(v)
and $\ffA=\fiA\implies\fibA=\fbA\su\fiA\implies\fibA\su\fbiA$ by
Lemma~\ref{lem:subset_implications}(iv) it follows that $\fibA=\fifA\implies\psi A=8$.
By Proposition~\ref{prop:order_implications}(v) and Lemma~\ref{lem:op_relations}(iii),
$\fbiA=(\fibA\mbox{ or }\fifA)\implies\fbA=\fibA\su\fbiA\su\fiA\implies\ffA=\fiA$.
It follows that $\fbiA=\fifA\implies\psi A=10$ and $\fibA=\fbiA$ implies $\psi A=11$ when
$\fifA\neq\es$, $\psi A=13$ when $\fifA=\es$.
Clearly $\mathcal{E}_3A\implies\psi A=12$.

\case{Case}~$3$.~$(\phi A=6)$
Let $o\in\{\fb,\ei,\fib,\fbi\}$.
Claim $(\fbA=\fibA$ and $\fiA=\fbiA$ and $oA=\fifA)\implies\ffA=\fifA$.
The hypothesis implies $\fibA\sm\fifA=\fbiA\sm\fifA=\es$ by
Proposition~\ref{prop:order_implications}(i).
Hence $\ffA\sm\fifA=(\fbA\cup\fiA)\sm\fifA=(\fibA\cup\fbiA)\sm\fifA=\es$ by
Lemma~\ref{lem:op_relations}(iii), giving us the claim.
Lemma~\ref{lem:op_relations}(iii) also implies $(\ffA\neq\fbA\mbox{ or }\ffA\neq\fiA)
\implies\fbA\neq\fiA$.
Thus when $\ffA\supsetneq\fifA\supsetneq\es$, the four combinations of $\ff=\fb$ and
$\ff=\ei$ being satisfied or not produce \sn s $23$-$26$.
If $\ffA=\fifA$, then $\fifA\supseteq\fbiA\neq\es$ by
Lemma~\ref{lem:equations_kf}(iii); in this case the four combinations produce \sn s
$27$-$30$.
Finally, $\fifA=\es\implies\psi A=31$ by Proposition~\ref{prop:order_implications}(vi) and
Lemma~\ref{lem:op_relations}(iii).

A space in which all $70$ \sn s occur is given in the next section (see
Proposition~\ref{prop:all_psi}).
\end{proof}

\begin{cor}\label{cor:kfz_collapses}
\mth{$\KFZ$} has exactly \mth{$62$} local collapses.
\end{cor}

\begin{proof}
By Table~\ref{tab:subset_types} the only \pn s that extend to multiple \sn s with equal
collapses of $\KFZ$ are the seven in the set $\Phi_2$ above.
Since exactly one of them (\pn\ $25)$ extends twice in this way, the result follows.
\end{proof}

\begin{cor}\label{cor:other_columns}
Columns~\mth{$7$-$12$} are correct in Table~\nit{\ref{tab:subset_types}.}
\end{cor}

\begin{proof}
Except for \sn\ $61$, $k=14-2e$ where $e$ is the number of equal signs in the collapse of
$\KZ$.
Similarly $\kf=34-2e_{\hspace{-.5pt}f}$ where $e_{\hspace{-.5pt}f}$ is $e$ plus the
number of equal signs in the remainder (see the footnote below
Table~\ref{tab:subset_types}).
Since $i=aba$ the value of $\psi i$ is determined by columns $5$ $(\psi a)$ and $9$
$(\psi b)$.

The four closed \pn s are characterized by their intersection with the set
$\{\{bi,b\},\{bi,i\}\}$: $21$ neither, $26$ $bi=b$, $29$ $bi=i$, $30$ both.
Thus, since $bi(bA)=b(bA)\iff bibA=bA$ and $bi(bA)=i(bA)\iff bibA=ibA$, the value of
$\phi(bA)$ is determined by the intersection of $\phi A$ with $\{\{bib,b\},\{bib,ib\}\}$.
If $\phi(bA)\in\{21,26\}$ we get $\psi(bA)$ immediately.
Otherwise, since $i(bA)=\es\ify ibA=\es$, $b(bA)=X\ify bA=X$, and $bA=\es\ify A=\es$, the
value of $\psi(bA)$ is determined by $\phi(bA)$ and the intersection of $\psi(A)$ with
$\{\{ib,0\},\{b,1\},\{\ident,0\}\}$.

The value of $\psi(\fA)$ is determined similarly since $bi(\fA)=b(\fA)\ify \bifA=\fA$,
$bi(\fA)=i(\fA)\ify \bifA=\ifA$, $i(\fA)=\es\ify\ifA=\es$, $b(\fA)=X\ify\fA=X$, and
$b(\fA)=\es\ify\fA=\es$.

\bgroup
\small
\renewcommand{\arraystretch}{.958}
\renewcommand{\tabcolsep}{1.2pt}
\begin{longtable}{|c|c|c|c|c|c|c|c|c|c|c|c|}
\caption{The $70$ local collapses of $\KF$.*}
\label{tab:subset_types}\\\hline
$\phi$&$\phi a$&collapse of $\KZ$\vrule width0pt height10pt depth5pt&$\psi$&
$\psi a$&remainder&$k$&$\kf$&$\psi b$&$\psi i$&$\psi\!\f$&$\psi g$\\
\Xcline{1-12}{2\arrayrulewidth}
\endfirsthead
\caption*{Table~\ref{tab:subset_types} (cont.): The $70$ local
collapses of $\KF$.}\\\hline
$\phi$&$\phi a$&collapse of $\KZ$\vrule width0pt height10pt depth5pt&$\psi$&
$\psi a$&remainder&$k$&$\kf$&$\psi b$&$\psi i$&$\psi\!\f$&$\psi g$\\
\Xcline{1-12}{2\arrayrulewidth}
\endhead
\multirow{6}{*}{\z{$1$}}&\multirow{6}{*}{\z{$1$}}&
\multirow{6}{*}{\z{$\es$}}&\z{$1$}&\z{$1$}&\z{$\es$}&
\multirow{6}{*}{\z{$14$}}&\z{$34$}&\multirow{6}{*}{\z{$52$}}&\multirow{6}{*}{\z{$51$}}&
\multirow{5}{*}{\z{$52$}}&\multirow{5}{*}{\z{$37$}}\\\cline{4-6}\cline{8-8}
&&&\z{$2$}&\z{$2$}&$\fib\eqy\fbi$&&\z{$32$}&&&&\\\cline{4-6}\cline{8-8}
&&&\z{$3$}&\z{$4$}&$\fib\eqy\fif$&&\z{$32$}&&&&\\\cline{4-6}\cline{8-8}
&&&\z{$4$}&\z{$3$}&$\fbi\eqy\fif$&&\z{$32$}&&&&\\\cline{4-6}\cline{8-8}
&&&\z{$5$}&\z{$5$}&$\fib\eqy\fbi\eqy\fif$&&\z{$30$}&&&&\\\cline{4-6}\cline{8-8}
\cline{11-12}
&&&\z{$6$}&\z{$6$}&$\fib\eqy\fbi,\,\fif\eqy0,\,\bif\eqy\ie$&&\z{$28$}&&&\z{$66$}&\z{$48$}
\\\hline
\multirow{8}{*}{\z{$2$}}&\multirow{8}{*}{\z{$3$}}&
\multirow{8}{*}{$bib\eqy b$}&\z{$7$}&\z{$14$}&$\fb\eqy\fib$&\multirow{8}{*}{\z{$12$}}&
\z{$30$}&
\multirow{8}{*}{\z{$62$}}&\multirow{8}{*}{\z{$51$}}&\multirow{6}{*}{\z{$52$}}&
\multirow{6}{*}{\z{$37,44$}}
\\\cline{4-6}\cline{8-8}
&&&\z{$8$}&\z{$15$}&$\fb\eqy\fib\eqy\fif$&&\z{$28$}&&&&\\\cline{4-6}\cline{8-8}
&&&\z{$9$}&\z{$16$}&$\fb\eqy\fib,\,\ff\eqy\ei$&&\z{$28$}&&&&\\\cline{4-6}\cline{8-8}
&&&\z{$10$}&\z{$17$}&$\fb\eqy\fib,\,\ff\eqy\ei,\,\fbi\eqy\fif$&&\z{$26$}&&&&\\\cline{4-6}
\cline{8-8}
&&&\z{$11$}&\z{$18$}&$\fb\eqy\fib\eqy\fbi,\,\ff\eqy\ei$&&\z{$26$}&&&&\\\cline{4-6}
\cline{8-8}
&&&\z{$12$}&\z{$19$}&$\fb\eqy\fib\eqy\fbi\eqy\fif,\,\ff\eqy\ei$&&\z{$24$}&&&&\\
\cline{4-6}
\cline{8-8}\cline{11-12}
&&&\z{$13$}&\z{$20$}&\makecell{$\fb\eqy\fib\eqy\fbi,\,\ff\eqy\ei,\,\fif\eqy0$,
\vspace{-1pt}\\$\bif\eqy\ie$}&&
\z{$22$}&&&\z{$66$}&\z{$48,60$}\\\hline
\multirow{8}{*}{\z{$3$}}&\multirow{8}{*}{\z{$2$}}&
\multirow{8}{*}{$ibi\eqy i$}&
\z{$14$}&\z{$7$}&$\ei\eqy\fbi$&\multirow{8}{*}{\z{$12$}}&\z{$30$}&
\multirow{8}{*}{\z{$52$}}&
\multirow{8}{*}{\z{$63$}}&\multirow{6}{*}{\z{$52$}}&\multirow{6}{*}{\z{$37$}}\\
\cline{4-6}\cline{8-8}
&&&\z{$15$}&\z{$8$}&$\ei\eqy\fbi\eqy\fif$&&\z{$28$}&&&&\\\cline{4-6}\cline{8-8}
&&&\z{$16$}&\z{$9$}&$\ei\eqy\fbi,\,\ff\eqy\fb$&&\z{$28$}&&&&\\\cline{4-6}\cline{8-8}
&&&\z{$17$}&\z{$10$}&$\ei\eqy\fbi,\,\ff\eqy\fb,\,\fib\eqy\fif$&&\z{$26$}&&&&\\\cline{4-6}
\cline{8-8}
&&&\z{$18$}&\z{$11$}&$\ei\eqy\fbi\eqy\fib,\,\ff\eqy\fb$&&\z{$26$}&&&&\\\cline{4-6}
\cline{8-8}
&&&\z{$19$}&\z{$12$}&$\ei\eqy\fbi\eqy\fib\eqy\fif,\,\ff\eqy\fb$&&\z{$24$}&&&&\\
\cline{4-6}\cline{8-8}\cline{11-12}
&&&\z{$20$}&\z{$13$}&\makecell{$\ei\eqy\fbi\eqy\fib,\,\ff\eqy\fb,\,\fif\eqy0$,
\vspace{-1pt}\\$\bif\eqy\ie$}&&
\z{$22$}&&&\z{$66$}&\z{$48$}\\\hline
\z{$4$}&\z{$5$}&\z{$bib\eqy ib$}&\z{$21$}&\z{$22$}&$\fib\eqy0,\,\fbi\eqy\fif$&\z{$12$}&
\z{$28$}&\z{$66$}&\z{$51$}&\z{$52$}&\z{$37$}\\\hline
\z{$5$}&\z{$4$}&\z{$ibi\eqy bi$}&\z{$22$}&\z{$21$}&$\fbi\eqy0,\,\fib\eqy\fif$&\z{$12$}&
\z{$28$}&\z{$52$}&\z{$64$}&\z{$52$}&\z{$37$}\\\hline
\multirow{10}{*}{\z{$6$}}&\multirow{10}{*}{\z{$6$}}&
\multirow{10}{*}{$bib\eqy b,\,ibi\eqy i$}&
\z{$23$}&\z{$23$}&$\fb\eqy\fib,\,\ei\eqy\fbi$&\multirow{10}{*}{\z{$10$}}&\z{$26$}&
\multirow{10}{*}{\z{$62$}}&
\multirow{10}{*}{\z{$63$}}&\multirow{4}{*}{\z{$52$}}&\multirow{4}{*}{\z{$37,44$}}
\\\cline{4-6}\cline{8-8}
&&&\z{$24$}&\z{$25$}&$\ff\eqy\fb\eqy\fib,\,\ei\eqy\fbi$&&\z{$24$}&&&&\\\cline{4-6}
\cline{8-8}
&&&\z{$25$}&\z{$24$}&$\ff\eqy\ei\eqy\fbi,\,\fb\eqy\fib$&&\z{$24$}&&&&\\\cline{4-6}
\cline{8-8}
&&&\z{$26$}&\z{$26$}&$\ff\eqy\fb\eqy\fib\eqy\ei\eqy\fbi$&&\z{$22$}&&&&
\\\cline{4-6}\cline{8-8}\cline{11-12}
&&&\z{$27$}&\z{$27$}&$\fb\eqy\fib,\,\ei\eqy\fbi,\,\ff\eqy\fif,\,\f\eqy\bif$&&\z{$22$}&&&
\multirow{4}{*}{\z{$62$}}&\multirow{4}{*}{\z{$44$}}\\\cline{4-6}\cline{8-8}
&&&\z{$28$}&\z{$29$}&$\ei\eqy\fbi,\,\ff\eqy\fb\eqy\fib\eqy\fif,\,\f\eqy\bif$&&\z{$20$}&&
&&\\\cline{4-6}\cline{8-8}
&&&\z{$29$}&\z{$28$}&$\fb\eqy\fib,\,\ff\eqy\ei\eqy\fbi\eqy\fif,\,\f\eqy\bif$&&\z{$20$}&&
&&\\\cline{4-6}\cline{8-8}
&&&\z{$30$}&\z{$30$}&$\ff\eqy\fb\eqy\ei\eqy\fib\eqy\fbi\eqy\fif,\,\f\eqy\bif$&&\z{$18$}&
&&&\\\cline{4-6}\cline{8-8}\cline{11-12}
&&&\z{$31$}&\z{$31$}&\makecell{$\ff\eqy\fb\eqy\fib\eqy\ei\eqy\fbi,\,\fif\eqy0$,
\vspace{-1pt}\\$\bif\eqy\ie$}&&
\z{$18$}&&&\z{$66$}&\z{$48,60$}\\\hline
\z{$7$}&\z{$8$}&\z{$bib\eqy ib,\,ibi\eqy i$}&
\z{$32$}&\z{$33$}&$\ei\eqy\fbi\eqy\fif,\,\fib\eqy0$&\z{$10$}&\z{$24$}&\z{$66$}&\z{$63$}&
\z{$52$}&\z{$37$}\\\hline
\z{$8$}&\z{$7$}&\z{$ibi\eqy bi,\,bib\eqy b$}&
\z{$33$}&\z{$32$}&$\fb\eqy\fib\eqy\fif,\,\fbi\eqy0$&\z{$10$}&\z{$24$}&\z{$62$}&\z{$64$}&
\z{$52$}&\z{$37,44$}\\\hline
\multirow{3}{*}{\z{$9$}}&\multirow{3}{*}{\z{$10$}}&
\multirow{3}{*}{\z{$bib\eqy ib\eqy b$}}&
\z{$34$}\vrule width0pt height9pt depth3pt&\z{$36$}&$\fbi\eqy\fif,\,\ff\eqy\ei,\,\fb\eqy\fib\eqy0$&
\multirow{3}{*}{\z{$10$}}&\z{$22$}&\z{$68$}&
\multirow{3}{*}{\z{$51$}}&\multirow{3}{*}{\z{$52$}}&\multirow{3}{*}{\z{$37,44$}}
\\\cline{4-6}\cline{8-9}
&&&\z{$35$}&\z{$37$}&\makecell{$(34)$,\,$b\eqy1,\,i\eqy\af$,\vspace{-1pt}\\
$bi\eqy\aif,\,ibi\eqy\abif$}&&\z{$14$}&\z{$69$}&&&\\\hline
\multirow{2}{*}{\z{$10$}}&\multirow{2}{*}{\z{$9$}}&
\multirow{2}{*}{$ibi\eqy bi\eqy i$}&
\z{$36$}&\z{$34$}&$\fib\eqy\fif,\,\ff\eqy\fb,\,\ei\eqy\fbi\eqy0$&
\multirow{2}{*}{\z{$10$}}&\z{$22$}&
\multirow{2}{*}{\z{$52$}}&\z{$68$}&\multirow{2}{*}{\z{$52$}}&\multirow{2}{*}{\z{$37$}}
\\\cline{4-6}\cline{8-8}\cline{10-10}
&&&\z{$37$}&\z{$35$}&$(36)$,\,$b\eqy f,\,i\eqy0,\,ib\eqy\ie,\,bib\eqy\bif$&&\z{$14$}&&
\z{$70$}&&\\\hline
\z{$11$}&\z{$11$}&\z{$bib\eqy ib,\,ibi\eqy bi$}&
\z{$38$}&\z{$38$}&$\fib\eqy\fbi\eqy\fif\eqy0,\,\bif\eqy\ie$&\z{$10$}&\z{$22$}&\z{$66$}&
\z{$64$}&\z{$66$}&\z{$48$}\\\hline
\z{$12$}&\z{$12$}&\z{$bib\eqy bi,\,ibi\eqy ib$}&\z{$39$}&\z{$39$}&$\fib\eqy\fbi,\,\bif
\eqy\fif\eqy\ie\eqy0,\,\f\eqy\ff$&\z{$10$}&\z{$20$}&\z{$52$}&\z{$51$}&\z{$67$}&
\z{$56,67$}\\\hline
\z{$13$}&\z{$13$}&\z{$bib\eqy bi\eqy ibi\eqy ib$}&\z{$40$}&\z{$40$}&$\fib\eqy\fbi\eqy\bif
\eqy\fif\eqy\ie\eqy0,\,\f\eqy\ff$&\z{$8$}&\z{$16$}&\z{$66$}&\z{$64$}&\z{$67$}&\z{$56,67$}
\\\hline
\multirow{2}{*}{\z{$14$}}&\multirow{2}{*}{\z{$15$}}&
\multirow{2}{*}{\makecell{$bib\eqy ib\eqy b$,\vspace{-2pt}\\
$ibi\eqy i$}}&
\z{$41$}&\z{$43$}&\makecell{$\ff\eqy\ei\eqy\fbi\eqy\fif,\,\fb\eqy\fib\eqy0$,\vspace{-1pt}
\\$\f\eqy\bif$}&
\multirow{2}{*}{\z{$8$}}&\z{$16$}&
\z{$68$}&\multirow{2}{*}{\z{$63$}}&\multirow{2}{*}{\z{$62$}}&\multirow{2}{*}{\z{$44$}}
\\\cline{4-6}\cline{8-9}
&&&\z{$42$}&\z{$44$}&$(41)$,\,$b\eqy1,\,i\eqy\af,\,bi\eqy\aif$&&
\z{$10$}&\z{$69$}&&&\\\hline
\multirow{2}{*}{\z{$15$}}&\multirow{2}{*}{\z{$14$}}&
\multirow{2}{*}{\makecell{$ibi\eqy bi\eqy i$,\vspace{-2pt}\\
$bib\eqy b$}}&
\z{$43$}&\z{$41$}&\makecell{$\ff\eqy\fb\eqy\fib\eqy\fif,\,\ei\eqy\fbi\eqy0$,\vspace{-1pt}
\\$\f\eqy\bif$}&
\multirow{2}{*}{\z{$8$}}&\z{$16$}&
\multirow{2}{*}{\z{$62$}}&\z{$68$}&\multirow{2}{*}{\z{$62$}}&\multirow{2}{*}{\z{$44$}}
\\\cline{4-6}\cline{8-8}\cline{10-10}
&&&\z{$44$}&\z{$42$}&$(43)$,\,$b\eqy f,\,i\eqy0,\,ib\eqy\ie$&&\z{$10$}&&\z{$70$}&
&\\\hline\pagebreak
\multirow{2}{*}{\z{$16$}}&\multirow{2}{*}{\z{$17$}}&
\multirow{2}{*}{\makecell{$bib\eqy ib\eqy b$,\vspace{-2pt}\\
$ibi\eqy bi$}}&
\z{$45$}&\z{$47$}&\makecell{$\fb\eqy\fib\eqy\fbi\eqy\fif\eqy0,\,\ff\eqy\ei$,
\vspace{-1pt}\\$\bif\eqy\ie$}&
\multirow{2}{*}{\z{$8$}}&\z{$16$}&\z{$68$}&\multirow{2}{*}{\z{$64$}}&
\multirow{2}{*}{\z{$66$}}&
\multirow{2}{*}{\z{$48,60$}}\\\cline{4-6}\cline{8-9}
&&&\z{$46$}&\z{$48$}&$(45)$,\,$b\eqy1,\,i\eqy\af,\,bi\eqy\aif$&&\z{$10$}&\z{$69$}
&&&\\\hline
\multirow{2}{*}{\z{$17$}}&\multirow{2}{*}{\z{$16$}}&
\multirow{2}{*}{\makecell{$ibi\eqy bi\eqy i$,\vspace{-2pt}\\
$bib\eqy ib$}}&
\z{$47$}&\z{$45$}&\makecell{$\ei\eqy\fib\eqy\fbi\eqy\fif\eqy0,\,\ff\eqy\fb$,
\vspace{-1pt}\\$\bif\eqy\ie$}&
\multirow{2}{*}{\z{$8$}}&\z{$16$}&\multirow{2}{*}{\z{$66$}}&\z{$68$}&
\multirow{2}{*}{\z{$66$}}&
\multirow{2}{*}{\z{$48$}}\\\cline{4-6}\cline{8-8}\cline{10-10}
&&&\z{$48$}&\z{$46$}&$(47)$,\,$b\eqy f,\,i\eqy0,\,ib\eqy\ie$&&\z{$10$}&&\z{$70$}&&\\
\hline
\z{$18$}&\z{$19$}&\makecell{$bib\eqy bi\eqy b$,\vspace{-2pt}\\
$ibi\eqy ib$}&\z{$49$}&\z{$50$}&
\makecell{$\fb\eqy\fib\eqy\fbi,\,\bif\eqy\fif\eqy\ie\eqy0$,\vspace{-1pt}\\
$\f\eqy\ff\eqy\ei$}&\z{$8$}&\z{$14$}&
\z{$62$}&\z{$51$}&\z{$67$}&\z{$56,67$}\\\hline
\z{$19$}&\z{$18$}&\makecell{$ibi\eqy ib\eqy i$,\vspace{-2pt}\\
$bib\eqy bi$}&\z{$50$}&\z{$49$}&
\makecell{$\ei\eqy\fib\eqy\fbi,\,\bif\eqy\fif\eqy\ie\eqy0$,\vspace{-1pt}\\
$\f\eqy\ff\eqy\fb$}&\z{$8$}&\z{$14$}&
\z{$52$}&\z{$63$}&\z{$67$}&\z{$56,67$}\\\hline
\z{$20$}&\z{$21$}&$(18),\,\ident\eqy i$&\z{$51$}&\z{$52$}&$(49)$&\z{$6$}&\z{$12$}&
\z{$62$}&\z{$51$}&\z{$67$}&\z{$70$}\\\hline
\z{$21$}&\z{$20$}&$(19),\,\ident\eqy b$&\z{$52$}&\z{$51$}&$(50)$&\z{$6$}&\z{$12$}&
\z{$52$}&\z{$63$}&\z{$67$}&\z{$67$}\\\hline
\multirow{2}{*}{\z{$22$}}&\multirow{2}{*}{\z{$23$}}&
\multirow{2}{*}{\makecell{$bib\eqy bi\eqy$\vspace{-2pt}\\
$b\eqy ibi\eqy ib$}}&
\z{$53$}&\z{$55$}&$(49)$,\,$\fb\eqy0$&
\multirow{2}{*}{\z{$6$}}&\z{$10$}&\z{$68$}&\z{$64$}&\multirow{2}{*}{\z{$67$}}&
\multirow{2}{*}{\z{$56,67$}}\\\cline{4-6}\cline{8-10}
&&&\z{$54$}&\z{$56$}&$(53),\,b\eqy1,\,i\eqy\af$&&\z{$6$}&\z{$69$}&\z{$65$}&&\\\hline
\multirow{2}{*}{\z{$23$}}&\multirow{2}{*}{\z{$22$}}&
\multirow{2}{*}{\makecell{$ibi\eqy ib\eqy$\vspace{-2pt}\\
$i\eqy bib\eqy bi$}}&
\z{$55$}&\z{$53$}&$(50)$,\,$\ei\eqy0$&
\multirow{2}{*}{\z{$6$}}&\z{$10$}&\z{$66$}&\z{$68$}&\multirow{2}{*}{\z{$67$}}&
\multirow{2}{*}{\z{$56$}}\\\cline{4-6}\cline{8-10}
&&&\z{$56$}&\z{$54$}&$(55),\,b\eqy f,\,i\eqy0$&&\z{$6$}&\z{$67$}&\z{$70$}&&\\\hline
\z{$24$}&\z{$24$}&\makecell{$bib\eqy bi\eqy b$,\vspace{-2pt}\\
$ibi\eqy ib\eqy i$}&\z{$57$}&\z{$57$}&
\makecell{$\f\eqy\ff\eqy\fb\eqy\ei\eqy\fib\eqy\fbi$,\vspace{-1pt}\\
$\bif\eqy\fif\eqy\ie\eqy0$}&\z{$6$}&\z{$10$}&
\z{$62$}&\z{$63$}&\z{$67$}&\z{$56,67$}\\\hline
\multirow{4}{*}{\z{$25$}}&\multirow{4}{*}{\z{$25$}}&
\multirow{4}{*}{\makecell{$bib\eqy ib\eqy b$,\vspace{-2pt}\\
$ibi\eqy bi\eqy i$}}&
\z{$58$}&\z{$58$}&\makecell{$\ff\eqy\fb\eqy\ei\eqy\fib\eqy\fbi\eqy\fif\eqy0$,
\vspace{-1pt}\\$\f\eqy\bif\eqy\ie$}&
\multirow{3}{*}{\z{$6$}}&\z{$10$}&\z{$68$}&\multirow{2}{*}{\z{$68$}}&
\multirow{3}{*}{\z{$68$}}&
\multirow{3}{*}{\z{$60$}}\\\cline{4-6}\cline{8-9}
&&&\z{$59$}&\z{$60$}&$(58),\,b\eqy1,\,i\eqy\af$&&\z{$6$}&\z{$69$}&&&\\
\cline{4-6}\cline{8-10}
&&&\z{$60$}&\z{$59$}&$(58),\,b\eqy f,\,i\eqy0$&&\z{$6$}&\z{$68$}&
\multirow{2}{*}{\z{$70$}}&&\\\cline{4-9}\cline{11-12}
&&&\z{$61$}&\z{$61$}&$(60),\,f\eqy1$&\z{$4$}&\z{$4$}&\z{$69$}&&\z{$69$}&\z{$61$}\\\hline
\z{$26$}&\z{$27$}&$(24),\,\ident\eqy b$&\z{$62$}&\z{$63$}&$(57)$&\z{$4$}&\z{$8$}&
\z{$62$}&\z{$63$}&\z{$67$}&\z{$67$}\\\hline
\z{$27$}&\z{$26$}&$(24),\,\ident\eqy i$&\z{$63$}&\z{$62$}&$(57)$&\z{$4$}&\z{$8$}&
\z{$62$}&\z{$63$}&\z{$67$}&\z{$70$}\\\hline
\multirow{2}{*}{\z{$28$}}&\multirow{2}{*}{\z{$29$}}&
\multirow{2}{*}{$(22),\,\ident\eqy i$}&\z{$64$}&\z{$66$}&$(53)$&
\multirow{2}{*}{\z{$4$}}&\z{$8$}&\z{$68$}&\z{$64$}&\multirow{2}{*}{\z{$67$}}&
\multirow{2}{*}{\z{$70$}}\\\cline{4-6}\cline{8-10}
&&&\z{$65$}&\z{$67$}&$(53),\,b\eqy1,\,i\eqy\af$&&\z{$4$}&\z{$69$}&\z{$65$}&&\\\hline
\multirow{2}{*}{\z{$29$}}&\multirow{2}{*}{\z{$28$}}&
\multirow{2}{*}{$(23),\,\ident\eqy b$}&
\z{$66$}&\z{$64$}&$(55)$&
\multirow{2}{*}{\z{$4$}}&\z{$8$}&\z{$66$}&\z{$68$}&\multirow{2}{*}{\z{$67$}}&
\multirow{2}{*}{\z{$67$}}\\\cline{4-6}\cline{8-10}
&&&\z{$67$}&\z{$65$}&$(55),\,b\eqy f,\,i\eqy0$&&\z{$4$}&\z{$67$}&\z{$70$}&&\\\hline
\multirow{3}{*}{\z{$30$}}&\multirow{3}{*}{\z{$30$}}&
\multirow{3}{*}{\makecell{$\ident\eqy bib\eqy bi\eqy$\vspace{-2pt}\\
$b\eqy ibi\eqy ib\eqy i$}}&
\z{$68$}&\z{$68$}&$(58),\,\f\eqy0$&
\multirow{3}{*}{\z{$2$}}&\z{$4$}&\z{$68$}&\z{$68$}&\multirow{3}{*}{\z{$70$}}&
\multirow{3}{*}{\z{$70$}}\\\cline{4-6}\cline{8-10}
&&&\z{$69$}&\z{$70$}&$(68),\,b\eqy1$&&\z{$2$}&\z{$69$}&\z{$69$}&&\\
\cline{4-6}\cline{8-10}
&&&\z{$70$}&\z{$69$}&$(68),\,b\eqy0$&&\z{$2$}&\z{$70$}&\z{$70$}&&\\\hline
\multicolumn{12}{c}{\hspace{0pt}\begin{minipage}[t]{380pt}\specialpar{\footnotesize
\vrule width0pt height11pt*Let $o\in\KFZ$. If $oA=pA$ for some $p\in\KF$ besides $o$,
then $o=p$ appears for at least one such $p$.
Since $o$ and $ao$ never both appear, it follows that $\kf(A)$ equals $34$ minus twice
the number of equal signs that appear.
Implied equations $oA=pA$ hold either because $o=o'$ and $o'=p$ both appear for some
$o'\in\KZ$ or because $o=ao'$ and $o'=ap$ both appear for some $o'\in\FZ$.
The notation ``$(x)$'' either stands for \pn\ $x$ or the remainder in \sn\ $x$.}
\end{minipage}}\\
\end{longtable}
\egroup

\begin{table}[!ht]
\centering
\renewcommand{\arraystretch}{1}
\renewcommand{\tabcolsep}{3pt}
\caption{The smallest $|X|$ such that $A\su X$ exists with $\psi A=n$, up to duality.}
\begin{tabular}{|c|c|c|c|c|}
\multicolumn{1}{c}{$|X|$}&\multicolumn{1}{c}{$\psi$}&\multicolumn{1}{c}{}&
\multicolumn{1}{c}{$|X|$}&\multicolumn{1}{c}{$\psi$}
\\\cline{1-2}\cline{4-5}
$1$\vadj&$69$&&$5$&$30,31,35,39,41,45$\\\cline{1-2}\cline{4-5}
$2$\vadj&$61,65,68$&&$6$&$12,13,24,26,27,28,32,34,38$\\\cline{1-2}\cline{4-5}
$3$\vadj&$54,59,62,64$&&$7$&$5,6,7,8,9,10,11,21,23$\\\cline{1-2}\cline{4-5}
$4$\vadj&$40,42,46,49,51,53,57,58$&&$8$&$1,2,3$\\\cline{1-2}\cline{4-5}
\end{tabular}\label{tab:minimal_psi}
\end{table}

\bgroup
\begin{figure}[!b]
\centering\small
\begin{tikzpicture}
[auto,scale=.8,
 block/.style={rectangle,draw=black,align=center,minimum width={30pt},
 minimum height={16pt},scale=.8},
 topblock/.style={rectangle,draw=black,line width=1pt,align=center,minimum width={30pt},
 minimum height={16pt},scale=.8},
 wideblock/.style={rectangle,draw=black,align=center,minimum height={16pt},scale=.8},
 line/.style={draw}]
	\node[topblock,anchor=south] (1) at (3,0) {$1$};
	\node[topblock,anchor=south] (2) at (0,0) {$2$};
	\node[topblock,anchor=south] (3) at (4.5,0) {$3\com4$};
	\node[topblock,anchor=south] (5) at (6,0) {$5$};
	\node[topblock,anchor=south] (6) at (1.5,0) {$6$};
	\node[block,anchor=south] (7) at (3,.8) {$7\com14$};
	\node[block,anchor=south] (8) at (4.5,.8) {$8\com15$};
	\node[topblock,anchor=south] (9) at (9,0) {$9\com16$};
	\node[block,anchor=south] (10) at (5.7,2.5) {$10\com17$};
	\node[block,anchor=south] (11) at (0,1) {$11\com18$};
	\node[block,anchor=south] (12) at (7.3,.85) {$12\com19$};
	\node[block,anchor=south] (13) at (1.3,1.5) {$13\com20$};
	\node[topblock,anchor=south] (21) at (7.5,0) {$21\com22$};
	\node[topblock,anchor=south] (23) at (10.5,0) {$23$};
	\node[block,anchor=south] (24) at (8.6,.85) {$24\com25$};
	\node[block,anchor=south] (26) at (0,2) {$26$};
	\node[topblock,anchor=south] (27) at (8.7,3.5) {$27$};
	\node[block,anchor=south] (28) at (3.5,2.95) {$28\com29$};
	\node[block,anchor=south] (30) at (10,3.5) {$30$};
	\node[block,anchor=south] (31) at (7.2,2) {$31$};
	\node[block,anchor=south] (32) at (12,.9) {$32\com33$};
	\node[block,anchor=south] (34) at (10.5,.9) {$34\com36$};
	\node[block,anchor=south] (35) at (6,4.4) {$35\com37$};
	\node[topblock,anchor=south] (38) at (12,0) {$38$};
	\node[block,anchor=south] (39) at (3,2.2) {$39$};
	\node[block,anchor=south] (40) at (13.5,1.7) {$40$};
	\node[block,anchor=south] (41) at (12,1.7) {$41\com43$};
	\node[block,anchor=south] (42) at (10,5) {$42\com44$};
	\node[block,anchor=south] (45) at (10.5,2.5) {$45\com47$};
	\node[block,anchor=south] (46) at (12,3.3) {$46\com48$};
	\node[wideblock,anchor=south] (49) at (3,5) {$49\com50\com51\com52$};
	\node[topblock,anchor=south] (53) at (13.5,3.3) {$53\com55$};
	\node[block,anchor=south] (54) at (13.5,5) {$54\com56$};
	\node[block,anchor=south] (57) at (6,5.5) {$57$};
	\node[block,anchor=south] (58) at (11.25,4.2) {$58$};
	\node[block,anchor=south] (59) at (12,6.3) {$59\com60$};
	\node[topblock,anchor=south] (61) at (6,7) {$61$};
	\node[block,anchor=south] (62) at (7.5,6.3) {$62\com63$};
	\node[block,anchor=south] (64) at (12.75,4.2) {$64\com66$};
	\node[block,anchor=south] (65) at (9,7) {$65\com67$};
	\node[block,anchor=south] (68) at (10.5,7) {$68$};
	\node[block,anchor=south] (69) at (7.5,7.5) {$69\com70$};
	\node[anchor=south] (label) at (-.1,6.4) [label={[label distance=-5pt]0:(see
	Proposition~\ref{prop:always_together})}] {};
	\node (tail) at (2.6,5.9) {};
	\draw [->, transform canvas={xshift=-20pt,yshift=12pt}] (tail) -- ($(tail)!22pt!(49)$);
	\draw[line] (1.90) -- (7.-90);
	\draw[line] (2.90) -- (11.-90);
	\draw[line] (2.north east) -- (39.-90);
	\draw[line] (3.90) -- (8.-90);
	\draw[line] (3.north east) -- (10.-90);
	\draw[line] (5.north east) -- (12.south west);
	\draw[line] (5.50) -- (39.0);
	\draw[line] (6.124) -- (13.-90);
	\draw[line] (6.90) -- (39.-90);
	\draw[line] (7.north east) -- (35.south west);
	\draw[line] (7.90) -- (39.-90);
	\draw[line] (7.north east) -- (57.south west);
	\draw[line] (8.90) -- (28.-90);
	\draw[line] (8.90) -- (35.south west);
	\draw[line] (8.90) -- (39.-90);
	\draw[line] (9.135) -- (24.-70);
	\draw[line] (10.180) -- (28.0);
	\draw[line] (10.44) -- (35.-90);
	\draw[line] (11.90) -- (26.-90);
	\draw[line] (12.north east) -- (30.south west);
	\draw[line] (12.north west) -- (35.-90);
	\draw[line] (13.0) -- (31.180);
	\draw[line] (13.90) -- (49.210);
	\draw[line] (21.north east) -- (32.south west);
	\draw[line] (21.north east) -- (34.south west);
	\draw[line] (21.north west) -- (39.0);
	\draw[line] (23.north west) -- (35.south east);
	\draw[line] (23.north west) -- (57.south east);
	\draw[line] (24.90) -- (35.south east);
	\draw[line] (24.90) -- (57.south east);
	\draw[line] (26.north east) -- (35.180);
	\draw[line] (26.north east) -- (57.south west);
	\draw[line] (27.north east) -- (42.-90);
	\draw[line] (27.200) -- (49.210);
	\draw[line] (27.north west) -- (57.south east);
	\draw[line] (28.0) -- (42.180);
	\draw[line] (28.90) -- (49.210);
	\draw[line] (28.90) -- (57.south west);
	\draw[line] (30.90) -- (42.-90);
	\draw[line] (31.18) -- (46.174);
	\draw[line] (31.60) -- (62.-120);
	\draw[line] (32.north west) -- (35.south east);
	\draw[line] (32.north east) -- (40.south west);
	\draw[line] (32.90) -- (41.-90);
	\draw[line] (34.north west) -- (35.south east);
	\draw[line] (34.north east) -- (41.south west);
	\draw[line] (35.0) -- (42.180);
	\draw[line] (35.180) -- (49.south east);
	\draw[line] (38.north east) -- (40.-90);
	\draw[line] (38.north west) -- (45.south east);
	\draw[line] (39.110) -- (49.210);
	\draw[line] (40.north west) -- (64.-90);
	\draw[line] (41.90) -- (42.-90);
	\draw[line] (41.north east) -- (64.-90);
	\draw[line] (42.north west) -- (62.south east);
	\draw[line] (45.north east) -- (46.south west);
	\draw[line] (45.90) -- (58.-90);
	\draw[line] (46.90) -- (59.-90);
	\draw[line] (46.90) -- (64.-90);
	\draw[line] (49.90) -- (62.180);
	\draw[line] (53.90) -- (54.-90);
	\draw[line] (53.90) -- (64.-90);
	\draw[line] (54.180) -- (65.-90);
	\draw[line] (57.north east) -- (62.south west);
	\draw[line] (58.90) -- (59.-90);
	\draw[line] (59.180) -- (68.-90);
	\draw[line] (61.north east) -- (69.180);
	\draw[line] (62.north east) -- (65.south west);
	\draw[line] (64.north west) -- (65.-90);
	\draw[line] (64.north west) -- (68.-90);
	\draw[line] (65.north west) -- (69.0);
	\draw[line] (68.north west) -- (69.20);
\end{tikzpicture}
\caption{The relation $(A\su X\mbox{ satisfies }\psi A=m)\implies(B\su X\mbox{ exists
with }\psi B=n)$.}
\label{fig:psi-implications}
\end{figure}

\begin{table}[!hb]
\caption{Evidence supporting Figure~\ref{fig:psi-implications} in all spaces of
cardinality $\leq11$.}
\centering\small
\renewcommand{\arraystretch}{.9}
\renewcommand{\tabcolsep}{2pt}
\renewcommand{\vadj}{\vrule width0pt height8pt}
\begin{tabular}[t]{|c|c|c|c|c|c|c|c|c|c|c|}
\multicolumn{2}{c}{}&\multicolumn{1}{c}{for some}&\multicolumn{8}{c}{}\\
\multicolumn{1}{c}{$\psi A$}&\multicolumn{1}{c}{$\psi B$}&\multicolumn{1}{c}{$B$ among:}&
\multicolumn{1}{c}{}&\multicolumn{1}{c}{$\psi A$}&\multicolumn{1}{c}{$\psi B$}&
\multicolumn{1}{c}{$B$}&\multicolumn{1}{c}{}&\multicolumn{1}{c}{$\psi A$}&
\multicolumn{1}{c}{$\psi B$}&\multicolumn{1}{c}{$B$}\\
\Xcline{1-3}{2\arrayrulewidth}\Xcline{5-7}{2\arrayrulewidth}
\Xcline{9-11}{2\arrayrulewidth}
\multirow{2}{*}{\y{$1$}}&\multirow{2}{*}{\y{$7$}}&$A\sm\ffA$,&&\y{$13$}&
\y{$31$}&
$A\cup biA$&&\y{$34$}&\y{$41$}&$A\cup biA$\vrule width0pt height9pt depth0pt\\
\cline{5-7}\cline{9-11}
&&$aA\sm\ffA$&&\y{$13$}&\y{$51$}\vadj&$iA$&&\y{$35$}&\y{$44$}&
Proposition~\ref{prop:GE_characterization}\\\cline{1-3}\cline{5-7}\cline{9-11}
\y{$2$}&\y{$11$}&$A\cap ibA$&&\y{$21$}&\y{$32$}\vadj&$A\cup biA$&&\y{$35$}&\y{$51$}&$iA$\\
\cline{1-3}\cline{5-7}\cline{9-11}
\y{$2$}&\y{$39$}&$A\cup\ifA$&&\y{$21$}&\y{$34$}\vadj&$A\cap ibA$&&\y{$38$}&\y{$40$}&$A\cup
\ifA$\\\cline{1-3}\cline{5-7}\cline{9-11}
\y{$3$}&\y{$8$}&$A\cap ibA$&&\y{$21$}&\y{$39$}\vadj&$iA\cup\fbA$&&\y{$38$}&\y{$45$}&$A\cap
ibA$\\\cline{1-3}\cline{5-7}\cline{9-11}
\y{$3$}&\y{$10$}&$(A\cap\ifA)\cup iaA$&&\y{$23$}&\y{$35$}\vadj&$(A\cap ibA)\cup iaA$&&
\y{$39$}&\y{$51$}&$iA$\\\cline{1-3}\cline{5-7}\cline{9-11}
\y{$5$}&\y{$12$}&$A\cap ibA$&&\y{$23$}&\y{$57$}\vadj&$ibA\cup biA$&&\y{$40$}&\y{$64$}&$iA$
\\\cline{1-3}\cline{5-7}\cline{9-11}
\y{$5$}&\y{$39$}&$A\cup\ifA$&&\y{$24$}&\y{$35$}\vadj&$(A\cap ibA)\cup iaA$&&\y{$41$}&
\y{$44$}&Proposition~\ref{prop:GE_characterization}\\\cline{1-3}\cline{5-7}\cline{9-11}
\y{$6$}&\y{$13$}&$A\cap ibA$&&\y{$24$}&\y{$57$}\vadj&$ibA\cup biA$&&\y{$41$}&\y{$64$}&
$iA\cup\ifA$\\\cline{1-3}\cline{5-7}\cline{9-11}
\y{$6$}&\y{$39$}&$A\cup\ifA$&&\y{$26$}&\y{$35$}\vadj&$(A\cap ibA)\cup iaA$&&\y{$42$}&
\y{$62$}&Lemma~\ref{lem:k-number_26}\\\cline{1-3}\cline{5-7}\cline{9-11}
\y{$7$}&\y{$37$}&$gaA$&&\y{$26$}&\y{$57$}\vadj&$ibA\cup\bifA$&&\y{$45$}&\y{$48$}&$gaA$\\
\cline{1-3}\cline{5-7}\cline{9-11}
\y{$7$}&\y{$39$}&$iA\cup\fbA$&&\y{$27$}&\y{$44$}\vadj&
Proposition~\ref{prop:GE_characterization}&&\y{$45$}&\y{$58$}&$A\cup biA$\\
\cline{1-3}\cline{5-7}\cline{9-11}
\multirow{3}{*}{\y{$7$}\vadj}&\multirow{3}{*}{\y{$57$}\vadj}&$ibA\cap \bifA$,&&\y{$27$}&
\y{$51$}\vadj&$iA\cup\ifA$&&\y{$46$}&\y{$60$}&$A\cap\ifA$\\\cline{5-7}\cline{9-11}
&&$ibA\cup\bifA$,&&\y{$27$}&\y{$57$}\vadj&$\ifA\cup\fiA$&&\y{$46$}&\y{$64$}&$iA$\\
\cline{5-7}\cline{9-11}
&&$ib(aA)\cup\bifA$&&\y{$28$}&\y{$44$}\vadj&Proposition~\ref{prop:GE_characterization}&&
\y{$49$}&\y{$62$}&$bA$\\\cline{1-3}\cline{5-7}\cline{9-11}
\y{$8$}&\y{$28$}&$A\cup biA$&&\y{$28$}&\y{$51$}\vadj&$\ifA\cup iaA$&&\y{$53$}&\y{$54$}&
$agaA$\\\cline{1-3}\cline{5-7}\cline{9-11}
\y{$8$}&\y{$37$}&$gaA$&&\y{$28$}&\y{$57$}\vadj&$ibA\cup biA$&&\y{$53$}&\y{$64$}&$iA$\\
\cline{1-3}\cline{5-7}\cline{9-11}
\y{$8$}&\y{$39$}&$iA\cup\fibA$&&\y{$30$}&\y{$44$}\vadj&
Proposition~\ref{prop:GE_characterization}&&\y{$54$}&\y{$65$}&$iA$\\
\cline{1-3}\cline{5-7}\cline{9-11}
\y{$9$}&\y{$24$}&$(A\cap\ifA)\cup iaA$&&\y{$31$}&\y{$46$}\vadj&$(A\cap ibA)\cup iaA$&&
\y{$57$}&\y{$62$}&$bA$\\\cline{1-3}\cline{5-7}\cline{9-11}
\y{$10$}&\y{$28$}&$(A\cap\ifA)\cup iaA$&&\y{$31$}&\y{$62$}\vadj&
Lemma~\ref{lem:k-number_26}&&\y{$58$}&\y{$60$}&$gA$\\
\cline{1-3}\cline{5-7}\cline{9-11}
\y{$10$}&\y{$37$}&$gaA$&&\y{$32$}&\y{$37$}\vadj&$gA$&&\y{$59$}&\y{$68$}&$iA$\\
\cline{1-3}\cline{5-7}\cline{9-11}
\y{$11$}&\y{$26$}&$A\cup biA$&&\y{$32$}&\y{$40$}\vadj&$\fiA\cup iaA$&&\y{$62$}&\y{$67$}&
$\fA$\\\cline{1-3}\cline{5-7}\cline{9-11}
\y{$12$}&\y{$30$}&$A\cup biA$&&\y{$32$}&\y{$41$}\vadj&$A\cap ibA$&&\y{$64$}&\y{$67$}&
$\fA$\\\cline{1-3}\cline{5-7}\cline{9-11}
\y{$12$}&\y{$37$}&$gaA$&&\y{$34$}&\y{$37$}\vadj&$gaA$&&\y{$64$}&\y{$68$}&$bA$\\
\cline{1-3}\cline{5-7}\cline{9-11}
\end{tabular}
\label{tab:psi-implications}
\end{table}
\egroup

Since $ig=0$ we have $\psi(gA)\in\{37,44,48,56,60,61,67,70\}$.
Clearly $\psi(gA)=70\iff A=iA$.
Since $\ifg=\ie$ we have $\ie(gA)=X\iff\ifA=X$.
Hence $\psi(gA)=61\iff\psi A=61$.
Suppose $\psi(gA)\cn\in\{61,70\}$.
Since $\fifA=\es\iff\fib(gA)=i(gA)$ and $\bifA=\es\iff \bif(gA)=i(gA)$ we get $\fifA\neq
\es\iff\psi(gA)\in\{37,44\}$, $\fifA=\es\neq\bifA\iff\psi(gA)\in\{48,60\}$, and $\bifA=
\es\iff\psi(gA)\in\{56,67\}$.
Have $b(gA)=\bif(gA)=\bifA\implies bibA=bA$ by Lemma~\ref{lem:equations_kf}(xii).
Thus $(\fifA\neq\es\mbox{ and }bibA\neq bA)\implies\psi(gA)=37$ and $(\fifA=\es\neq\bifA
\mbox{ and }bibA\neq bA)\implies\psi(gA)=48$.
By $\Ord\KFG$, $\fA=\bifA\implies b(gA)=\bifA=\bif(gA)$.
Thus $(\fifA\neq\es\mbox{ and }\fA=\bifA)\implies\psi(gA)=44$ and
$\fifA=\es\neq\bifA=\fA\implies\psi(gA)=60$.
Note that $(b(gA)=(gA)\mbox{ and }biA=iA)\implies A=gA\cup iA=bgA\cup biA=bA$.
It follows that $\phi A=23\implies\psi(gA)=56$.
Since $A=bA\implies gA=gbA=\fbA\implies(gA)=b(gA)$ it follows that $A=bA\neq iA\implies
\psi(gA)=67$.
We verified by computer that in all remaining cases, both possible values of $\psi(gA)$
occur.
\end{proof}

Table~\ref{tab:subset_types} shows that every even number from $2$ to $34$ occurs as the
value of $\kf(A)$ for some $A$.
There are $16$ self-dual \sn s and $27$ dual pairs.
The following \sn s represent collapses of $\KF$ that are both local and global:
$1$~(GE), $6$~(KD), $38$~(ED), $39$~(OU), $40$~(EO), $58$~(P), $68$~(D).

Non-Kuratowski spaces satisfy equations that exclude certain \pn s.
Up to duality and excluding \pn\ $30$, the possibilities are:
ED $11$, $13$, $16$, $22$, $25$, $28$;
OU $12$, $13$, $18$, $20$, $22$, $24$, $26$, $28$;
EO\ $13$, $22$, $28$;
P $25$;
D (none).

We verified that all \sn s extended from these \pn s occur, thus no \sn\ is excluded
nontrivially from any non-Kuratowski space type.
However, \sn\ $61$ is nontrivially excluded from KD spaces
(see Proposition~\ref{prop:irresolvable}).
The possible \sn s in KD spaces are thus: $6$, $13$, $20$, $31$, $38$-$40$, $45$-$60$,
$62$-$70$.
We verified that each occurs in some KD space.

For each $n\leq11$ we computed the relative frequencies of \sn s over all nonhomeomorphic
spaces on $n$ points.
The results point strongly to $38$ as the asymptotically rarest \sn\ and $49$ the
commonest (in a tie with its dual, \sn\ $50$) as $n\to\infty$.
The same conclusions hold for the corresponding \pn s, $11$ and $\{18,19\}$.
Table~\ref{tab:minimal_psi} shows the minimum space cardinality required for any given
\sn\ to occur.

Let $m\preceq n$ if and only if $(A\su X\mbox{ satisfies }\psi A=m)\implies(B\su X
\mbox{ exists with }\psi B=n)$.
We verified that the relation $\preceq$ is contained in the partial order in
Figure~\ref{fig:psi-implications}.
The two relations must surely be equal; entries with Boolean set operators in
Table~\ref{tab:psi-implications} are posed to the reader as challenging exercises.

The next corollary formally justifies our use of the term \define{completely full}.

\begin{cor}\label{cor:completely_full}
Every completely full space is full\nit.
The converse is not true in general\nit.
\end{cor}

\begin{proof}
Using Table~\ref{tab:subset_types} it is easy to verify that $\kf(\XT)=4,10,16,20,22,34$
implies $k(\XT)\geq2,6,8,10,10,14$, respectively, and in KD spaces $\kf(A)=28\implies\psi
A=6\implies k(A)=14$.
Thus the first sentence holds.
The minimal non-indiscrete partition space is full but not completely full (see
Table~\ref{tab:psi_topsum}).
\end{proof}

The next corollary gives every pair $o_1,o_2$ in $\KFZ$ such that $ao_1A\neq o_2A$ holds
in general.

\begin{cor}\label{cor:excluded}
\nit{(i)}~\mth{$ao_1A\nsu o_2A$} if \mth{$o_1,o_2\in{\downarrow}\{\ff\hspace{.6pt}\}$.}
\nit{(ii)}~\mth{$ao_1A\neq o_2A$} if \mth{$o_1=\ident$} and \mth{$o_2\in\KZ\sm\{\ident\}$}
or \mth{$o_1,o_2$} both belong to \mth{$\{i,ibi,bi\}$,} \mth{$\{ib,bib,b\}$,} or
\mth{$\{\ie,\bif,\f\hspace{1pt}\}$.}
\end{cor}

\begin{proof}
(i)~Since the inclusions $\affA\su ao_1A\su o_2A\su\ffA$ imply $\ffA=X$ the result
holds by Lemma~\ref{lem:equations_kf}(x).
(ii)~Suppose $aA=oA$ where $o\in\{i,ibi,bi\}$.
Since $iA\su oA\su aA\implies iA=\es\implies oA=\es\implies A=X\implies oA\neq\es$ we
conclude $aA\neq oA$.
The same holds for $o\in\{ib,bib,b\}$ by a dual argument.

In the remaining cases we can assume $o_1\leq o_2$.
Suppose $ao_1A=o_2A$.
Then $o_1A\su o_2A\cap ao_2A=\es$.
Hence $o_2A=X$.
Thus $(\KZ o_1)A=\{\es\}$ and $(\KZ o_2)A=\{X\}$.
But $o_1\in\KZ o_2$ or $o_2\in\KZ o_1$, contradicting $X\neq\es$.
\end{proof}

A topological space is connected if and only if it contains no subset with \sn\,~$68$.
It is resolvable (see Hewitt \cite{1943_hewitt}) if and only if it contains a subset with
\sn\,~$61$ (otherwise it is irresolvable).
A non-clopen subset $A$ is regular closed if and only if $\phi A=26$ and regular open if
and only if $\phi A=27$.

\subsection{Inclusions in \texorpdfstring{$\KFA$}{KFA}}\label{sub:inclusions}

We now find all local orderings on $\KF$.

\begin{dfn}\label{def:optional}
Suppose \emph{$A$} satisfies a collapse \emph{$C$} of \emph{$\Y$} and \emph{$o_1A\su
o_2A$} \emph{$(o_1,o_2\in\Y)$.}
If \emph{$C$} implies the inclusion we call the inequality \emph{$o_1\leq o_2$} a
\emph{base inequality} of \emph{$C$,} otherwise we call it \emph{optional with respect to
$C$} and say it is \emph{optionally satisfied} by \emph{$A$.}
These definitions apply similarly to partial orders on \emph{$\Y$.}
\end{dfn}

Figure~\ref{fig:subset_types} gives all disjointness and base inequalities in $\KFZ$ for
each of the $70$ local collapses of $\KF$.

\begin{prop}\label{prop:kfz_inequalities}
\mth{$\KFZ$} has exactly \mth{$274$} local orderings\nit.
\end{prop}

\bgroup
\setlength\tabcolsep{0pt}
\renewcommand{\hd}{\hspace{-1pt},}
\renewcommand{\f}{f\hskip-1pt}
\renewcommand{\ff}{f\hskip-2ptf}
\renewcommand{\fff}{f\hskip-2ptf\hskip-2ptf}
\renewcommand{\ffA}{f\hskip-2ptf\hskip-1.5ptA}
\renewcommand{\ei}{f\hskip-1pti}
\renewcommand{\ie}{i\hskip-1ptf}
\renewcommand{\fif}{f\hskip-1pti\hskip-1ptf}
\renewcommand{\fifA}{f\hskip-1pti\hskip-1ptf\hskip-1.5ptA}
\renewcommand{\ifA}{i\hskip-1ptf\hskip-1.5ptA}
\renewcommand{\fA}{f\hskip-1.5ptA}
\renewcommand{\fb}{f\hskip-.5ptb}
\begin{figure}
\tiny


\medskip
\caption{Disjointness and base inequalities in $\KFZ$ for the $70$ \sn s.\\
\smallskip\small\begin{tabular}{l}
The only base inequalities not implied by $\Ord\KFZ$ are for \sn s\\
$2,3,4,8,10,11,15,17,18$ by
Proposition~\ref{prop:order_implications}(v) and $9,16$ by
Lemma~\ref{lem:subset_implications}(iv)-(v).
\end{tabular}}\label{fig:subset_types}
\end{figure}
\egroup

\bgroup
\captionsetup[table]{labelfont=sc,labelsep=period,margin={0.02\textwidth,0.02\textwidth},
justification=justified}
\begin{table}[!ht]
\caption{For incomparable, non-disjoint $o_1,o_2\in\KFZ$, the inclusion $o_1A\su o_2A$
only depends on $\psi A$ and\\\centering$I_{opt}=\{(\fbi,\fib),(\fif,\fbi),(\fib,\fif),
(bi, ib),(\fif, ib),(\ident, ib),(\fbi,\ident),(\fif,\ident),(\fib,\ident),(\fb,\ident),
(\ff,\ident)\}$.}
\centering\small
\renewcommand{\arraystretch}{1}
\renewcommand{\tabcolsep}{3pt}
\begin{tabular}{ccc}
\begin{tabular}[t]{|c|c|c|c|}
\multicolumn{1}{c}{row}&\multicolumn{1}{c}{$o_1$}&\multicolumn{1}{c}{$o_2$}&
\multicolumn{1}{c}{$o_1A\su o_2A\ify$}\\\Xhline{2\arrayrulewidth}
$1$&$\ident$&$bib$\vadj&$bibA=bA$\\\hline
$2$&$\f$&$bib$\vadj&$bibA=bA$\\\hline
$3$&$\ff$&$bib$\vadj&$bibA=bA$\\\hline
$4$&$\fb$&$bib$\vadj&$bibA=bA$\\\hline
$5$&$ib$&$bi$\vadj&$\ifA=\es$\\\hline
$6$&$\f$&$bi$\vadj&$biA=bA$\\\hline
$7$&$\bif$&$bi$\vadj&$\ifA=\es\mbox{ and }\fifA\su\fbiA$\\\hline
$8$&$\ident$&$bi$\vadj&$biA=bA$\\\hline
$9$&$\ff$&$bi$\vadj&$\ffA=\fiA$\\\hline
$10$&$\fb$&$bi$\vadj&$\ffA=\fiA$\\\hline
$11$&$\fib$&$bi$\vadj&$\fifA\su\fbiA$\\\hline
$12$&$\fif$&$bi$\vadj&$\fifA\su\fbiA$\\\hline
$13$&$\f$&$ib$\vadj&$ibA=bA$\\\hline
$14$&$\ff$&$ib$\vadj&$ibA=bA$\\\hline
$15$&$\ei$&$ib$\vadj&$biA\su ibA$\\\hline
$16$&$\fbi$&$ib$\vadj&$biA\su ibA$\\\hline
$17$&$\bif$&$ib$\vadj&$\fifA\su ibA$\\\hline
$18$&$\f$&$ibi$\vadj&$ibiA=bA$\\\hline
$19$&$\ff$&$ibi$\vadj&$ibiA=biA\mbox{ and }\ffA=\fiA$\\\hline
$20$&$\ei$&$ibi$\vadj&$ibiA=biA$\\\hline
$21$&$\ident$&$ibi$\vadj&$A\su ibA$\\\hline
$22$&$bib$&$\ident$\vadj&$ibiA=iA\mbox{ and }\fbiA\su A$\\\hline
$23$&$bi$&$\ident$\vadj&$ibiA=iA\mbox{ and }\fbiA\su A$\\\hline
$24$&$ib$&$\ident$\vadj&$ibA=iA$\\\hline
$25$&$ibi$&$\ident$\vadj&$ibiA=iA$\\\hline
$26$&$\f$&$\ident$\vadj&$A=bA$\\\hline
$27$&$\ei$&$\ident$\vadj&$ibiA=iA\mbox{ and }\fbiA\su A$\\\hline
\end{tabular}&
\begin{tabular}[t]{|c|c|c|}
\multicolumn{1}{c}{$o_1$}&\multicolumn{1}{c}{$o_2$}&
\multicolumn{1}{c}{$o_1A\su o_2A\ify$}\\\Xhline{2\arrayrulewidth}
$\ie$&$\ident$\vadj&$\ifA=\es$\\\hline
$\bif$&$\ident$\vadj&$\ifA=\es\mbox{ and }\fifA\su A$\\\hline
$bib$&$\f$\vadj&$ibA=\ifA$\\\hline
$bi$&$\f$\vadj&$ibiA=\es$\\\hline
$ib$&$\f$\vadj&$ibA=\ifA$\\\hline
$ibi$&$\f$\vadj&$ibiA=\es$\\\hline
$\ident$&$\f$\vadj&$bA=\fA$\\\hline
$bib$&$\ff$\vadj&$ibA=\es$\\\hline
$bi$&$\ff$\vadj&$ibiA=\es$\\\hline
$ib$&$\ff$\vadj&$ibA=\es$\\\hline
$ibi$&$\ff$\vadj&$ibiA=\es$\\\hline
$\ident$&$\ff$\vadj&$bA=\ffA$\\\hline
$\bif$&$\ff$\vadj&$\ifA=\es$\\\hline
$bib$&$\fb$\vadj&$ibA=\es$\\\hline
$bi$&$\fb$\vadj&$ibiA=\es$\\\hline
$\ident$&$\fb$\vadj&$bA=\fbA$\\\hline
$\bif$&$\fb$\vadj&$\ifA=\es\mbox{ and }\fbiA\su\fibA$\\\hline
$\ei$&$\fb$\vadj&$\ffA=\fbA$\\\hline
$\fbi$&$\fb$\vadj&$\fbiA\su\fibA$\\\hline
$\fif$&$\fb$\vadj&$\fbiA\su\fibA$\\\hline
$bi$&$\fib$\vadj&$ibiA=\es$\\\hline
$\ident$&$\fib$\vadj&$bA=\fibA$\\\hline
$\bif$&$\fib$\vadj&$\ifA=\es\mbox{ and }\fbiA\su\fibA$\\\hline
$\ei$&$\fib$\vadj&$\ffA=\fbA$\\\hline
$\fif$&$\fib$\vadj&$\fbiA\su\fibA$\\\hline
$ib$&$\ei$\vadj&$ibA=\es$\\\hline
$ibi$&$\ei$\vadj&$ibiA=\es$\\\hline
\end{tabular}&
\begin{tabular}[t]{|c|c|c|}
\multicolumn{1}{c}{$o_1$}&\multicolumn{1}{c}{$o_2$}&
\multicolumn{1}{c}{$o_1A\su o_2A\ify$}\\\Xhline{2\arrayrulewidth}
$\ident$&$\ei$\vadj&$bA=\fiA$\\\hline
$\bif$&$\ei$\vadj&$\ifA=\es\mbox{ and }\fifA\su\fbiA$\\\hline
$\fb$&$\ei$\vadj&$\ffA=\fiA$\\\hline
$\fib$&$\ei$\vadj&$\fifA\su\fbiA$\\\hline
$\fif$&$\ei$\vadj&$\fifA\su\fbiA$\\\hline
$ib$&$\fbi$\vadj&$ibA=\es$\\\hline
$\ident$&$\fbi$\vadj&$bA=\fbiA$\\\hline
$\bif$&$\fbi$\vadj&$\ifA=\es\mbox{ and }\fifA\su\fbiA$\\\hline
$\fb$&$\fbi$\vadj&$\ffA=\fiA$\\\hline
$\fib$&$\fbi$\vadj&$\fifA\su\fbiA$\\\hline
$\ident$&$\ie$\vadj&$iA=\es\mbox{ and }A\su ibA$\\\hline
$bi$&$\bif$\vadj&$ibiA=\es\mbox{ and }\fibA\su\fifA$\\\hline
$ib$&$\bif$\vadj&$ibA=\ifA$\\\hline
$\ident$&$\bif$\vadj&$bA=\bifA$\\\hline
$\ff$&$\bif$\vadj&$\ffA=\fifA$\\\hline
$\fb$&$\bif$\vadj&$bibA=bA\mbox{ and }\fibA\su\fifA$\\\hline
$\fib$&$\bif$\vadj&$\fibA\su\fifA$\\\hline
$\ei$&$\bif$\vadj&$ibiA=iA\mbox{ and }\fibA\su\fifA$\\\hline
$\fbi$&$\bif$\vadj&$\fibA\su\fifA$\\\hline
$bi$&$\fif$\vadj&$ibiA=\es$\\\hline
$ib$&$\fif$\vadj&$ibA=\es$\\\hline
$\ident$&$\fif$\vadj&$bA=\fifA$\\\hline
$\fb$&$\fif$\vadj&$bibA=bA\mbox{ and }\fibA\su\fifA$\\\hline
$\fbi$&$\fif$\vadj&$\fibA\su\fifA$\\\hline
$\ei$&$\fif$\vadj&$ibiA=iA\mbox{ and }\fibA\su\fifA$\\\hline
\end{tabular}\end{tabular}\label{tab:incomparable}
\end{table}
\egroup

\begin{proof}We verified by computer that $274$ occur.
Claim: (i)~Whether $A$ satisfies $o_1\leq o_2$ in $\KFZ$ is determined by $\psi A$ and
the action of $I_{opt}$ on $A$ (see Table~\ref{tab:incomparable}).
\emph{Note}. It is easy to check that no inequality in $I_{opt}$ is optional with
respect to any \sn\ $>50$, hence each generates only one local ordering on $\KFZ$.
(ii)~Table~\ref{tab:optional} is correct.
(iii)~If $\psi A\leq50$ every possible collection of inequalities in $\KFZ$ satisfied
optionally by $A$ is represented by some combination of zero or one parenthesized lists
from column~$1$ with the same from column~$2$ in Table~\ref{tab:ineqs}.

(i)~We can assume $o_1,o_2$ are neither disjoint nor comparable, for $o_1\leq o_2\implies
o_1A\su o_2A$, $o_2\leq o_1\implies(o_1A\su o_2A\iff o_1A=o_2A)$, $(o_1\wedge o_2=0
\mbox{ and }o_1A\neq\es)\implies o_1A\nsu o_2A$, and $o_1A=\es\implies o_1A\su o_2A$.
The result follows by Table~\ref{tab:incomparable}.

(ii)~We verified by computer that each non-blank entry in Table~\ref{tab:optional}
occurs optionally and claim that no blank ones do.

Let $\Y=\{\fib,\fif,\fbi\}$.
Note that the \sn s without an equation in $\Y$ are: $1$, $7$, $9$, $14$, $16$,
$23$-$25$, $27$.
By Proposition~\ref{prop:order_implications}(v) when a subset satisfies one or more
equations in $\Y$ the inequalities in $\Y$ it satisfies are determined.
When no equations in $\Y$ are satisfied then either zero or two inequalities in $\Y$ are.

\bgroup
\begin{table}[!ht]
\vspace{10pt}
\caption{Optional inequalities in $I_{opt}$ (see Table~\ref{tab:incomparable}) with
respect to \sn s $\leq50$.}
\renewcommand{\arraystretch}{1}
\renewcommand{\tabcolsep}{.8pt}
\small
\centering\vspace{-6pt}
\begin{tabular}{rc|c|c|c|c|c|c|c|c|c|c|c|}
\multicolumn{2}{c}{}&\multicolumn{11}{c}{\lower3pt\hbox{\hspace{26pt}$\overbrace{\vrule width340pt height0pt depth0pt}^{\mbox{\small$\lfloor\psi A/10\rfloor$}}$}}\\
\multicolumn{1}{c}{}
&\multicolumn{1}{c}{\h}&\multicolumn{1}{c}{$0$}
&\multicolumn{1}{c}{\hz}&\multicolumn{1}{c}{$1$}
&\multicolumn{1}{c}{\hz}&\multicolumn{1}{c}{$2$}
&\multicolumn{1}{c}{\hz}&\multicolumn{1}{c}{$3$}
&\multicolumn{1}{c}{\hz}&\multicolumn{1}{c}{$4$}
&\multicolumn{1}{c}{\hz}&\multicolumn{1}{c}{$5$}\\
\cline{3-3}\cline{5-5}\cline{7-7}\cline{9-9}\cline{11-11}\cline{13-13}
\p{\fbi\leq\fib}
&\h&\begin{tabular}{ccccccccc}\p{1}&\ph{2}&\ph{3}&\ph{4}&\ph{5}&\ph{6}&\p{7}&\ph{8}&\ph{9}\end{tabular}
&\hz&\begin{tabular}{cccccccccc}\ph{0}&\ph{1}&\ph{2}&\ph{3}&\ph{4}&\ph{5}&\ph{6}&\ph{7}&\ph{8}&\ph{9}\end{tabular}
&\hz&\begin{tabular}{cccccccccc}\ph{0}&\ph{1}&\ph{2}&\ph{3}&\ph{4}&\ph{5}&\ph{6}&\ph{7}&\ph{8}&\ph{9}\end{tabular}
&\hz&\begin{tabular}{cccccccccc}\ph{0}&\ph{1}&\ph{2}&\ph{3}&\ph{4}&\ph{5}&\ph{6}&\ph{7}&\ph{8}&\ph{9}\end{tabular}
&\hz&\begin{tabular}{cccccccccc}\ph{0}&\ph{1}&\ph{2}&\ph{3}&\ph{4}&\ph{5}&\ph{6}&\ph{7}&\ph{8}&\ph{9}\end{tabular}
&\hz&\begin{tabular}{c}\ph{0}\end{tabular}\\
\p{\fif\leq\fbi}
&\h&\begin{tabular}{ccccccccc}\p{1}&\ph{2}&\ph{3}&\ph{4}&\ph{5}&\ph{6}&\ph{7}&\ph{8}&\ph{9}\end{tabular}
&\hz&\begin{tabular}{cccccccccc}\ph{0}&\ph{1}&\ph{2}&\ph{3}&\p{4}&\ph{5}&\ph{6}&\ph{7}&\ph{8}&\ph{9}\end{tabular}
&\hz&\begin{tabular}{cccccccccc}\ph{0}&\ph{1}&\ph{2}&\ph{3}&\ph{4}&\ph{5}&\ph{6}&\ph{7}&\ph{8}&\ph{9}\end{tabular}
&\hz&\begin{tabular}{cccccccccc}\ph{0}&\ph{1}&\ph{2}&\ph{3}&\ph{4}&\ph{5}&\ph{6}&\ph{7}&\ph{8}&\ph{9}\end{tabular}
&\hz&\begin{tabular}{cccccccccc}\ph{0}&\ph{1}&\ph{2}&\ph{3}&\ph{4}&\ph{5}&\ph{6}&\ph{7}&\ph{8}&\ph{9}\end{tabular}
&\hz&\begin{tabular}{c}\ph{0}\end{tabular}\\
\p{\fib\leq\fif}
&\h&\begin{tabular}{ccccccccc}\p{1}&\ph{2}&\ph{3}&\ph{4}&\ph{5}&\ph{6}&\p{7}&\ph{8}&\ph{9}\end{tabular}
&\hz&\begin{tabular}{cccccccccc}\ph{0}&\ph{1}&\ph{2}&\ph{3}&\p{4}&\ph{5}&\ph{6}&\ph{7}&\ph{8}&\ph{9}\end{tabular}
&\hz&\begin{tabular}{cccccccccc}\ph{0}&\ph{1}&\ph{2}&\ph{3}&\ph{4}&\ph{5}&\ph{6}&\ph{7}&\ph{8}&\ph{9}\end{tabular}
&\hz&\begin{tabular}{cccccccccc}\ph{0}&\ph{1}&\ph{2}&\ph{3}&\ph{4}&\ph{5}&\ph{6}&\ph{7}&\ph{8}&\ph{9}\end{tabular}
&\hz&\begin{tabular}{cccccccccc}\ph{0}&\ph{1}&\ph{2}&\ph{3}&\ph{4}&\ph{5}&\ph{6}&\ph{7}&\ph{8}&\ph{9}\end{tabular}
&\hz&\begin{tabular}{c}\ph{0}\end{tabular}\\
\cline{3-3}\cline{5-5}\cline{7-7}\cline{9-9}\cline{11-11}\cline{13-13}
\p{bi\leq ib}
&\h&\begin{tabular}{ccccccccc}\p{1}&\ph{2}&\ph{3}&\ph{4}&\ph{5}&\ph{6}&\p{7}&\ph{8}&\ph{9}\end{tabular}
&\hz&\begin{tabular}{cccccccccc}\ph{0}&\ph{1}&\ph{2}&\ph{3}&\p{4}&\ph{5}&\ph{6}&\ph{7}&\ph{8}&\ph{9}\end{tabular}
&\hz&\begin{tabular}{cccccccccc}\ph{0}&\ph{1}&\ph{2}&\ph{3}&\ph{4}&\ph{5}&\ph{6}&\p{7}&\ph{8}&\ph{9}\end{tabular}
&\hz&\begin{tabular}{cccccccccc}\ph{0}&\ph{1}&\ph{2}&\ph{3}&\ph{4}&\ph{5}&\ph{6}&\ph{7}&\ph{8}&\ph{9}\end{tabular}
&\hz&\begin{tabular}{cccccccccc}\ph{0}&\ph{1}&\ph{2}&\ph{3}&\ph{4}&\ph{5}&\ph{6}&\ph{7}&\ph{8}&\ph{9}\end{tabular}
&\hz&\begin{tabular}{c}\ph{0}\end{tabular}\\
\p{\fif\leq ib}
&&\begin{tabular}{ccccccccc}\p{1}&\ph{2}&\ph{3}&\ph{4}&\ph{5}&\ph{6}&\ph{7}&\ph{8}&\p{9}\end{tabular}
&\h&\begin{tabular}{cccccccccc}\ph{0}&\ph{1}&\ph{2}&\ph{3}&\p{4}&\ph{5}&\ph{6}&\ph{7}&\ph{8}&\ph{9}\end{tabular}
&\hz&\begin{tabular}{cccccccccc}\ph{0}&\ph{1}&\ph{2}&\ph{3}&\ph{4}&\p{5}&\ph{6}&\ph{7}&\ph{8}&\ph{9}\end{tabular}
&\hz&\begin{tabular}{cccccccccc}\ph{0}&\ph{1}&\ph{2}&\ph{3}&\ph{4}&\ph{5}&\ph{6}&\ph{7}&\ph{8}&\ph{9}\end{tabular}
&\hz&\begin{tabular}{cccccccccc}\ph{0}&\ph{1}&\ph{2}&\ph{3}&\ph{4}&\ph{5}&\ph{6}&\ph{7}&\ph{8}&\ph{9}\end{tabular}
&\hz&\begin{tabular}{c}\ph{0}\end{tabular}\\
\p{\ident\leq ib}
&\h&\begin{tabular}{ccccccccc}\ph{1}&\ph{2}&\ph{3}&\ph{4}&\ph{5}&\ph{6}&\p{7}&\p{8}&\p{9}\end{tabular}
&\hz&\begin{tabular}{cccccccccc}\p{0}&\p{1}&\p{2}&\p{3}&\ph{4}&\ph{5}&\ph{6}&\ph{7}&\ph{8}&\ph{9}\end{tabular}
&\hz&\begin{tabular}{cccccccccc}\ph{0}&\ph{1}&\ph{2}&\p{3}&\p{4}&\p{5}&\p{6}&\p{7}&\p{8}&\p{9}\end{tabular}
&\hz&\begin{tabular}{cccccccccc}\p{0}&\p{1}&\ph{2}&\p{3}&\ph{4}&\ph{5}&\ph{6}&\ph{7}&\ph{8}&\ph{9}\end{tabular}
&\hz&\begin{tabular}{cccccccccc}\ph{0}&\ph{1}&\ph{2}&\p{3}&\p{4}&\ph{5}&\ph{6}&\ph{7}&\ph{8}&\p{9}\end{tabular}
&\hz&\begin{tabular}{c}\ph{0}\end{tabular}\\
\Xcline{3-3}{1pt}\Xcline{5-5}{1pt}\Xcline{7-7}{1pt}\Xcline{9-9}{1pt}\Xcline{11-11}{1pt}\Xcline{13-13}{1pt}
\p{\fbi\leq\ident}
&\h&\begin{tabular}{ccccccccc}\p{1}&\p{2}&\p{3}&\p{4}&\p{5}&\p{6}&\p{7}&\p{8}&\p{9}\end{tabular}
&\hz&\begin{tabular}{cccccccccc}\p{0}&\p{1}&\p{2}&\p{3}&\p{4}&\p{5}&\p{6}&\p{7}&\p{8}&\p{9}\end{tabular}
&\hz&\begin{tabular}{cccccccccc}\p{0}&\p{1}&\ph{2}&\p{3}&\p{4}&\p{5}&\p{6}&\p{7}&\p{8}&\p{9}\end{tabular}
&\hz&\begin{tabular}{cccccccccc}\p{0}&\p{1}&\p{2}&\ph{3}&\p{4}&\p{5}&\ph{6}&\ph{7}&\ph{8}&\p{9}\end{tabular}
&\hz&\begin{tabular}{cccccccccc}\ph{0}&\p{1}&\p{2}&\ph{3}&\ph{4}&\ph{5}&\ph{6}&\ph{7}&\ph{8}&\p{9}\end{tabular}
&\hz&\begin{tabular}{c}\p{0}\end{tabular}\\
\p{\fif\leq\ident}
&&\begin{tabular}{ccccccccc}\p{1}&\p{2}&\p{3}&\p{4}&\p{5}&\ph{6}&\p{7}&\p{8}&\p{9}\end{tabular}
&\h&\begin{tabular}{cccccccccc}\p{0}&\p{1}&\p{2}&\ph{3}&\p{4}&\p{5}&\p{6}&\p{7}&\p{8}&\p{9}\end{tabular}
&\hz&\begin{tabular}{cccccccccc}\ph{0}&\p{1}&\p{2}&\p{3}&\p{4}&\p{5}&\p{6}&\p{7}&\p{8}&\p{9}\end{tabular}
&\hz&\begin{tabular}{cccccccccc}\p{0}&\ph{1}&\p{2}&\p{3}&\p{4}&\p{5}&\p{6}&\p{7}&\ph{8}&\ph{9}\end{tabular}
&\hz&\begin{tabular}{cccccccccc}\ph{0}&\p{1}&\p{2}&\p{3}&\p{4}&\ph{5}&\ph{6}&\ph{7}&\ph{8}&\ph{9}\end{tabular}
&\hz&\begin{tabular}{c}\ph{0}\end{tabular}\\
\p{\fib\leq\ident}
&\h&\begin{tabular}{ccccccccc}\p{1}&\p{2}&\p{3}&\p{4}&\p{5}&\p{6}&\p{7}&\p{8}&\p{9}\end{tabular}
&\hz&\begin{tabular}{cccccccccc}\p{0}&\p{1}&\p{2}&\p{3}&\p{4}&\p{5}&\p{6}&\p{7}&\p{8}&\p{9}\end{tabular}
&\hz&\begin{tabular}{cccccccccc}\p{0}&\ph{1}&\p{2}&\p{3}&\p{4}&\p{5}&\p{6}&\p{7}&\p{8}&\p{9}\end{tabular}
&\hz&\begin{tabular}{cccccccccc}\p{0}&\p{1}&\ph{2}&\p{3}&\ph{4}&\ph{5}&\p{6}&\p{7}&\ph{8}&\p{9}\end{tabular}
&\hz&\begin{tabular}{cccccccccc}\ph{0}&\ph{1}&\ph{2}&\p{3}&\p{4}&\ph{5}&\ph{6}&\ph{7}&\ph{8}&\p{9}\end{tabular}
&\hz&\begin{tabular}{c}\p{0}\end{tabular}\\
\cline{3-3}\cline{5-5}\cline{7-7}\cline{9-9}\cline{11-11}\cline{13-13}
\p{\fb\leq\ident}
&&\begin{tabular}{ccccccccc}\p{1}&\p{2}&\p{3}&\p{4}&\p{5}&\p{6}&\p{7}&\p{8}&\p{9}\end{tabular}
&\h&\begin{tabular}{cccccccccc}\p{0}&\p{1}&\p{2}&\p{3}&\p{4}&\p{5}&\p{6}&\p{7}&\p{8}&\p{9}\end{tabular}
&\hz&\begin{tabular}{cccccccccc}\p{0}&\p{1}&\p{2}&\p{3}&\p{4}&\p{5}&\p{6}&\p{7}&\p{8}&\p{9}\end{tabular}
&\hz&\begin{tabular}{cccccccccc}\p{0}&\p{1}&\p{2}&\p{3}&\ph{4}&\ph{5}&\p{6}&\p{7}&\p{8}&\p{9}\end{tabular}
&\hz&\begin{tabular}{cccccccccc}\p{0}&\ph{1}&\ph{2}&\p{3}&\p{4}&\ph{5}&\ph{6}&\p{7}&\p{8}&\p{9}\end{tabular}
&\hz&\begin{tabular}{c}\ph{0}\end{tabular}\\
\p{\ff\leq\ident}
&\h&\begin{tabular}{ccccccccc}\ph{1}&\ph{2}&\ph{3}&\ph{4}&\ph{5}&\ph{6}&\ph{7}&\ph{8}&\ph{9}\end{tabular}
&\hz&\begin{tabular}{cccccccccc}\ph{0}&\ph{1}&\ph{2}&\ph{3}&\p{4}&\p{5}&\p{6}&\p{7}&\p{8}&\p{9}\end{tabular}
&\hz&\begin{tabular}{cccccccccc}\p{0}&\ph{1}&\ph{2}&\p{3}&\p{4}&\p{5}&\p{6}&\p{7}&\p{8}&\p{9}\end{tabular}
&\hz&\begin{tabular}{cccccccccc}\p{0}&\p{1}&\p{2}&\ph{3}&\ph{4}&\ph{5}&\p{6}&\p{7}&\ph{8}&\ph{9}\end{tabular}
&\hz&\begin{tabular}{cccccccccc}\ph{0}&\p{1}&\p{2}&\p{3}&\p{4}&\ph{5}&\ph{6}&\p{7}&\p{8}&\ph{9}\end{tabular}
&\hz&\begin{tabular}{c}\ph{0}\end{tabular}\mystrut\\[-\arrayrulewidth]
\cline{3-3}\cline{5-5}\cline{7-7}\cline{9-9}\cline{11-11}\cline{13-13}
\end{tabular}\label{tab:optional}
\end{table}
\egroup

\bgroup
\renewcommand{\tabcolsep}{1pt}
\captionsetup[table]{labelfont=sc,labelsep=period}
\begin{table}[!ht]
\vspace{6pt}
\caption{Subsets of $I_{opt}$ that $A$ can satisfy optionally.\\\vspace{4pt}\hspace{0pt}
\small
\begin{tabular}{l}
Lone operators $o$ represent $o\leq ib$ in column~1 and $o\leq\ident$ in column~2.\\
Inequalities implied by $\Ord\KFZ$, e.g., $\ff\leq\ident\implies\fb\leq\ident$, are
not shown.\\
\end{tabular}}
\small
\centering
\renewcommand{\arraystretch}{1}
\renewcommand{\tabcolsep}{2pt}
\renewcommand{\arraystretch}{1.0045}
\begin{tabular}[t]{|c|c!{\vrule width1pt}c|c|c|c|c!{\vrule width1pt}c|c|c|c|c!{\vrule width1pt}c|c|}
\multicolumn{1}{c}{$\psi$}&
\multicolumn{1}{c}{$1$}&
\multicolumn{1}{c}{$2$}&
\multicolumn{1}{c}{\#}&
\multicolumn{1}{c}{}\hspace{2pt}&
\multicolumn{1}{c}{$\psi$}&
\multicolumn{1}{c}{$1$}&
\multicolumn{1}{c}{$2$}&
\multicolumn{1}{c}{\#}&
\multicolumn{1}{c}{}\hspace{2pt}&
\multicolumn{1}{c}{$\psi$}&
\multicolumn{1}{c}{$1$}&
\multicolumn{1}{c}{$2$}&
\multicolumn{1}{c}{\#}\\
\Xcline{1-4}{1pt}\Xcline{6-9}{1pt}\Xcline{11-14}{1pt}
\multirow{3}{*}{$1$}&$((\fib\leq\fif),bi)$,&
\multirow{3}{*}{\begin{tabular}{c}$((\fbi),(\fif),(\fib))$,\\
$(\fbi,\fif,(\fb))$\end{tabular}}&\multirow{3}{*}{$36$}&&
$19$&&$(\fib),(\fb)$&$3$&&
$43$&$(\ident)$&$(\fib)$&$3$\vadj\\\cline{6-9}\cline{11-14}
&$((\fif\leq\fbi),\fif)$,&&&&
$20$&&$(\fib),(\fb)$&$3$&&
$44$&$(\ident)$&$(\fib)$&$3$\\\cline{6-9}\cline{11-14}
&$(\fbi\leq\fib)$&&&&
$21$&&$((\fbi),(\fb))$&$4$&&
$45,46$&&&$1$\\\cline{1-4}\cline{6-9}\cline{11-14}
$2$&&$(((\fif),\fib),\fb)$&$4$&&
$22$&&$(\fib),(\fb)$&$3$&&
$47$&&$(\fb)$&$2$\\\cline{1-4}\cline{6-9}\cline{11-14}
$3$&&$(((\fbi),\fib),\fb)$&$4$&&
$23$&$(\ident)$&$(\fbi),(\fif),(\fib),(\ff)$&$6$&&
$48$&&$(\fb)$&$2$\\\cline{1-4}\cline{6-9}\cline{11-14}
$4$&&$(((\fib),\fbi),(\fb))$&$5$&&
$24$&$(\ident)$&$(\fbi),(\fif),(\fib)$&$5$&&
$49$&$(\ident)$&$(\fbi)$&$3$\\\cline{1-4}\cline{6-9}\cline{11-14}
$5$&&$(\fib),(\fb)$&$3$&&
$25$&$((\fif),(\ident))$&$(\fbi),(\fif),(\fib)$&$11$&&
$50$&&$(\fbi)$&$2$\\\cline{1-4}\cline{6-9}\cline{11-14}
$6$&&$(\fib),(\fb)$&$3$&&
$26$&$(\ident)$&$(\fif),(\fbi)$&$4$&&
$51$&&&$1$\\\cline{1-4}\cline{6-9}\cline{11-14}
\multirow{2}{*}{$7$}&$(((\fib\leq\fif),bi),(\ident))$,&
\multirow{2}{*}{$((\fbi),(\fif),(\fib))$}&\multirow{2}{*}{$22$}&&
$27$&$((bi),(\ident))$&$(\fbi),(\fif),(\fib)$&$11$&&
$52$&&&$1$\\\cline{6-9}\cline{11-14}
&$((\fbi\leq\fib),\ident)$&&&&
$28$&$(\ident)$&$(\fbi),(\fib)$&$4$&&
$53,54$&&&$1$\\\cline{1-4}\cline{6-9}\cline{11-14}
$8$&$(\ident)$&$((\fbi),\fib)$&$4$&&
$29$&$(\ident)$&$(\fbi),(\fib)$&$4$&&
$55$&&&$1$\\\cline{1-4}\cline{6-9}\cline{11-14}
$9$&$((\fif),(\ident))$&$((\fif),(\fib),\fbi)$&$11$&&
$30$&$(\ident)$&$(\fbi)$&$3$&&
$56$&&&$1$\\\cline{1-4}\cline{6-9}\cline{11-14}
$10$&$(\ident)$&$((\fib),\fbi)$&$4$&&
$31$&$(\ident)$&$(\fbi)$&$3$&&
$57$&&&$1$\\\cline{1-4}\cline{6-9}\cline{11-14}
$11$&$(\ident)$&$((\fif),\fbi)$&$4$&&
$32$&&$(\fbi),(\fb),(\ff)$&$4$&&
$58,59$&&&$1$\\\cline{1-4}\cline{6-9}\cline{11-14}
$12$&$(\ident)$&$(\fbi)$&$3$&&
$33$&$(\ident)$&$(\fib)$&$3$&&
$60,61$&&&$1$\\\cline{1-4}\cline{6-9}\cline{11-14}
$13$&$(\ident)$&$(\fbi)$&$3$&&
$34,35$&&$(\fbi)$&$2$&&
$62$&&&$1$\\\cline{1-4}\cline{6-9}\cline{11-14}
\multirow{2}{*}{$14$}&$((\fib\leq\fif),bi)$,
&$((\fbi),(\fif),(\fib))$,&\multirow{2}{*}{$31$}&&
$36$&&$(\fib),(\fb)$&$3$&&
$63$&&&$1$\\\cline{6-9}\cline{11-14}
&$((\fif\leq\fbi),\fif)$&$(\fb),(\ff)$&&&
$37$&&$(\fib),(\fb)$&$3$&&
$64,65$&&&$1$\\\cline{1-4}\cline{6-9}\cline{11-14}
$15$&&$(\fbi,(\fib)),(\fb),(\ff)$&$5$&&
$38$&&$(\fb)$&$2$&&
$66$&&&$1$\\\cline{1-4}\cline{6-9}\cline{11-14}
$16$&&$((\fbi),(\fif),\fib),(\fb)$&$5$&&
$39$&&$(\fib),(\fb)$&$3$&&
$67$&&&$1$\\\cline{1-4}\cline{6-9}\cline{11-14}
$17$&&$((\fbi),\fib),(\fb)$&$4$&&
$40$&&$(\fb)$&$2$&&
$68,69$&&&$1$\\\cline{1-4}\cline{6-9}\cline{11-14}
$18$&&$((\fif),\fib),(\fb)$&$4$&&
$41,42$&&$(\fbi)$&$2$&&
$70$&&&$1$\mystrut\\[-\arrayrulewidth]\cline{1-4}\cline{6-9}\cline{11-14}
\end{tabular}\label{tab:ineqs}
\end{table}
\egroup

Since $\ffA=\fiA\iff\fbA\su\fbiA$ it follows that $\psi A\in\{7,23\}\implies\fibA\nsu
\fbiA$ and $\psi A=9\implies\fibA\su\fbiA$.
Hence, while subsets with \sn\ $9$ satisfy exactly two inequalities in $\Y$, they satisfy
none \emph{optionally.}
Dual results hold for \sn s in $\{14,23\}$ and $\{16\}$, respectively.
Suppose $\psi A=23$.
Since the inclusions $\fibA\su\fifA\iff\fbiA\su\fifA$ imply $\ffA\su\fifA$ it follows
that $A$ satisfies zero inequalities in $\Y$.
Suppose $\psi A\in\{24,25,27\}$.
Since $\ffA=oA$ for some $o\in\Y$ exactly two inequalities in $\Y$ are satisfied by $A$.
Thus all blank entries in rows $1$-$3$ of Table~\ref{tab:optional} are correct.

Suppose $biA\su ibA$ holds optionally.
Then $bibA\neq ibA$ and $ibiA\neq biA$.
Hence $\psi A\cn\in\{21$, $22$, $32$-$38$, $40$-$48\}$.
Note, $\ffA=\fiA\implies\fbA\su\fiA\su ibA\implies\fbA=\es$.
This and its dual imply $\psi A\cn\in\{9$, $16$, $24$, $25$, $49$, $50\}$.
Since $\fbiA\su\fibA\implies\fbiA=\es$, $\psi A\cn\in\{2$, $3$, $5$, $6$, $8$, $11$-$13$,
$17$-$20$, $26$, $28$, $30$, $31$, $39\}$.
Since $\fbiA\su\afibA$ Proposition~\ref{prop:order_implications}(vii) implies $\fbiA\su
\fifA$.
Hence $\psi A\neq23$.
Since $\fibA\su\fifA$ we have $\fifA\su\fbiA\implies\fibA=\es$.
Thus $\psi A\cn\in\{4$, $10$, $15$, $29\}$.
Thus all blank entries in row~$4$ are correct.

Suppose $\fifA\su ibA$ holds optionally. Since $bibA\neq ibA$, $\psi A\cn\in\{21$,
$32$, $34$, $35$, $38$, $40$-$42$, $45$-$48\}$. Since $\fifA\neq\es$, $\psi A\cn\in\{6$,
$13$, $20$, $31$, $39$, $49$, $50\}$. Clearly $\fbiA\nsu\fibA$ and $\fbiA\nsu\fifA$.
It follows that $\psi A\cn\in\{2$, $3$-$5$, $8$, $10$-$12$, $15$, $17$-$19$, $22$,
$26$, $28$-$30$, $33$, $36$, $37$, $43$, $44\}$. Since $\fifA\su\afibA$ we have $\fibA\su
\fbiA$ by Proposition~\ref{prop:order_implications}(vii).
Thus $\fbA=\fibA\implies\fbA\su\fiA\implies\ffA=\fiA$.
Hence $A$ cannot have \sn\ $7$, $23$, $24$, or $27$ since each contains $\fb=\fib$ but
not $\ff=\ei$.
Since $\ffA=\fbA\implies\fbiA\su\fiA\su\fibA$, $\psi A\neq16$.
Thus all blank entries in row~$5$ are correct.

The implications $A\su ibA\implies bibA=bA$ and $ibA=bA\implies A\su ibA$ give us the
blank entries in row~$6$.
Those in rows~$7$-$9$ follow from the equations $o=0$ for $o\in\Y$.
Since $\fA\su A\implies bA=A$ we have $\psi A=50\implies(\fbA\nsu A\mbox{ and }\ffA\nsu
A)$.
Since $\ffA\su A\implies\fiA\su A\implies biA\su A\implies ibiA=iA$ the other
eliminations in rows~$10$ and~$11$ follow from the equation $\fb=0$ and inequation
$ibi\neq i$.

(iii)~We leave cases with six or fewer combinations to the reader.

\emph{Case~1.}~$(\psi A=1)$
We found above that $biA\su ibA\implies\fibA\su\fifA$ and $\fifA\su ibA\implies\fifA\su
\fbiA$.
By Proposition~\ref{prop:order_implications}(ii) any two inclusions among $\fibA\su A$,
$\fifA\su A$, $\fbiA\su A$ imply the third.
Thus by Tables~\ref{tab:incomparable} and~\ref{tab:optional} all possible lists appear.
Suppose one list is selected from each column.
List~1 contains exactly one inequality $o_1\leq o_2$ with $o_1,o_2\in\Y$.
It cannot be combined with the list $(o_2\leq\ident)$ since then $o_1\leq\ident$ giving us
the whole row.
When $o_2=\fib$ the list $(\fb\leq\ident)$ is similarly unavailable.
Since $o_2\in\{\fbi,\fif\hspace{.6pt}\}$ occurs four times and $o_2=\fib$ just once,
there are $(4\times5)+(1\times4)=24$ possible combinations.
Clearly $12$ combinations are possible when the empty list is selected from at least one
column, giving us a total of $36$ possible combinations.

\emph{Case~2.}~$(\psi A=7)$
We again get $12$ combinations by selecting zero lists from one or both columns.
Suppose one list is selected from each column.
We claim that neither $\fif\leq\ident$ nor $\fib\leq\ident$ can be combined with $\ident\leq
ib$.
The latter is obvious since $\fibA\neq\es$.
We showed above that $\psi A=7\implies\fibA\nsu\fbiA$.
Thus $\fifA\nsu\afibA$ by Proposition~\ref{prop:order_implications}(vii).
Hence the claim holds.
Since $(\fbiA\su A\mbox{ and }A\su ibA)\implies\fbiA\su ibA\implies biA\su ibA$ the
claim implies that only one combination is possible when $A$ satisfies $\ident\leq ib$.
Since lists without $\ident\leq ib$ contain an inequality $o_1\leq o_2$ as in Case~$1$,
they each produce three combinations.
Thus $(3\times3)+1=10$ combinations are possible.

\emph{Case~3.}~$(\psi A=9)$
Since $\ffA=\fiA\iff\fbA\su\fbiA$, if $A$ satisfies $\fbi\leq\ident$ it satisfies every
inequality in column~$2$.
Selecting zero lists from at least one column clearly produces seven combinations.
Suppose one list is selected from each column.
Since $\fbiA\su ibA\iff biA\su ibA$ and $\fibA\su ibA\implies\fibA=\es$ neither $\fbi
\leq\ident$ nor $\fib\leq\ident$ can be combined with $\ident\leq ib$.
Since $(\fifA\su A\mbox{ and }A\su ibA)\implies\fifA\su ibA$ the inequality $\ident\leq
ib$ therefore produces only one combination.
Hence only four are possible giving us a total of $11$.

\emph{Case~4.}~$(\psi A=14)$
Since $\fiA=\fbiA$ and $\fbA\cup\fiA\su A\iff\ffA\su A$ it follows that combining $\fb
\leq\ident$ with either $\fbi\leq\ident$ or $\fif\leq\ident$ produces $\ff\leq\ident$.
Apply the Case~$1$ proof to get $11+20=31$ possible combinations.

\emph{Case~5.}~$(\psi A=25)$
Since $\ffA=\fbiA$ and $\fibA\neq\es$ the inequality $\ident\leq ib$ cannot be combined
with either $\fib\leq\ident$ or $\fbi\leq\ident$.
When it is combined with $\fif\leq\ident$ we get $\fif\leq ib$.
Note that $(\fbi)$ is the same as $(\fbi,\fif,\fib)$.
It follows that $11$ combinations are possible.

\emph{Case~6.}~$(\psi A=27)$
Since $\ffA=\fifA$ and $\fibA\neq\es$ the inequality $\ident\leq ib$ cannot be combined
with either $\fib\leq\ident$ or $\fif\leq\ident$.
When it is combined with $\fbi\leq\ident$ we get $bi\leq ib$.
It follows that $11$ combinations are possible.
\end{proof}

\begin{table}[!ht]
\caption{Inclusions $o_1A\su ao_2A$ are equivalent to inclusions in $\KFZA$ or in
$\KFZ(aA)$.}
\centering\small
\renewcommand{\arraystretch}{1}
\renewcommand{\tabcolsep}{3pt}
\begin{tabular}{ccc}
\begin{tabular}[t]{|c|c|c|c|}
\multicolumn{1}{c}{row}&\multicolumn{1}{c}{$o_1$}&\multicolumn{1}{c}{$o_2$}&
\multicolumn{1}{c}{$o_1A\su ao_2A\ify$}\\\Xhline{2\arrayrulewidth}
$1$&$\ident$&$bib$\vadj&$ibA=\es$\\\hline
$2$&$\f$&$bib$\vadj&$bibA=iA$\\\hline
$3$&$\ff$&$bib$\vadj&$biA=iA$ and $\bifA=\ifA$\\\hline
$4$&$\fb$&$bib$\vadj&$\fibA=\es$\\\hline
$5$&$ib$&$bi$\vadj&$iA=\es$\\\hline
$6$&$\f$&$bi$\vadj&$biA=iA$\\\hline
$7$&$\bif$&$bi$\vadj&$\fif(aA)\su ib(aA)$\\\hline
$8$&$\ident$&$bi$\vadj&$iA=\es$\\\hline
$9$&$\ff$&$bi$\vadj&$biA=iA$\\\hline
$10$&$\fb$&$bi$\vadj&$biA\su ibA$\\\hline
$11$&$\fib$&$bi$\vadj&$biA\su ibA$\\\hline
$12$&$\fif$&$bi$\vadj&$\fif(aA)\su ib(aA)$\\\hline
$13$&$\f$&$ib$\vadj&$ibA=iA$\\\hline
$14$&$\ff$&$ib$\vadj&$\ffA=\fbA$\\\hline
$15$&$\ident$&$ib$\vadj&$ibA=\es$\\\hline
\end{tabular}&
\begin{tabular}[t]{|c|c|c|}
\multicolumn{1}{c}{$o_1$}&\multicolumn{1}{c}{$o_2$}&
\multicolumn{1}{c}{$o_1A\su ao_2A\ify$}\\\Xhline{2\arrayrulewidth}
$\ei$&$ib$\vadj&$\ffA=\fbA$\\\hline
$\fbi$&$ib$\vadj&$\fbiA\su\fibA$\\\hline
$\bif$&$ib$\vadj&$\ifA=\es\mbox{ and }\fbiA\su\fibA$\\\hline
$\fif$&$ib$\vadj&$\fbiA\su\fibA$\\\hline
$\f$&$ibi$\vadj&$ibiA=iA$\\\hline
$\ff$&$ibi$\vadj&$ibiA=iA$\\\hline
$\ei$&$ibi$\vadj&$\fiA=\fbiA$\\\hline
$\ident$&$ibi$\vadj&$iA=\es$\\\hline
$\f$&$\ident$\vadj&$A=iA$\\\hline
$\ie$&$\ident$\vadj&$\ifA=\es$\\\hline
$\ff$&$\ident$\vadj&$\ff(aA)\su aA$\\\hline
$\fb$&$\ident$\vadj&$A\su ibA$\\\hline
$\fib$&$\ident$\vadj&$\fbi(aA)\su aA$\\\hline
$\ei$&$\ident$\vadj&$\fb(aA)\su aA$\\\hline
$\fbi$&$\ident$\vadj&$\fib(aA)\su aA$\\\hline
\end{tabular}&
\begin{tabular}[t]{|c|c|c|}
\multicolumn{1}{c}{$o_1$}&\multicolumn{1}{c}{$o_2$}&
\multicolumn{1}{c}{$o_1A\su ao_2A\ify$}\\\Xhline{2\arrayrulewidth}
$\bif$&$\ident$\vadj&$\ifA=\es$\\\hline
$\fif$&$\ident$\vadj&$\fif(aA)\su aA$\\\hline
$\bif$&$\ff$\vadj&$\bifA=\ifA$\\\hline
$\bif$&$\fb$\vadj&$\fifA\su ibA$\\\hline
$\ei$&$\fb$\vadj&$biA\su ibA$\\\hline
$\fbi$&$\fb$\vadj&$biA\su ibA$\\\hline
$\fif$&$\fb$\vadj&$\fifA\su ibA$\\\hline
$\bif$&$\fib$\vadj&$\fifA\su ibA$\\\hline
$\ei$&$\fib$\vadj&$biA\su ibA$\\\hline
$\fbi$&$\fib$\vadj&$biA\su ibA$\\\hline
$\fif$&$\fib$\vadj&$\fifA\su ibA$\\\hline
$\bif$&$\ei$\vadj&$\fif(aA)\su ib(aA)$\\\hline
$\fif$&$\ei$\vadj&$\fif(aA)\su ib(aA)$\\\hline
$\bif$&$\fbi$\vadj&$\fif(aA)\su ib(aA)$\\\hline
$\fif$&$\fbi$\vadj&$\fif(aA)\su ib(aA)$\\\hline
\end{tabular}\end{tabular}\label{tab:incomparable_kf_akf}
\end{table}

As one might guess, $\KFZA$ and $\KFZ(aA)$ determine the local ordering of $\KF$ that
$A$ satisfies.

\begin{lem}\label{lem:kfa_inequalities}
The partial order that \mth{$A$} satisfies on \mth{$\KF$} is determined by the partial
orders that \mth{$A$} and \mth{$aA$} satisfy on \mth{$\KFZ$.}
\end{lem}

\begin{proof}
Let $A\su X$.
We claim every inclusion in $\KFA$ that is optional with respect to $\psi A$ is
equivalent to some inclusions in $\KFZA$ or some inclusions in $\KFZ(aA)$.
Let $o_1,o_2\in\KFZ\sm\{0\}$.
\emph{Case~1.}~$(o_1A\su ao_2A)$
Optionality with respect to $\psi A$ implies $o_1,o_2$ are neither disjoint nor
comparable.
The claim follows by Table~\ref{tab:incomparable_kf_akf}.
\emph{Case~2.}~$(ao_1A\su o_2A)$
Have $ao_1A\su o_2A\iff ao_1a(aA)\su o_2a(aA)$.
Suppose $o_1\in\KZ$.
Then $ao_1a\in\KZ$.
The claim follows since $o_2\in\KZ\implies o_2a\in a\KZ$ (apply Case~1) and $o_2\in\FZ
\implies o_2a\in\FZ$.
Suppose $o_1\in\{\ie,\bif,\f\}$.
Since Table~\ref{tab:kf_akf} implies $bA=X$, the claim holds by
Lemma~\ref{lem:equations_kf}(x).
The case $o_1\in{\downarrow}\{\ff\hspace{.6pt}\}$ follows, for $ao_1A\su o_2A\iff ao_2A
\su o_1A$ and we have $o_2\in\KZ\cup\{\ie,\bif,\f\}$ by Corollary~\ref{cor:excluded}.
\end{proof}

Exactly $496$ local orderings of $\KF$ are satisfied by subsets in spaces of cardinality
$10$.
No further ones are satisfied in spaces of cardinality $11$; does this imply exactly
$496$ exist in general?
This may be a deep question.
Lacking such a proof, we instead proved it both directly and by applying
Proposition~\ref{prop:kfz_inequalities} and Lemma~\ref{lem:kfa_inequalities}.
The details are left to the reader (each proof is similar to the proof of
Proposition~\ref{prop:kfz_inequalities}).

\begin{thm}\label{thm:kfa_inclusions}
\mth{$\KF$} has exactly \mth{$496$} local orderings\nit.
\end{thm}

\noindent
The next corollary follows easily by computer.

\begin{cor}\label{cor:k_inclusions}
\mth{$\K$} has exactly \mth{$66$} local orderings\nit.
\end{cor}

\noindent
Given the $496$ local orderings on $\KF$, the ordering under implication on all equations
and inclusions in $\KFA$ is easily obtained by computer.
Those that are neither impossible nor hold in general can be described and counted as
follows.
Let $o_1,o_2\in\KFZ$.
We naturally select the equation $o_1A=o_2A$ to represent both itself and $ao_1A=ao_2A$.
We also need consider only one of the two equations $o_1A=ao_2A$ and $ao_1A=o_2A$.
Hence there are at most $2\times\binom{17}{2}=272$ equations to consider.
Corollary~\ref{cor:excluded} excludes $\binom{7}{2}+15=36$ of them, reducing the total to
$236$.
If $o_1\leq o_2$ then $o_1A\su ao_2A\iff o_1A=\es$ and $ao_1A\su o_2A\iff o_2A=X$.
If $o_1\leq ao_2$ then $ao_1A\su o_2A\iff ao_1A=o_2A$.
The only remaining inclusions in $\KFA$ of the forms $ao_1A\su o_2A$ and $o_1A\su ao_2A$
are those where $o_1$ and $o_2$ are incomparable.
Up to left duality there are $90$ such inclusions, eight of which are impossible by
Corollary~\ref{cor:excluded}(i).
We similarly need consider only $90$ inclusions of the form $o_1A\su o_2A$.
Thus $236+172=408$ relations suffice to represent all non-general equations and
inclusions that occur in $\KFA$.

\begin{cor}\label{cor:general_implications}
The $408$ representative relations described above break into \mth{$74$} equivalence
classes under logical equivalence \nit(see Table~\nit{\ref{tab:implications_poset}).}
Figure~\nit{\ref{fig:implications_poset}} gives their ordering under implication\nit.
Figure~\nit{\ref{fig:implications_space}} is the operator analogue.
\end{cor}

Duality accounts for the vertical symmetry in Figure~\ref{fig:implications_poset}.
By coincidence, the number of relational classes it contains equals the number of local
collapses of $\KF$ when both are counted up to duality ($43$).

\begin{table}
\caption{Equivalences among the $408$ representative relations in $\KFA$.\\\smallskip
\small$\KF$ counterparts are shown to save space.
Relations are indexed in Table~\ref{tab:index}.}
\small
\centering
\vspace{4pt}
\renewcommand{\arraystretch}{1}
\renewcommand{\tabcolsep}{2pt}
\renewcommand{\arraystretch}{1.0045}
\begin{tabular}[t]{|c|l|c|c|l|c|c|l|c|c|l|c|c|l|}
\cline{1-2}\cline{4-5}\cline{7-8}\cline{10-11}\cline{13-14}
\y{$1$}&$\fbi\leqy\ident$&&\y{$23$}&$\fb\leqy\fbi$&&\multirow{8}{*}{\y{$46$}}&$bi\leqy
a,\,\ident\leqy\f$&&\y{$60$}&$\ident\leqy\ie$&&\y{$70$}&$bib\eqy\afi,\,bib\eqy\aff$
\\\cline{1-2}\cline{4-5}\cline{10-11}\cline{13-14}
\multirow{3}{*}{\y{$2$}}&$bib\eqy b,\,\fb\eqy\fib$&&\multirow{2}{*}{\y{$24$}}&$bib\eqy
ib,\,\fib\eqy0$&&&$bi\leqy\f,\,ibi\leqy\f$&&\multirow{14}{*}{\y{$61$}}&$ib\eqy0,\,bib\eqy
0$&&\multirow{2}{*}{\y{$71$}}&$\ident\eqy\af$\\
&$\ident\leqy bib,\,\f\leqy bib$&&&$bib\leqy\afb$&&&$ib\leqy\f,\,bib\leqy\f$&&&
$b\eqy\fb,\,b\eqy\ff$&&&$\ident\eqy\afi,\,\ident\eqy\aff$\\\cline{4-5}\cline{13-14}
&$\ff\leqy bib,\,\fb\leqy bib$&&\y{$25$}&$\fb\leqy\ident$&&&$bi\leqy\ff,\,ibi\leqy\ff$&&&
$ib\eqy\fbi,\,ib\eqy\ei$&&\multirow{12}{*}{\y{$72$}}&$\ident\eqy0,\,b\eqy0$
\\\cline{1-2}\cline{4-5}
\y{$3$}&$\fif\leqy a$&&\multirow{10}{*}{\y{$26$}}&$bib\eqy bi,\,ibi\eqy ib$&&&$bi\leqy
\fb,\,bi\leqy\fib$&&&$ib\eqy\fib,\,ib\eqy\fif$&&&$\ident\eqy\fbi,\,\ident\eqy\ei$\\
\cline{1-2}
\multirow{3}{*}{\y{$4$}}&$\fbi\leqy\fib,\,\fif\leqy\fib$&&&$\bif\eqy0,\,\ie\eqy0$&&&$bi
\leqy\bif,\,ib\leqy\bif$&&&$bib\eqy\fbi,\,bib\eqy\ei$&&&$\ident\eqy\fib,\,\ident\eqy
\fif$\\
&$\fbi\leqy\fb,\,\fif\leqy\fb$&&&$\bif\eqy\fif,\,\ie\eqy\fif$&&&$bi\leqy\fif,\,ibi\leqy
\ei$&&&$bib\eqy\fib,\,bib\eqy\fif$&&&$\ident\eqy\bif,\,\ident\eqy\ie$\\
&$\fbi\leqy aib,\,\fif\leqy aib$&&&$\f\eqy\ff,\,ib\leqy bi$&&&$bi\leqy aib,\,ibi\leqy a$&
&&$i\eqy\bif,\,i\eqy\ie$&&&$b\eqy\fbi,\,b\eqy\ei$\\\cline{1-2}\cline{7-8}
\multirow{2}{*}{\y{$5$}}&$\fbi\leqy\fif,\,\fib\leqy\fif$&&&$\ie\leqy\ident,\,\ident\leqy
\aif$&&\y{$47$}&$bib\leqy\aff$&&&$bi\eqy\bif,\,bi\eqy\ie$&&&$b\eqy\fib,\,b\eqy\fif$
\\\cline{7-8}
&$\fbi\leqy\bif,\,\fib\leqy\bif$&&&$\bif\leqy\ident,\,\bif\leqy a$&&
\multirow{2}{*}{\y{$48$}}&$\f\eqy\bif,\,\ff\eqy\fif$&&&$ibi\eqy\bif,\,ibi\eqy\ie$&&&$ib
\eqy\fb,\,ib\eqy\ff$\\\cline{1-2}
\multirow{3}{*}{\y{$6$}}&$\fib\leqy\fbi,\,\fif\leqy\fbi$&&&$\bif\leqy bi,\,\bif\leqy
aib$&&&$\ff\leqy\bif$&&&$\ident\leqy\fb,\,\ident\leqy\ff$&&&$bib\eqy\fb,\,bib\eqy\ff$
\\\cline{7-8}
&$\fib\leqy\ei,\,\fif\leqy\ei$&&&$\bif\leqy\fb,\,\bif\leqy\ei$&&\y{$49$}&$\ff\leqy ibi$&&&
$ib\leqy a,\,bib\leqy a$&&&$i\eqy\f$\\\cline{7-8}
&$\fib\leqy bi,\,\fif\leqy bi$&&&$\bif\leqy\fbi,\,\bif\leqy\fib$&&
\multirow{14}{*}{\y{$50$}}&$b\eqy1,\,i\eqy\af$&&&$bib\leqy\fb,\,bib\leqy\ff$&&&$bi\eqy
\f,\,ibi\eqy\f$\\\cline{1-2}
\y{$7$}&$\fif\leqy\ident$&&&$\bif\leqy\ff$&&&$ib\eqy1,\,bib\eqy1$&&&$ib\leqy\ei,\,ib\leqy
\ff$&&&$\ident\leqy\fbi,\,\ident\leqy\ei$\\\cline{1-2}\cline{4-5}
\multirow{3}{*}{\y{$8$}}&$ibi\eqy i,\,\ei\eqy\fbi$&&\y{$27$}&$\ei\leqy a$&&&$b\eqy
\afb,\,b\eqy\afib$&&&$ib\leqy\fbi,\,ib\leqy\fif$&&&$\ident\leqy\fib,\,\ident\leqy\fif$
\\\cline{4-5}\cline{10-11}\cline{13-14}
&$ibi\leqy\ident,\,\f\leqy aibi$&&\y{$28$}&$\ff\eqy\fib$&&&$ib\eqy\afb,\,ib\eqy\afib$&&
\y{$62$}&\y{$\ident\eqy bi,\,\ident\eqy bib$}&&\multirow{18}{*}{\y{$73$}}&$\f\eqy1,\,\ie
\eqy1$\\\cline{4-5}\cline{10-11}
&$\ff\leqy aibi,\,\ei\leqy aibi$&&\y{$29$}&$\ff\eqy\fbi$&&&$bib\eqy\afb,\,bib\eqy\afib$&&
\y{$63$}&$\f\eqy\ie,\,\ff\eqy0$&&&$\bif\eqy1,\,b\eqy ai$\\\cline{1-2}\cline{4-5}
\cline{10-11}
\y{$9$}&$\fib\leqy a$&&\y{$30$}&$\fib\eqy\ei$&&&$bi\eqy\aif,\,ibi\eqy\abif$&&\y{$64$}&
\y{$\ident\eqy ib,\,\ident\eqy ibi$}&&&$b\eqy abi,\,b\eqy aibi$\\\cline{1-2}\cline{4-5}
\cline{10-11}
\multirow{5}{*}{\y{$10$}}&$bi\leqy ib$&&\y{$31$}&$\fb\leqy\bif,\,\fb\leqy\fif$&&&$a\leqy
ib,\,a\leqy\f$&&\multirow{14}{*}{\y{$65$}}&$bi\eqy1,\,ibi\eqy1$&&&$bi\eqy abib,\,bi\eqy
aib$\\\cline{4-5}
&$bi\leqy\afb,\,\ei\leqy ib$&&\multirow{3}{*}{\y{$32$}}&$bi\eqy ib,\,bib\eqy ibi$&&&$aib
\leqy\f,\,abib\leqy\f$&&&$i\eqy\afi,\,i\eqy\aff$&&&$bib\eqy aibi,\,i\eqy aib$\\
&$\fbi\leqy ib,\,bi\leqy\afib$&&&$\bif\eqy\fbi,\,\bif\eqy\fib$&&&$abi\leqy\f,\,aibi\leqy
\f$&&&$bi\eqy\afib,\,bi\eqy\afb$&&&$i\eqy abib,\,ib\eqy aibi$\\
&$\ei\leqy\afib,\,\fb\leqy\afbi$&&&$\ie\eqy\fbi,\,\ie\eqy\fib$&&&$aib\leqy\ff,\,abib\leqy
\ff$&&&$bi\eqy\afbi,\,bi\eqy\afif$&&&$\f\eqy\afb,\,\f\eqy\aff$\\\cline{4-5}
&$\fb\leqy\afi,\,\fbi\leqy\afib$&&\y{$33$}&$\ei\leqy\bif,\,\ei\leqy\fif$&&&$aib\leqy
\ei,\,aib\leqy\fbi$&&&$ibi\eqy\afib,\,ibi\eqy\afb$&&&$\f\eqy\afbi,\,\f\eqy\afi$
\\\cline{1-2}\cline{4-5}
\multirow{3}{*}{\y{$11$}}&$bi\leqy\abif,\,bi\leqy\afif$&&\y{$34$}&$\fb\eqy\fbi$&&&$aib
\leqy\bif,\,abi\leqy\bif$&&&$ibi\eqy\afif,\,ibi\eqy\afbi$&&&$\f\eqy\afib,\,\f\eqy\afif$
\\\cline{4-5}
&$\ei\leqy\afif,\,\bif\leqy\afi$&&\multirow{2}{*}{\y{$35$}}&$bi\eqy i,\,\ei\eqy0$&&&$aib
\leqy\fif,\,abib\leqy\fb$&&&$b\eqy\abif,\,b\eqy\aif$&&&$\ff\eqy\abif,\,\ff\eqy\aif$\\
&$\fbi\leqy\afif,\,\bif\leqy\afbi$&&&$bi\leqy\af,\,bi\leqy\aff$&&&$abi\leqy ib,\,a\leqy
bib$&&&$ib\eqy\abif,\,bib\eqy\aif$&&&$\bif\eqy\afbi,\,\bif\eqy\afi$
\\\cline{1-2}\cline{4-5}\cline{7-8}
\y{$12$}&$\fib\eqy\fif$&&\y{$36$}&$\ff\leqy a$&&\y{$51$}&\y{$\ident\leqy ibi$}&&&$bib\eqy
\abif,\,ib\eqy\aif$&&&$\bif\eqy\afb,\,\bif\eqy\afib$\\\cline{1-2}\cline{4-5}\cline{7-8}
\y{$13$}&$\fib\leqy\ident$&&\y{$37$}&$\fb\eqy\ei$&&\multirow{4}{*}{\y{$52$}}&$bib\eqy
\ie,\,ib\eqy\bif$&&&$a\leqy\ei,\,a\leqy\ff$&&&$\bif\eqy\afif,\,\ie\eqy\afb$\\\cline{1-2}
\cline{4-5}
\y{$14$}&$\fib\eqy\fbi$&&\y{$38$}&$\ff\leqy\ident$&&&$i\eqy\fib,\,i\eqy\fif$&&&$a\leqy
bi,\,a\leqy ibi$&&&$\ie\eqy\afbi,\,\ie\eqy\afi$\\\cline{1-2}\cline{4-5}
\y{$15$}&$\fbi\leqy a$&&\multirow{2}{*}{\y{$39$}}&$ib\eqy b,\,\fb\eqy0$&&&$bi\eqy
\fib,\,bi\eqy\fif$&&&$aibi\leqy\ei,\,aibi\leqy\ff$&&&$\ie\eqy\afib,\,\ie\eqy\afif$\\
\cline{1-2}
\y{$16$}&$\fbi\eqy\fif$&&&$\f\leqy ib,\,\ff\leqy ib$&&&$ibi\eqy\fib,\,ibi\eqy\fif$&&&
$abi\leqy\fb,\,abi\leqy\ff$&&&$\abif\leqy\fb,\,\abif\leqy\ei$
\\\cline{1-2}\cline{4-5}\cline{7-8}
\multirow{3}{*}{\y{$17$}}&$\bif\leqy ib,\,\fif\leqy ib$&&\y{$40$}&$\fb\eqy\fif$&&
\y{$53$}&$\ident\eqy b,\,\f\leqy\ident$&&&$abi\leqy\fib,\,abi\leqy\fif$&&&$\abif\leqy
\fbi,\,\abif\leqy\fib$\\\cline{4-5}\cline{7-8}\cline{10-11}
&$\fb\leqy\afif,\,\bif\leqy\afb$&&\multirow{2}{*}{\y{$41$}}&$ib\eqy i,\,\f\eqy\fb$&&
\multirow{2}{*}{\y{$54$}}&$bib\eqy\f,\,b\eqy\bif$&&\y{$66$}&$a\leqy\ie$&&&$\aff\leqy\bif$
\\\cline{10-11}\cline{13-14}
&$\fib\leqy\afif,\,\bif\leqy\afib$&&&$ib\leqy\ident,\,\f\leqy aib$&&&$\ident\leqy\bif$&&
\multirow{2}{*}{\y{$67$}}&$\ident\eqy\f$&&\multirow{12}{*}{\y{$74$}}&$\ident\eqy1,\,i
\eqy1$\\\cline{1-2}\cline{4-5}\cline{7-8}
\multirow{2}{*}{\y{$18$}}&$bi\eqy ibi,\,\fbi\eqy0$&&\multirow{2}{*}{\y{$42$}}&$bi\eqy
b,\,\f\eqy\ei$&&\multirow{2}{*}{\y{$55$}}&$bib\eqy i,\,\bif\eqy\ei$&&&$\ident\eqy
\fb,\,\ident\eqy\ff$&&&$\ident\eqy\afib,\,\ident\eqy\afb$\\\cline{10-11}
&$\ei\leqy ibi$&&&$\ident\leqy bi,\,\f\leqy bi$&&&$\ie\eqy\ei,\,bib\leqy\af$&&
\multirow{4}{*}{\y{$68$}}&$b\eqy\ie,\,ib\eqy\f$&&&$\ident\eqy\afbi,\,\ident\eqy\afif$
\\\cline{1-2}\cline{4-5}\cline{7-8}
\multirow{3}{*}{\y{$19$}}&$\ff\eqy\fb,\,\ei\leqy\fb$&&\y{$43$}&$\ei\eqy\fif$&&
\multirow{2}{*}{\y{$56$}}&$ibi\eqy b,\,\bif\eqy\fb$&&&$i\eqy\fb,\,i\eqy\ff$&&&$\ident\eqy
\abif,\,\ident\eqy\aif$\\\cline{4-5}
&$\ff\leqy aib,\,\ei\leqy aib$&&\y{$44$}&$\f\eqy\fbi,\,\f\eqy\fib$&&&$\ie\eqy\fb,\,\f\leqy
ibi$&&&$bi\eqy\fb,\,bi\eqy\ff$&&&$i\eqy\afib,\,i\eqy\afb$\\\cline{4-5}\cline{7-8}
&$\ei\leqy\fib$&&\y{$45$}&\y{$bib\leqy\ident$}&&\multirow{2}{*}{\y{$57$}}&$ibi\eqy\af,\,i
\eqy\abif$&&&$ibi\eqy\fb,\,ibi\eqy\ff$&&&$i\eqy\afbi,\,i\eqy\afif$
\\\cline{1-2}\cline{4-5}\cline{10-11}
\y{$20$}&$bi\leqy\ident,\,\ei\leqy\ident$&&\multirow{6}{*}{\y{$46$}}&$i\eqy0,\,b\eqy\f$&&
&$a\leqy\bif$&&\multirow{3}{*}{\y{$69$}}&$b\eqy i$&&&$bi\eqy\afi,\,bi\eqy\aff$
\\\cline{1-2}\cline{7-8}
\multirow{2}{*}{\y{$21$}}&$\bif\eqy\ie,\,\fif\eqy0$&&&$bi\eqy0,\,ibi\eqy0$&&\y{$58$}&
$\ident\eqy i,\,\f\leqy a$&&&$\f\eqy0,\,\f\eqy\fif$&&&$ibi\eqy\afi,\,ibi\eqy\aff$\\
\cline{7-8}
&$\bif\leqy\aff$&&&$i\eqy\ei,\,i\eqy\fbi$&&\multirow{4}{*}{\y{$59$}}&$ibi\eqy\aif,\,bi\eqy
\abif$&&&$\ff\eqy\bif,\,\ff\eqy\ie$&&&$b\eqy\af$\\\cline{1-2}\cline{10-11}
\y{$22$}&$\ident\leqy ib,\,\fb\leqy a$&&&$bi\eqy\ei,\,bi\eqy\fbi$&&&$b\eqy\afbi,\,b\eqy
\afif$&&\multirow{3}{*}{\y{$70$}}&$i\eqy\aif,\,bi\eqy\af$&&&$ib\eqy\af,\,bib\eqy\af$
\\\cline{1-2}
\multirow{2}{*}{\y{$23$}}&$\ff\eqy\ei,\,\fb\leqy\ei$&&&$ibi\eqy\ei,\,ibi\eqy\fbi$&&&$ib\eqy
\afbi,\,ib\eqy\afif$&&&$b\eqy\afi,\,b\eqy\aff$&&&$a\leqy\fib,\,a\leqy\fb$\\
&$\ff\leqy bi,\,\fb\leqy bi$&&&$ib\eqy\ie,\,bib\eqy\bif$&&&$bib\eqy\afbi,\,bib\eqy\afif$&
&&$ib\eqy\afi,\,ib\eqy\aff$&&&$a\leqy\fbi,\,a\leqy\fif$
\\\cline{1-2}\cline{4-5}\cline{7-8}\cline{10-11}\cline{13-14}
\end{tabular}\label{tab:implications_poset}
\end{table}

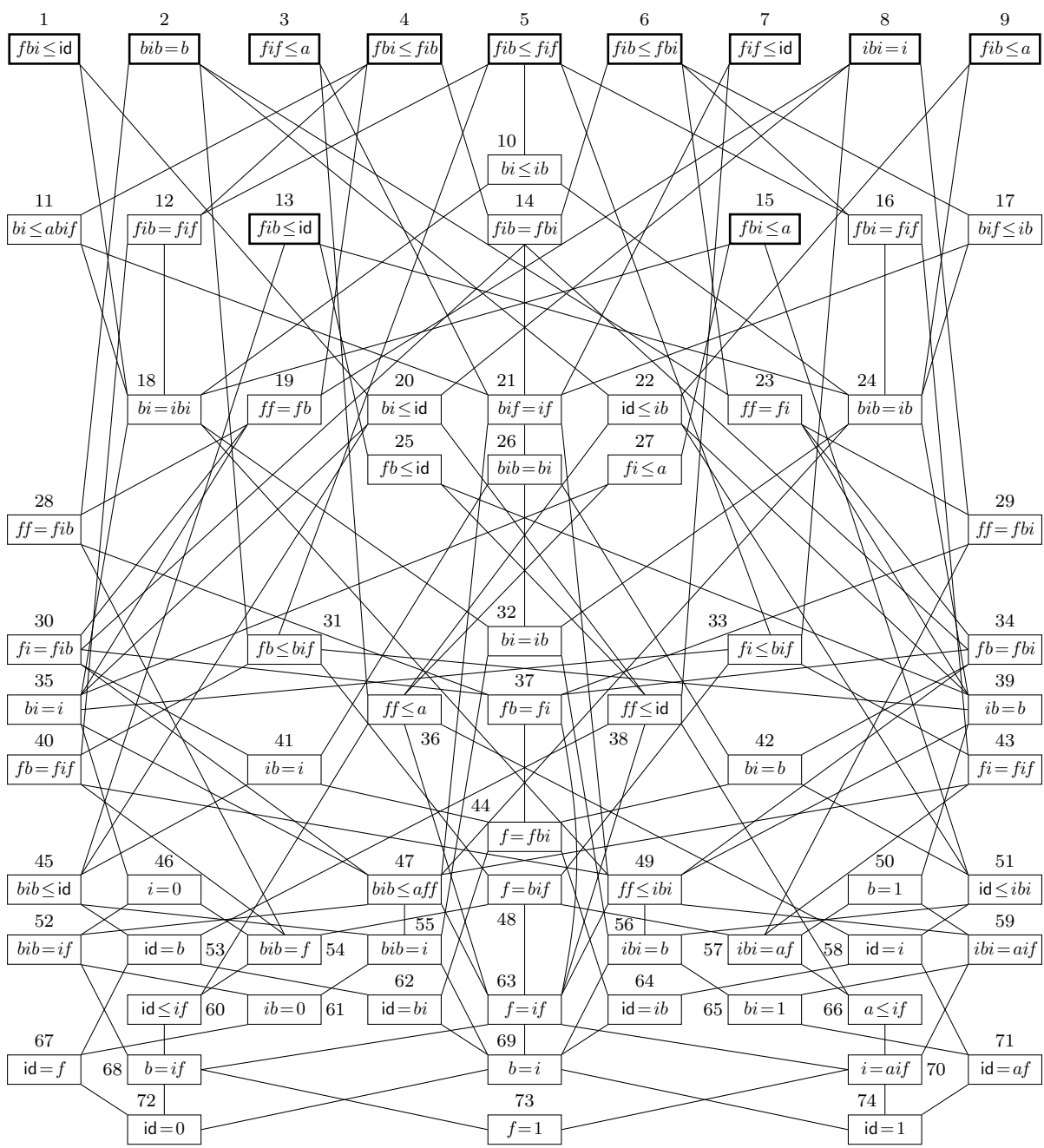
\begin{figure}
\centering
\small
\begin{tikzpicture}
[auto,
 block/.style={rectangle,draw=black,align=center,minimum width={40pt},
 minimum height={16pt},inner sep=2pt,scale=.8},
 topblock/.style={rectangle,draw=black,line width=1pt,align=center,minimum width={38pt},
 minimum height={16pt},inner sep=2pt,scale=.8},
 line/.style ={draw},
 scale=.92]

	\node[block,anchor=south] (0) at (-6,2) {\vphantom{$\f$}$\ident\eqy b$};
	\node[block,anchor=south] (1) at (6,3) {\vphantom{$\f$}$b\eqy1$};
	\node[block,anchor=south] (2) at (6,2) {\vphantom{$\f$}$\ident\eqy i$};
	\node[block,anchor=south] (3) at (6,-1) {\vphantom{$\f$}$\ident\eqy1$};
	\node[block,anchor=south] (4) at (-2,1) {\vphantom{$\f$}$\ident\eqy bi$};
	\node[block,anchor=south] (5) at (4,1) {\vphantom{$\f$}$bi\eqy1$};
	\node[block,anchor=south] (6) at (2,1) {\vphantom{$\f$}$\ident\eqy ib$};
	\node[block,anchor=south] (7) at (-8,0) {$\ident\eqy\f$};
	\node[block,anchor=south] (8) at (-6,-1) {\vphantom{$\f$}$\ident\eqy0$};
	\node[block,anchor=south] (9) at (6,1) {$a\leqy\ie$};
	\node[block,anchor=south] (10) at (4,2) {$ibi\eqy\af$};
	\node[block,anchor=south] (11) at (8,0) {$\ident\eqy\af$};
	\node[block,anchor=south] (12) at (0,0) {\vphantom{$\f$}$b\eqy i$};
	\node[block,anchor=south] (13) at (4,5) {\vphantom{$\f$}$bi\eqy b$};
	\node[block,anchor=south] (14) at (8,6) {\vphantom{$\f$}$ib\eqy b$};
	\node[topblock,anchor=south] (15) at (-6,17) {\vphantom{$\f$}$bib\eqy b$};
	\node[block,anchor=south] (16) at (2,2) {\vphantom{$\f$}$ibi\eqy b$};
	\node[block,anchor=south] (17) at (-6,3) {\vphantom{$\f$}$i\eqy0$};
	\node[block,anchor=south] (18) at (-4,1) {\vphantom{$\f$}$ib\eqy0$};
	\node[block,anchor=south] (19) at (-6,0) {$b\eqy\ie$};
	\node[block,anchor=south] (20) at (-4,2) {$bib\eqy\f$};
	\node[block,anchor=south] (21) at (-8,6) {\vphantom{$\f$}$bi\eqy i$};
	\node[block,anchor=south] (22) at (-4,5) {\vphantom{$\f$}$ib\eqy i$};
	\node[block,anchor=south] (23) at (-2,2) {\phantom{$y$}$bib\eqy i$\phantom{$y$}};
	\node[topblock,anchor=south] (24) at (6,17) {\vphantom{$\f$}$ibi\eqy i$};
	\node[block,anchor=south] (25) at (-8,2) {$bib\eqy\ie$};
	\node[block,anchor=south] (26) at (0,7.14) {\vphantom{$\f$}$bi\eqy ib$};
	\node[block,anchor=south] (27) at (0,10) {\vphantom{$\f$}$bib\eqy bi$};
	\node[block,anchor=south] (28) at (-6,11) {\vphantom{$\f$}$bi\eqy ibi$};
	\node[block,anchor=south] (29) at (6,11) {\vphantom{$\f$}$bib\eqy ib$};
	\node[block,anchor=south] (30) at (0,1) {$\f\eqy\ie$};
	\node[block,anchor=south] (31) at (0,3) {$\f\eqy\bif$};
	\node[block,anchor=south] (32) at (0,3.88) {$\f\eqy\fbi$};
	\node[block,anchor=south] (33) at (-4,11) {$\ff\eqy\fb$};
	\node[block,anchor=south] (34) at (-8,9) {$\ff\eqy\fib$};
	\node[block,anchor=south] (35) at (4,11) {$\ff\eqy\ei$};
	\node[block,anchor=south] (36) at (8,9) {$\ff\eqy\fbi$};
	\node[block,anchor=south] (37) at (0,11) {$\bif\eqy\ie$};
	\node[block,anchor=south] (38) at (0,6) {$\fb\eqy\ei$};
	\node[block,anchor=south] (39) at (8,7) {$\fb\eqy\fbi$};
	\node[block,anchor=south] (40) at (-8,5) {$\fb\eqy\fif$};
	\node[block,anchor=south] (41) at (-8,7) {$\ei\eqy\fib$};
	\node[block,anchor=south] (42) at (0,14) {$\fib\eqy\fbi$};
	\node[block,anchor=south] (43) at (-6,14) {$\fib\eqy\fif$};
	\node[block,anchor=south] (44) at (8,5) {$\ei\eqy\fif$};
	\node[block,anchor=south] (45) at (6,14) {$\fbi\eqy\fif$};
	\node[block,anchor=south] (46) at (0,-1) {$\f\eqy1$};
	\node[block,anchor=south] (47) at (6,0) {$i\eqy\aif$};
	\node[block,anchor=south] (48) at (8,2) {$ibi\eqy\aif$};
	\node[block,anchor=south] (49) at (-8,14) {$bi\leqy\abif$};
	\node[block,anchor=south] (50) at (0,15) {\vphantom{$\f$}$bi\leqy ib$};
	\node[block,anchor=south] (51) at (-2,11) {\vphantom{$\f$}$bi\leqy\ident$};
	\node[block,anchor=south] (52) at (-2,3) {$bib\leqy\aff$};
	\node[block,anchor=south] (53) at (-8,3) {\vphantom{$\f$}$bib\leqy\ident$};
	\node[block,anchor=south] (54) at (8,14) {$\bif\leqy ib$};
	\node[block,anchor=south] (55) at (2,11) {\vphantom{$\f$}$\ident\leqy ib$};
	\node[block,anchor=south] (56) at (-4,7) {$\fb\leqy\bif$};
	\node[block,anchor=south] (57) at (-2,10) {$\fb\leqy\ident$};
	\node[topblock,anchor=south] (58) at (4,14) {$\fbi\leqy a$};
	\node[topblock,anchor=south] (59) at (-2,17) {$\fbi\leqy\fib$};
	\node[topblock,anchor=south] (60) at (0,17) {$\fib\leqy\fif$};
	\node[topblock,anchor=south] (61) at (-8,17) {$\fbi\leqy\ident$};
	\node[block,anchor=south] (62) at (-2,6) {$\ff\leqy a$};
	\node[block,anchor=south] (63) at (2,3) {$\ff\leqy ibi$};
	\node[block,anchor=south] (64) at (2,6) {$\ff\leqy\ident$};
	\node[block,anchor=south] (65) at (2,10) {$\ei\leqy a$};
	\node[block,anchor=south] (66) at (4,7) {$\ei\leqy\bif$};
	\node[topblock,anchor=south] (67) at (8,17) {$\fib\leqy a$};
	\node[topblock,anchor=south] (68) at (2,17) {$\fib\leqy\fbi$};
	\node[topblock,anchor=south] (69) at (-4,14) {$\fib\leqy\ident$};
	\node[topblock,anchor=south] (70) at (-4,17) {$\fif\leqy a$};
	\node[topblock,anchor=south] (71) at (4,17) {$\fif\leqy\ident$};
	\node[block,anchor=south] (72) at (8,3) {\vphantom{$\f$}$\ident\leqy ibi$};
	\node[block,anchor=south] (73) at (-6,1) {$\ident\leqy\ie$};

	\node[anchor=south] at (61.north) {\scriptsize$1$};
	\node[anchor=south] at (15.north) {\scriptsize$2$};
	\node[anchor=south] at (70.north) {\scriptsize$3$};
	\node[anchor=south] at (59.north) {\scriptsize$4$};
	\node[anchor=south] at (60.north) {\scriptsize$5$};
	\node[anchor=south] at (68.north) {\scriptsize$6$};
	\node[anchor=south] at (71.north) {\scriptsize$7$};
	\node[anchor=south] at (24.north) {\scriptsize$8$};
	\node[anchor=south] at (67.north) {\scriptsize$9$};
	\node[anchor=south] at (50.140) {\scriptsize$10$};
	\node[anchor=south] at (49.north) {\scriptsize$11$};
	\node[anchor=south] at (43.north) {\scriptsize$12$};
	\node[anchor=south] at (69.north) {\scriptsize$13$};
	\node[anchor=south] at (42.north) {\scriptsize$14$};
	\node[anchor=south] at (58.north) {\scriptsize$15$};
	\node[anchor=south] at (45.north) {\scriptsize$16$};
	\node[anchor=south] at (54.north) {\scriptsize$17$};
	\node[anchor=south] at (28.140) {\scriptsize$18$};
	\node[anchor=south] at (33.north) {\scriptsize$19$};
	\node[anchor=south] at (51.north) {\scriptsize$20$};
	\node[anchor=south] at (37.140) {\scriptsize$21$};
	\node[anchor=south] at (55.north) {\scriptsize$22$};
	\node[anchor=south] at (35.north) {\scriptsize$23$};
	\node[anchor=south] at (29.140) {\scriptsize$24$};
	\node[anchor=south] at (57.north) {\scriptsize$25$};
	\node[anchor=south] at (27.140) {\scriptsize$26$};
	\node[anchor=south] at (65.north) {\scriptsize$27$};
	\node[anchor=south] at (34.north) {\scriptsize$28$};
	\node[anchor=south] at (36.north) {\scriptsize$29$};
	\node[anchor=south] at (41.north) {\scriptsize$30$};
	\node[anchor=south] at (56.north east) {\quad\scriptsize$31$};
	\node[anchor=south] at (26.140) {\scriptsize$32$};
	\node[anchor=south] at (66.north west) {\scriptsize$33$\quad\quad};
	\node[anchor=south] at (39.north) {\scriptsize$34$};
	\node[anchor=south] at (21.north) {\scriptsize$35$};
	\node[anchor=north] at (62.330) {\scriptsize$36$};
	\node[anchor=south] at (38.north) {\scriptsize$37$};
	\node[anchor=north] at (64.210) {\scriptsize$38$};
	\node[anchor=south] at (14.north) {\scriptsize$39$};
	\node[anchor=south] at (40.north) {\scriptsize$40$};
	\node[anchor=south] at (22.north) {\scriptsize$41$};
	\node[anchor=south] at (13.north) {\scriptsize$42$};
	\node[anchor=south] at (44.north) {\scriptsize$43$};
	\node[anchor=east] at (32.150) {\raise14pt\hbox{\scriptsize$44$}};
	\node[anchor=south] at (53.north) {\scriptsize$45$};
	\node[anchor=south] at (17.north) {\scriptsize$46$};
	\node[anchor=south] at (52.north) {\scriptsize$47$};
	\node[anchor=north] at (31.220) {\scriptsize$48$};
	\node[anchor=south] at (63.north) {\scriptsize$49$};
	\node[anchor=south] at (1.north) {\scriptsize$50$};
	\node[anchor=south] at (72.north) {\scriptsize$51$};
	\node[anchor=south] at (25.north) {\scriptsize$52$};
	\node[anchor=west] at (0.0) {\!\scriptsize$53$};
	\node[anchor=west] at (20.0) {\!\scriptsize$54$};
	\node[anchor=west] at (23.80) {\raise9pt\hbox{\scriptsize$55$}};
	\node[anchor=east] at (16.100) {\raise9pt\hbox{\scriptsize$56$}};
	\node[anchor=east] at (10.180) {\scriptsize$57$\!};
	\node[anchor=east] at (2.180) {\scriptsize$58$\!};
	\node[anchor=south] at (48.north) {\scriptsize$59$};
	\node[anchor=west] at (73.0) {\!\scriptsize$60$};
	\node[anchor=west] at (18.0) {\!\scriptsize$61$};
	\node[anchor=south] at (4.north) {\scriptsize$62$};
	\node[anchor=south] at (30.140) {\scriptsize$63$};
	\node[anchor=south] at (6.north) {\scriptsize$64$};
	\node[anchor=east] at (5.180) {\scriptsize$65$\!};
	\node[anchor=east] at (9.180) {\scriptsize$66$\!};
	\node[anchor=south] at (7.north) {\scriptsize$67$};
	\node[anchor=east] at (19.180) {\scriptsize$68$\!};
	\node[anchor=south] at (12.140) {\scriptsize$69$};
	\node[anchor=west] at (47.0) {\!\scriptsize$70$};
	\node[anchor=south] at (11.north) {\scriptsize$71$};
	\node[anchor=south] at (8.140) {\scriptsize$72$};
	\node[anchor=south] at (46.north) {\scriptsize$73$};
	\node[anchor=south] at (3.140) {\scriptsize$74$};

	\draw[line] (0.north west) -- (53.south east);
	\draw[line] (0.north east) -- (64.south west);
	\draw[line] (1.north east) -- (14.south west);
	\draw[line] (2.north west) -- (62.south east);
	\draw[line] (2.north east) -- (72.south west);
	\draw[line] (3.north east) -- (11.south west);
	\draw[line] (3.west) -- (12.east);
	\draw[line] (3) -- (47);
	\draw[line] (4.north west) -- (0.south east);
	\draw[line] (4.north east) -- (32.south west);
	\draw[line] (5.north west) -- (16.south east);
	\draw[line] (5.north east) -- (48.south west);
	\draw[line] (6.north east) -- (2.south west);
	\draw[line] (6.north west) -- (32.south east);
	\draw[line] (7.north east) -- (0.south west);
	\draw[line] (7.north east) -- (18.south west);
	\draw[line] (8.north west) -- (7.south east);
	\draw[line] (8.east) -- (12.west);
	\draw[line] (8) -- (19);
	\draw[line] (9.north west) -- (10.south east);
	\draw[line] (9.north west) -- (64.south east);
	\draw[line] (10.north east) -- (1.south west);
	\draw[line] (10.north west) -- (31.south east);
	\draw[line] (10.north) -- (36.south west);
	\draw[line] (10.north) -- (44.south west);
	\draw[line] (11.north west) -- (2.south east);
	\draw[line] (11.north west) -- (5.south east);
	\draw[line] (12.north west) -- (4.south east);
	\draw[line] (12.north east) -- (6.south west);
	\draw[line] (12.north east) -- (16.south west);
	\draw[line] (12.north west) -- (23.south east);
	\draw[line] (12) -- (30);
	\draw[line] (13.north west) -- (27.south east);
	\draw[line] (13.north east) -- (39.south west);
	\draw[line] (14.north west) -- (29.south east);
	\draw[line] (14.north west) -- (35.south east);
	\draw[line] (14.north west) -- (55.south east);
	\draw[line] (14.west) -- (56.east);
	\draw[line] (14.north west) -- (57.south east);
	\draw[line] (16.north west) -- (26.south east);
	\draw[line] (16) -- (63);
	\draw[line] (16.north east) -- (72.south west);
	\draw[line] (17.north west) -- (21.south east);
	\draw[line] (18.north east) -- (23.south west);
	\draw[line] (18.north west) -- (25.south east);
	\draw[line] (19.north west) -- (25.south east);
	\draw[line] (19.east) -- (30.south west);
	\draw[line] (19) -- (73);
	\draw[line] (20.north west) -- (17.south east);
	\draw[line] (20.north east) -- (31.south west);
	\draw[line] (20.north) -- (34.south east);
	\draw[line] (20.north) -- (40.south east);
	\draw[line] (21.north east) -- (28.south west);
	\draw[line] (21.north east) -- (33.south west);
	\draw[line] (21.north east) -- (51.south west);
	\draw[line] (21.north east) -- (65.south west);
	\draw[line] (21.east) -- (66.west);
	\draw[line] (22.north east) -- (27.south west);
	\draw[line] (22.north west) -- (41.south east);
	\draw[line] (23.north east) -- (26.south west);
	\draw[line] (23) -- (52);
	\draw[line] (23.north west) -- (53.south east);
	\draw[line] (25.north east) -- (17.south west);
	\draw[line] (25.north east) -- (52.south west);
	\draw[line] (26) -- (27);
	\draw[line] (26.north west) -- (28.south east);
	\draw[line] (26.north east) -- (29.south west);
	\draw[line] (27) -- (37);
	\draw[line] (28) -- (43);
	\draw[line] (28.north west) -- (49.south east);
	\draw[line] (28.north east) -- (50.south west);
	\draw[line] (28.north east) -- (58.south west);
	\draw[line] (28.north west) -- (61.south east);
	\draw[line] (29) -- (45);
	\draw[line] (29.north west) -- (50.south east);
	\draw[line] (29.north east) -- (54.south west);
	\draw[line] (29.north east) -- (67.south west);
	\draw[line] (29.north west) -- (69.south east);
	\draw[line] (30) -- (31);
	\path (30.north east) edge[out=79,in=281] (38.south east);
	\draw[line] (30.north west) -- (52.south east);
	\draw[line] (30.north west) -- (62.south);
	\draw[line] (30.north east) -- (63.south west);
	\draw[line] (30.north east) -- (64.south);
	\draw[line] (31.north west) -- (56.south east);
	\draw[line] (31.north east) -- (66.south west);
	\draw[line] (32.north east) -- (13.south west);
	\draw[line] (32.north west) -- (22.south east);
	\draw[line] (32) -- (38);
	\draw[line] (33.north east) -- (24.south west);
	\draw[line] (33.north east) -- (59.south west);
	\draw[line] (34.north east) -- (15.south west);
	\draw[line] (34.north east) -- (33.south west);
	\draw[line] (35.north west) -- (15.south east);
	\draw[line] (35.north west) -- (68.south east);
	\draw[line] (36.north west) -- (24.south east);
	\draw[line] (36.north west) -- (35.south east);
	\draw[line] (37) -- (42);
	\draw[line] (37.north west) -- (49.south east);
	\draw[line] (37.north east) -- (54.south west);
	\draw[line] (37.north west) -- (70.south east);
	\draw[line] (37.north east) -- (71.south west);
	\draw[line] (38.north west) -- (34.south east);
	\draw[line] (38.north east) -- (36.south west);
	\draw[line] (38.north east) -- (39.west);
	\draw[line] (38.north west) -- (41.east);
	\draw[line] (39.north west) -- (35.south east);
	\draw[line] (39.west) -- (42.south);
	\draw[line] (40.north east) -- (43.south west);
	\draw[line] (40.north east) -- (56.south west);
	\draw[line] (41.north east) -- (33.south west);
	\draw[line] (41.east) -- (42.south);
	\draw[line] (42.north west) -- (59.south east);
	\draw[line] (42.north east) -- (68.south west);
	\draw[line] (43.north east) -- (59.south west);
	\draw[line] (43.north east) -- (60.south west);
	\draw[line] (44.north west) -- (45.south east);
	\draw[line] (44.north west) -- (66.south east);
	\draw[line] (45.north west) -- (60.south east);
	\draw[line] (45.north west) -- (68.south east);
	\draw[line] (46.west) -- (19.east);
	\draw[line] (46.east) -- (47.west);
	\draw[line] (47) -- (9);
	\draw[line] (47.west) -- (30.south east);
	\draw[line] (47.north east) -- (48.south west);
	\draw[line] (48.north west) -- (1.south east);
	\draw[line] (48.north west) -- (63.south east);
	\draw[line] (49.north east) -- (59.south west);
	\draw[line] (50) -- (60);
	\draw[line] (51.north east) -- (24.south west);
	\draw[line] (51.north west) -- (61.south east);
	\draw[line] (52.north west) -- (21.south east);
	\draw[line] (52.north east) -- (29.south west);
	\draw[line] (52.north east) -- (37.south west);
	\draw[line] (52.north west) -- (41.south east);
	\draw[line] (52.north east) -- (44.south west);
	\draw[line] (53.north east) -- (22.south west);
	\draw[line] (53.north east) -- (51.south west);
	\draw[line] (53.north east) -- (69.south);
	\draw[line] (54.north west) -- (68.south east);
	\draw[line] (55.north west) -- (15.south east);
	\draw[line] (55.north east) -- (67.south west);
	\draw[line] (56.north west) -- (15.south east);
	\draw[line] (56.110) -- (60.south west);
	\draw[line] (57.north west) -- (69.south east);
	\draw[line] (62.north) -- (55.south west);
	\draw[line] (62.north) -- (65.south west);
	\draw[line] (62.north west) -- (70.south east);
	\draw[line] (63.north east) -- (14.south west);
	\draw[line] (63.north west) -- (28.south east);
	\draw[line] (63.north west) -- (37.south east);
	\draw[line] (63.north east) -- (39.south west);
	\draw[line] (63.north west) -- (40.south east);
	\draw[line] (64.north) -- (51.south east);
	\draw[line] (64.north) -- (57.south east);
	\draw[line] (64.north east) -- (71.south west);
	\draw[line] (65.north east) -- (58.south west);
	\draw[line] (66.north east) -- (24.south west);
	\draw[line] (66.70) -- (60.south east);
	\draw[line] (72.north west) -- (13.south east);
	\draw[line] (72.north west) -- (55.south east);
	\draw[line] (72.north west) -- (58.south);
	\draw[line] (73.north east) -- (20.south west);
	\draw[line] (73.north east) -- (62.south west);

\end{tikzpicture}
\vspace{6pt}
\caption{The partial order on relational classes in $\KFA$ under logical implication.\\
\smallskip\small$\KF$ counterparts are shown to save space.
Relations are indexed in Table~\ref{tab:index}.}
\label{fig:implications_poset}\vspace{-6pt}
\end{figure}

\begin{table}
\caption{Index to Figure~\ref{fig:implications_poset} and Table~\ref{tab:implications_poset}.}
\small
\centering
\renewcommand{\tabcolsep}{2pt}
\renewcommand{\arraystretch}{1}
\begin{tabular}{llcllcllcllcllcllcllcll}
$b\eqy0$&$72$&&$bib\eqy\abif$&$65$&&$\f\eqy b$&$46$&&$\ei\eqy bib$&$61$&&$i\eqy\abif$&$57$&&$ib\eqy\afib$&$50$&&$ibi\eqy bib$&$32$&&$\ident\eqy\ff$&$67$\\
$b\eqy1$&$50$&	&$bib\eqy\af$&$74$&	&$\f\eqy bi$&$72$&	&$\ei\eqy\bif$&$55$&	&$i\eqy\af$&$50$&	&$ib\eqy\afif$&$59$&	&$ibi\eqy\bif$&$61$&	&$\ident\eqy\ei$&$72$\\
$b\eqy abi$&$73$&	&$bib\eqy\afb$&$50$&	&$\f\eqy bib$&$54$&	&$\ei\eqy\f$&$42$&	&$i\eqy\afb$&$74$&	&$ib\eqy aibi$&$73$&	&$ibi\eqy\f$&$72$&	&$\ident\eqy\fib$&$72$\\
$b\eqy\abif$&$65$&	&$bib\eqy\afbi$&$59$&	&$\f\eqy\bif$&$48$&	&$\ei\eqy\fb$&$37$&	&$i\eqy\afbi$&$74$&	&$ib\eqy\aif$&$65$&	&$ibi\eqy\fb$&$68$&	&$\ident\eqy\fif$&$72$\\
$b\eqy\af$&$74$&	&$bib\eqy\aff$&$70$&	&$\fb\eqy0$&$39$&	&$\ei\eqy\fbi$&$\phantom{0}8$&	&$i\eqy\aff$&$65$&	&$ib\eqy b$&$39$&	&$ibi\eqy\fbi$&$46$&	&$\ident\eqy i$&$58$\\
$b\eqy\afb$&$50$&	&$bib\eqy\afi$&$70$&	&$\fb\eqy b$&$61$&	&$\ei\eqy\ff$&$23$&	&$i\eqy\afi$&$65$&	&$ib\eqy bi$&$32$&	&$ibi\eqy\ff$&$68$&	&$\ident\eqy ib$&$64$\\
$b\eqy\afbi$&$59$&	&$bib\eqy\afib$&$50$&	&$\fb\eqy bi$&$68$&	&$\fib\eqy0$&$24$&	&$i\eqy\afib$&$74$&	&$ib\eqy bib$&$24$&	&$ibi\eqy\ei$&$46$&	&$\ident\eqy ibi$&$64$\\
$b\eqy\aff$&$70$&	&$bib\eqy\afif$&$59$&	&$\fb\eqy bib$&$72$&	&$\fib\eqy b$&$72$&	&$i\eqy\afif$&$74$&	&$ib\eqy\bif$&$52$&	&$ibi\eqy\fib$&$52$&	&$\ie\eqy0$&$26$\\
$b\eqy\afi$&$70$&	&$bib\eqy aibi$&$73$&	&$\fb\eqy\bif$&$56$&	&$\fib\eqy bi$&$52$&	&$i\eqy aib$&$73$&	&$ib\eqy\f$&$68$&	&$ibi\eqy\fif$&$52$&	&$\ie\eqy1$&$73$\\
$b\eqy\afib$&$50$&	&$bib\eqy\aif$&$65$&	&$\fb\eqy\f$&$41$&	&$\fib\eqy bib$&$61$&	&$i\eqy\aif$&$70$&	&$ib\eqy\fb$&$72$&	&$ibi\eqy i$&$\phantom{0}8$&	&$\ie\eqy\afb$&$73$\\
$b\eqy\afif$&$59$&	&$bib\eqy b$&$\phantom{0}2$&	&$\fbi\eqy0$&$18$&	&$\fib\eqy\bif$&$32$&	&$i\eqy b$&$69$&	&$ib\eqy\fbi$&$61$&	&$ibi\eqy ib$&$26$&	&$\ie\eqy\afbi$&$73$\\
$b\eqy ai$&$73$&	&$bib\eqy bi$&$26$&	&$\fbi\eqy b$&$72$&	&$\fib\eqy\f$&$44$&	&$i\eqy bi$&$35$&	&$ib\eqy\ff$&$72$&	&$\ident\eqy0$&$72$&	&$\ie\eqy\afi$&$73$\\
$b\eqy aibi$&$73$&	&$\bif\eqy0$&$26$&	&$\fbi\eqy bi$&$46$&	&$\fib\eqy\fb$&$\phantom{0}2$&	&$i\eqy bib$&$55$&	&$ib\eqy\ei$&$61$&	&$\ident\eqy1$&$74$&	&$\ie\eqy\afib$&$73$\\
$b\eqy\aif$&$65$&	&$\bif\eqy1$&$73$&	&$\fbi\eqy bib$&$61$&	&$\fib\eqy\fbi$&$14$&	&$i\eqy\bif$&$61$&	&$ib\eqy\fib$&$61$&	&$\ident\eqy\abif$&$74$&	&$\ie\eqy\afif$&$73$\\
$bi\eqy0$&$46$&	&$\bif\eqy\afb$&$73$&	&$\fbi\eqy\bif$&$32$&	&$\fib\eqy\ff$&$28$&	&$i\eqy\f$&$72$&	&$ib\eqy\fif$&$61$&	&$\ident\eqy\af$&$71$&	&$\ie\eqy b$&$68$\\
$bi\eqy1$&$65$&	&$\bif\eqy\afbi$&$73$&	&$\fbi\eqy\f$&$44$&	&$\fib\eqy\ei$&$30$&	&$i\eqy\fb$&$68$&	&$ib\eqy i$&$41$&	&$\ident\eqy\afb$&$74$&	&$\ie\eqy bi$&$61$\\
$bi\eqy abib$&$73$&	&$\bif\eqy\afi$&$73$&	&$\fbi\eqy\fb$&$34$&	&$\fif\eqy0$&$21$&	&$i\eqy\fbi$&$46$&	&$ibi\eqy0$&$46$&	&$\ident\eqy\afbi$&$74$&	&$\ie\eqy bib$&$52$\\
$bi\eqy\abif$&$59$&	&$\bif\eqy\afib$&$73$&	&$\ff\eqy0$&$63$&	&$\fif\eqy b$&$72$&	&$i\eqy\ff$&$68$&	&$ibi\eqy1$&$65$&	&$\ident\eqy\aff$&$71$&	&$\ie\eqy\bif$&$21$\\
$bi\eqy\af$&$70$&	&$\bif\eqy\afif$&$73$&	&$\ff\eqy\abif$&$73$&	&$\fif\eqy bi$&$52$&	&$i\eqy\ei$&$46$&	&$ibi\eqy\abif$&$50$&	&$\ident\eqy\afi$&$71$&	&$\ie\eqy\f$&$63$\\
$bi\eqy\afb$&$65$&	&$\bif\eqy b$&$54$&	&$\ff\eqy\aif$&$73$&	&$\fif\eqy bib$&$61$&	&$i\eqy\fib$&$52$&	&$ibi\eqy\af$&$57$&	&$\ident\eqy\afib$&$74$&	&$\ie\eqy\fb$&$56$\\
$bi\eqy\afbi$&$65$&	&$\bif\eqy bi$&$61$&	&$\ff\eqy b$&$61$&	&$\fif\eqy\bif$&$26$&	&$i\eqy\fif$&$52$&	&$ibi\eqy\afb$&$65$&	&$\ident\eqy\afif$&$74$&	&$\ie\eqy\fbi$&$32$\\
$bi\eqy\aff$&$74$&	&$\bif\eqy bib$&$46$&	&$\ff\eqy bi$&$68$&	&$\fif\eqy\f$&$69$&	&$ib\eqy0$&$61$&	&$ibi\eqy\afbi$&$65$&	&$\ident\eqy\aif$&$74$&	&$\ie\eqy\ff$&$69$\\
$bi\eqy\afi$&$74$&	&$\f\eqy0$&$69$&	&$\ff\eqy bib$&$72$&	&$\fif\eqy\fb$&$40$&	&$ib\eqy1$&$50$&	&$ibi\eqy\aff$&$74$&	&$\ident\eqy b$&$53$&	&$\ie\eqy\ei$&$55$\\
$bi\eqy\afib$&$65$&	&$\f\eqy1$&$73$&	&$\ff\eqy\bif$&$69$&	&$\fif\eqy\fbi$&$16$&	&$ib\eqy\abif$&$65$&	&$ibi\eqy\afi$&$74$&	&$\ident\eqy bi$&$62$&	&$\ie\eqy\fib$&$32$\\
$bi\eqy\afif$&$65$&	&$\f\eqy\afb$&$73$&	&$\ff\eqy\f$&$26$&	&$\fif\eqy\ff$&$48$&	&$ib\eqy\af$&$74$&	&$ibi\eqy\afib$&$65$&	&$\ident\eqy bib$&$62$&	&$\ie\eqy\fif$&$26$\\
$bi\eqy aib$&$73$&	&$\f\eqy\afbi$&$73$&	&$\ff\eqy\fb$&$19$&	&$\fif\eqy\ei$&$43$&	&$ib\eqy\afb$&$50$&	&$ibi\eqy\afif$&$65$&	&$\ident\eqy\bif$&$72$&	&$\ie\eqy i$&$61$\\
$bi\eqy\aif$&$50$&	&$\f\eqy\aff$&$73$&	&$\ff\eqy\fbi$&$29$&	&$\fif\eqy\fib$&$12$&	&$ib\eqy\afbi$&$59$&	&$ibi\eqy\aif$&$59$&	&$\ident\eqy\f$&$67$&	&$\ie\eqy ib$&$46$\\
$bi\eqy b$&$42$&	&$\f\eqy\afi$&$73$&	&$\ei\eqy0$&$35$&	&$i\eqy0$&$46$&	&$ib\eqy\aff$&$70$&	&$ibi\eqy b$&$56$&	&$\ident\eqy\fb$&$67$&	&$\ie\eqy ibi$&$61$\\
$bib\eqy0$&$61$&	&$\f\eqy\afib$&$73$&	&$\ei\eqy b$&$72$&	&$i\eqy1$&$74$&	&$ib\eqy\afi$&$70$&	&$ibi\eqy bi$&$18$&	&$\ident\eqy\fbi$&$72$&	&$\ie\eqy\ident$&$72$\\
$bib\eqy1$&$50$&	&$\f\eqy\afif$&$73$&	&$\ei\eqy bi$&$46$&	&$i\eqy abib$&$73$&&&&&&&&&&&&\\\hline\vrule width0pt height9pt
$a\leqy bi$&$65$&	&$abib\leqy\ff$&$50$&	&$bi\leqy aib$&$46$&	&$\bif\leqy\afib$&$17$&	&$\fb\leqy bi$&$23$&	&$\ff\leqy bib$&$\phantom{0}2$&	&$\fib\leqy\fbi$&$\phantom{0}6$&	&$ib\leqy\ident$&$41$\\
$a\leqy bib$&$50$&	&$\abif\leqy\fb$&$73$&	&$bi\leqy\bif$&$46$&	&$\bif\leqy aib$&$26$&	&$\fb\leqy bib$&$\phantom{0}2$&	&$\ff\leqy\bif$&$48$&	&$\fib\leqy\ei$&$\phantom{0}6$&	&$ibi\leqy a$&$46$\\
$a\leqy\bif$&$57$&	&$\abif\leqy\fbi$&$73$&	&$bi\leqy\f$&$46$&	&$\bif\leqy bi$&$26$&	&$\fb\leqy\bif$&$31$&	&$\ff\leqy ib$&$39$&	&$\fib\leqy\fif$&$\phantom{0}5$&	&$ibi\leqy\f$&$46$\\
$a\leqy\f$&$50$&	&$\abif\leqy\ei$&$73$&	&$bi\leqy\fb$&$46$&	&$\bif\leqy\fb$&$26$&	&$\fb\leqy\fbi$&$23$&	&$\ff\leqy ibi$&$49$&	&$\fib\leqy\ident$&$13$&	&$ibi\leqy\ff$&$46$\\
$a\leqy\fb$&$74$&	&$\abif\leqy\fib$&$73$&	&$bi\leqy\ff$&$46$&	&$\bif\leqy\fbi$&$26$&	&$\fb\leqy\ei$&$23$&	&$\ff\leqy\ident$&$38$&	&$\fif\leqy a$&$\phantom{0}3$&	&$ibi\leqy\ei$&$46$\\
$a\leqy\fbi$&$74$&	&$\aff\leqy\bif$&$73$&	&$bi\leqy\fib$&$46$&	&$\bif\leqy\ff$&$26$&	&$\fb\leqy\fif$&$31$&	&$\ei\leqy a$&$27$&	&$\fif\leqy aib$&$\phantom{0}4$&	&$ibi\leqy\ident$&$\phantom{0}8$\\
$a\leqy\ff$&$65$&	&$aib\leqy\bif$&$50$&	&$bi\leqy\fif$&$46$&	&$\bif\leqy\ei$&$26$&	&$\fb\leqy\ident$&$25$&	&$\ei\leqy\afib$&$10$&	&$\fif\leqy bi$&$\phantom{0}6$&	&$\ident\leqy\aif$&$26$\\
$a\leqy\ei$&$65$&	&$aib\leqy\f$&$50$&	&$bi\leqy ib$&$10$&	&$\bif\leqy\fib$&$26$&	&$\fbi\leqy a$&$15$&	&$\ei\leqy\afif$&$11$&	&$\fif\leqy\fb$&$\phantom{0}4$&	&$\ident\leqy bi$&$42$\\
$a\leqy\fib$&$74$&	&$aib\leqy\fbi$&$50$&	&$bi\leqy\ident$&$20$&	&$\bif\leqy ib$&$17$&	&$\fbi\leqy\afib$&$10$&	&$\ei\leqy aib$&$19$&	&$\fif\leqy\fbi$&$\phantom{0}6$&	&$\ident\leqy bib$&$\phantom{0}2$\\
$a\leqy\fif$&$74$&	&$aib\leqy\ff$&$50$&	&$bib\leqy a$&$61$&	&$\bif\leqy\ident$&$26$&	&$\fbi\leqy\afif$&$11$&	&$\ei\leqy aibi$&$\phantom{0}8$&	&$\fif\leqy\ei$&$\phantom{0}6$&	&$\ident\leqy\bif$&$54$\\
$a\leqy ib$&$50$&	&$aib\leqy\ei$&$50$&	&$bib\leqy\af$&$55$&	&$\f\leqy a$&$58$&	&$\fbi\leqy aib$&$\phantom{0}4$&	&$\ei\leqy\bif$&$33$&	&$\fif\leqy\fib$&$\phantom{0}4$&	&$\ident\leqy\f$&$46$\\
$a\leqy ibi$&$65$&	&$aib\leqy\fif$&$50$&	&$bib\leqy\afb$&$24$&	&$\f\leqy aib$&$41$&	&$\fbi\leqy\bif$&$\phantom{0}5$&	&$\ei\leqy\fb$&$19$&	&$\fif\leqy ib$&$17$&	&$\ident\leqy\fb$&$61$\\
$a\leqy\ie$&$66$&	&$aibi\leqy\f$&$50$&	&$bib\leqy\aff$&$47$&	&$\f\leqy aibi$&$\phantom{0}8$&	&$\fbi\leqy\fb$&$\phantom{0}4$&	&$\ei\leqy\fib$&$19$&	&$\fif\leqy\ident$&$\phantom{0}7$&	&$\ident\leqy\fbi$&$72$\\
$abi\leqy\bif$&$50$&	&$aibi\leqy\ff$&$65$&	&$bib\leqy\f$&$46$&	&$\f\leqy bi$&$42$&	&$\fbi\leqy\fib$&$\phantom{0}4$&	&$\ei\leqy\fif$&$33$&	&$ib\leqy a$&$61$&	&$\ident\leqy\ff$&$61$\\
$abi\leqy\f$&$50$&	&$aibi\leqy\ei$&$65$&	&$bib\leqy\fb$&$61$&	&$\f\leqy bib$&$\phantom{0}2$&	&$\fbi\leqy\fif$&$\phantom{0}5$&	&$\ei\leqy ib$&$10$&	&$ib\leqy bi$&$26$&	&$\ident\leqy\ei$&$72$\\
$abi\leqy\fb$&$65$&	&$bi\leqy a$&$46$&	&$bib\leqy\ff$&$61$&	&$\f\leqy ib$&$39$&	&$\fbi\leqy ib$&$10$&	&$\ei\leqy ibi$&$18$&	&$ib\leqy\bif$&$46$&	&$\ident\leqy\fib$&$72$\\
$abi\leqy\ff$&$65$&	&$bi\leqy\abif$&$11$&	&$bib\leqy\ident$&$45$&	&$\f\leqy ibi$&$56$&	&$\fbi\leqy\ident$&$\phantom{0}1$&	&$\ei\leqy\ident$&$20$&	&$ib\leqy\f$&$46$&	&$\ident\leqy\fif$&$72$\\
$abi\leqy\fib$&$65$&	&$bi\leqy\af$&$35$&	&$\bif\leqy a$&$26$&	&$\f\leqy\ident$&$53$&	&$\ff\leqy a$&$36$&	&$\fib\leqy a$&$\phantom{0}9$&	&$ib\leqy\fbi$&$61$&	&$\ident\leqy ib$&$22$\\
$abi\leqy\fif$&$65$&	&$bi\leqy\afb$&$10$&	&$\bif\leqy\afb$&$17$&	&$\fb\leqy a$&$22$&	&$\ff\leqy aib$&$19$&	&$\fib\leqy\afif$&$17$&	&$ib\leqy\ff$&$61$&	&$\ident\leqy ibi$&$51$\\
$abi\leqy ib$&$50$&	&$bi\leqy\aff$&$35$&	&$\bif\leqy\afbi$&$11$&	&$\fb\leqy\afbi$&$10$&	&$\ff\leqy aibi$&$\phantom{0}8$&	&$\fib\leqy bi$&$\phantom{0}6$&	&$ib\leqy\ei$&$61$&	&$\ident\leqy\ie$&$60$\\
$abib\leqy\f$&$50$&	&$bi\leqy\afib$&$10$&	&$\bif\leqy\aff$&$21$&	&$\fb\leqy\afi$&$10$&	&$\ff\leqy bi$&$23$&	&$\fib\leqy\bif$&$\phantom{0}5$&	&$ib\leqy\fif$&$61$&	&$\ie\leqy\ident$&$26$\\
$abib\leqy\fb$&$50$&	&$bi\leqy\afif$&$11$&	&$\bif\leqy\afi$&$11$&	&$\fb\leqy\afif$&$17$&&&&&&&&&&&&
\end{tabular}\label{tab:index}
\end{table}

\begin{figure}[!ht]
\centering
\small
\begin{tikzpicture}
[auto,
 block/.style={rectangle,draw=black,align=center,minimum width={40pt},
 minimum height={16pt},inner sep=2pt,scale=.8},
 topblock/.style={rectangle,draw=black,line width=1pt,align=center,minimum width={38pt},
 minimum height={16pt},inner sep=2pt,scale=.8},
 line/.style ={draw}]

	\node[block,anchor=south] (0) at (0,0) {D: $41,44,45,53,55,62,69$};
	\node[block,anchor=east] (1) at (-.425,1) {EO,\,D: $32$};
	\node[block,anchor=west] (2) at (.368,1) {P,\,D: $2,19,20,28,30,31,35,36,37,40,47,48,63$};
	\node[block,anchor=east] (3) at (-.415,1.8) {OU,\,EO,\,D: $26$};
	\node[block,anchor=west] (4) at (.368,1.8) {P,\,D, and optionally in ED,\,EO: $25$};
	\node[block,anchor=west] (5) at (.368,2.6) {ED,\,EO,\,P,\,D: $1,5,10,12,13,18$};
	\node[block,anchor=south] (6) at (0.01,3.18) {KD,\,OU,\,ED,\,EO,\,P,\,D: $3,4,11,14,21$};
	\node[block,anchor=south] (7) at (0,-.8) {cannot hold for all $A\su X$: $46,52,54,60,61,67,68,72,73$};

	\draw[line] (0.168) -- (1.270);
	\draw[line] (0.12) -- (2.186);
	\draw[line] (1.88) -- (3.320);
	\draw[line] (1.north east) -- (5.south west);
	\draw[line] (2.164) -- (4.248);
	\draw[line] (3.40) -- (6.192);
	\draw[line] (4.113) -- (5.302);
	\draw[line] (5.170) -- (6.347);
	\draw[line] (7.90) -- (0.270);

\end{tikzpicture}
\caption{Operator relations in $\KF$ up to duality, ordered by logical implication (see
Table~\ref{tab:implications_poset}).}
\label{fig:implications_space}
\end{figure}
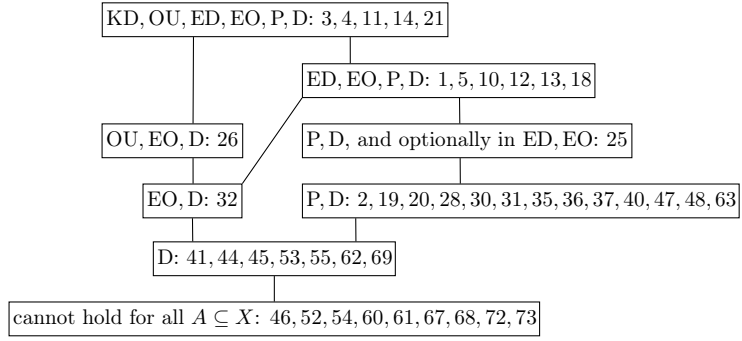

\section{The Interplay Between \texorpdfstring{$\mathbf\KF$}{KF} and
\texorpdfstring{$\KFA$}{KFA}}\label{sec:interplay}

\subsection{Introduction.}\label{sub:interplay_intro}

Definition~$3$.$3$ and Lemma~$3$.$15$ in GJ are reprinted below.

\begin{dfn}\label{dfn:topsum} \nit{(GJ, p\mbox{.} $22$)}
Let \mth{$\{(X_i,\mathcal{T}_i):i\in I\}$} be a family of topological spaces\nit.
The \define{sum space} \mth{$\sum_{i\in I}(X_i,\mathcal{T}_i)$} is the space on the
disjoint union \mth{$\dot{\cup}_{i\in I}X_i:=\cup_{i\in I}(X_i\times\{i\})$} with base
\mth{$\cup_{i\in I}\{S\times\{i\}:S\in\mathcal{T}_i\}$.}
\end{dfn}

\begin{lem}\label{lem:topsum} \nit{(GJ, p\mbox{.} $27$)}
Let \mth{$\{(X_i,\mathcal{T}_i):i\in I\}$} be a family of spaces with the \mth{$X_i$}
pairwise disjoint and let \mth{$o_1$} and \mth{$o_2$} be two Kuratowski operators\nit.
Then \mth{$o_1$} and \mth{$o_2$} are equal on the sum space \mth{$\sum_{i\in I}
(X_i,\mathcal{T}_i)$} if and only if they are equal on each \mth{$(X_i,\mathcal{T}_i)$.}
Furthermore if \mth{$S:=\dot{\cup}_{i\in I}S_i$} is a subset of
\mth{$\dot{\cup}_{i\in I}X_i$} then \mth{$o_1$} and \mth{$o_2$} agree on \mth{$S$} if and
only if they agree on each \mth{$S_i$.}
\end{lem}

Let $X_n:=\textstyle\sum_{i=1}^n\XT$ for $n\geq1$.
The following is Proposition~$3$.$16$ in GJ.

\begin{prop}\label{prop:topsum} \nit{(GJ, p\mbox{.} $27$)}
Let \mth{$\XT$} have Kuratowski monoid \mth{$\K$.}

\noindent
\nit{\makebox[16pt]{\hfill(i)}} If \mth{$K(\XT)=2$} then \mth{$\XT$} is a full space\nit.

\noindent\hangindent=20pt
\nit{\makebox[16pt]{\hfill(ii)}} If \mth{$K(\XT)\in\{6,8\}$} then \mth{$X_2$} is full
with Kuratowski monoid \mth{$\K$.}

\noindent
\nit{\makebox[16pt]{\hfill(iii)}} If \mth{$K(\XT)=10$} then \mth{$X_3$} is full with
Kuratowski monoid \mth{$\K$.}

\noindent
\nit{\makebox[16pt]{\hfill(iv)}} If \mth{$K(\XT)=14$} then \mth{$X_4$} is full with
Kuratowski monoid \mth{$\K$.}
\end{prop}

As GJ point out, for $n\geq2$ Lemma~\ref{lem:topsum} implies $X_n$ has the same
Kuratowski monoid as $\XT$ and $K(\XT)>k(X_{n-1})\implies k(X_n)>
k(X_{n-1})$.\footnote{Since Lemma~\ref{lem:topsum} applies to all GE operators $o_1,o_2$
we similarly have $\Kf(\XT)>\kf(X_{n-1})\implies \kf(X_n)>\kf(X_{n-1})$.}
After proving Proposition~\ref{prop:topsum}(iv) for the hypothetical case $k(\XT)=6$ and
noting that the sum space on two copies of the minimal Kuratowski space does not contain
a $14$-set, they write:

\begin{center}
\begin{minipage}{320pt}\inlinequote{We do not know of a space where more than
three copies are required\nit, or in fact any Kuratowski space with k-number\/} $6$.
\end{minipage}
\end{center}

\noindent
No such spaces exist.
We prove this and investigate the $\KF$ analogues.
Several preliminary results are needed.

\begin{lem}\label{lem:i-union}\vrule width0pt

\noindent
\nit{\makebox[16pt]{\hfill(i)}}\phantom{t}If \mth{$A$} and \mth{$iB$} are closed then
\mth{$i(A\cup B)=iA\cup iB$} and \mth{$bi(A\cup B)=biA\cup biB$.}

\noindent
\nit{\makebox[16pt]{\hfill(ii)}}\phantom{t}If \mth{$A$} and \mth{$bB$} are open then
\mth{$b(A\cap B)=bA\cap bB$} and \mth{$ib(A\cap B)=ibA\cap ibB$.}

\noindent\hangindent=20pt
\nit{\makebox[16pt]{\hfill(iii)}}\phantom{t}For all \mth{$B\su X$,} \mth{$bibA=ibA
\implies ib(A\cup B)=ibA\cup ibB$} and \mth{$ibiA=biA\implies bi(A\cap B)=biA\cap biB$.}
\end{lem}

\begin{proof}
(i)~The hypothesis implies $i(A\cup B)=(i(A\cup B)\cap A)\cup(i(A\cup B)\sm A)\su A\cup
iB$.
Hence $i(A\cup B)=(i(A\cup B)\sm iB)\cup(i(A\cup B)\cap iB)\su iA\cup iB$.
The reverse inclusion holds in general.
The second equation follows.
(ii) is the dual of (i) and (iii) follows easily from (i) and (ii).
\end{proof}

\begin{lem}\label{lem:b_equals_bib}
\hspace{-8pt}\raisebox{-.5\baselineskip}{$\begin{array}[c]{r@{\hspace{3pt}}l}
\nit{(i)}&\mbox{If }A\mbox{ and }B\mbox{ each satisfy }bib=b\mbox{ then so does }A\cup
B\nit.\\
\nit{(ii)}&\mbox{If }A\mbox{ and }B\mbox{ each satisfy }ibi=i\mbox{ then so does }A\cap
B\nit.\\
\end{array}$}
\end{lem}

\begin{proof}
(i)~$b(A\cup B)=bA\cup bB=bibA\cup bibB=b(ibA\cup ibB)\su bi(bA\cup bB)=bib(A\cup B)$.
(ii) is the dual of (i).
\end{proof}

\begin{lem}\label{lem:general_lemma}
\hspace{-10pt}\raisebox{-\baselineskip}{$\begin{array}[c]{r@{\hspace{3pt}}l}
\nit{(i)}&ibA=ib(A\cap ibA)=ib(A\cap bibA)\nit.\\
\nit{(ii)}&A\cap ibA\mbox{ and }A\cap bibA\mbox{ each satisfy }bib=b\nit.\\
\nit{(iii)}&\f(A\cap B)=(b(A\cap B)\sm iA)\cup(b(A\cap B)\sm iB)\nit.\\
\end{array}$}
\end{lem}

\begin{proof}
(i)~Let $x\in ibA$ and $U$ be an open neighborhood of $x$.
Since $x\in bA$, $U\cap ibA\cap A\neq\es$.
Thus $ibA\su ib(A\cap ibA)\su ib(A\cap bibA)\su ibA$.
(ii)~By (i) we have $b(A\cap ibA)\su bA\cap bibA=bibA=bib(A\cap ibA)$.
The same applies to $A\cap bibA$.
(iii)~Have $ba(A\cap B)=b(aA\cup aB)=aiA\cup aiB$.
The result follows.
\end{proof}

\begin{lem}\label{lem:b_equals_bi}\vrule width0pt

\noindent
\nit{\makebox[16pt]{\hfill(i)}}\phantom{t}If \mth{$A\cup B$} satisfies \mth{$bib=bi$}
then \mth{$\ifA\su ibB$.}

\noindent
\nit{\makebox[16pt]{\hfill(ii)}}\phantom{t}If \mth{$\ifA$} and \mth{$iB$} are both empty
then \mth{$bi(A\cup B)=biA$.}

\noindent
\nit{\makebox[16pt]{\hfill(iii)}}\phantom{t}If \mth{$\ifA$} is empty and \mth{$\ie B$} is
closed then \mth{$\ie(A\cup gB)=\ie(A\cup gaB)$.}

\noindent
\nit{\makebox[16pt]{\hfill(iv)}}\phantom{t}If \mth{$\ifA$} is empty then \mth{$\ie B\su
\ie(A\sym B)$} for all \mth{$B\su X$.}
\end{lem}

\begin{proof}
(i)~Apply Lemma~\ref{lem:half_dist}(i) to get $ibA\su bib(A\cup B)=bi(A\cup B)\su
b(iA\cup B)=biA\cup bB$.
Hence $\ifA=ibA\sm biA\su bB$.
Thus $\ifA\su ibB$.
(ii)~Have $i(A\cup B)\su bi(A\cup B)\su b(A\cup iB)=bA$.
Thus $i(A\cup B)\su ibA$.
Conclude $bi(A\cup B)\sm biA\su bibA\sm biA\su\bifA=\es$.
(iii)~By Lemma~\ref{lem:i-union}(i) we have $ib(A\cup gB)=i(bA\cup bgB)=ibA\cup\ie B$.
Substitute $aB$ for $B$ to get $ib(A\cup gaB)=ibA\cup\ifa B=ib(A\cup gB)$.
By (ii), $bi(A\cup gB)=biA=bi(A\cup gaB)$.
(iv)~Let $x\in\ie B\cap biA$ and $U$ be an open neighborhood of $x$.
Have $\es\neq U\cap\ie B\cap iA\su\f B$.
Thus $U\cap(A\cap B)\neq\es\neq U\cap(A\sm B)$.
Hence $x\in b(A\sm B)\cap b(A\cap B)$ $\su$ $(b(A\sm B)\cup b(B\sm A))$ $\cap$ $(ba
(A\cup B)\cup b(A\cap B))=\f(A\sym B)$.
Thus $\ie B\cap bi(aA)\su\f(aA\sym B)=\fa(A\sym B)=\f(A\sym B)$.
But $biA\cup bi(aA)=\aifA=X$.
\end{proof}

\subsection{Interrelationships between the GE monoid and local collapses.}

Obviously, dual \pn s always occur together in any given space.
This also holds for two dual pairs.

\begin{prop}\label{prop:always_together}
\pn s \mth{$18$-$21$ (}equivalently\nit, \sn s \mth{$49$-$52$)} always occur together in
any given space\nit.
\end{prop}

\begin{proof}
Suppose $\phi A=21$.
There exists a point $x\in biA\sm iA$.
Let $B=A\sm\{x\}$.
Since $iA\su B$ we have $x\in biA\su bB\su bA=A$.
Hence $bB=bA$.
Since $iA\su a(\{x\})$ we have $iA\su ia(\{x\})$.
Thus $iB=i(A\sm\{x\})=iA\cap ia(\{x\})=iA$.
It follows that $oB=oA$ for all $o\in\KZ\sm\{\ident\}$.
Since $A=bA\supsetneq biA\supsetneq iA$ we have $|A\sm iA|\geq2$.
Thus $B\neq iA=iB$.
Conclude $\phi B=19$.
Conversely $\phi A=19\implies\phi(bA)=21$ by Table~\ref{tab:subset_types}.
\end{proof}

\begin{cor}\label{cor:kf_12}
\mth{$\kf(\XT)\neq12$.}
\end{cor}

\begin{proof}
By Table~\ref{tab:subset_types}, $\phi A\in\{20,21\}\ify k_f(A)=12$ and
$\phi A\in\{18,19\}\imp k_f(A)=14$.
Apply Proposition~\ref{prop:always_together}.
\end{proof}

\begin{lem}\label{lem:k-number_26}
Kuratowski and OU spaces always contain at least one non-clopen, regular closed subset,
i.e., a subset with \pn\ \mth{$26$.}
\end{lem}

\begin{proof}
For some $A\su X$, $bib(biA)=bi(biA)=biA\neq i(biA)=ib(biA)$.
Conclude $\phi(biA)=26$ by Table~\ref{tab:characterizations}.
\end{proof}

\begin{lem}\label{lem:psi_42_44}
If \mth{$\phi A\in\{24,26\}$} and \mth{$\psi B=61$} then \mth{$\psi(A\cup B)=42$.}
\end{lem}

\begin{proof}
Have $X=bB\su b(A\cup B)$.
Thus $A\cup B$ satisfies $ab=0$.
Have $iB=\es$.
Thus by Lemma~\ref{lem:half_dist}(i), $ibi(A\cup B)\su ib(A\cup iB)=ibA=iA=iA\cup iB\su
i(A\cup B)$.
Hence $A\cup B$ satisfies $ibi=i=ab\vee i=\af$.
Left-multiplying $ibi(A\cup B)=iA$ by $b$ yields $bi(A\cup B)=biA\neq iA=i(A\cup B)$.
By Table~\ref{tab:characterizations} we conclude $\psi(A\cup B)=42$.
\end{proof}

\begin{prop}\label{prop:irresolvable}
KD spaces are irresolvable\nit.
\end{prop}

\begin{proof}
Suppose $\XT$ is a Kuratowski space containing a subset $B$ such that $\psi B=61$.
By Lemma~\ref{lem:k-number_26}, $X$ contains a subset $A$ with \pn\ $26$.
By Lemma~\ref{lem:psi_42_44} we have $\psi(A\cup B)=42$.
Conclude $\XT$ is not KD.
\end{proof}

\begin{lem}\label{lem:psi_60}
KD spaces always contain subsets \mth{$A\nsu B$} such that \mth{$\psi A=60$} and
\mth{$\psi B=62$.}
\end{lem}

\begin{proof}
Let $X$ be a KD space.
Claim $\psi A=60$ for some $A\su X$.
We have $bibE\neq biE$ for some $E\su X$.
Tables~\ref{tab:multiplication} and~\ref{tab:implications_poset} imply $\ie(gE)=\ie E
\neq\es$.
Since $X$ is KD and $gE$ satisfies $i=0\neq\ie$ it follows by
Table~\ref{tab:subset_types} that $\psi(gE)\in\{48,60\}$.
The claim holds if $\psi(gE)=60$ so we assume $\psi(gE)=48$.
Let $U=agE$.
Have $\bif=\ie$ since $X$ is KD.
Thus, since $abiU$ is open and $bU=X$, Lemma~\ref{lem:i-union}(ii) implies
$b(U\sm biU)=bU\cap babiU=aibiU=\bif E=\ie E\neq\es$.
It follows that $U\sm biU$ satisfies $ib=b$ and $\ident\neq0$.
Since we also have $i(U\sm biU)=iU\cap iabiU=iU\sm biU=\es$, Table~\ref{tab:subset_types}
and Proposition~\ref{prop:irresolvable} imply $\psi(U\sm biU)=60$.
Hence the claim holds.

By Lemma~\ref{lem:k-number_26} there exists $V\su X$ such that $\psi V=62$.
If $A\nsu V$ we are done so we assume $A\su V$.
Let $B=V\sm bA$.
Note that $\psi(iaA)=68$.
Since $iV$ and $biaA$ are each open Lemma~\ref{lem:i-union}(ii) implies $B=V\cap iaA=biV
\cap biaA=b(iV\cap iaA)=bi(V\cap iaA)=biB$.
Since $V$ is closed and $A\su V$ we have $bA\su V$.
Hence $bA\su iV$ since $bA$ is open.
Thus, since $V$ is not open, we get $i(V\sm bA)=iV\cap iabA=iV\sm bA\subsetneq V\sm bA$.
Conclude $\psi B=62$ by Table~\ref{tab:subset_types}.
Since $A\cap B=\es$ the result follows.
\end{proof}

\begin{lem}\label{lem:sufficient_31}
If \mth{$\bifA=\ifA\neq\es$,} \mth{$iA=ibiA\neq biA$,} and \mth{$bibA=bA$,} then
\mth{$\psi A=31$.}
\end{lem}

\begin{proof}
(i)~By Table~\ref{tab:subset_types}, $\bifA=\ifA\neq\es\implies(\psi A\in
\{6,13,20,31,38\}$ or $\phi A\in\{16,17,25\})$, $ibiA=iA\implies\psi A\cn\in\{6,13,38\}$,
$ibiA\neq biA\implies\phi A\cn\in\{16,17,25\}$, and $bibA=bA\implies\psi A\neq20$.
\end{proof}

\begin{lem}\label{lem:psi_31}
KD spaces always contain subsets with \sn s \mth{$31$} and \mth{$48$.}
\end{lem}

\begin{proof}
Let $X$ be a KD space.
By Lemma~\ref{lem:psi_60} there exist subsets $A\nsu B$ in $X$ such that $\psi A=60$ and
$\psi B=62$.
Claim $\psi(A\cup B)=31$.
Since $iA$ ($=\es$) and $B$ are closed, Lemma~\ref{lem:i-union}(i) implies $i(A\cup B)=iA\cup
iB=iB\neq biB=bi(A\cup B)$.
Hence $ibi(A\cup B)=ibiB=iB=i(A\cup B)$.
Lemma~\ref{lem:b_equals_bib}(i) implies $bib(A\cup B)=b(A\cup B)$.
Since $ibA$ and $B$ are closed and $baA$ and $aB$ are open we have $\ie(A\cup B)=ib(A\cup
B)\cap abi(A\cup B)=(ibA\cup iB)\cap ib(aA\cap aB)=(bA\cup iB)\cap ibaA\cap ibaB=bA\sm
B\neq\es$.
Since $X$ is KD, $\bif(A\cup B)=\ie(A\cup B)$.
Hence the claim holds by Lemma~\ref{lem:sufficient_31}.
This implies $\psi(g(A\cup B))\in\{48,60\}$.
Have $g(A\cup B)=(A\cup B)\sm i(A\cup B)=(A\cup B)\sm iB=(A\sm iB)\cup gB$.
Thus, since $gB\neq\es$, $g(A\cup B)\cap B\neq\es$.
Hence $ibg(A\cup B)=\ie(A\cup B)=bA\sm B\neq bg(A\cup B)$.
Conclude $\psi(g(A\cup B))=48$.
\end{proof}

\begin{lem}\label{lem:psi_37}
\mth{$\psi A=37\implies\psi(A\cap ibA)=\psi(A\cap bibA)=44$.}
\end{lem}

\begin{proof}
Have $i(A\cap ibA)\su iA=\es$.
Lemma~\ref{lem:general_lemma}(i)-(ii) imply $b(A\cap ibA)=bib(A\cap ibA)=bibA\neq ibA=
ib(A\cap ibA)$.
Thus $\psi(A\cap ibA)=44$ by Table~\ref{tab:subset_types}.
The other proof is similar.
\end{proof}

None of the six possible Kuratowski monoids is characterized by the presence or absence
of a subset with any specific $\phi$- or \sn.
However, one of the seven GE monoids is.

\begin{prop}\label{prop:GE_characterization}
A topological space is GE if and only if it contains a subset with \sn\ \mth{$44$.}
\end{prop}

\begin{proof}
The ``if'' was established in Section~\ref{sec:kfa}.
Conversely suppose $X$ is GE.
Some $A\su X$ satisfies $\fif g=\fif\neq0$.
Since $gA$ also satisfies $i=0$, Table~\ref{tab:subset_types} implies $\psi(gA)\in
\{37,44\}$.
Apply Lemma~\ref{lem:psi_37}.
\end{proof}

Since Lemma~\ref{lem:psi_31} and Proposition~\ref{prop:GE_characterization} imply that
every Kuratowski space contains a subset $A$ with $\psi A\in\{31,44\}$, no Kuratowski
space has \knum\ $6$.

\begin{table}
\footnotesize
\centering
\caption{Characterizations of the local collapses of $\KZ$ and $\KF$.}
\renewcommand{\tabcolsep}{1pt}
\renewcommand{\arraystretch}{1}
\begin{tabular}[t]{|c|@{\hspace{1pt}}c@{\hspace{2pt}}|c|c|@{\hspace{1pt}}
c@{\hspace{1pt}}|}
\multicolumn{1}{c}{$\phi A$}&
\multicolumn{1}{c}{$\ify A$ satisfies}&
\multicolumn{1}{c}{}&
\multicolumn{1}{c}{$\phi A$}&
\multicolumn{1}{c}{$\ify A$ satisfies}\\
\Xcline{1-2}{1pt}\Xcline{4-5}{1pt}
\raisebox{-.5pt}{$1$}\vrule width0pt height8pt&$bib\neqy b,ibi\neqy i,bib\neqy bi,bib\neqy ib,ibi\neqy
bi$&\vrule height0pt depth0pt width10pt&
\raisebox{-.5pt}{$16$}&$ibi\neqy i,ibi\eqy bi,bib\neqy bi,ib\eqy b$\\\cline{1-2}\cline{4-5}
\y{$2$}&$bib\eqy b,ibi\neqy i,bib\neqy bi,bib\neqy ib,ibi\neqy bi$&&
\y{$17$}&$bib\neqy b,bib\eqy ib,ibi\neqy ib,bi\eqy i$\\\cline{1-2}\cline{4-5}
\y{$3$}&$ibi\eqy i,bib\neqy b,ibi\neqy ib,ibi\neqy bi,bib\neqy ib$&&
\y{$18$}&$\ident\neqy i,ibi\neqy i,bib\neqy ib,bi\eqy b$\\\cline{1-2}\cline{4-5}
\y{$4$}&$bib\neqy b,ibi\neqy i,bib\eqy ib,ibi\neqy bi$&&
\y{$19$}&$\ident\neqy b,bib\neqy b,ibi\neqy bi,ib\eqy i$\\\cline{1-2}\cline{4-5}
\y{$5$}&$ibi\neqy i,bib\neqy b,ibi\eqy bi,bib\neqy ib$&&
\y{$20$}&$\ident\eqy i,ibi\neqy i,bib\neqy ib$\\\cline{1-2}\cline{4-5}
\y{$6$}&$bib\eqy b,ibi\eqy i,bib\neqy bi,bib\neqy ib,ibi\neqy bi$&&
\y{$21$}&$\ident\eqy b,bib\neqy b,ibi\neqy bi$\\\cline{1-2}\cline{4-5}
\y{$7$}&$bib\neqy b,ibi\eqy i,bib\eqy ib,ibi\neqy bi$&&
\y{$22$}&$\ident\neqy i,ibi\eqy b$\\\cline{1-2}\cline{4-5}
\y{$8$}&$ibi\neqy i,bib\eqy b,ibi\eqy bi,bib\neqy ib$&&
\y{$23$}&$\ident\neqy b,bib\eqy i$\\\cline{1-2}\cline{4-5}
\y{$9$}&$ibi\neqy i,ibi\neqy bi,ib\eqy b$&&
\y{$24$}&$\ident\neqy b,\ident\neqy i,\f\eqy\fbi$\\\cline{1-2}\cline{4-5}
\y{$10$}&$bib\neqy b,bib\neqy ib,bi\eqy i$&&
\y{$25$}&$bib\neqy bi,\f\eqy\ie$\\\cline{1-2}\cline{4-5}
\y{$11$}&$bib\neqy b,ibi\neqy i,bib\neqy bi,bib\eqy ib,ibi\eqy bi$&&
\y{$26$}&$\ident\eqy bi,bib\neqy ib$\\\cline{1-2}\cline{4-5}
\y{$12$}&$bib\neqy b,ibi\neqy i,bib\eqy bi,bib\neqy ib$&&
\y{$27$}&$\ident\eqy ib,ibi\neqy bi$\\\cline{1-2}\cline{4-5}
\y{$13$}&$bib\neqy b,ibi\neqy i,bi\eqy ib$&&
\y{$28$}&$\ident\eqy i,ibi\neqy i,bib\eqy ib$\\\cline{1-2}\cline{4-5}
\y{$14$}&$ibi\eqy i,ibi\neqy bi,ib\eqy b$&&
\y{$29$}&$\ident\eqy b,bib\neqy b,ibi\eqy bi$\\\cline{1-2}\cline{4-5}
\y{$15$}&$bib\eqy b,bib\neqy ib,bi\eqy i$&&
\y{$30$}&\y{$b\eqy i$}\mystrut\\[-\arrayrulewidth]\cline{1-2}\cline{4-5}
\multicolumn{5}{c}{}\\
\multicolumn{1}{c}{$\psi A$}&
\multicolumn{1}{c}{$\ify A$ satisfies}&
\multicolumn{1}{c}{}&
\multicolumn{1}{c}{$\psi A$}&
\multicolumn{1}{c}{$\ify A$ satisfies}\\
\Xcline{1-2}{1pt}\Xcline{4-5}{1pt}
\raisebox{-.5pt}{$1$}\vrule width0pt height8pt&$bib\neqy b,ibi\neqy i,\fib\neqy\fbi,\fib\neqy\fif,\fbi\neqy
\fif$&&
\raisebox{-.5pt}{$36$}&$bib\neqy b,bib\neqy ib,bi\eqy i,i\neqy0$\\\cline{1-2}\cline{4-5}
\y{$2$}&$bib\neqy b,ibi\neqy i,\fib\eqy\fbi,\fib\neqy\fif,\bif\neqy\ie$&&
\y{$37$}&$bib\neqy b,bib\neqy ib,i\eqy 0$\\\cline{1-2}\cline{4-5}
\y{$3$}&$ibi\neqy i,ibi\neqy bi,bib\neqy b,\fib\neqy\fbi,\fib\eqy\fif$&&
\y{$38$}&see $\phi A\eqy11$\\\cline{1-2}\cline{4-5}
\y{$4$}&$bib\neqy b,bib\neqy ib,ibi\neqy i,\fib\neqy\fbi,\fbi\eqy\fif$&&
\y{$39$}&see $\phi A\eqy12$\\\cline{1-2}\cline{4-5}
\y{$5$}&$bib\neqy b,bib\neqy ib,ibi\neqy i,\fib\eqy\fbi,\fbi\eqy\fif$&&
\y{$40$}&see $\phi A\eqy13$\\\cline{1-2}\cline{4-5}
\y{$6$}&$bib\neqy b,bib\neqy bi,bib\neqy ib,ibi\neqy i,\bif\eqy\ie$&&
\y{$41$}&$ibi\eqy i,ibi\neqy bi,ib\eqy b,b\neqy1$\\\cline{1-2}\cline{4-5}
\y{$7$}&$ibi\neqy i,bib\eqy b,\fib\neqy\fif,\ff\neqy\ei$&&
\y{$42$}&$bi\neqy i,ibi\eqy\af$\\\cline{1-2}\cline{4-5}
\y{$8$}&$ibi\neqy i,ibi\neqy bi,\fib\neqy\fbi,\fb\eqy\fif$&&
\y{$43$}&$bib\eqy b,bib\neqy ib,bi\eqy i,i\neqy0$\\\cline{1-2}\cline{4-5}
\y{$9$}&$ibi\neqy i,\ff\eqy\ei,\fib\neqy\fbi,\fbi\neqy\fif$&&
\y{$44$}&$ib\neqy b,bib\eqy\f$\\\cline{1-2}\cline{4-5}
\y{$10$}&$ibi\neqy i,bib\neqy ib,bib\eqy b,\fbi\eqy\fif,\fib\neqy\fif$&&
\y{$45$}&$ibi\eqy bi,ibi\neqy i,ibi\neqy ib,ib\eqy b,b\neqy1$\\\cline{1-2}\cline{4-5}
\y{$11$}&$ibi\neqy i,\fib\neqy\fif,\fb\eqy\fbi,\bif\neqy\ie$&&
\y{$46$}&$bi\neqy ib,ibi\neqy i,ibi\eqy\aif$\\\cline{1-2}\cline{4-5}
\y{$12$}&$ibi\neqy i,ibi\neqy bi,\fib\eqy\fif,\ff\eqy\ei$&&
\y{$47$}&$bib\eqy ib,bib\neqy b,bib\neqy bi,bi\eqy i,i\neqy0$\\\cline{1-2}\cline{4-5}
\y{$13$}&$ibi\neqy i,ibi\neqy ib,ibi\neqy bi,bib\eqy b,\bif\eqy\ie$&&
\y{$48$}&$bi\neqy ib,bib\neqy b,bib\eqy\ie$\\\cline{1-2}\cline{4-5}
\y{$14$}&$bib\neqy b,ibi\eqy i,\fbi\neqy\fif,\ff\neqy\fb$&&
\y{$49$}&see $\phi A\eqy18$\\\cline{1-2}\cline{4-5}
\y{$15$}&$bib\neqy b,bib\neqy ib,\fib\neqy\fbi,\ei\eqy\fif$&&
\y{$50$}&see $\phi A\eqy19$\\\cline{1-2}\cline{4-5}
\y{$16$}&$bib\neqy b,\ff\eqy\fb,\fib\neqy\fbi,\fib\neqy\fif$&&
\y{$51$}&see $\phi A\eqy20$\\\cline{1-2}\cline{4-5}
\y{$17$}&$bib\neqy b,ibi\neqy bi,ibi\eqy i,\fib\eqy\fif,\fbi\neqy\fif$&&
\y{$52$}&see $\phi A\eqy21$\\\cline{1-2}\cline{4-5}
\y{$18$}&$bib\neqy b,\fbi\neqy\fif,\ei\eqy\fib,\bif\neqy\ie$&&
\y{$53$}&$\ident\neqy i,ibi\eqy b,b\neqy1$\\\cline{1-2}\cline{4-5}
\y{$19$}&$bib\neqy b,bib\neqy ib,\fbi\eqy\fif,\ff\eqy\fb$&&
\y{$54$}&$\ident\neqy i,bi\eqy 1$\\\cline{1-2}\cline{4-5}
\y{$20$}&$bib\neqy b,bib\neqy bi,bib\neqy ib,ibi\eqy i,\bif\eqy\ie$&&
\y{$55$}&$\ident\neqy b,bib\eqy i,i\neqy0$\\\cline{1-2}\cline{4-5}
\y{$21$}&see $\phi A\eqy4$&&
\y{$56$}&$\ident\neqy b,ib\eqy 0$\\\cline{1-2}\cline{4-5}
\y{$22$}&see $\phi A\eqy5$&&
\y{$57$}&see $\phi A\eqy24$\\\cline{1-2}\cline{4-5}
\y{$23$}&$bib\eqy b,ibi\eqy i,\ff\neqy\fb,\ff\neqy\ei,\f\neqy\bif$&&
\y{$58$}&$bi\neqy ib,b\neqy1,i\neqy0,\f\eqy\ie$\\\cline{1-2}\cline{4-5}
\y{$24$}&$\fb\neqy\ei,\ff\eqy\fib,\f\neqy\bif$&&
\y{$59$}&$\ident\neqy i,bi\eqy\af,i\neqy0$\\\cline{1-2}\cline{4-5}
\y{$25$}&$\fb\neqy\ei,\ff\eqy\fbi,\f\neqy\bif$&&
\y{$60$}&$\ident\neqy b,ib\eqy\f,b\neqy1$\\\cline{1-2}\cline{4-5}
\y{$26$}&$\fb\eqy\ei,\bif\neqy\ie,\f\neqy\bif$&&
\y{$61$}&$\f\eqy 1$\\\cline{1-2}\cline{4-5}
\y{$27$}&$\ff\neqy\fb,\ff\neqy\ei,\f\eqy\bif$&&
\y{$62$}&see $\phi A\eqy26$\\\cline{1-2}\cline{4-5}
\y{$28$}&$ibi\eqy i,ibi\neqy bi,\fib\neqy\fbi,\fb\eqy\fif$&&
\y{$63$}&see $\phi A\eqy27$\\\cline{1-2}\cline{4-5}
\y{$29$}&$bib\eqy b,bib\neqy ib,\fib\neqy\fbi,\ei\eqy\fif$&&
\y{$64$}&$\ident\eqy i,\ident\neqy b,ibi\eqy bi,b\neqy1$\\\cline{1-2}\cline{4-5}
\y{$30$}&$\fib\eqy\fbi,\bif\neqy\ie,\f\eqy\bif$&&
\y{$65$}&$\ident\eqy\af,\ident\neqy1$\\\cline{1-2}\cline{4-5}
\y{$31$}&$bi\neqy b,ib\neqy b,\fb\eqy\ei,\bif\eqy\ie$&&
\y{$66$}&$\ident\eqy b,\ident\neqy i,bib\eqy ib,i\neqy0$\\\cline{1-2}\cline{4-5}
\y{$32$}&see $\phi A\eqy7$&&
\y{$67$}&$\ident\eqy\f,\ident\neqy0$\\\cline{1-2}\cline{4-5}
\y{$33$}&see $\phi A\eqy8$&&
\y{$68$}&$\ident\neqy0,\ident\neqy1,b\eqy i$\\\cline{1-2}\cline{4-5}
\y{$34$}&$ibi\neqy i,ibi\neqy bi,ib\eqy b,b\neqy1$&&
\y{$69$}&\y{$\ident\eqy 1$}\\\cline{1-2}\cline{4-5}
\y{$35$}&$ibi\neqy i,ibi\neqy bi,b\eqy 1$&&
\y{$70$}&\y{$\ident\eqy 0$}\mystrut\\[-\arrayrulewidth]\cline{1-2}\cline{4-5}
\end{tabular}\label{tab:characterizations}
\end{table}

\begin{thm}\label{thm:GE_implications}
The GE monoid of a space implies it has subsets satisfying the following classes of dual
\sn s\nit:
discrete with \mth{$|X|>1$,} \mth{$\{68\}$;}
indiscrete partition\nit, \mth{$\{61\}$;}
non-indiscrete partition\nit, \mth{$\{59,60\}$,} \mth{$\{68\}$;}
EO\nit, ED\nit, \mth{$\{65,67\}$;}
OU\nit, \mth{$\{62,63\}$,} \mth{$\{65,67\}$;}
KD\nit, \mth{$\{31\}$,} \mth{$\{46,48\}$,} \mth{$\{59,60\}$,} \mth{$\{62,63\}$,}
\mth{$\{64,66\}$,} \mth{$\{65,67\}$,} \mth{$\{68\}$;}
GE\nit, \mth{$\{42,44\}$,} \mth{$\{62,63\}$,} \mth{$\{65,67\}$.}
Each list is sharp.
\end{thm}

\begin{proof}
Nonempty proper subsets of discrete (indiscrete) spaces have \sn\ $68$ $(61)$.
Every non-indiscrete partition space contains distinct points $x,y,z$ such that $b(\{x\})
\neq b(\{y\})=b(\{z\})$.
We have $\psi(\{x,y\})=60$ and $\psi(b(\{x\}))=68$.
By Table~\ref{tab:subset_types} the boundary of every open set that is not closed has
\sn\ $67$.
Thus every non-discrete, non-partition space contains a subset with \sn\ $67$.
Kuratowski and OU spaces contain a subset with \sn\ $62$ by Lemma~\ref{lem:k-number_26}.
GE spaces contain a subset with \sn\ $44$ by Proposition~\ref{prop:GE_characterization}.
Lemma~\ref{lem:psi_31} implies KD spaces contain a subset with \sn\ $31$, the interior of
which has \sn\ $63$, and a subset with \sn\ $46$, the interior of which has \sn\ $64$.
Lemma~\ref{lem:psi_60} implies KD spaces contain a subset with \sn\ $60$, the closure of
which has \sn\ $68$.
We verified sharpness by computer.
\end{proof}

Table~2.1 in GJ points out that \pn s $4$, $7$, $11$, $13$, $16$ and their duals cannot
occur in connected spaces since they imply $iA\subsetneq ibiA=biA\subsetneq bA$ and/or
$iA\subsetneq ibA=bibA\subsetneq bA$.
Similarly, \sn s $6$, $13$, $31$, $58$, $59$ and their duals imply $\es\subsetneq\ifA=
\bifA\subsetneq bA$,\footnote{Since connected spaces do not admit \sn\ $6$,
Table~\ref{tab:subset_types} proves a conjecture of Moslehian and Tavallaii
\cite{1995_moslehian_tavallaii} that states every Kuratowski $14$-set generates $8$
distinct sets under $\{\f,i\}$ in connected spaces.} \sn s $34$, $41$, $53$, $64$ and
their duals imply either $\es\subsetneq iA=biA\subsetneq bA$ or $iA\subsetneq ibA=bA
\subsetneq X$, and \sn\ $68$ implies $\es\subsetneq iA=bA\subsetneq X$.
We verified by computer that each of the remaining \sn s occurs in at least one connected
space.

The next corollary holds by the above and Theorem~\ref{thm:GE_implications}.

\begin{cor}\label{cor:kd_disconnected}
KD\nit, non-indiscrete partition\nit, and discrete spaces of cardinality \mth{$>1$} are
disconnected\nit.
There exist extremally disconnected spaces\nit, i.e.\nit, spaces that satisfy
\mth{$ibi=bi$,} that are not disconnected\nit.
\end{cor}

\begin{prop}\label{prop:psi_39}
Every KD space that contains a set \mth{$A$} with \mth{$\psi A=39$} contains a set
\mth{$B$} with \mth{$\psi B\in\{6,13\}$.}
\end{prop}

\noindent
To prove this we need several preliminary results.

\begin{lem}\label{lem:kd_void_interior}
Suppose \mth{$iA=\es$} in a KD space \mth{$X$.}
Then \mth{$ibA\cap biB=\ie(A\cap B)$} for all \mth{$B\su X$.}
\end{lem}

\begin{proof}
$(\su)$~The hypotheses imply $ib(aA)=X$ and $\bif=\ie$.
Note that $bibA$ is open by Lemma~\ref{lem:kd_bib}.
Thus by Lemmas~\ref{lem:i-union}(ii), \ref{lem:half_dist}(ii),
and~\ref{lem:general_lemma}(iii), $ibA\cap biB=bi(bA\cap B)\su bib(A\cap B)=
b(ib(A\cap B)\cap ib(aA))=bi(b(A\cap B)\sm iA)\su\bif(A\cap B)=\ie(A\cap B)$.
$(\supseteq)$~$\ie(A\cap B)\su ib(A\cap B)\su ibA\cap ibB$.
\end{proof}

\begin{lem}\label{lem:a_union_gb}
Suppose \mth{$bibA=biA\neq bA$} in a KD space \mth{$X$.}
If \mth{$A\sym gB$} and \mth{$A\sym gaB$} each satisfy \mth{$bib=b$} then \mth{$\ie(A\cup
gB)=\ie(A\cup gaB)\neq\es$.}
\end{lem}

\begin{proof}
Since $\ifA=\es$ and every subset satisfies $\bif=\ie$ the equation holds by
Lemma~\ref{lem:b_equals_bi}(iii).
Suppose $\ie(A\cup gB)=\es$.
By Lemmas~\ref{lem:kd_bib} and~\ref{lem:b_equals_bi}(ii), $b(A\cap agB)\su b(A\sym gB)=
bib(A\sym gB)\su bib(A\cup gB)=bi(A\cup gB)=biA$.
Similarly $b(A\cap agaB)\su biA$.
Since $gB\cap gaB=\es$ we have $agB\cup agaB=X$.
Thus $bA=b((A\cap agB)\cup(A\cap agaB))\su biA$.
Conclude $\ie(A\cup gB)\neq\es$.
\end{proof}

\begin{lem}\label{lem:sym}
\mth{$bibA=ibA\implies bi(aA)\cap biB\cap bi(aB)\su bi(A\sym B)\sm ib(A\sym B)$.}
\end{lem}

\begin{proof}
Since $bi(aA)$ is open, every open neighborhood $U$ of $x\in bi(aA)\cap biB\cap bi(aB)$
satisfies $i(A\sym B)$ $=$ $i(A\cup B)\cap i(aA\cup aB)$ $\supseteq$ $U\cap i(aA)\cap iB$
$\neq$ $\es$ $\neq$ $U\cap i(aA)\cap i(aB)$ $\su$ $i(A\cup aB)\cap i(aA\cup B)$ $=$ $ia(A
\sym B)$.
\end{proof}

\begin{lem}\label{lem:psi_39}
If \mth{$iA=\es$} and $bibA=ibA$ in a KD space \mth{$X$} then \mth{$\fib B\su\fib(A\sym
B)\cap\fib(A\cup B)$} for all \mth{$B\su X$.}
\end{lem}

\begin{figure}[!ht]
\small
\centering
\begin{tabular}[t]{c@{\hspace{15pt}}c}
\raise122pt\hbox{\begin{tikzpicture}
[auto,
 block/.style={rectangle,draw=black,align=center,minimum width={20pt},
 minimum height={16pt},inner sep=2pt,scale=.64},
 topblock/.style={rectangle,draw=black,line width=1pt,align=center,minimum width={18pt},
 minimum height={16pt},inner sep=2pt,scale=.64},
 line/.style ={draw},scale=.8]

	\node[block,anchor=south] (p1) at (0,0) {$1$};
	\node[block,anchor=south] (p2) at (-1.25,1) {$2$};
	\node[block,anchor=south] (p3) at (1.25,1) {$3$};
	\node[block,anchor=south] (p4) at (-3,1) {$4$};
	\node[block,anchor=south] (p5) at (3,1) {$5$};
	\node[block,anchor=south] (p6) at (0,2) {$6$};
	\node[block,anchor=south] (p7) at (-3,2) {$7$};
	\node[block,anchor=south] (p8) at (3,2) {$8$};
	\node[block,anchor=south] (p9) at (-2,3) {$9$};
	\node[block,anchor=south] (p10) at (2,3) {$10$};
	\node[block,anchor=south] (p11) at (0,3) {$11$};
	\node[block,anchor=south] (p12) at (0,4) {$12$};
	\node[block,anchor=south] (p13) at (0,5) {$13$};
	\node[block,anchor=south] (p14) at (-3,4) {$14$};
	\node[block,anchor=south] (p15) at (3,4) {$15$};
	\node[block,anchor=south] (p16) at (-3,5) {$16$};
	\node[block,anchor=south] (p17) at (3,5) {$17$};
	\node[block,anchor=south] (p18) at (-1.25,4.5) {$18$};
	\node[block,anchor=south] (p19) at (1.25,4.5) {$19$};
	\node[block,anchor=south] (p20) at (-1.25,5.35) {$20$};
	\node[block,anchor=south] (p21) at (1.25,5.35) {$21$};
	\node[block,anchor=south] (p22) at (-2,6.5) {$22$};
	\node[block,anchor=south] (p23) at (2,6.5) {$23$};
	\node[block,anchor=south] (p24) at (0,5.75) {$24$};
	\node[block,anchor=south] (p25) at (0,7) {$25$};
	\node[block,anchor=south] (p26) at (-1.25,7.25) {$26$};
	\node[block,anchor=south] (p27) at (1.25,7.25) {$27$};
	\node[block,anchor=south] (p28) at (-2,8) {$28$};
	\node[block,anchor=south] (p29) at (2,8) {$29$};
	\node[topblock,anchor=south] (p30) at (0,8.5) {$30$};
	\node[anchor=south] (caption) at (0,-1) {\pn s (height $=$ $6$)};

	\draw[line] (p1.north west) -- (p2.south east);
	\draw[line] (p1.north east) -- (p3.south west);
	\draw[line] (p1.180) -- (p4.0);
	\draw[line] (p1.0) -- (p5.180);
	\path (p1.north east) edge[out=79,in=281] (p12.south east);
	\draw[line] (p2.north east) -- (p6.south west);
	\draw[line] (p2.0) -- (p8.180);
	\draw[line] (p2.90) -- (p9.south east);
	\draw[line] (p2) -- (p18);
	\draw[line] (p3.north west) -- (p6.south east);
	\draw[line] (p3.180) -- (p7.0);
	\draw[line] (p3.90) -- (p10.south west);
	\draw[line] (p3) -- (p19);
	\draw[line] (p4) -- (p7);
	\draw[line] (p4.north east) -- (p9.south west);
	\draw[line] (p4.north east) -- (p11.south west);
	\draw[line] (p5) -- (p8);
	\draw[line] (p5.north west) -- (p10.south east);
	\draw[line] (p5.north west) -- (p11.south east);
	\draw[line] (p6.north west) -- (p14.0);
	\draw[line] (p6.north east) -- (p15.180);
	\path (p6.north west) edge[out=101,in=259] (p24.south west);
	\draw[line] (p7) -- (p14);
	\draw[line] (p7.north east) -- (p17.south west);
	\draw[line] (p8) -- (p15);
	\draw[line] (p8.north west) -- (p16.south east);
	\draw[line] (p9.north west) -- (p14.south east);
	\draw[line] (p9.north west) -- (p16.south east);
	\draw[line] (p10.north east) -- (p15.south west);
	\draw[line] (p10.north east) -- (p17.south west);
	\path (p11.north west) edge[out=101,in=259] (p13.south west);
	\draw[line] (p11.north west) -- (p16.south east);
	\draw[line] (p11.north east) -- (p17.south west);
	\draw[line] (p12) -- (p13);
	\draw[line] (p12.180) -- (p18.0);
	\draw[line] (p12.0) -- (p19.180);
	\draw[line] (p13.north west) -- (p22.south east);
	\draw[line] (p13.north east) -- (p23.south west);
	\draw[line] (p14.110) -- (p25.south west);
	\draw[line] (p15.70) -- (p25.south east);
	\draw[line] (p16.north east) -- (p22.south west);
	\draw[line] (p16.north east) -- (p25.south west);
	\draw[line] (p17.north west) -- (p23.south east);
	\draw[line] (p17.north west) -- (p25.south east);
	\draw[line] (p18) -- (p20);
	\draw[line] (p18.north west) -- (p22.270);
	\draw[line] (p18.north east) -- (p24.south west);
	\draw[line] (p19) -- (p21);
	\draw[line] (p19.north east) -- (p23.270);
	\draw[line] (p19.north west) -- (p24.south east);
	\draw[line] (p20.north east) -- (p27.south west);
	\draw[line] (p20.90) -- (p28.270);
	\draw[line] (p21.north west) -- (p26.south east);
	\draw[line] (p21.90) -- (p29.270);
	\draw[line] (p22) -- (p28);
	\draw[line] (p23) -- (p29);
	\draw[line] (p24.north west) -- (p26.south east);
	\draw[line] (p24.north east) -- (p27.south west);
	\draw[line] (p25) -- (p30);
	\draw[line] (p26.north east) -- (p30.south west);
	\draw[line] (p27.north west) -- (p30.south east);
	\draw[line] (p28.0) -- (p30.180);
	\draw[line] (p29.180) -- (p30.0);

\end{tikzpicture}}&
\begin{tikzpicture}
[auto,
 block/.style={rectangle,draw=black,align=center,minimum width={20pt},
 minimum height={16pt},inner sep=2pt,scale=.64},
 topblock/.style={rectangle,draw=black,line width=1pt,align=center,minimum width={18pt},
 minimum height={16pt},inner sep=2pt,scale=.64},
 line/.style ={draw},scale=.8]

	\node[block,anchor=south] (1) at (0,.5) {$1$};
	\node[block,anchor=south] (2) at (0,1.5) {$2$};
	\node[block,anchor=south] (3) at (-2.5,1.65) {$3$};
	\node[block,anchor=south] (4) at (2.5,1.65) {$4$};
	\node[block,anchor=south] (5) at (0,2.5) {$5$};
	\node[block,anchor=south] (6) at (0,4.5) {$6$};
	\node[block,anchor=south] (7) at (-5,1.5) {$7$};
	\node[block,anchor=south] (8) at (-5,3) {$8$};
	\node[block,anchor=south] (9) at (-6,3) {$9$};
	\node[block,anchor=south] (10) at (-6,6) {$10$};
	\node[block,anchor=south] (11) at (-4,6) {$11$};
	\node[block,anchor=south] (12) at (-4,7.5) {$12$};
	\node[block,anchor=south] (13) at (-2,9.85) {$13$};
	\node[block,anchor=south] (14) at (5,1.5) {$14$};
	\node[block,anchor=south] (15) at (5,3) {$15$};
	\node[block,anchor=south] (16) at (6,3) {$16$};
	\node[block,anchor=south] (17) at (6,6) {$17$};
	\node[block,anchor=south] (18) at (4,6) {$18$};
	\node[block,anchor=south] (19) at (4,7.5) {$19$};
	\node[block,anchor=south] (20) at (2,9.85) {$20$};
	\node[block,anchor=south] (21) at (-1.75,2.65) {$21$};
	\node[block,anchor=south] (22) at (1.75,2.65) {$22$};
	\node[block,anchor=south] (23) at (0,6.25) {$23$};
	\node[block,anchor=south] (24) at (-2.5,6.88) {$24$};
	\node[block,anchor=south] (25) at (2.5,6.88) {$25$};
	\node[block,anchor=south] (26) at (0,7.5) {$26$};
	\node[block,anchor=south] (27) at (0,9) {$27$};
	\node[block,anchor=south] (28) at (-4,9.65) {$28$};
	\node[block,anchor=south] (29) at (4,9.65) {$29$};
	\node[block,anchor=south] (30) at (0,11.25) {$30$};
	\node[block,anchor=south] (31) at (0,12.125) {$31$};
	\node[block,anchor=south] (32) at (-1.75,8.25) {$32$};
	\node[block,anchor=south] (33) at (1.75,8.25) {$33$};
	\node[block,anchor=south] (34) at (-6,7.5) {$34$};
	\node[block,anchor=south] (35) at (-6,9) {$35$};
	\node[block,anchor=south] (36) at (6,7.5) {$36$};
	\node[block,anchor=south] (37) at (6,9) {$37$};
	\node[block,anchor=south] (38) at (0,13) {$38$};
	\node[block,anchor=south] (39) at (0,14) {$39$};
	\node[block,anchor=south] (40) at (0,15) {$40$};
	\node[block,anchor=south] (41) at (-4,11.5) {$41$};
	\node[block,anchor=south] (42) at (-6,13) {$42$};
	\node[block,anchor=south] (43) at (4,11.5) {$43$};
	\node[block,anchor=south] (44) at (6,13) {$44$};
	\node[block,anchor=south] (45) at (-2,13.5) {$45$};
	\node[block,anchor=south] (46) at (-4,14) {$46$};
	\node[block,anchor=south] (47) at (2,13.5) {$47$};
	\node[block,anchor=south] (48) at (4,14) {$48$};
	\node[block,anchor=south] (49) at (-6,15.25) {$49$};
	\node[block,anchor=south] (50) at (6,15.25) {$50$};
	\node[block,anchor=south] (51) at (-2,16.25) {$51$};
	\node[block,anchor=south] (52) at (2,16.25) {$52$};
	\node[block,anchor=south] (53) at (-6,16.25) {$53$};
	\node[block,anchor=south] (54) at (-4,16.75) {$54$};
	\node[block,anchor=south] (55) at (6,16.25) {$55$};
	\node[block,anchor=south] (56) at (4,16.75) {$56$};
	\node[block,anchor=south] (57) at (0,16) {$57$};
	\node[block,anchor=south] (58) at (0,17) {$58$};
	\node[block,anchor=south] (59) at (-2,17.5) {$59$};
	\node[block,anchor=south] (60) at (2,17.5) {$60$};
	\node[topblock,anchor=south] (61) at (0,18) {$61$};
	\node[block,anchor=south] (62) at (-1.5,18.5) {$62$};
	\node[block,anchor=south] (63) at (1.5,18.5) {$63$};
	\node[block,anchor=south] (64) at (-4,18.5) {$64$};
	\node[block,anchor=south] (65) at (-6,19.25) {$65$};
	\node[block,anchor=south] (66) at (4,18.5) {$66$};
	\node[block,anchor=south] (67) at (6,19.25) {$67$};
	\node[block,anchor=south] (68) at (0,19.5) {$68$};
	\node[topblock,anchor=south] (69) at (-2,20) {$69$};
	\node[topblock,anchor=south] (70) at (2,20) {$70$};
	\node[anchor=south] (caption) at (0,-.5) {\sn s (height $=$ $10$)};

	\draw[line] (1) -- (2);
	\draw[line] (1.180) -- (3.0);
	\draw[line] (1.0) -- (4.180);
	\draw[line] (1.180) -- (7.0);
	\draw[line] (1.0) -- (14.180);
	\draw[line] (2) -- (5);
	\path (2.north east) edge[out=79,in=281] (6.south east);
	\draw[line] (2.north west) -- (11.south east);
	\draw[line] (2.north east) -- (18.south west);
	\draw[line] (3.0) -- (5.180);
	\draw[line] (3.north west) -- (8.south east);
	\draw[line] (3.north east) -- (17.south west);
	\draw[line] (3.0) -- (22.south west);
	\draw[line] (4.180) -- (5.0);
	\draw[line] (4.north west) -- (10.south east);
	\draw[line] (4.north east) -- (15.south west);
	\draw[line] (4.180) -- (21.south east);
	\draw[line] (5.north west) -- (12.south east);
	\draw[line] (5.north east) -- (19.south west);
	\path (5.north west) edge[out=108,in=252] (38.south west);
	\draw[line] (6.north west) -- (13.south east);
	\draw[line] (6.north east) -- (20.south west);
	\path (6.north west) edge[out=108,in=252] (38.south west);
	\path (6.north east) edge[out=79,in=281] (39.south east);
	\draw[line] (7) -- (8);
	\draw[line] (7.north west) -- (9.270);
	\draw[line] (7.north east) -- (23.south west);
	\draw[line] (8.90) -- (12.south west);
	\draw[line] (8.90) -- (28.south west);
	\draw[line] (8.north east) -- (33.south west);
	\draw[line] (9) -- (10);
	\draw[line] (9.north east) -- (11.south west);
	\draw[line] (9.north east) -- (25.south west);
	\draw[line] (10.north east) -- (12.south west);
	\draw[line] (10.north east) -- (29.south west);
	\draw[line] (10) -- (34);
	\draw[line] (11) -- (12);
	\draw[line] (11.north east) -- (13.south west);
	\draw[line] (11.0) -- (26.south west);
	\draw[line] (12.north east) -- (30.270);
	\draw[line] (12.north east) -- (45.south west);
	\draw[line] (13.north east) -- (31.south west);
	\draw[line] (13) -- (45);
	\draw[line] (13.north west) -- (49.south east);
	\draw[line] (14) -- (15);
	\draw[line] (14.north east) -- (16.270);
	\draw[line] (14.north west) -- (23.south east);
	\draw[line] (15.90) -- (19.south east);
	\draw[line] (15.90) -- (29.south east);
	\draw[line] (15.north west) -- (32.south east);
	\draw[line] (16) -- (17);
	\draw[line] (16.north west) -- (18.south east);
	\draw[line] (16.north west) -- (24.south east);
	\draw[line] (17.north west) -- (19.south east);
	\draw[line] (17.north west) -- (28.south east);
	\draw[line] (17) -- (36);
	\draw[line] (18) -- (19);
	\draw[line] (18.north west) -- (20.south east);
	\draw[line] (18.180) -- (26.south east);
	\draw[line] (19.north west) -- (30.270);
	\draw[line] (19.north west) -- (47.south east);
	\draw[line] (20.north west) -- (31.south east);
	\draw[line] (20) -- (47);
	\draw[line] (20.north east) -- (50.south west);
	\draw[line] (21) -- (32);
	\draw[line] (21.north west) -- (34.south east);
	\draw[line] (21.90) -- (38.south west);
	\draw[line] (22) -- (33);
	\draw[line] (22.north east) -- (36.south west);
	\draw[line] (22.90) -- (38.south east);
	\draw[line] (23.180) -- (24.0);
	\draw[line] (23.0) -- (25.180);
	\path (23.north east) edge[out=79,in=281] (27.south east);
	\draw[line] (24.0) -- (26.180);
	\draw[line] (24.north west) -- (28.south east);
	\draw[line] (25.180) -- (26.0);
	\draw[line] (25.north east) -- (29.south west);
	\path (26.north west) edge[out=101,in=259] (30.south west);
	\path (26.north west) edge[out=101,in=259] (31.south west);
	\draw[line] (27.180) -- (28.0);
	\draw[line] (27.0) -- (29.180);
	\draw[line] (28.north east) -- (30.180);
	\draw[line] (28.north east) -- (43.180);
	\draw[line] (29.north west) -- (30.0);
	\draw[line] (29.north west) -- (41.0);
	\path (30.north west) edge[out=108,in=252] (58.south west);
	\path (31.north west) edge[out=108,in=252] (57.south west);
	\path (31.north east) edge[out=79,in=281] (58.south east);
	\draw[line] (32.north west) -- (41.south east);
	\draw[line] (32.north east) -- (47.270);
	\draw[line] (33.north east) -- (43.south west);
	\draw[line] (33.north west) -- (45.270);
	\draw[line] (34) -- (35);
	\draw[line] (34.north east) -- (41.south west);
	\draw[line] (34.north east) -- (45.south west);
	\draw[line] (35) -- (42);
	\draw[line] (35.north east) -- (46.270);
	\draw[line] (36) -- (37);
	\draw[line] (36.north west) -- (43.south east);
	\draw[line] (36.north west) -- (47.south east);
	\draw[line] (37) -- (44);
	\draw[line] (37.north west) -- (48.270);
	\path (38.north west) edge[out=101,in=259] (40.south west);
	\draw[line] (38.180) -- (45.0);
	\draw[line] (38.0) -- (47.180);
	\draw[line] (39) -- (40);
	\draw[line] (39.180) -- (49.0);
	\draw[line] (39.0) -- (50.180);
	\draw[line] (40.180) -- (53.0);
	\draw[line] (40.0) -- (55.180);
	\draw[line] (41.north west) -- (42.south east);
	\draw[line] (41.north east) -- (58.south west);
	\draw[line] (42.north east) -- (59.south west);
	\draw[line] (43.north east) -- (44.south west);
	\draw[line] (43.north west) -- (58.south east);
	\draw[line] (44.north west) -- (60.south east);
	\draw[line] (45.180) -- (46.0);
	\draw[line] (45) -- (53);
	\draw[line] (45.north east) -- (58.south west);
	\draw[line] (46) -- (54);
	\draw[line] (46.north east) -- (59.south west);
	\draw[line] (47.0) -- (48.180);
	\draw[line] (47) -- (55);
	\draw[line] (47.north west) -- (58.south east);
	\draw[line] (48) -- (56);
	\draw[line] (48.north west) -- (60.south east);
	\draw[line] (49.0) -- (51.180);
	\draw[line] (49) -- (53);
	\draw[line] (49.0) -- (57.180);
	\draw[line] (50.180) -- (52.0);
	\draw[line] (50) -- (55);
	\draw[line] (50.180) -- (57.0);
	\draw[line] (51.north east) -- (63.south west);
	\draw[line] (51.north west) -- (64.south east);
	\draw[line] (52.north west) -- (62.south east);
	\draw[line] (52.north east) -- (66.south west);
	\draw[line] (53.0) -- (54.180);
	\draw[line] (53.north east) -- (64.south west);
	\draw[line] (54.north west) -- (65.south east);
	\draw[line] (55.180) -- (56.0);
	\draw[line] (55.north west) -- (66.south east);
	\draw[line] (56.north east) -- (67.south west);
	\draw[line] (57.north west) -- (62.south east);
	\draw[line] (57.north east) -- (63.south west);
	\draw[line] (58.180) -- (59.0);
	\draw[line] (58.0) -- (60.180);
	\path (58.north east) edge[out=79,in=281] (68.south east);
	\draw[line] (59.0) -- (61.180);
	\draw[line] (59) -- (69);
	\draw[line] (60.180) -- (61.0);
	\draw[line] (60) -- (70);
	\draw[line] (62.north east) -- (68.180);
	\draw[line] (63.north west) -- (68.0);
	\draw[line] (64.180) -- (65.0);
	\draw[line] (64.0) -- (68.180);
	\draw[line] (65.0) -- (69.180);
	\draw[line] (66.0) -- (67.180);
	\draw[line] (66.180) -- (68.0);
	\draw[line] (67.180) -- (70.0);
	\draw[line] (68.180) -- (69.0);
	\draw[line] (68.0) -- (70.180);

\end{tikzpicture}\end{tabular}
\caption{The meet semilattices of $\phi$- and $\psi$-numbers under set inclusion.\\\small
\smallskip\hspace{42.5pt}Note that $\phi_1\su\phi_2\implies\phi_1\leq\phi_2$; the same is
true of \sn s.}
\label{fig:meet_semilattices}
\end{figure}

\begin{proof}
Let $x\in\fib B$.
Have $ibA\cap biB=\ie(A\cap B)$ by Lemma~\ref{lem:kd_void_interior}.
Hence $ibA\cap biB\su ibiB$.
Since $X$ is KD, this implies $ibA\cap\fib B=ibA\cap\fbi B=\es$.
Thus $x\in bi(aA)$.
Since $\fib B=\fib B\cap\fbi B=biB\cap bi(aB)$ by Lemma~\ref{lem:op_relations}(iv),
$x\in\fib(A\sym B)$ by Lemma~\ref{lem:sym}.
Have $x\in biB\su bi(A\cup B)$.
Since $bi(aA)$ is open by Lemma~\ref{lem:kd_bib}, Lemma~\ref{lem:i-union}(ii) implies $x
\in bi(aA)\cap bi(aB)=bi(aA\cap aB)=aib(A\cup B)$.
\end{proof}

\begin{cor}\label{cor:psi_39_sym}
In KD spaces\nit, \mth{$(\psi A=39$} and \mth{$\psi B\in\{48,60\})\implies\psi(A\sym
B)\in\{6,13,20,31\}$.}
\end{cor}

\begin{proof}
Lemmas~\ref{lem:b_equals_bi}(iv) and~\ref{lem:psi_39} imply $\ie(A\sym B)\neq\es$ and
$\fib(A\sym B)\neq\es$.
The result follows by Table~\ref{tab:subset_types}.
\end{proof}

\bgroup
\renewcommand{\tabcolsep}{8pt}
\begin{table}
\caption{All $k$- and $k_f$-numbers that occur for each GE monoid.}
\centering
\begin{tabular}{|c|l|l|}
\multicolumn{1}{c}{\raisebox{1pt}{space type}}&\multicolumn{1}{c}{\raisebox{1pt}
{$k(\XT)$}}&
\multicolumn{1}{c}{\raisebox{1pt}{$k_f(\XT)$}}\\\Xhline{2\arrayrulewidth}
\vrule width0pt height11pt depth0ptGE&$8,10,12,14$&$10,14,16,18,\dots,34$\\\hline
\vrule width0pt height11pt depth0ptKD&$10,12,14$&$18,22,28$\\\hline
\vrule width0pt height11pt depth0ptED&$4,6,8,10$&$4,6,8,10,16,22$\\\hline
\vrule width0pt height11pt depth0ptOU&$4,6,8,10$&$8,10,14,16,20$\\\hline
\vrule width0pt height11pt depth0ptEO&$4,6,8$&$4,6,8,10,16$\\\hline
\vrule width0pt height11pt depth0ptpartition&$4,6$&$4,6,10$\\\hline
\vrule width0pt height11pt depth0ptdiscrete&$2$&$2,4$\\\hline
\end{tabular}\label{tab:kf_kfa}
\end{table}
\egroup

\begin{table}
\caption{Named spaces satisfying various space type and $k$-number combinations.\\\small
\smallskip\hspace{-95pt}For definitions see~Steen and Seebach \cite{1970_steen_seebach}.}
\small\centering
\renewcommand{\tabcolsep}{2pt}
\begin{tabular}{|c|c|c|c|c|c|c|}
\multicolumn{1}{c}{\raisebox{1pt}{type}}&\multicolumn{1}{c}{\raisebox{1pt}{$k$}}&
\multicolumn{1}{c}{\raisebox{1pt}{$\XT$}}&\multicolumn{1}{c}{}&
\multicolumn{1}{c}{\raisebox{1pt}{type}}&\multicolumn{1}{c}{\raisebox{1pt}{$k$}}&
\multicolumn{1}{c}{\raisebox{1pt}{$\XT$}}\\\Xcline{1-3}{1pt}\Xcline{5-7}{1pt}
\multirow{2}{*}{\vrule width0pt height11pt depth0ptED}&\vrule width0pt height11pt
depth0pt$6$&
$\R$, right order topology&&
\multirow{2}{*}{\vrule width0pt height11pt depth0ptEO}&\vrule width0pt height11pt
depth0pt$6$&
$\R$, compact complement topology\\\cline{2-3}\cline{6-7}
&\vrule width0pt height11pt depth0pt$4$&$\N$, cofinite topology&&
&\vrule width0pt height11pt depth0pt$4$&Sierpiński space\\\cline{1-3}\cline{5-7}
&\vrule width0pt height11pt depth0pt$\y{10}$&$\N\sm\{1\}$, divisor topology&&
\multirow{2}{*}{P}&\vrule width0pt height11pt depth0pt$6$&$\N$, odd-even topology
\\\cline{2-3}\cline{6-7}
OU&\vrule width0pt height11pt depth0pt$6$&$\N$, excluded set topology&&
&\vrule width0pt height11pt depth0pt$4$&$\N$, indiscrete topology
\\\cline{2-3}\cline{5-7}
&\vrule width0pt height11pt depth0pt$4$&$\N$, excluded point topology&&
D&\vrule width0pt height11pt depth0pt$2$&$\N$, discrete topology
\mystrut\\[-\arrayrulewidth]\cline{1-3}\cline{5-7}
\end{tabular}\label{tab:named}
\end{table}

We are now ready to prove Proposition~\ref{prop:psi_39}.
Since Table~\ref{tab:subset_types} implies $\kf(A)=20\implies\psi A=39$ in KD spaces,
it follows that no KD space has $\kf$-number $20$.

\begin{proof}
Suppose $X$ is KD and $\psi A=39$ for some $A\su X$.
By Theorem~\ref{thm:GE_implications}, $\psi E=\psi(aE)=31$ for some $E\su X$.
Have $\psi(gE),\psi(gaE)\in\{48,60\}$.
Hence $\psi(A\sym gE),\psi(A\sym gaE)\in\{6,13,20,31\}$ by
Corollary~\ref{cor:psi_39_sym}.
We can assume that $\psi(A\sym gE)=\psi(A\sym gaE)=31$.
Lemma~\ref{lem:a_union_gb} implies $\ie(A\cup gE)=\ie(A\cup gaE)\neq\es$.
Since $ibiA\neq iA$, Lemma~\ref{lem:b_equals_bib}(ii) implies $A\cup gE$ and $A\cup gaE$
cannot both satisfy $ibi=i$.
Suppose $A\cup gE$ satisfies $ibi\neq i$.
Since $\psi(gE)\in\{48,60\}$ we have $\fib(A\cup gE)\neq\es$ by Lemma~\ref{lem:psi_39}.
Add $ibi\neq i$ to the Corollary~\ref{cor:psi_39_sym} argument to get $\psi(A\cup
gE)\in\{6,13\}$.
Similarly, $\psi(A\cup gaE)\in\{6,13\}$ if $A\cup gaE$ satisfies $ibi\neq i$.
\end{proof}

\begin{thm}\label{thm:KF_kf-numbers}
Table~\nit{\ref{tab:kf_kfa}} lists the values of \mth{$k(\XT)$} and \mth{$\kf(\XT)$} that
occur for each GE monoid\nit.
\end{thm}

\begin{proof}
We verified by computer that each combination occurs.
The rest are excluded by Table~\ref{tab:subset_types},  Corollary~\ref{cor:kf_12},
Theorem~\ref{thm:GE_implications}, and Proposition~\ref{prop:psi_39}.
\end{proof}

The largest \knum\ (\kfnum) in each row of Table~\ref{tab:kf_kfa} is the space type's
\Knum\ (\Kfnum) since full (completely full) spaces of each type exist.

Table~\ref{tab:named} lists named spaces satisfying some of the possible space type
and \knum\ combinations (proofs are left to the reader).
Except for the Sierpiński space, each space with \knum\ $4$ is cited in GJ's proof of
Theorem~2.1.

\subsection{Topological sums.}

Lemma~\ref{lem:topsum} implies the following corollary.

\begin{cor}\label{cor:topsum_phi_and_psi}
The collapse of \mth{$\KZ$ $(\KF)$} satisfied by \mth{$A_1\ds A_2$ in $X_2$} is the
intersection of the collapses of \mth{$\KZ$} \mth{$(\KF)$} satisfied by \mth{$A_1$} and
\mth{$A_2$} in \mth{$\XT$.}
The poset under set inclusion of all \pn s \nit(\sn s\nit) is thus a meet semilattice
\nit{(}see Figure~\nit{\ref{fig:meet_semilattices}).}
Tables~\mth{\ref{tab:disjoint_union_phi}} and~\mth{\ref{tab:disjoint_union_psi}} list all
\mth{$\phi$-} and \sn s of disjoint unions \mth{$A_1\ds A_2$} in \mth{$X_2$} given
those of \mth{$A_1$} and \mth{$A_2$} in \mth{$\XT$.}
Note that \mth{$\alpha(A_1\ds A_2)\leq\min\{\alpha(A_1),\alpha(A_2)\}$} for
\mth{$\phi,\psi$} in place of \mth{$\alpha$.} 
\end{cor}

As we noted earlier, GJ showed that adding one copy to a non-full sum of copies of a
given space strictly increases the sum's \knum.
We show in Corollary~\ref{cor:increased_k-number} below that it always increases by $2$
or $4$.
The proof calls for several preliminary results.

\begin{lem}\label{lem:bi_ib_subsym}
If \mth{$\phi A\in\{16,17\}$} and \mth{$A\cup B$} or \mth{$A\cap B$} satisfies
\mth{$bib=bi$} then \mth{$biB\sm ibB\su bi(A\sym B)\sm ib(A\sym B)$.}
\end{lem}

\begin{proof}
Suppose $\phi A\in\{16,17\}$ and $A\cup B$ satisfies $bib=bi$.
Let $x\in biB\cap bi(aB)$.
Have $\ie A\su ibB$ by Lemma~\ref{lem:b_equals_bi}(i).
Thus $bi(aB)\su biA\cup bi(aA)$.
Since $aA\sym B=a(A\sym B)$ and $bi\sm ib=bia\sm iba$ the result follows by
Lemma~\ref{lem:sym}.
The remaining case holds since $aA\cup aB$ satisfies $bib=bi$ when $A\cap B$ does and
$aA\sym aB=A\sym B$.
\end{proof}

\begin{lem}\label{lem:dual_implications}
\noindent\hangindent=74pt
\nit{(i)}~If \mth{$bibB=biB$} and \mth{$A,A\cap B$} each satisfy \mth{$\fib=0$} then
\mth{$\fib B\su\fbi(A\cup B)\cap\fib(A\cup B)$.}

\noindent\hangindent=74pt\phantom{\textbf{Lemma 35.\!}}
\mth{(ii)}~If \mth{$bibB=biB$} and \mth{$A,A\cup B$} each satisfy \mth{$\fbi=0$} then
\mth{$\fbi B\su\fbi(A\cap B)\cap\fib(A\cap B)$.}
\end{lem}

\begin{proof}
(i)~$bibB=biB\implies\fib B=\fbi B\su biB\su bi(A\cup B)$.
Let $x\in\fib B$.
Since $bib(A\cap B)=ib(A\cap B)\su ibB$, Lemma~\ref{lem:half_dist}(ii) implies
$U\cap ibA\cap iB=U\cap i(bA\cap B)\su U\cap ib(A\cap B)=\es$ for some open neighborhood
$U$ of $x$.
Since $x\in bibB=biB$, it follows that $x\cn\in ibA\cup ibB=ib(A\cup B)$ by
Lemma~\ref{lem:i-union}(iii).
Apply Lemma~\ref{lem:op_relations}(iv) to get the result.
(ii) is the dual of (i).
\end{proof}

\begin{lem}\label{lem:final_implication}
Suppose \mth{$A\cup B$} and \mth{$A\cap B$} each satisfy \mth{$ibi=i$.}
If \mth{$biB=bB$} then \mth{$ibiA\sm iA\su\fbi(A\cap B)$.}
\end{lem}

\begin{proof}
Let $x\in ibiA\sm iA$.
Since $ibi(A\cap B)\su i(A\cap B)\su iA$, $x\cn\in ibi(A\cap B)$.
Since $x\in b(aA)$ and $ibiA\su biA\cap ibi(A\cup B)\su i(A\cup B)\su A\cup B\su A\cup
biB$, $U\cap ibiA\cap i(A\cap B)\neq\es$ for every open $U$ containing $x$.
\end{proof}

\begin{prop}\label{prop:phi-16}
\mth{$k(\XT)\geq10$} if \mth{$\phi A=16$} and \mth{$\phi B=26$} for some
\mth{$A,B\su X$.}
\end{prop}

\begin{proof}
Let $E=A\sym B$.
Have $ibE\nsu biE$ by Lemma~\ref{lem:b_equals_bi}(iv).
Suppose $A\cup B$ or $A\cap B$ satisfies $bib=bi$.
Then $biE\nsu ibE$ by Lemma~\ref{lem:bi_ib_subsym}.
Hence $ibE,biE,ibiE,bibE$ are pairwise distinct.
By Table~\ref{tab:subset_types} this implies $E$ is neither open nor closed.
Thus $k(E)\geq10$. 
Suppose $A\cup B$ and $A\cap B$ both satisfy $bib\neq bi$.
Since they both satisfy $ibi\neq ib$, if either additionally satisfies both $\fib\neq0$
and $\fbi\neq0$ then it has \knum\ $\geq10$ by the argument above.
Hence we can assume each satisfies $\fib=0$ or $\fbi=0$.
Since $\fib B=\fbi B\neq\es$, Lemma~\ref{lem:dual_implications} implies $A\cup B$
satisfies $\fbi\neq0$ and $A\cap B$ satisfies $\fib\neq0$.
Hence $A\cap B$ satisfies $\fbi=0$.
It follows by Lemma~\ref{lem:final_implication} that $A\cup B$ and $A\cap B$ do not both
satisfy $ibi=i$.
Thus $|\{\ident,i,ibi,ib,bib\}(A\cap B)|=5$ or $|\{\ident,i,ibi,bi,bib\}(A\cup B)|=5$.
\end{proof}

\begin{cor}\label{cor:increased_k-number}
If \mth{$\XT$} is full then so is \mth{$X_n$} for all \mth{$n$.}
If \mth{$X_n$} is not full then \mth{$2\leq k(X_{n+1})-k(X_n)\leq4$.}
\end{cor}

\begin{proof}
The first assertion holds by Lemma~\ref{lem:topsum}.
Suppose $X_n$ is not full.
Some $A\su X_n$ then satisfies $k(A)=k(X_n)<K(X_n)=K(\XT)$.
For some $B\su X$ Kuratowski operators $o_1,o_2$ exist such that $o_1A=o_2A$ and
$o_1B\neq o_2B$.
By Lemma~\ref{lem:topsum}, $k(A\ds B)\geq k(A)+2$ in $X_{n+1}$.
Hence $k(X_{n+1})\geq k(X_n)+2$.

It remains to show that $k(X_{n+1})\leq k(X_n)+4$.
Since $k(X_{n+1})\leq K(X_{n+1})=K(X_n)$ we are done if $K(X_n)\leq k(X_n)+4$ so
we assume $K(X_n)\geq k(X_n)+6$.
Theorem~\ref{thm:KF_kf-numbers} implies $K(X_n)=k(X_n)+6$.
Since $X_n$ is not full, if $n\geq2$ we have $k(X_n)\geq k(X_{n-1})+2$, hence
$K(X_{n-1})=K(X_n)=k(X_n)+6\geq k(X_{n-1})+8$.
Theorem~\ref{thm:KF_kf-numbers} disallows this, so $n=1$.
By Theorem~\ref{thm:KF_kf-numbers} only two cases are possible.

\noindent
\case{Case}~$1$.~$(K(\XT)=10)$\;
Since ED spaces do not admit \pn s $26,27$ and OU spaces do not admit \pn\ $25$ it
follows from Table~\ref{tab:subset_types} and columns $25$-$30$ in
Table~\ref{tab:disjoint_union_phi} that $k(\XT)\geq6$.

\noindent
\case{Case}~$2$.~$(K(\XT)=14)$\;
By Table~\ref{tab:subset_types} and columns $13$-$30$ in
Table~\ref{tab:disjoint_union_phi}, $k(X_2)=14$ only if $X$ contains a subset with \pn\
$16$.
It follows by Lemma~\ref{lem:k-number_26} and Proposition~\ref{prop:phi-16} that $k(\XT)
\geq10$.

\noindent
Since both cases contradict $K(\XT)=k(\XT)+6$ the result follows.
\end{proof}

\begin{cor}\label{cor:increased_kf-number}
If \mth{$\XT$} is completely full then so is \mth{$X_n$} for all \mth{$n$.}
If \mth{$X_n$} is not completely full then \mth{$2\leq k_f(X_{n+1})-k_f(X_n)\leq20$.}
\end{cor}

\begin{proof}
Adjust the proof above for $\KF$ to get the first sentence and lower bound.
Since $\kf(\XT)=10\implies k(\XT)\leq8$ (see Table~\ref{tab:subset_types}),
Corollary~\ref{cor:increased_k-number} and Theorem~\ref{thm:KF_kf-numbers} imply the
upper bound.
\end{proof}

\bgroup
\begin{table}[!ht]
\caption{All $\psi$-numbers $\leq68$ in $X_n$ but not $X_{n-1}$ where $X_1$ is minimal
and $X_0:=\{0\}$.\\\smallskip\small\hspace{-137.5pt}Each list is followed by
$(k(X_n),k_{\hspace{-.5pt}f}(X_n))$.}
\renewcommand{\tabcolsep}{1pt}
\footnotesize
\centering\vspace{-12pt}
\begin{tabular}{|C{10pt}|L{82pt}|L{66pt}|L{50pt}|L{52pt}|L{26pt}|L{32pt}|L{26pt}|L{34pt}|L{34pt}|}
\multicolumn{1}{c}{}&\multicolumn{1}{c}{}&\multicolumn{1}{c}{}&\multicolumn{1}{c}{}&\multicolumn{1}{c}{}&\multicolumn{1}{c}{}&\multicolumn{2}{c}{partition}&\multicolumn{2}{c}{discrete}\\[-.5mm]
\multicolumn{1}{c}{$n$}&\multicolumn{1}{c}{GE}&\multicolumn{1}{c}{KD}&\multicolumn{1}{c}{OU}&
\multicolumn{1}{c}{ED}&\multicolumn{1}{c}{EO}&\multicolumn{1}{c}{non-i.}&\multicolumn{1}{c}{ind.}&\multicolumn{1}{c}{$|X|>1$}&\multicolumn{1}{c}{$|X|=1$}\\\Xhline{2\arrayrulewidth}
\multirow{2}{*}{$1$}&
\begin{minipage}[t]{82pt}\raggedright$42$,\bksp$44$,\bksp$62$,\bksp$63$,\bksp$65$,\bksp$67$\bksp$(8,10)$\end{minipage}&
\begin{minipage}[t]{66pt}\raggedright$31$,\bksp$46$,\bksp$48$,\bksp$59$,\bksp$60$,\bksp$62\mbox{-}68$\bksp$(10,18)$\end{minipage}&
\begin{minipage}[t]{50pt}\raggedright$62$,\bksp$63$,\bksp$65$,\bksp$67$\bksp$(4,8)$\end{minipage}&
\begin{minipage}[t]{52pt}\raggedright$61$,\bksp$65$,\bksp$67$\bksp$(4,4)$\end{minipage}&
\begin{minipage}[t]{26pt}\raggedright$65$,\bksp$67$\bksp$(4,4)$\end{minipage}&
\begin{minipage}[t]{32pt}\raggedright$59$,\bksp$60$,\bksp$68$\bksp$(6,6)$\end{minipage}&
\begin{minipage}[t]{26pt}\raggedright$61$\bksp$(4,4)$\end{minipage}&
\begin{minipage}[t]{34pt}\raggedright$68$\bksp$(2,4)$\end{minipage}&
\begin{minipage}[t]{34pt}\raggedright$(2,2)$\end{minipage}\\\hline
\multirow{3}{*}{$2$}&
\begin{minipage}[t]{82pt}\raggedright$24$,\bksp$25$,\bksp$27$,\bksp$32$,\bksp$33$,\bksp$35$,\bksp$37$,\bksp$40$,\bksp$41$,\bksp$43$,\bksp$49\mbox{-}52$,\bksp$57$,\bksp$64$,\bksp$66$,\bksp$68$\bksp$(10,24)$\end{minipage}&
\begin{minipage}[t]{56pt}\raggedright$13$,\bksp$20$,\bksp$38$,\bksp$40$,\bksp$45$,\bksp$47$,\bksp$49\mbox{-}52$,\bksp$57$,\bksp$58$\bksp$(12,22)$\end{minipage}&
\begin{minipage}[t]{50pt}\raggedright$40$,\bksp$49\mbox{-}52$,\bksp$57$,\bksp$64$,\bksp$66$,\bksp$68$\bksp$(8,16)$\end{minipage}&
\begin{minipage}[t]{52pt}\raggedright$40$,\bksp$46$,\bksp$48$,\bksp$59$,\bksp$60$,\bksp$64$,\bksp$66$,\bksp$68$\bksp$(8,16)$\end{minipage}&
\begin{minipage}[t]{26pt}\raggedright$40$,\bksp$64$,\bksp$66$,\bksp$68$\bksp$(8,16)$\end{minipage}&
\begin{minipage}[t]{26pt}\raggedright$58$\bksp$(6,10)$\end{minipage}&
\begin{minipage}[t]{26pt}\raggedright$59$,\bksp$60$,\bksp$68$\bksp$(6,6)$\end{minipage}&&
\begin{minipage}[t]{34pt}\raggedright$68$\bksp$(2,4)$\end{minipage}\\\hline
\multirow{2}{*}{$3$}&
\begin{minipage}[t]{82pt}\raggedright$1$,\bksp$7$,\bksp$9$,\bksp$14$,\bksp$16$,\bksp$21\mbox{-}23$,\bksp$34$,\bksp$36$,\bksp$39$\bksp$(14$,\bksp$34)$\end{minipage}&
\begin{minipage}[t]{56pt}\raggedright$6$,\bksp$39$\bksp$(14,28)$\end{minipage}&
\begin{minipage}[t]{50pt}\raggedright$39$\bksp$(10,20)$\end{minipage}&
\begin{minipage}[t]{52pt}\raggedright$38$,\bksp$45$,\bksp$47$,\bksp$58$\bksp$(10,22)$\end{minipage}&&&
\begin{minipage}[t]{26pt}\raggedright$58$\bksp$(6,10)$\end{minipage}&&\\\hline
\end{tabular}\label{tab:psi_topsum}
\end{table}
\egroup

Let $X=\{v,w,x,y,z\}$ and $\mathcal{T}$ be the topology on $X$ with base
$\{\{v\}$, $\{v,w\}$, $\{x,y\}$, $\{v,x,y,z\}\}$.
There exists $A\su X$ with $\psi A=35$.
The only other \sn s $<49$ that occur in $X$ are $37,42,44$.
It follows by Table~\ref{tab:subset_types} and Figure~\ref{fig:meet_semilattices} that
$k(X_1)=10$, $k(X_2)=14$, $\kf(X_1)=14$, and $\kf(X_2)=34$.
Hence the upper bounds are sharp in Corollaries~\ref{cor:increased_k-number}
and~\ref{cor:increased_kf-number}.
Since the sum space on two copies of the minimal indiscrete partition space only
increases the space's $k$- and \kfnum s by $2$ (see Table~\ref{tab:psi_topsum}),
the lower bounds are sharp.

Table~\ref{tab:psi_topsum} follows directly from Theorem~\ref{thm:GE_implications} and
Corollary~\ref{cor:topsum_phi_and_psi}.
It is complete since no further \sn s appear in $X_n$ for $n\geq4$.
Propositions~\ref{prop:topsum} and~\ref{prop:kf_topsum} follow directly from
Theorem~\ref{thm:GE_implications}, Lemma~\ref{lem:topsum}, its $\KF$ analogue, and
Table~\ref{tab:psi_topsum}.
The number of copies in each sum is sharp by Table~\ref{tab:psi_topsum}.

\begin{prop}\label{prop:kf_topsum}
Let \mth{$\XT$} have GE monoid \mth{$\KF$.}

\noindent\hangindent=20pt
\nit{\makebox[16pt]{\hfill(i)}}\phantom{t}If \mth{$\Kf(\XT)\in\{4,16\}$} then
\mth{$X_2$} is completely full with GE monoid \mth{$\KF$.}

\noindent\hangindent=20pt
\nit{\makebox[16pt]{\hfill(ii)}}\phantom{t}If \mth{$\Kf(\XT)\in\{10,20,22,28,34\}$} then
\mth{$X_3$} is completely full with GE monoid \mth{$\KF$.}
\end{prop}

\bgroup
\vspace{-2pt}
\begin{table}[!ht]
\caption{Intersections of $\phi$- and \sn s $\leq30$.\\\smallskip\small$\phi$-numbers are
below the diagonal, $\psi$-numbers above it.}
\centering
\renewcommand{\arraystretch}{1}
\renewcommand{\tabcolsep}{1.5pt}
\footnotesize
\begin{tabular}{c!{\vrule width1pt}cacacacacacacacacacacacacacaca!{\vrule width1pt}c}\multicolumn{1}{c!{\vrule width1pt}}{}&
\p{1}&\p{2}&\p{3}&\p{4}&\p{5}&\p{6}&\p{7}&\p{8}&\p{9}&\p{10}&\p{11}&\p{12}&\p{13}&\p{14}&\p{15}&\p{16}&\p{17}&\p{18}&\p{19}&\p{20}&\p{21}&\p{22}&\p{23}&\p{24}&\p{25}&\p{26}&\p{27}&\p{28}&\p{29}&\multicolumn{1}{a!{\vrule width1pt}}{\p{30}}&\emph{$\psi$}\\\Xhline{1pt}
\p{1}&\makebox[9pt]{\hfill\p{\mathbf{1}}\hfill}&\p{1}&\p{1}&\p{1}&\p{1}&\p{1}&\p{1}&\p{1}&\p{1}&\p{1}&\p{1}&\p{1}&\p{1}&\p{1}&\p{1}&\p{1}&\p{1}&\p{1}&\p{1}&\p{1}&\p{1}&\p{1}&\p{1}&\p{1}&\p{1}&\p{1}&\p{1}&\p{1}&\p{1}&\p{1}&\p{1}\\
\p{2}&\p{1}&\makebox[9pt]{\hfill\p{\mathbf{2}}\hfill}&\p{1}&\p{1}&\p{2}&\p{2}&\p{1}&\p{1}&\p{1}&\p{1}&\p{2}&\p{2}&\p{2}&\p{1}&\p{1}&\p{1}&\p{1}&\p{2}&\p{2}&\p{2}&\p{1}&\p{1}&\p{1}&\p{1}&\p{1}&\p{2}&\p{1}&\p{1}&\p{1}&\p{2}&\p{2}\\\hline
\p{3}&\p{1}&\p{1}&\makebox[9pt]{\hfill\p{\mathbf{3}}\hfill}&\p{1}&\p{3}&\p{1}&\p{1}&\p{3}&\p{1}&\p{1}&\p{1}&\p{3}&\p{1}&\p{1}&\p{1}&\p{1}&\p{3}&\p{1}&\p{3}&\p{1}&\p{1}&\p{3}&\p{1}&\p{1}&\p{1}&\p{1}&\p{1}&\p{3}&\p{1}&\p{3}&\p{3}\\
\p{4}&\p{1}&\p{1}&\p{1}&\makebox[9pt]{\hfill\p{\mathbf{4}}\hfill}&\p{4}&\p{1}&\p{1}&\p{1}&\p{1}&\p{4}&\p{1}&\p{4}&\p{1}&\p{1}&\p{4}&\p{1}&\p{1}&\p{1}&\p{4}&\p{1}&\p{4}&\p{1}&\p{1}&\p{1}&\p{1}&\p{1}&\p{1}&\p{1}&\p{4}&\p{4}&\p{4}\\\hline
\p{5}&\p{1}&\p{1}&\p{1}&\p{1}&\makebox[9pt]{\hfill\p{\mathbf{5}}\hfill}&\p{2}&\p{1}&\p{3}&\p{1}&\p{4}&\p{2}&\p{5}&\p{2}&\p{1}&\p{4}&\p{1}&\p{3}&\p{2}&\p{5}&\p{2}&\p{4}&\p{3}&\p{1}&\p{1}&\p{1}&\p{2}&\p{1}&\p{3}&\p{4}&\p{5}&\p{5}\\
\p{6}&\p{1}&\p{2}&\p{3}&\p{1}&\p{1}&\makebox[9pt]{\hfill\p{\mathbf{6}}\hfill}&\p{1}&\p{1}&\p{1}&\p{1}&\p{2}&\p{2}&\p{6}&\p{1}&\p{1}&\p{1}&\p{1}&\p{2}&\p{2}&\p{6}&\p{1}&\p{1}&\p{1}&\p{1}&\p{1}&\p{2}&\p{1}&\p{1}&\p{1}&\p{2}&\p{6}\\\hline
\p{7}&\p{1}&\p{1}&\p{3}&\p{4}&\p{1}&\p{3}&\makebox[9pt]{\hfill\p{\mathbf{7}}\hfill}&\p{7}&\p{7}&\p{7}&\p{7}&\p{7}&\p{7}&\p{1}&\p{1}&\p{1}&\p{1}&\p{1}&\p{1}&\p{1}&\p{1}&\p{1}&\p{7}&\p{7}&\p{7}&\p{7}&\p{7}&\p{7}&\p{7}&\p{7}&\p{7}\\
\p{8}&\p{1}&\p{2}&\p{1}&\p{1}&\p{5}&\p{2}&\p{1}&\makebox[9pt]{\hfill\p{\mathbf{8}}\hfill}&\p{7}&\p{7}&\p{7}&\p{8}&\p{7}&\p{1}&\p{1}&\p{1}&\p{3}&\p{1}&\p{3}&\p{1}&\p{1}&\p{3}&\p{7}&\p{7}&\p{7}&\p{7}&\p{7}&\p{8}&\p{7}&\p{8}&\p{8}\\\hline
\p{9}&\p{1}&\p{2}&\p{1}&\p{4}&\p{1}&\p{2}&\p{4}&\p{2}&\makebox[9pt]{\hfill\p{\mathbf{9}}\hfill}&\p{9}&\p{9}&\p{9}&\p{9}&\p{1}&\p{1}&\p{1}&\p{1}&\p{1}&\p{1}&\p{1}&\p{1}&\p{1}&\p{7}&\p{7}&\p{9}&\p{9}&\p{7}&\p{7}&\p{9}&\p{9}&\p{9}\\
\p{10}&\p{1}&\p{1}&\p{3}&\p{1}&\p{5}&\p{3}&\p{3}&\p{5}&\p{1}&\makebox[9pt]{\hspace{-.8pt}\p{\mathbf{10}}}&\p{9}&\p{10}&\p{9}&\p{1}&\p{4}&\p{1}&\p{1}&\p{1}&\p{4}&\p{1}&\p{4}&\p{1}&\p{7}&\p{7}&\p{9}&\p{9}&\p{7}&\p{7}&\p{10}&\p{10}&\p{10}\\\hline
\p{11}&\p{1}&\p{1}&\p{1}&\p{4}&\p{5}&\p{1}&\p{4}&\p{5}&\p{4}&\p{5}&\makebox[9pt]{\hspace{-.5pt}\p{\mathbf{11}}}&\p{11}&\p{11}&\p{1}&\p{1}&\p{1}&\p{1}&\p{2}&\p{2}&\p{2}&\p{1}&\p{1}&\p{7}&\p{7}&\p{9}&\p{11}&\p{7}&\p{7}&\p{9}&\p{11}&\p{11}\\
\p{12}&\p{1}&\p{1}&\p{1}&\p{1}&\p{1}&\p{1}&\p{1}&\p{1}&\p{1}&\p{1}&\p{1}&\makebox[9pt]{\hspace{-1pt}\p{\mathbf{12}}}&\p{11}&\p{1}&\p{4}&\p{1}&\p{3}&\p{2}&\p{5}&\p{2}&\p{4}&\p{3}&\p{7}&\p{7}&\p{9}&\p{11}&\p{7}&\p{8}&\p{10}&\p{12}&\p{12}\\\hline
\p{13}&\p{1}&\p{1}&\p{1}&\p{4}&\p{5}&\p{1}&\p{4}&\p{5}&\p{4}&\p{5}&\p{11}&\p{12}&\makebox[9pt]{\hspace{-.8pt}\p{\mathbf{13}}}&\p{1}&\p{1}&\p{1}&\p{1}&\p{2}&\p{2}&\p{6}&\p{1}&\p{1}&\p{7}&\p{7}&\p{9}&\p{11}&\p{7}&\p{7}&\p{9}&\p{11}&\p{13}\\
\p{14}&\p{1}&\p{2}&\p{3}&\p{4}&\p{1}&\p{6}&\p{7}&\p{2}&\p{9}&\p{3}&\p{4}&\p{1}&\p{4}&\makebox[9pt]{\hspace{-1pt}\p{\mathbf{14}}}&\p{14}&\p{14}&\p{14}&\p{14}&\p{14}&\p{14}&\p{1}&\p{1}&\p{14}&\p{14}&\p{14}&\p{14}&\p{14}&\p{14}&\p{14}&\p{14}&\p{14}\\\hline
\p{15}&\p{1}&\p{2}&\p{3}&\p{1}&\p{5}&\p{6}&\p{3}&\p{8}&\p{2}&\p{10}&\p{5}&\p{1}&\p{5}&\p{6}&\makebox[9pt]{\hspace{-1pt}\p{\mathbf{15}}}&\p{14}&\p{14}&\p{14}&\p{15}&\p{14}&\p{4}&\p{1}&\p{14}&\p{14}&\p{14}&\p{14}&\p{14}&\p{14}&\p{15}&\p{15}&\p{15}\\
\p{16}&\p{1}&\p{2}&\p{1}&\p{4}&\p{5}&\p{2}&\p{4}&\p{8}&\p{9}&\p{5}&\p{11}&\p{1}&\p{11}&\p{9}&\p{8}&\makebox[9pt]{\hspace{-1pt}\p{\mathbf{16}}}&\p{16}&\p{16}&\p{16}&\p{16}&\p{1}&\p{1}&\p{14}&\p{16}&\p{14}&\p{16}&\p{14}&\p{16}&\p{14}&\p{16}&\p{16}\\\hline
\p{17}&\p{1}&\p{1}&\p{3}&\p{4}&\p{5}&\p{3}&\p{7}&\p{5}&\p{4}&\p{10}&\p{11}&\p{1}&\p{11}&\p{7}&\p{10}&\p{11}&\makebox[9pt]{\hspace{-.8pt}\p{\mathbf{17}}}&\p{16}&\p{17}&\p{16}&\p{1}&\p{3}&\p{14}&\p{16}&\p{14}&\p{16}&\p{14}&\p{17}&\p{14}&\p{17}&\p{17}\\
\p{18}&\p{1}&\p{2}&\p{1}&\p{1}&\p{1}&\p{2}&\p{1}&\p{2}&\p{2}&\p{1}&\p{1}&\p{12}&\p{12}&\p{2}&\p{2}&\p{2}&\p{1}&\makebox[9pt]{\hspace{-1pt}\p{\mathbf{18}}}&\p{18}&\p{18}&\p{1}&\p{1}&\p{14}&\p{16}&\p{14}&\p{18}&\p{14}&\p{16}&\p{14}&\p{18}&\p{18}\\\hline
\p{19}&\p{1}&\p{1}&\p{3}&\p{1}&\p{1}&\p{3}&\p{3}&\p{1}&\p{1}&\p{3}&\p{1}&\p{12}&\p{12}&\p{3}&\p{3}&\p{1}&\p{3}&\p{12}&\makebox[9pt]{\hspace{-.9pt}\p{\mathbf{19}}}&\p{18}&\p{4}&\p{3}&\p{14}&\p{16}&\p{14}&\p{18}&\p{14}&\p{17}&\p{15}&\p{19}&\p{19}\\
\p{20}&\p{1}&\p{2}&\p{1}&\p{1}&\p{1}&\p{2}&\p{1}&\p{2}&\p{2}&\p{1}&\p{1}&\p{12}&\p{12}&\p{2}&\p{2}&\p{2}&\p{1}&\p{18}&\p{12}&\makebox[9pt]{\hspace{-.6pt}\p{\mathbf{20}}}&\p{1}&\p{1}&\p{14}&\p{16}&\p{14}&\p{18}&\p{14}&\p{16}&\p{14}&\p{18}&\p{20}\\\hline
\p{21}&\p{1}&\p{1}&\p{3}&\p{1}&\p{1}&\p{3}&\p{3}&\p{1}&\p{1}&\p{3}&\p{1}&\p{12}&\p{12}&\p{3}&\p{3}&\p{1}&\p{3}&\p{12}&\p{19}&\p{12}&\makebox[9pt]{\hspace{-.3pt}\p{\mathbf{21}}}&\p{1}&\p{1}&\p{1}&\p{1}&\p{1}&\p{1}&\p{1}&\p{4}&\p{4}&\p{21}\\
\p{22}&\p{1}&\p{2}&\p{1}&\p{4}&\p{5}&\p{2}&\p{4}&\p{8}&\p{9}&\p{5}&\p{11}&\p{12}&\p{13}&\p{9}&\p{8}&\p{16}&\p{11}&\p{18}&\p{12}&\p{18}&\p{12}&\makebox[9pt]{\hspace{-.6pt}\p{\mathbf{22}}}&\p{1}&\p{1}&\p{1}&\p{1}&\p{1}&\p{3}&\p{1}&\p{3}&\p{22}\\\hline
\p{23}&\p{1}&\p{1}&\p{3}&\p{4}&\p{5}&\p{3}&\p{7}&\p{5}&\p{4}&\p{10}&\p{11}&\p{12}&\p{13}&\p{7}&\p{10}&\p{11}&\p{17}&\p{12}&\p{19}&\p{12}&\p{19}&\p{13}&\makebox[9pt]{\hspace{-.6pt}\p{\mathbf{23}}}&\p{23}&\p{23}&\p{23}&\p{23}&\p{23}&\p{23}&\p{23}&\p{23}\\
\p{24}&\p{1}&\p{2}&\p{3}&\p{1}&\p{1}&\p{6}&\p{3}&\p{2}&\p{2}&\p{3}&\p{1}&\p{12}&\p{12}&\p{6}&\p{6}&\p{2}&\p{3}&\p{18}&\p{19}&\p{18}&\p{19}&\p{18}&\p{19}&\makebox[9pt]{\hspace{-.8pt}\p{\mathbf{24}}}&\p{23}&\p{24}&\p{23}&\p{24}&\p{23}&\p{24}&\p{24}\\\hline
\p{25}&\p{1}&\p{2}&\p{3}&\p{4}&\p{5}&\p{6}&\p{7}&\p{8}&\p{9}&\p{10}&\p{11}&\p{1}&\p{11}&\p{14}&\p{15}&\p{16}&\p{17}&\p{2}&\p{3}&\p{2}&\p{3}&\p{16}&\p{17}&\p{6}&\makebox[9pt]{\hspace{-.6pt}\p{\mathbf{25}}}&\p{25}&\p{23}&\p{23}&\p{25}&\p{25}&\p{25}\\
\p{26}&\p{1}&\p{2}&\p{3}&\p{1}&\p{1}&\p{6}&\p{3}&\p{2}&\p{2}&\p{3}&\p{1}&\p{12}&\p{12}&\p{6}&\p{6}&\p{2}&\p{3}&\p{18}&\p{19}&\p{18}&\p{21}&\p{18}&\p{19}&\p{24}&\p{6}&\makebox[9pt]{\hspace{-.8pt}\p{\mathbf{26}}}&\p{23}&\p{24}&\p{25}&\p{26}&\p{26}\\\hline
\p{27}&\p{1}&\p{2}&\p{3}&\p{1}&\p{1}&\p{6}&\p{3}&\p{2}&\p{2}&\p{3}&\p{1}&\p{12}&\p{12}&\p{6}&\p{6}&\p{2}&\p{3}&\p{18}&\p{19}&\p{20}&\p{19}&\p{18}&\p{19}&\p{24}&\p{6}&\p{24}&\makebox[9pt]{\hspace{-.8pt}\p{\mathbf{27}}}&\p{27}&\p{27}&\p{27}&\p{27}\\
\p{28}&\p{1}&\p{2}&\p{1}&\p{4}&\p{5}&\p{2}&\p{4}&\p{8}&\p{9}&\p{5}&\p{11}&\p{12}&\p{13}&\p{9}&\p{8}&\p{16}&\p{11}&\p{18}&\p{12}&\p{20}&\p{12}&\p{22}&\p{13}&\p{18}&\p{16}&\p{18}&\p{20}&\makebox[9pt]{\hspace{-.8pt}\p{\mathbf{28}}}&\p{27}&\p{28}&\p{28}\\\hline
\p{29}&\p{1}&\p{1}&\p{3}&\p{4}&\p{5}&\p{3}&\p{7}&\p{5}&\p{4}&\p{10}&\p{11}&\p{12}&\p{13}&\p{7}&\p{10}&\p{11}&\p{17}&\p{12}&\p{19}&\p{12}&\p{21}&\p{13}&\p{23}&\p{19}&\p{17}&\p{21}&\p{19}&\p{13}&\makebox[9pt]{\hspace{-.8pt}\p{\mathbf{29}}}&\p{29}&\p{29}\\
\p{30}&\p{1}&\p{2}&\p{3}&\p{4}&\p{5}&\p{6}&\p{7}&\p{8}&\p{9}&\p{10}&\p{11}&\p{12}&\p{13}&\p{14}&\p{15}&\p{16}&\p{17}&\p{18}&\p{19}&\p{20}&\p{21}&\p{22}&\p{23}&\p{24}&\p{25}&\p{26}&\p{27}&\p{28}&\p{29}&\makebox[9pt]{\hspace{-.8pt}\p{\mathbf{30}}}&\p{30}\\\Xhline{1pt}
\multicolumn{1}{c!{\vrule width1pt}}{$\phi$}&
\p{1}&\p{2}&\p{3}&\p{4}&\p{5}&\p{6}&\p{7}&\p{8}&\p{9}&\p{10}&\p{11}&\p{12}&\p{13}&\p{14}&\p{15}&\p{16}&\p{17}&\p{18}&\p{19}&\p{20}&\p{21}&\p{22}&\p{23}&\p{24}&\p{25}&\p{26}&\p{27}&\p{28}&\p{29}&\multicolumn{1}{a!{\vrule width1pt}}{\p{30}}\\
\end{tabular}\label{tab:disjoint_union_phi}
\end{table}
\egroup

Our final result requires a computer to verify.
See the footnote for details.

\begin{prop}\label{prop:all_psi}
There exists a topological space in which all \mth{$70$} \sn s occur\nit. 
\end{prop}

\begin{proof}
Let $X$ be the $11$-point set $\{p,q,r,s,t,u,v,w,x,y,z\}$.
Resolvable topologies $\mathcal{T}_1,\mathcal{T}_2,\mathcal{T}_3,\mathcal{T}_4$ exist on
$X$ with the following property: For each $1\leq n\leq68$ there exists $A_n\su X$ such
that $\psi A_n=n$ in $(X,\mathcal{T}_{m(n)})$ for some $1\leq m(n)\leq4$.\footnote{Bases
for $\mathcal{T}_1,\dots,\mathcal{T}_4$ are, respectively:
$\{\{p,q\}$, $\{r,s\}$, $\{t,u\}$, $\{v,w\}$, $\{p,q,r,s,x\}$, $\{v,w,y\}$,
$\{p,q,t,u,v,w,y,z\}\}$, $\{\{p,q\}$, $\{r,s\}$, $\{t,u\}$, $\{v,w\}$, $\{r,s,x\}$,
$\{v,w,y\}$, $\{t,u,v,w,y,z\}\}$, $\{\{p,q\}$, $\{r,s\}$, $\{t,u\}$, $\{p,q,r,s,t,u,v\}$,
$\{p,q,r,s,w\}$, $\{p,q,x\}$, $\{r,s,y\}$, $\{t,u,z\}\}$, $\{\{p,q\}$, $\{r,s\}$,
$\{t,u\}$, $\{p,q,r,s,v\}$, $\{p,q,t,u,w\}$, $\{p,q,x\}$, $\{r,s,t,u,y\}$, $\{r,s,z\}\}$.
We verified by computer that $\mathcal{T}_j$ generates \sn\ $61$ for $j=1,2,3,4$;
$\mathcal{T}_1$ generates \sn s $5$-$19$ and $24$-$30$ except $6$ and $13$;
$\mathcal{T}_2$ generates \sn s $6$, $13$, $20$-$22$, and $31$-$68$ except $57$;
$\mathcal{T}_3$ generates \sn s $2$-$4$ and $57$; and $\mathcal{T}_4$ generates \sn s $1$
and $23$.
For $j=4,\dots,11$, respectively, the largest number of \sn s satisfied in one space on
$j$ points is $12,17,25,32,38,43,52,59$.
The smallest two-step increase is $11$, thus the smallest cardinality admitting all $70$
is likely to be $13$.
Since there are approximately $16.5$ billion nonhomeomorphic non-$T_0$ spaces on $13$
points, the probability of finding such a space using a home computer is very low at
present (Kuratowski $14$-sets do not occur in finite $T_0$ spaces
\cite{1966_herda_metzler}).
Ruling out cardinality $12$ might be even more challenging.
The smallest cardinality admitting all $30$ \pn s is $10$.
One such $10$-space has base $\{\{q\}$, $\{r\}$, $\{s\}$, $\{t,u\}$, $\{v,w\}$,
$\{q,x\}$, $\{r,s,y\}$, $\{q,v,w,x,z\}\}$.}

For each $n$ such that \sn\ $n$ is a subset of \sn\ $69$, let $U_n$ be the disjoint
union of three copies of $X$ in positions $j\neq m(n)$ with $A_n$ in the $m(n)$th
position.
For all other $n$ except $n=61$ define $U_n$ similarly with $\es$ in place of $X$.
By resolvability there exist $Q_j\su X$ such that $\psi Q_j=61$ in $(X,\mathcal{T}_j)$
for $1\leq j\leq4$.
Let $U_{61}=Q_1\ds Q_2\ds Q_3\ds Q_4$.
Lemma~\ref{lem:topsum} implies $\psi U_n=n$ in $\sum_{j=1}^4(X,\mathcal{T}_j)$ for
$1\leq n\leq68$.
\end{proof}

\section*{Closing Remarks}

Section~4 of GJ deals with the \define{Kuratowski $\Y$-problem}: Does the set
$\{|\YA|$\,$:$\,all $\XT$ and all $A\su X\}$ have a finite supremum and if so what is it?
We generalize this problem by calling any optimization problem that involves a given
\KurZar{collection} $\Y$ of set operators on the power set of a \KurZar{space} $X$
defined in terms of a \KurZar{system} $\mathcal{S}\su2^X$ satisfying certain
properties a \define{Kuratowski\textendash Zarycki} (or \define{KZ}) \define{problem}.
A \KurZar{seed} $A$ and/or \KurZar{family} $\YA$ may also be involved.

Associated with every Kuratowski $\Y$-problem is the KZ~problem that asks to minimize
the cardinality of a \KurZar{space} that maximizes the \KurZar{family}.
GJ showed the answers are $14$ and $9$, respectively, for $\Y=\{b,i,\vee,\wedge\}$ and
$\{ia,\vee,\wedge\}$ (see Propositions~4.3 and~4.4) and cited Moser \cite{1977_moser},
Herda and Metzler \cite{1966_herda_metzler}, and Anusiak and Shum
\cite{1971_anusiak_shum} for proving the answers of $9$, $7$, and $6$ for $\Y=\{b,i,
\vee\}$, $\{a,b\}$, and $\{a,$ generalized closure$\}$, respectively.\footnote{KZ
problems that minimize the space can be solved by a present-day home computer when the
answer is $\leq11$.}

In closure spaces, Soltan \cite{1982_soltan} addressed the above KZ problem for various
$\Y\su\{a,b,i,\f\hspace{1pt}\}$ as well as the one that minimizes the cardinality of a
\KurZar{system} that maximizes the \KurZar{family} (for $\{a,b\}$ the answer is $14$ and
for $\{a,b,\f\hspace{1pt}\}$ it is $24$).
Bowron \cite{2012_bowron_burdick} minimized the cardinality of a \KurZar{seed} that
maximizes the \KurZar{family} for $\Y=\{a,b\}$ in topological spaces (the answer is $3$).

The following conjecture is supported by computer results: \emph{every ED space contains
a nonempty subset $A$ such that $bA=\ifA$}.
If this is true, then every connected ED space contains a set $A$ such that $\ifA=X$.

\FloatBarrier
\begin{table}
\renewcommand{\arraystretch}{.99}
\renewcommand{\tabcolsep}{1.5pt}
\scriptsize
\centering
\begin{tabular}{c!{\vrule width1pt}cacacacacacacacacacacacacacacacacacacaca!{\vrule width1pt}c}\multicolumn{1}{c!{\vrule width1pt}}{$\psi$}&
\p{31}&\p{32}&\p{33}&\p{34}&\p{35}&\p{36}&\p{37}&\p{38}&\p{39}&\p{40}&\p{41}&\p{42}&\p{43}&\p{44}&\p{45}&\p{46}&\p{47}&\p{48}&\p{49}&\p{50}&\p{51}&\p{52}&\p{53}&\p{54}&\p{55}&\p{56}&\p{57}&\p{58}&\p{59}&\p{60}&\p{61}&\p{62}&\p{63}&\p{64}&\p{65}&\p{66}&\p{67}&\p{68}&\p{69}&\multicolumn{1}{a!{\vrule width1pt}}{\p{70}}&\emph{$\psi$}\\\Xhline{1pt}
\p{1}&\p{1}&\p{1}&\p{1}&\p{1}&\p{1}&\p{1}&\p{1}&\p{1}&\p{1}&\p{1}&\p{1}&\p{1}&\p{1}&\p{1}&\p{1}&\p{1}&\p{1}&\p{1}&\p{1}&\p{1}&\p{1}&\p{1}&\p{1}&\p{1}&\p{1}&\p{1}&\p{1}&\p{1}&\p{1}&\p{1}&\p{1}&\p{1}&\p{1}&\p{1}&\p{1}&\p{1}&\p{1}&\p{1}&\p{1}&\p{1}&\p{1}\\
\p{2}&\p{2}&\p{1}&\p{1}&\p{1}&\p{1}&\p{1}&\p{1}&\p{2}&\p{2}&\p{2}&\p{1}&\p{1}&\p{1}&\p{1}&\p{2}&\p{2}&\p{2}&\p{2}&\p{2}&\p{2}&\p{2}&\p{2}&\p{2}&\p{2}&\p{2}&\p{2}&\p{2}&\p{2}&\p{2}&\p{2}&\p{2}&\p{2}&\p{2}&\p{2}&\p{2}&\p{2}&\p{2}&\p{2}&\p{2}&\p{2}&\p{2}\\\hline
\p{3}&\p{1}&\p{1}&\p{3}&\p{1}&\p{1}&\p{3}&\p{3}&\p{3}&\p{1}&\p{3}&\p{1}&\p{1}&\p{3}&\p{3}&\p{3}&\p{3}&\p{3}&\p{3}&\p{1}&\p{1}&\p{1}&\p{1}&\p{3}&\p{3}&\p{3}&\p{3}&\p{1}&\p{3}&\p{3}&\p{3}&\p{3}&\p{1}&\p{1}&\p{3}&\p{3}&\p{3}&\p{3}&\p{3}&\p{3}&\p{3}&\p{3}\\
\p{4}&\p{1}&\p{4}&\p{1}&\p{4}&\p{4}&\p{1}&\p{1}&\p{4}&\p{1}&\p{4}&\p{4}&\p{4}&\p{1}&\p{1}&\p{4}&\p{4}&\p{4}&\p{4}&\p{1}&\p{1}&\p{1}&\p{1}&\p{4}&\p{4}&\p{4}&\p{4}&\p{1}&\p{4}&\p{4}&\p{4}&\p{4}&\p{1}&\p{1}&\p{4}&\p{4}&\p{4}&\p{4}&\p{4}&\p{4}&\p{4}&\p{4}\\\hline
\p{5}&\p{2}&\p{4}&\p{3}&\p{4}&\p{4}&\p{3}&\p{3}&\p{5}&\p{2}&\p{5}&\p{4}&\p{4}&\p{3}&\p{3}&\p{5}&\p{5}&\p{5}&\p{5}&\p{2}&\p{2}&\p{2}&\p{2}&\p{5}&\p{5}&\p{5}&\p{5}&\p{2}&\p{5}&\p{5}&\p{5}&\p{5}&\p{2}&\p{2}&\p{5}&\p{5}&\p{5}&\p{5}&\p{5}&\p{5}&\p{5}&\p{5}\\
\end{tabular}
\vspace{-60pt}\caption{Intersections of \sn s not both $\leq30$.}\label{tab:disjoint_union_psi}
\end{table}
\FloatBarrier

Our original goal was to prove or disprove that no Kuratowski space has \knum\ $6$.
Shortly after finding a proof,\footnote{Our first proof does not involve the
boundary operator.
It can be found at \url{https://mathoverflow.net/questions/325995}.}
the arXiv preprint of Canilang et~al\text{.} \cite{2021_canilang_cohen_graese_seong}
led us to investigate similar questions about $\KF$.

Figure~$11$ gives a graphical summary
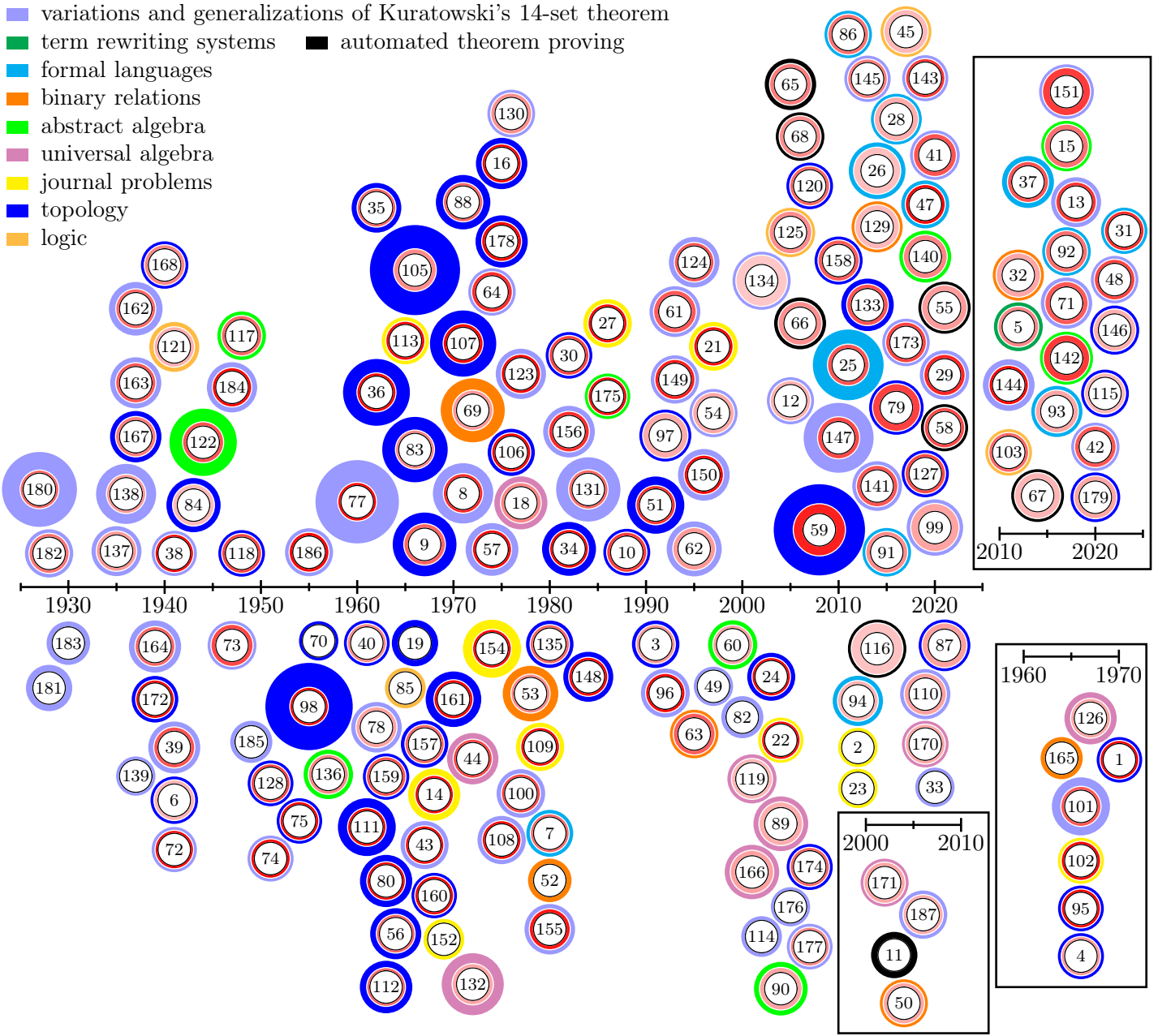
of the first century of literature related to the closure-complement theorem.
In addition to the paper above, the following recent ones are especially noteworthy:
Banakh et~al\text{.} (2018)
\cite{2018_banakh_chervak_martynyuk_pylypovych_ravsky_simkiv},
Berghammer (2017) \cite{2017_berghammer},
Cohen et~al\text{.} (2020) \cite{2020_cohen_johnson_kral_li_soll},
Gupta and Sarma (2017) \cite{2017_gupta_sarma},
Santiago (2019) \cite{2019_santiago}, and
Schwiebert (2017) \cite{2017_schwiebert}.
The following papers cite the use of computers:
\citep{2012_al-hassani_mahesar_coen_sorge_a,
2018_banakh_chervak_martynyuk_pylypovych_ravsky_simkiv,
2017_berghammer,
2021_canilang_cohen_graese_seong,
2013_charlier_domaratzki_harju_shallit,
2021_fuenmayor,
2005_grabowski,
2006_grabowski,
2015_jirasek_jiraskova,
2021_ma,
2007_mccluskey_mcintyre_watson}.
Papers with at least one author from a computer science
department include:
\citep{2012_al-hassani_mahesar_coen_sorge_a,
2017_berghammer,
2011_brzozowski_grant_shallit,
2014_brzozowski_jiraskova_zou,
2016_campeanu_moreira_reis,
2013_charlier_domaratzki_harju_shallit,
2019_dassow,
1980_fischer_paterson,
1972_graham_knuth_motzkin,
2015_jirasek_jiraskova,
2017_jirasek_palmovsky_sebej,
2016_jirasek_sebej,
2012_jiraskova_shallit,
2014_mahesar,
2011_shallit_willard}.
Many papers besides those of Kuratowski~\citep{1922_kuratowski},
Zarycki~\citep{1927_zarycki}, and GE
have been authored or co-authored by graduate students
\citep{2012_al-hassani_mahesar_coen_sorge_a,
1966_herda_metzler,
1943_hewitt,
2014_mahesar,
2005_muhm,
1967_nelson,
2013_plewik_walczynska,
1935_sanders,
2019_santiago,
1963_staley,
1937_stopher_d} and undergraduates
\citep{2018_banakh_chervak_martynyuk_pylypovych_ravsky_simkiv,
2011_brzozowski_grant_shallit,
2021_canilang_cohen_graese_seong,
1962_chapman_a,
1962_chapman_b,
1967_coleman}.
The author of the present paper was a graduate student at the University of Virginia
when Buchman \citep{1986_buchman_ferrer} introduced him to the
closure-complement-boundary problem.
\phantom{\cite{1988_aull}}

Links in the bibliography are colored \textcolor{red}{red} (not free),
\textcolor{purple}{purple} (conditionally free), and \textcolor{blue}{blue} (free).

\section*{Acknowledgments}
Brendan~D.~McKay provided invaluable help with computing topological spaces.
Alexandre~Eremenko kindly passed along firsthand details about the history of GE.
Dave~Renfro sent many useful references.

\advance\baselineskip by -.7pt

\interlinepenalty=10000

\bibliography{kuratowski}

\bigskip
{\small
\textsc{Las Vegas, NV USA}

\textit{Email address}: \texttt{mathematrucker@gmail.com}
}

\end{document}